\documentclass[a4paper,reqno]{amsart}

\usepackage[T1]{fontenc}
\usepackage{amssymb, mathtools, enumitem}
\usepackage[numbers]{natbib}
\usepackage[dvipsnames]{xcolor}
\usepackage{relsize}
\usepackage{todonotes}
\usepackage[colorlinks=true,unicode]{hyperref}

\hypersetup{
linktoc=page,
linkcolor=Violet,
urlcolor=NavyBlue,
citecolor=NavyBlue}
\allowdisplaybreaks

\newtheorem{theorem}{Theorem}[section]
\newtheorem{lemma}[theorem]{Lemma}
\newtheorem*{lemma*}{Lemma}
\newtheorem{proposition}[theorem]{Proposition}
\newtheorem{corollary}[theorem]{Corollary}
\theoremstyle{definition}
\newtheorem{definition}[theorem]{Definition}
\theoremstyle{remark}
\newtheorem{remark}[theorem]{Remark}
\numberwithin{equation}{section}

\newcommand{\itref}[2]{%
  \hyperref[#2]{\ref*{#1}\ref*{#2}}%
}

\def\mb{\mathbb}
\def\mc{\mathcal}
\def\mbf{\mathbf}

\def\mcT{\mathcal{T}}
\def\mcP{\mathcal{P}}
\def\mcQ{\mathcal{Q}}
\def\wh{\widehat}
\def\pp{\perp}
\def\ti{\widetilde}

\def\for{\quad\mbox{for}\quad}
\def\fors{\mbox{for }}
\def\andd{\quad\mbox{and}\quad}
\def\with{\quad\mbox{with}\quad}
\def\where{\quad\mbox{where}\quad}
\def\step[#1,#2]{\textbf{Step #1}:{\it#2}.}

\def\mbN{\mb{N}}
\def\mbR{\mb{R}}
\def\mbZ{\mb{Z}}
\def\mbS{\mb{S}}
\def\mcS{\mc{S}}

\def\sch{Schr{\"o}dinger}
\def\one{\mbf{1}}
\def\les{\lesssim}
\def\gtr{\gtrsim}

\def\pr{\partial}
\def\prr{\pr_r}
\def\prt{\pr_t}
\def\prs{\pr_s}
\def\pry{\pr_y}
\def\vg{\Gamma}

\def\la{\lambda}
\def\vla{\Lambda}
\def\tha{\theta}
\def\vtha{\Theta}

\def\vv{\mathrm{v}}
\def\hv{\wh{\vv}}

\def\w{{\mathrm{w}}}
\def\wp{\w^{\pp}}
\def\hw{\wh{\w}}
\def\wz{\w_{0}}
\def\wzp{\wz^{\pp}}
\def\wzt{\ti{\w}_{0}}
\def\wztp{\ti{\w}_{0}^{\pp}}
\def\wzto{\wzt^{1}}
\def\wztt{\wzt^{2}}
\def\ws{\w_{\ast}}
\def\ra{r_{\alpha}}
\def\rb{r_{\beta}}
\def\rc{r_{\gamma}}
\def\rs{r_{\ast}}
\def\bs{b_{1}^{\ast}}
\def\dbs{\delta(\bs)}
\def\W{W}
\def\Wp{\W^{\pp}}
\def\hW{\wh{\W}}
\def\hp{\wh{\rm p}}
\def\lphi{\vla\phi}
\def\mod{{\rm Mod}}
\def\modt{\ti{\mod}}
\def\psiz{\Psi_{0}}
\def\psizone{\Psi_{0,1}}
\def\psiztwo{\Psi_{0,2}}
\def\psizt{\ti{\Psi}_{0}}
\def\QQ{\mathcal{Q}}
\def\smw{\mbox{\larger[-2]$\mathcal{W}$}}
\def\hs[#1]{\hspace{#1em}}
\def\mhso{\hs[-0.1]}
\def\jbk[#1]{\langle #1\rangle}

\DeclareMathOperator{\diag}{diag}
\DeclareMathOperator{\spann}{span}
\DeclareMathOperator{\id}{Id}
\DeclareMathOperator{\GL}{GL}
\DeclareMathOperator{\SO}{SO}
\DeclareMathOperator{\supp}{supp}
\DeclareMathOperator{\R}{\hs[-.05]R\hs[-.05]}
\DeclareMathOperator{\hj}{\hs[-.05]J\hs[-.05]}
\DeclareMathOperator{\PP}{P\hs[-.05]}
\DeclareMathOperator{\T}{T\hs[-.05]}
\DeclareMathOperator{\ma}{\hs[-.05]\mb{A}\hs[-.05]}
\DeclareMathOperator{\mas}{\hs[-.05]\mb{A}^{\ast}\hs[-.05]}
\DeclareMathOperator{\mham}{\hs[-.02]\mb{H}\hs[-.02]}
\DeclareMathOperator{\mhamp}{\hs[-.02]\vphantom{\mham}\smash{\mb{\mathring{H}}}\hs[-.02]}
\DeclareMathOperator{\mfg}{\mathfrak{g}}
\newcommand{\oa}{A}
\newcommand{\oaz}{A_{0}}
\newcommand{\as}{A^{\ast}}
\newcommand{\asz}{A_{0}^{\ast}}
\newcommand{\ham}{H}
\def\mhaml{\mb{H}_{\la}}

\def\dd{{\rm{d}}}
\def\z{Z}
\def\zt{\bar{Z}}

\def\mcA{\mc{A}}
\def\mcAb{\mc{A}_{b}}
\def\mcAg{\mc{A}_{g}}
\def\mcH{\mc{H}}
\def\mcHb{\mc{H}_{b}}
\def\mcHg{\mc{H}_{g}}
\def\mcD{\mc{D}}
\def\mcN{\mc{N}}

\def\mF{\mc{F}}
\def\mFb{\bar{\mF}}
\def\vr{\mc{V}}
\def\mcJ{\mc{J}}
\def\mora{\mc{M}}
\def\E{\mc{E}}

\def\xm{X}
\def\up{\Upsilon}

\def\eone{e_{1}}
\def\etwo{e_{2}}
\def\ethr{e_{3}}

\def\sd[#1]{_{(#1)}}
\def\lgb{\left|\log b_{1}\right|}
\def\cb{c_{b_{1}}}
\def\vp{\Phi}
\def\x{\Xi}
\def\xx[#1]{\x_{#1}}
\def\phia[#1]{\vp_{#1,0}}
\def\phib[#1]{\vp_{0,#1}}
\def\phiat[#1]{\ti{\vp}_{#1,0}}
\def\phibt[#1]{\ti{\vp}_{0,#1}}
\def\adb{admissible}

\def\homo{homogeneous}
\def\rrr{\mc{R}}
\def\ve{\varepsilon}
\def\si{\ve} 
\def\sib{\bar{\ve}} 
\def\vsig{\Sigma}
\def\tea{E}

\def\Eplus{\biggl(\E_{2L+2}+\frac{b_1^{2L+2}}{\lgb^2}\biggr)}
\def\EplusM{\biggl(\frac{\E_{2L+2}}{\sqrt{\log M}} +C(M)\frac{b_1^{2L+2}}{\lgb^2}\biggr)}

\begin{document}

\title[Quantized slow blowup dynamics]{Quantized slow blowup dynamics for the energy-critical Schr{\"o}dinger map flow}

\author{Jacek Jendrej}
\address{Sorbonne Universit\'e, Institut de Math\'ematiques de Jussieu-Paris Rive Gauche, 4 place Jussieu, 75005 Paris, France
  \& Faculty of Applied Mathematics, AGH University of Krak\'ow, al. Adama Mickiewicza 30, 30-059 Krak\'ow, Poland.}
\email{jendrej@imj-prg.fr}

\author{Jitao Xu}
\address{Morningside Center of Mathematics, Chinese Academy of Sciences, Beijing 100190, P. R. China.}
\email{xu.jitao@amss.ac.cn}

\author{Xu Yuan}
\address{State Key Laboratory of Mathematical Sciences, Academy of Mathematics and Systems Science, Chinese Academy of Sciences, Beijing 100190, P. R. China.}
\email{xu.yuan@amss.ac.cn}

\author{Lifeng Zhao}
\address{School of Mathematical Sciences, University of Science and Technology of China, Hefei, Anhui,
  230026, P. R. China.}
\email{zhaolf@ustc.edu.cn}

\subjclass[2020]{Primary 35B44; Secondary 35Q55, 35B40}

\begin{abstract}
  We study the Cauchy problem for the energy-critical Schr\"{o}dinger map into the 2-sphere:
  \[
    \partial_{t}u = u\times\Delta u,
    \quad (t,x)\in\mathbb{R}^{1+2},
    \quad u(t,x)\in\mathbb{S}^2
  \]
  for equivariant initial data of homotopy degree $k=1$.
  Given any integer $L\in\mathbb{N}_+$, we exhibit a set of initial data arbitrarily close to the ground state harmonic map in the energy norm such that the corresponding solution blows up in finite time with a quantized rate
  \[
    \lambda(t)\sim \frac{\left(T-t\right)^L}{\left|\log(T-t)\right|^{\frac{2L}{2L-1}}}
    \quad\mbox{as}\ \ t\uparrow T.
  \]
  Our proof is inspired by the techniques developed for the energy-critical harmonic map heat flow in a similar context by Rapha{\"e}l and Schweyer~[Anal. PDE 7 (2014), 1713--1805]. However, its adaptation to a Schr\"{o}dinger-type equation requires the introduction of a new energy-Virial-Morawetz Lyapunov functional tailored to quantized slow blowup dynamics.
\end{abstract}

\maketitle

\tableofcontents

\section{Introduction}
\label{sec:intro}
\subsection{Setting of the problem}
\label{sub:setting}
Consider the energy-critical \sch\ map equation,
\begin{equation}
  \label{eq:sm}
  \left\{\begin{aligned}
    & \partial_{t}u = u\times\Delta u, \\
    & u|_{t=0}=u_{0}\in \dot{H}^{1},
  \end{aligned}\right.
  \quad (t,x)\in\mbR^{1+2},
  \quad u(t,x)\in\mbS^2,
\end{equation}
where $\times$ denotes the cross product and $\mbS^{2}\subset\mbR^{3}$ is the 2-sphere. This equation is a special case of the Landau--Lifshitz equation for ferromagnets. It can also be viewed as a geometric analogue of the \sch\ equation.

Let us recall some basic facts on the \sch\ map. For any $\dot{H}^{1}$ solution $u$, \eqref{eq:sm} admits the conserved Dirichlet energy
\begin{equation}
  \label{eq:dirichlet}
  E(u) = \frac12\int|\nabla u(t,x)|^2\dd x=E(u_{0}).
\end{equation}
It also admits the scaling, rotation and translation symmetries: the mixed transformation
\[
  u \mapsto O u\biggl(\frac{t}{\la^2},\frac{x-\bar{x}}{\la}\biggr)
  \with (O,\la,\bar{x})\in\SO(3)\times\mbR_+\times\mbR^2,
\]
maps solutions of \eqref{eq:sm} to solutions. Since the Dirichlet energy is invariant under the scaling transformation, \eqref{eq:sm} is said to be \emph{energy-critical}.

This article focuses on $k$-equivariant maps, which take the form
\[
  u(t,r,\tha) = e^{k\tha\R}
  \begin{pmatrix}
    u_1(t,r) \\ u_2(t,r) \\ u_3(t,r)
  \end{pmatrix}.
\]
Here, $(r,\tha)$ are polar coordinates on $\mbR^2$, and $\R$ is the generator of the horizontal rotation given by
\[
  \R = e_3\times = \begin{pmatrix}
    0 & -1 & 0 \\ 1 & 0 & 0\\ 0 & 0 & 0
  \end{pmatrix},
\]
where $(e_1,e_2,e_3)$ denotes the standard orthonormal basis of $\mbR^{3}$. For $k\neq 0$, we restrict attention to the finite-energy class of maps with pole limits
\[
  u(t,0) = e_3, \quad
  \lim_{r\to+\infty}u(t,r,\theta) = -e_3.
\]
In this class, $k$ marks the homotopy degree:
\[
  k = \frac{1}{4\pi}\int u\cdot(\partial_{x_1}u\times\partial_{x_2}u)\,\dd x.
\]
A well-known example of such maps is the harmonic map $Q_k$, which satisfies
\begin{equation}
  \label{eq:grd-state-eq}
  \Delta Q_k = -|\nabla Q_k|^2Q_k.
\end{equation}
Its explicit expression is
\[
  Q_k(r,\tha) = e^{k\tha\R}
  \begin{pmatrix}
    \sin\phi_k(r) \\ 0 \\ \cos\phi_k(r)
  \end{pmatrix}
  \with
  \phi_k(r)=2\arctan(r^{|k|}).
\]
By the Bogomol'nyi factorization \citep{bogomolnyi_stability_1976}, $Q_k$ minimizes the Dirichlet energy \eqref{eq:dirichlet} (up to symmetries) in its homotopy class, with $E(Q_k)=4\pi|k|$. For notational convenience, we use the conventions $Q=Q_1$ and $\phi=\phi_1$ in this article.

\subsection{Related results}
\label{sub:related-results}
Extensive literature has been devoted to the \sch\ map flow. Early foundational works developed moving-frame and gauge formulations for the equation and the associated well-posedness theory \citep{chang_schrodinger_2000,nahmod_schrodinger_2003}. Further local well-posedness results were established in \citep{ding_local_2001,mcgahagan_approximation_2007,sulem_continuous_1986}. For small data in the critical space with target $\mbS^2$, Bejenaru, Ionescu, Kenig, and Tataru \citep{bejenaru_global_2011} proved global well-posedness and scattering. Li \citep{li_global_2021,li_global_2022} then extended these results to general compact K{\"a}hler targets. At the energy-critical level relevant to the present work, namely the $\dot H^1$ scaling in two spatial dimensions, Smith \citep{smith_geometric_2012} developed the caloric-gauge renormalization below the ground-state energy, while Dodson and Smith \citep{dodson_controlling_2015} identified a critical space-time norm controlling continuation of the flow. In the equivariant setting, Bejenaru, Ionescu, Kenig, and Tataru \citep{bejenaru_equivariant_2013} proved that solutions with energy below $E(Q)=4\pi$ are global and scatter.

The long-time behavior near the ground states $Q_k$ displays a sharp dependence on the degree $k$. For $|k|\ge 3$, the harmonic maps are stable in the energy space, and in fact asymptotically stable; see \citep{gustafson_schrodinger_2007,gustafson_asymptotic_2008,gustafson_asymptotic_2010}. Consequently, no blowup can occur near $Q_k$ in these equivariance classes. For $k=1$, Bejenaru and Tataru \citep{bejenaru_soliton_2010} showed that $Q$ is stable in a stronger topology $X\subset\dot{H}^{1}$, but unstable in $\dot{H}^{1}$, so that stability alone cannot rule out blowup. Very recently, Bejenaru, Pillai, and Tataru \citep{bejenaru_soliton_2024} settled the remaining 2-equivariant case. They proved that small $\dot{H}^{1}$ perturbations of $Q_2$ give rise to global solutions with dispersive bounds, and that $Q_2$ is stable in a slightly stronger topology. The near-ground-state picture does not preclude more complicated dynamics at higher energy. In particular, Yin \citep{yin_bubble_tower_2026} constructed, for $k\ge 3$, infinite-time bubble towers with an arbitrary number of asymptotically decoupled harmonic maps.

The \sch\ map problem is closely related to the energy-critical wave map equation, where some of the earliest rigorous type II blowup constructions were obtained. Krieger, Schlag, and Tataru \citep{krieger_renormalization_2008} constructed charge-one equivariant solutions concentrating a harmonic map at a continuum of polynomial rates. Rodnianski and Sterbenz \citep{rodnianski_formation_2010} subsequently exhibited a stable resonant concentration mechanism in the equivariance classes $k\ge 4$, with a logarithmically modified self-similar upper bound. Rapha{\"e}l and Rodnianski \citep{raphael_stable_2012} later constructed stable blowup for the critical 1-corotational wave map and equivariant Yang--Mills problems. These works provide an important geometric precedent for the concentration mechanisms considered here.

The instability of $Q$ motivates the study of blowup solutions. This is part of the broader question of near-soliton dynamics discussed in \citep[Section~6, Open Problem~6]{tataru_schrodinger_2013}. In the $1$-equivariant setting, the relevant symmetries of \eqref{eq:sm} are scaling and rotation, so it is natural to seek solutions of the form
\[
  e^{\vtha(t)\R}Q\biggl(\frac{x}{\la(t)}\biggr)+u_0^{\ast}
  \with \|u_0^{\ast}\|_{\dot{H}^1}\ll 1,
\]
where $\la(t)\to0$ near the blowup time. Such finite-time blowup solutions were constructed by Merle, Rapha{\"e}l, and Rodnianski \citep{merle_blowup_2013}. They exhibited a codimension-one set of $1$-equivariant initial data that blow up with the rate
\begin{equation}
  \label{eq:l1-rate}
  \la(t)=C(u_0)(1+o(1))\frac{T-t}{\left|\log(T-t)\right|^2}
  \quad\mbox{as}\quad t\uparrow T,
\end{equation}
where $C(u_0)$ is a constant depending on the initial data and $T\in\mbR_+$ is the blowup time; the phase $\vtha(t)$ converges to a fixed value. In a complementary direction, Perelman \citep{perelman_blow_2014} constructed a solution that blows up backwards in time on $(0,t_0]$, with the polynomial rate
\begin{equation}
  \label{eq:cont-rate}
  \la(t)=t^{\frac12+\nu}
  \quad\mbox{as}\quad t\downarrow0
\end{equation}
for a prescribed constant $\nu>1$ (in the normalization of \citep{perelman_blow_2014}, $\la(t)=t^{-\nu}$ with $\nu>3/2$). The phase exhibits the accelerated shift
\[
  \vtha(t)=\vtha_0\log t
  \quad\mbox{as}\quad t\downarrow0.
\]
Given these two mechanisms, the following question becomes natural:
\begin{gather*}
  \mbox{\emph{Do there exist blowup solutions to~\eqref{eq:sm} with rates}}\\
  \mbox{\emph{other than~\eqref{eq:l1-rate} and~\eqref{eq:cont-rate}?}}
\end{gather*}
The present article gives an affirmative answer: for each positive integer $L$, there exist blowup solutions with the quantized rate~\eqref{th:main-thm-quant-rate} below, which corresponds to \eqref{eq:l1-rate} when $L=1$. Related logarithmically modified concentration dynamics also occur in the self-dual Chern--Simons--Schr{\"o}dinger equation: Kim, Kwon, and Oh \citep{kim_blowup_2024} constructed smooth finite-energy radial solutions blowing up near the ground state and developed a nonlinear covariant conjugation framework adapted to that problem.

Methodologically, the construction in \citep{merle_blowup_2013} is an instance of the \emph{energy method}, in which a blowup profile is modulated by a finite system of parameters driven by an underlying spectral problem. In the mass-critical NLS setting, Bourgain and Wang \citep{bourgain_wang_blowup_1997} constructed special blowup solutions at the pseudoconformal scale, while Merle and Rapha{\"e}l developed the universal log--log blowup dynamics and obtained sharp bounds on the blowup rate \citep{merle_universality_2004,merle_blowup_dynamic_2005,merle_raphael_lower_2006}. These works supplied several of the modulation, spectral, and localized energy ideas that underlie later geometric constructions. The energy method also has roots in the critical gKdV theory developed by Martel and Merle \citep{martel_stability_2002,martel_blow_2014}. It was subsequently adapted to several energy-supercritical equations: the nonlinear \sch\ equation by Merle, Rapha{\"e}l, and Rodnianski \citep{merle_type_2015}, the wave equation by Collot \citep{collot_type_2018}, and the 1-corotational wave maps by Ghoul, Ibrahim, and Nguyen \citep{ghoul_construction_2018}. A more recent breakthrough of this kind is due to Merle, Rapha{\"e}l, Rodnianski, and Szeftel \citep{merle_blow_2022}. Their construction relies on a front mechanism developed in the companion articles \citep{merle_implosion_2022a,merle_implosion_2022} on the implosion of the 3D compressible Euler equations. Analogous mechanisms appear in parabolic settings as well, for instance in the energy-supercritical heat equation \citep{collot_strongly_2020}.

The present work belongs to the \emph{quantized-rate} blowup constructions initiated by Rapha{\"e}l and Schweyer \citep{raphael_quantized_2014}. They extended their stable blowup construction for the corotational energy-critical harmonic heat flow \citep{raphael_stable_2013} to a discrete family of rates indexed by a natural number. Since then, the quantized mechanism has been established in various contexts. We list some representative results. In the energy-critical dispersive setting, Jeong \citep{jeong_quantized_2023} exhibited quantized blowup for the 1-corotational wave maps with the rate
\[
  \la(t)=C(u_0,\dot{u}_0)(1+o(1))
  \frac{(T-t)^\ell}{\left|\log(T-t)\right|^{\ell/(\ell-1)}},
  \quad 2\le\ell\in\mbN_+.
\]
Jeong and Kim \citep{jeong_quantized_2024} obtained quantized blowup for the Calogero--Moser derivative nonlinear \sch\ equation. In that problem, the control of high-order energies relies on the hierarchy of conservation laws furnished by the Lax pair structure. In the parabolic setting, del Pino, Musso, Wei, Q. Zhang, and Y. Zhou \citep{pino_type_2020} constructed type II blowup for the 3D energy-critical heat equation by the inner-outer gluing method, with the quantized rates $\|u\|_{L^\infty}\sim(T-t)^{-k}$, $k\in\mbN_+$. Higher speeds were subsequently obtained for the five-dimensional energy-critical heat equation by Harada \citep{harada_higher_2020}. Quantized dynamics also arise in free-boundary problems: Had{\v{z}}i{\'c} and Rapha{\"e}l \citep{hadzic_melting_2019} treated the 2D radial Stefan problem, while Zhang \citep{zhang_melting_2025} constructed both the stable fundamental melting regime and a sequence of codimension-$k$ excited regimes for the 3D radial Stefan problem. For the $\SO(d)$-equivariant Yang--Mills heat flow in high dimensions, Yi \citep{yi_stable_2025} constructed blowup concentrating a universal stationary profile with quantized polynomial rates. Bensouilah, Duong, and Ghoul \citep{bensouilah_nonself_2025} constructed non-self-similar blowup with a quantized family of polynomial rates.
\subsection{Main result}

We now state the main theorem.

\begin{theorem}\label{th:main-thm}
  For any $L\in\mbN_+$, there exists a codimension-$(2L-1)$ set of smooth $1$-equivariant initial data $u_0\in\dot{H}^{1}$ arbitrarily close to the ground state $Q$ in the energy norm, such that the corresponding solution to~\eqref{eq:sm} blows up in finite time $T=T(u_0)\in\mbR_+$ by bubbling off $Q$ near the blowup time, in the sense that
  \[
    u(t,x) - e^{\vtha(t)\R} Q\biggl(\frac{x}{\la(t)}\biggr) \to u^{\ast}\in\dot{H}^{1}
    \quad\mbox{as}\ \ t\uparrow T,
  \]
  where the geometric parameters $(\la,\vtha)\in C^1([0,T),\mbR_+\times\mbR)$ yield both the quantized blowup rate
  \begin{equation}
    \label{th:main-thm-quant-rate}
    \la(t) = C(u_0)(1+o(1))\frac{(T-t)^L}{\left|\log(T-t)\right|^{\frac{2L}{2L-1}}}
    \quad\mbox{as}\ \ t\uparrow T,
  \end{equation}
  and the convergence of the phase
  \[
    \vtha(t)\to \vtha(u_0)\in\mbR
    \quad\mbox{as}\ \ t\uparrow T.
  \]
\end{theorem}

\begin{remark}[Generalization to quantized-rate regime]
  The blowup solution constructed in~\citep{merle_blowup_2013} corresponds to the case $L=1$ of Theorem~\ref{th:main-thm}. Hence our result generalizes it to a higher-order quantized family of blowup rates. The main novelty of our construction is twofold: the new Virial functional and the higher codimension of the initial data.

\end{remark}

\begin{remark}[Codimension of the initial data]
  We use the term \emph{codimension-$(2L-1)$ set of initial data} in the following parametrized sense. With the initial normalization $\la(0)=1$, $\vtha(0)=0$, for any smooth 1-equivariant perturbation $\vv_0\in\dot{H}^1$ satisfying the initial smallness conditions of Subsection~\ref{sub:mod-boot}, there exist parameters $(a_{\ell})_{\ell=1}^L$, $(b_{\ell})_{\ell=1}^L$ such that the initial data
  \[
    u_0=Q+\vv_0 + (e_r,e_\tau,Q) \Biggl( \sum_{\ell=1}^{L}a_{\ell}(\vv_0)\phiat[\ell] + \sum_{\ell=1}^{L}b_{\ell}(\vv_0)\phibt[\ell] + \text{h.o.t.} \Biggr)
  \]
  leads to blowup in the regime given by Theorem~\ref{th:main-thm}. Here,
  \(\phiat[\ell], \phibt[\ell]\) are certain directions (the localized resonances; see Corollary~\ref{co:resona} and \eqref{eq:local-phi}), and the $h.o.t.$ represents the higher order terms. The choice of $(a_1,\dots,a_L,b_1,\dots,b_L)$ can be traced back to that of parameters $(U,V)$. The linearized system for \((U,V)\) has one stable direction and \(2L-1\) unstable directions. We specify the initial data in the unstable directions while leaving the stable parameter free, which yields the parametrized notion of codimension; see~Proposition~\ref{pro:boot} and Remark~\ref{re:fixed-scaling}.
\end{remark}

\medskip
\noindent\textbf{Comments on the proof.}
The overall strategy is the energy method, as in \citep{merle_blowup_2013, raphael_quantized_2014, merle_type_2015, collot_type_2018, ghoul_construction_2018}. One seeks the solution in the form of an explicit blowup profile plus a corrector (the radiation $\w$ in our case). The existence of the corrector is proved by a bootstrap on adapted energies that control its $\dot{H}^{1}\cap\dot{H}^{2L+2}$ regularity. The main points of the proof are the following.

(i) \emph{Approximate solution}.
Since $L\in\mbN_+$ is arbitrary, we need a blowup profile accurate to the order prescribed by $L$. To this end, we review the notion of \adb\ functions from \citep{raphael_quantized_2014} and extend this notion into a systematic framework of function classes. This enables us to construct a high-order approximate solution and efficiently estimate the nonlinear effects of~\eqref{eq:sm}.

(ii) \emph{Energy-Virial-Morawetz functional}.
When $L>2$, the direct estimate \citep[Lemma~4.7]{merle_blowup_2013} for the leading quadratic term in the energy estimate is no longer sufficient to close the bootstrap. This failure cannot be remedied by minor adjustments within the trapped regime. To overcome this, we introduce the Virial and Morawetz functionals to modify the energy. This is inspired by the works of the first author~\citep{jendrej_construction_2017a, jendrej_construction_2019}. The critical novelty is the weight $q(y)$ in the definition of the Virial correction. It exhibits an extra logarithmic decay: $(\T\pry)^2 q \le c \jbk[y]^{-2}\jbk[\log y]^{-2}$ for $y\ge1$, where the linear operator $\T$ is given by \eqref{eq:t-p} and $(\T\pry)^2\approx \Delta_{\rm rad}^2$. This decay is required by the coercivity estimate for the conjugate operator $\ti{\ham}$ (see Lemma~\ref{le:coer-ti-ham}).

(iii) \emph{Integrability near the origin}.
Because of the complexity of geometric flows, the error terms involve the iterated operator $(\hj\mham)^{k}$ for $k\le L+1$, which can generate singular behavior at $y=0$ and render the resulting integrals non-integrable. To restore integrability, we introduce the so-called Dore series and track their order under these operators; see Definition~\ref{def:dore} and Lemma~\ref{le:preserv}.

\medskip
\noindent\textbf{Notation.} We collect here the notation used throughout the article.
\noindent
Given a vector $\w\in\mbR^3$, we denote by $\wp$ its projection onto the plane spanned by $e_1$ and $e_2$:
\[
  \wp = -\R^2\w = \w^{(1)}e_1 +\w^{(2)}e_2.
\]
\noindent
Given a function $f(y)$ of the radial variable, we denote by $\vla$ the generator of dilations:
\[
  \vla f = y\pry f.
\]
\noindent
Given a parameter $0<b_1\ll 1$, we define the scales\footnotemark
\[
  S_0 = \frac{1}{\sqrt{b_1}},
  \quad S_1 = \frac{\lgb}{\sqrt{b_1}}.
\]
Given infinitesimal $b_1$, we denote by $\delta(b_1)$ a generic positive quantity satisfying $\delta(b_1)\to 0$ as $b_1\to 0$. \footnotetext{In~\cite{merle_blowup_2013}, the scales $S_0, S_1$ are denoted by $B_0, B_1$.}

\noindent
Given positive numbers $a$ and $b$, we write $a\simeq b$ if $a = (1+o(1))b$, and $a\sim b$ if $C_1b \le a \le C_2b$ for some constants $C_2>C_1>0$.

\noindent
Given a real number $y$, we write $\jbk[y] = (1+y^2)^{1/2}$ for the Japanese bracket.

\noindent
We denote by $\chi$ a fixed radial smooth cutoff function satisfying
\[
  \chi(y) = 1 \quad \mbox{for}\;\; y\le 1,
  \qquad
  \chi(y) = 0 \quad \mbox{for}\;\; y\ge 2,
\]
and write $\chi_M = \chi\bigl(\frac{\cdot}{M}\bigr)$ for any $M>0$.

\noindent
We denote by $\one_A$ the indicator function of a set $A$. If the set $A$ is defined by an inequality, for instance $A = \{y\le 10\}$, then we also use the abbreviation $\one_{y\le 10}$.

\noindent
Finally, we write $\delta_{j,k}$ for the Kronecker delta.

\subsection{Outline of the proof}
\label{sub:outline}
We now sketch the proof.

\step[1, Normalization] We seek a blowup solution of the form
\[
  u = e^{\vtha\R}(Q+\hv) \Bigl(t,\frac{r}{\la}\Bigr),
\]
where the phase $\vtha=\vtha(t)$ and the scale $\la=\la(t)$ are modulation parameters to be determined, and $\hv$ is a corrector to be determined. The solution satisfies the equation~\eqref{eq:sm} and remains $\mbS^2$-valued\footnote{This is a geometric property of the \sch\ map flow; see Lemma~\ref{le:geo-prop}.}, and hence
\[
  u(t,x)\in\mbS^2
  \quad\Longrightarrow\quad
  |Q +\hv| = 1, \quad
  \forall (t,r)\in\mbR\times\mbR_+.
\]
Since $1$-equivariance is preserved by the flow, we use the Frenet basis $(e_r, e_\tau, Q)$, whose three vectors are $1$-equivariant, and express the corrector $\hv$ in this basis, with coefficient vector $\hw\in\mbR^3$. The latter satisfies
\begin{equation}
  \label{eq:constra}
  |e_3 +\hw| =1.
\end{equation}
In the self-similar variables
\begin{equation}
  \label{eq:ss-trans}
  \frac{\dd s}{\dd t} = \frac{1}{\la^2},
  \quad y = \frac{r}{\la},
\end{equation}
the equation \eqref{eq:sm} reduces to the normalized equation
\begin{equation}
  \label{eq:vec-eq-outline}
  \prs\hw -\frac{\la_s}{\la}\vla\hw +\vtha_s\z\R\hw +\hj[\hw](\mham\hw +\hp\lphi) =0.
\end{equation}
Here $\z(y)=(1-y^2)/(1+y^2)$, $\hj[\hw]\approx\R$ is the complex structure from \eqref{eq:hj}, and $\hp$ denotes a vector depending on $\vtha_s$ and $\la_s$. The operator $\mham$, defined in \eqref{eq:mham}, is essentially the diagonal operator
\[
  \mham \approx
  \begin{pmatrix}
    \ham & 0    & 0       \\
    0    & \ham & 0       \\
    0    & 0    & -\Delta
  \end{pmatrix}
  \with
  \ham = -\Delta +\frac{y^4-6y^2+1}{y^2(1+y^2)^2}.
\]
Thus the leading part of the dynamics is governed by the self-adjoint \sch\ operator $\ham$. This is the normalization adopted in \citep[Section~2]{merle_blowup_2013}.

\step[2, Construction of an approximate solution] To solve \eqref{eq:vec-eq-outline}, we seek an approximate solution in the spirit of separation of variables. We introduce the modulation parameters
\[
  b_1 = -\frac{\la_s}{\la}
  \andd
  a_1 = -\vtha_s,\footnotemark
\]
and make the ansatz
\footnotetext{This is only a heuristic approximation; see Lemma~\ref{le:mod-eq} for the rigorous statement.}
\[
  \hw = \wz :=a_1 f_{1,0}(y) +b_1 f_{0,1}(y) +o(|a_1|+|b_1|)\,f_{h.o.t.}(y).
\]
At the formal level, the derivatives of $a_1$ and $b_1$ contribute only lower-order terms, in view of the expected asymptotics $\la\downarrow 0$ and $\vtha\to \vtha_{\infty}\in\mbR$ as $s\to +\infty$. The leading-order part of \eqref{eq:vec-eq-outline} is therefore
\[
  \mhamp\wz +\hp\lphi = a_1(\mhamp f_{1,0} -\lphi e_1) +b_1(\mhamp f_{0,1} -\lphi e_2).
\]
Requiring this expression to vanish, we choose
\[
  f_{1,0} = \phia[1], \quad f_{0,1} = \phib[1],
\]
where $\phia[1]$ and $\phib[1]$ are the first iterated resonances of $\mhamp$; see Corollary~\ref{co:resona}. Higher blowup rates $\la_L(t)$, with $L\ge 2$, are obtained by including the higher-order resonances $\phia[k]$ and $\phib[k]$, and hence
\[
  \wz = \sum_{k=1}^{L}(a_k \phia[k] +b_k \phib[k]) +o(|a_1|+|b_1|)\,f_{h.o.t.}(y).
\]
All of these profiles are \adb\ functions; see Definition~\ref{def:adb}. The actions of $\ham$ and $\ham^{-1}$ on \adb\ functions are given in Lemma~\ref{le:ham-act-adb}.

The higher-order part of $\wz$ is chosen to control the quasilinear error $\hj[\hw]\mham\wz$ to the required order while enforcing the constraint \eqref{eq:constra} with high precision. To organize the construction and track nonlinear interactions, we introduce suitable classes of admissible functions (and also homogeneous functions) that encode both the parameter dependence and the spatial behavior of the profiles. Within this framework, we solve for the planar correction $\x$ and the radial correction $\gamma_0$; see Subsection~\ref{sub:homo-func} and Lemma~\ref{le:app-sol}.

\step[3, Construction of exact solutions] We add a radiation corrector and seek solutions to \eqref{eq:vec-eq-outline} of the form
\[
  \hw = \wzt +\w,
\]
where $\wzt$ is a localized version of $\wz$. The existence of $\w$ is proved using an energy method combined with a bootstrap argument; see Proposition~\ref{pro:boot}. We introduce adapted energies $\E_1$ and $\E_{2k}$, $1\le k\le L+1$, which are comparable to the squared $\dot{H}^1$, $\dot{H}^{2k}$ norms of $\wp$, and propagate the bootstrap bounds on the rescaled time interval $[s_0,+\infty)$. This yields $\w\in C([s_0,+\infty),\dot{H}^1)$. The key point is a Gronwall-type estimate for the top energy $\E_{2L+2}$; see Lemma~\ref{le:eng}. Its proof requires controlling three types of terms: the leading quadratic term, the high-order commutator term, and the remaining errors.

A Morawetz-type estimate analogous to those in \citep{merle_blowup_2013,xu_blowup_2025} cancels the commutator term. To control the leading quadratic term, we introduce the virial functional
\[
  \vr(t) = b_1 \int\ma_{\la}\W_{2L}^0\cdot \prr (q_{\la})\W_{2L+2}^0,
\]
where $\ma_{\la}\approx -\prr +\z_{\la}/r$ acts on the first two components and $\W_{2k}(t,r)$ denotes the renormalized radiation. The novelty lies in the weight $q_{\la}(r)=q(y)$, which, for any prescribed $0<c<1$ and all sufficiently large $R>1$, satisfies the following properties:
\begin{enumerate}[label={\rm (\roman*)}]
  \item $q(y) = y^2/2$ for $y\le R$; \label{eq:inner-strat}
  \item $q''(y)\ge -c$ and $\frac{1}{y}q'(y)\ge -c$ for $y>0$;  \label{eq:outer-strat}
  \item $(\T\pry)^2 q(y) \le c \jbk[y]^{-2}\jbk[\log y]^{-2}$ for a.e. $y>0$. \label{eq:delta2-strat}
\end{enumerate}
Property~\ref{eq:inner-strat} yields the virial-type cancellation in $\vr$, while property~\ref{eq:outer-strat} keeps the leading quadratic term small for $y\ge R$.
This construction can be traced back to those in \citep{jendrej_construction_2017a, jendrej_twobubble_2018, jendrej_construction_2019}, which use localized, almost convex weights $q(y)$ satisfying a biharmonic upper bound $\Delta_{\rm rad}^2 q\le c/y^2$, with the radial Laplacian taken in the relevant spatial dimension. In our setting, the commutator structure involves $(\T\pry)^2$ instead of $\Delta_{\rm rad}^2$, and property~\ref{eq:delta2-strat} requires the additional logarithmic decay $\jbk[\log y]^{-2}$, which is necessary in view of the sharp coercivity of the conjugate operator $\ti{\ham}$ (Lemma~\ref{le:coer-ti-ham}). These requirements motivate the construction of the new piecewise-smooth weight $q(y)$ in Lemma~\ref{le:vr-weig}.

The remaining error terms involve iterates $(\hj\mham)^K$, where $\hj = (e_3+\wzt+\w)\times$ and the potential $V/y^2\sim y^{-2}$ of $\ham=-\Delta+V/y^2$ may develop a singularity at the origin. We treat these errors systematically as in Subsection~\ref{sub:eng-func}. For $y\le1$, we use finite-order Dore expansions (Definition~\ref{def:dore} and Lemma~\ref{le:preserv}) and the highest-order energy bound. For $y\ge1$, the integrand is reformulated via the pointwise bounds for $\Omega^k_\sigma$ (Lemma~\ref{le:omega-pw-bnd}). Combining these estimates with the Morawetz and virial corrections, we obtain the mixed energy estimate of Proposition~\ref{pro:mix-eng} and close the bootstrap.

\step[4, Blowup dynamics] The asymptotic behavior of the parameters is governed by the modulation equations. We project the equation for $\w$ onto the first two components and test it against the truncated profiles $\xm_{k,0}$ and $\xm_{0,k}$, $1\le k\le L$. This yields the modulation system (Lemma~\ref{le:mod-eq}), which reads
\begin{equation}\begin{split}
    \label{eq:mod-eq-strat}
    & \bigl|\prs a_k +(2k-1+\cb)b_1a_k -a_1b_k -a_{k+1}\bigr|   \\
    & + \bigl|\prs b_k +(2k-1+\cb)b_1b_k +a_1a_k -b_{k+1}\bigr|
    \les \sqrt{\E_{2L+2}} +\frac{b_1^{L+1}}{\lgb},
  \end{split}\end{equation}
where $\cb$ is given by \eqref{eq:cb} and satisfies $\cb\approx 2\lgb^{-1}$.
Using the refined modulation equations in Lemma~\ref{le:refined-mod-eq}, we determine the asymptotics of $a_1$ and $b_1$ as $s\to+\infty$, and hence the blowup rate
\[
  \la(t) \sim \frac{(T-t)^L}{\left|\log (T-t)\right|^{\frac{2L}{2L-1}}}.
\]
We analyze the stability\footnote{A blowup solution is said to be stable if it preserves its blowup dynamics under perturbations of the initial data.} of the construction by linearizing the modulation system near the explicit approximate solution $(a^e,b^e)$ of \eqref{eq:mod-eq-strat}. The corresponding linearized system of Subsection~\ref{sub:ab-sys} has one negative eigenvalue, corresponding to the stable mode, and $2L-1$ positive eigenvalues, corresponding to the unstable modes. Together with the two symmetry directions of the $1$-equivariant flow, $\la$ and $\vtha$, this shows that the blowup data form a codimension-$(2L-1)$ family in $\dot{H}^1$, in the sense of the remark following Theorem~\ref{th:main-thm}.

\medskip
\noindent\textbf{Structure of this article.} The rest of the article is organized as follows. In Section~\ref{sec:prelim}, we introduce the Frenet basis and derive the vectorial form of~\eqref{eq:sm}. Section~\ref{sec:app-sol} constructs the approximate solution, while Section~\ref{sec:trap-reg} sets up the trapped regime and the bootstrap argument. In Section~\ref{sec:energy-est}, we establish the energy, virial, and Morawetz estimates. Section~\ref{sec:compl-constr} closes the bootstrap and completes the proof of Theorem~\ref{th:main-thm}. The appendices collect the geometric identities and technical estimates used in the argument.

\section{Vectorial formulation}
\label{sec:prelim}
In this section, we introduce the Frenet basis associated with the ground state $Q=Q_1$, and derive the vectorial equation~\eqref{eq:vect-form} for the anticipated solution. We then study the linear Hamiltonian arising therein. This section follows the treatment in~\citep{merle_blowup_2013}, which we recall for completeness.

\subsection{The Frenet basis}
\label{sub:frenet-basis}
We recall some basics of the Frenet basis associated with the ground state $Q$.
Define the following functions:
\begin{equation}
  \label{eq:vla-phi-z}
  \begin{gathered}
    \lphi(r) = r\prr(2\arctan r) =\frac{2r}{1+r^2} =\sin\phi(r), \\
    \z(r) = \frac{1-r^2}{1+r^2} =\cos\phi(r), \quad
    \zt(r) = 2(1+\z).
  \end{gathered}
\end{equation}
Then from \eqref{eq:grd-state-eq}, we rewrite $Q$ as
\[
  Q(r,\tha) = e^{\tha\R}
  \begin{pmatrix}
    \lphi(r) \\ 0 \\ \z(r)
  \end{pmatrix}.
\]
The Frenet basis $(e_r, e_\tau, Q)$ is given by
\begin{equation}
  \label{eq:frenet}
  e_r = \frac{\prr Q}{|\prr Q|} = e^{\tha\R}
  \begin{pmatrix}
    \z \\ 0 \\ -\lphi
  \end{pmatrix}
  \andd
  e_\tau = \frac{\pr_\tau Q}{|\pr_\tau Q|} = e^{\tha\R}
  \begin{pmatrix}
    \hs[.1]0\hs[.1] \\ 1 \\ 0
  \end{pmatrix}.
\end{equation}

\begin{lemma}[Properties of the Frenet basis \citep{merle_blowup_2013}]
  \label{le:frenet-basis}
  The following hold:
  \begin{enumerate}[label={\rm{(\roman*)}}]
    \item derivatives:
      \begin{gather*}
        \prr e_r = -\frac{\zt}{2} Q,\quad
        \prr e_\tau = 0,     \quad
        \prr Q = \frac{\zt}{2} e_r, \\
        \vla e_r = -\lphi\,Q, \quad
        \vla e_\tau = 0, \quad
        \vla Q = \lphi\,e_r, \\
        \pr_{\tau} e_r =\frac{\z}{r}e_\tau, \quad
        \pr_{\tau} e_\tau =-\frac{\z}{r}e_r - \frac{\zt}{2} Q, \quad
        \pr_{\tau} Q = \frac{\zt}{2} e_\tau, \\
        \Delta e_r =-\frac{1}{r^2} e_r -\frac{\z\zt}{r}Q, \quad
        \Delta e_\tau = -\frac{1}{r^2} e_\tau, \quad
        \Delta Q =- \frac{\zt^2}{2} Q.
      \end{gather*}
    \item rotations:
      \[
        \R e_r = \z e_\tau, \quad
        \R e_\tau = -\z e_r -\lphi Q,\quad
        \R Q = \lphi e_\tau.
      \]
    \item infinitesimal generators: for $Q_{\vtha,\la}(r,\tha) := e^{\vtha\R}Q(\tfrac{r}{\la},\tha)$, we have
      \[
        \frac{\dd}{\dd \la}(Q_{\vtha,\la})\big|_{\la=1,\vtha=0}= -\lphi e_r,\quad
        \frac{\dd}{\dd \vtha}(Q_{\vtha,\la})\big|_{\la=1,\vtha=0}= \lphi e_\tau.
      \]
  \end{enumerate}
\end{lemma}

\subsection{Derivation of the vectorial equation}
\label{sub:norm-eq}
Following the strategy outlined in Subsection~\ref{sub:outline}, we seek solutions of the form
\begin{equation}
  \label{eq:sol-form}
  u = \mcS (Q+\hv)
  := e^{\vtha(t)\R}(Q+\hv)\biggl(t,\frac{r}{\la(t)}\biggr),
\end{equation}
where $\vtha, \la$ are time-dependent functions to be determined, and $\mcS$ denotes the mixed transformation induced by them. Expressed in the Frenet basis, the vector $\hv$ reads
\[
  \hv = (e_r,e_\tau,Q)\hw
  \with
  \hw = \begin{pmatrix}
    \hs[.1]\wh\alpha\hs[.1] \\ \wh\beta \\ \wh\gamma
  \end{pmatrix}.
\]
Then the \sch\ map equation \eqref{eq:sm} reduces to the equation for $\hw$, which we now derive. In view of~\eqref{eq:ss-trans}, we work in the self-similar variables and define the rescaled vector by
\begin{equation}
  \label{eq:la-trans}
  \hv_\la(t,r) = \hv(s,y).
\end{equation}
We first consider the term $\prt u$. By Lemma~\ref{le:frenet-basis},
\begin{align*}
  \prt u
  & = \frac{1}{\la^2} \mcS \Bigl(\prs\hv +\vtha_s\R(Q+\hv) -\frac{\la_s}{\la}\vla(Q+\hv) \Bigr)                                                               \\
  & = \frac{1}{\la^2} \mcS \biggl\{ \biggl( \prs\wh\alpha -\frac{\la_s}{\la}\vla\wh\alpha -\vtha_s \z\wh\beta -\frac{\la_s}{\la}\lphi(1+\wh\gamma) \biggr)e_r \\
  & \qquad\qquad + \biggl( \prs\wh\beta -\frac{\la_s}{\la} \vla\wh\beta +\vtha_s \z\wh\alpha +\vtha_s\lphi(1+\wh\gamma) \biggr)e_\tau                         \\
  & \qquad\qquad + \biggl( \prs\wh\gamma -\frac{\la_s}{\la}\vla\wh\gamma +\frac{\la_s}{\la}\lphi\wh\alpha -\vtha_s\lphi\wh\beta \biggr)\, Q\biggr\}           \\
  & = \frac{1}{\la^2} \mcS \biggl\{(e_r, e_\tau, Q) \biggl(\prs\hw -\frac{\la_s}{\la}\vla\hw +\vtha_s\z\R\hw +(e_3+\hw)\times\hp\lphi\biggr)\biggr\},
\end{align*}
where we set the $\mbR^3$-vector
\[
  \hp = \begin{pmatrix}
    \vtha_s \\ \la_s/\la \\ 0
  \end{pmatrix}.
\]
Consider next the term $u\times\Delta u$. Again, from Lemma~\ref{le:frenet-basis},
\[
  \Delta Q = -\frac{\zt^2}{2}Q
  \quad\Longrightarrow\quad
  (Q+\hv)\times\Delta(Q+\hv)
  = (Q+\hv)\times\biggl(\Delta\hv+\frac{\zt^2}{2} \hv\biggr),
\]
where
\begin{align*}
  \Delta\hv+\frac{\zt^2}{2} \hv
  & = \biggl( \Delta\wh\alpha +\biggl( \frac{\zt^2}{2} -\frac{1}{r^2} \biggr)\wh\alpha +\zt\pr_r\wh\gamma \biggr) e_r                                                                    \\
  & \quad +\biggl( \Delta\wh\beta +\biggl( \frac{\zt^2}{2} -\frac{1}{r^2} \biggr)\wh\beta \biggr) e_\tau +\biggl( \Delta\wh\gamma -\zt\prr\wh\alpha -\frac{\z\zt}{r}\wh\alpha \biggr) Q.
\end{align*}
Observe that the leading linear parts in the $e_1, e_2$ coordinates are governed by $\Delta -\frac{1}{r^2}$. This motivates the definition of the linear Hamiltonian
\[
  \ham = -\Delta +\frac{V(r)}{r^2}
  \where
  V(r) = -r^2\biggl(\frac{\zt^2}{2} -\frac{1}{r^2}\biggr) = 2\vla\z +1.
\]
Note that for any $\mbR^3$-valued functions $\vv, \w$,
\[
  \Delta \vv_{\la} =\frac{1}{\la^2} (\Delta \vv)_{\la}, \quad
  \mcS \vv\times\mcS\w = \mcS (\vv\times\w).
\]
Therefore,
\[
  u\times\Delta u
  = \frac{1}{\la^2}\mcS \Bigl\{ (Q+\hv)\times\Delta(Q+\hv) \Bigr\}
  = -\frac{1}{\la^2}\mcS \Bigl\{ (e_r, e_\tau, Q) \Bigl((e_3+\hw)\times\mham\hw\Bigr) \Bigr\},
\]
where the vectorial Hamiltonian is defined by
\begin{equation}
  \label{eq:mham}
  \mham\hw = \begin{pmatrix*}[l]
    \ham\wh\alpha -\zt\pry\wh\gamma \\
    \ham\wh\beta                   \\
    -\Delta\wh\gamma +\zt\bigl(\pry+\tfrac{\z}{y}\bigr)\wh\alpha
  \end{pmatrix*},
\end{equation}
and the scalar Hamiltonian by
\begin{equation}
  \label{eq:ham}
  \ham = -\Delta +\frac{V(y)}{y^2},
  \quad
  V(y) = -y^2\biggl(\frac{\zt^2}{2} -\frac{1}{y^2}\biggr)
  =\frac{y^4-6y^2+1}{{(1+y^2)}^2}.
\end{equation}
In addition, for any vector-valued function $\w$, we define the corresponding complex structure
\begin{equation}
  \label{eq:hj}
  \hj[\w] =(e_3 +\w)\times.
\end{equation}
Then, gathering the above, we obtain the vectorial equation
\begin{equation}
  \label{eq:vect-form}
  \prs\hw -\frac{\la_s}{\la}\vla\hw +\vtha_s\z\R\hw +\hj[\hw](\mham\hw +\hp\lphi) =0.
\end{equation}

\begin{remark}
  In the presence of the self-similar transform~\eqref{eq:ss-trans}, the spatial variable in \eqref{eq:vect-form} is now $y=\frac{r}{\la}$ instead of $r$.

\end{remark}

\subsection{The linear Hamiltonian}
\label{sub:lin-ham}
We collect some properties of the linear Hamiltonian $\ham$, which are used in the following sections. These are taken from \citep{merle_blowup_2013}.

The operator $\ham$ admits the decomposition
\begin{equation}
  \label{eq:ham-factori}
  \ham = \as\oa,
\end{equation}
where the first-order linear operators are defined by
\[
  \oa=-\pry+\frac{\z}{y} = -\lphi\pry\biggl(\frac{\cdot}{\lphi}\biggr),
  \quad \as=\pry+\frac{1+\z}{y} = \frac{1}{y\lphi}\pry(y\lphi\,\cdot).
\]
Here $\z$ and $\lphi$ are given in~\eqref{eq:vla-phi-z}, and $\as$ is the adjoint of $\oa$.
For radial functions, the operators $\ham, \oa, \as$ are viewed as ODE operators on $\mbR_+$.
A direct computation gives
\[
  \ker \oa =\spann(\lphi)
  \andd
  \ker \as =\spann\biggl(\frac{1}{y\lphi}\biggr).
\]
Correspondingly,
\[
  \ker\ham =\spann(\lphi, \vg)
  \with
  \vg(y) := \lphi(y)\int_1^y\frac{\dd x}{x \lphi^2(x)}.
\]
The asymptotics are as follows:
\begin{equation}
  \begin{split}
    \label{eq:lphi}
    \lphi(y) & = \begin{cases}
      2y+O(y^3)                               & \mbox{as $y\to0$},       \\
      \frac{2}{y}+ O\bigl(\frac{1}{y^3}\bigr) & \mbox{as $y\to+\infty$},
    \end{cases}    \\
    \vg(y)   & = \begin{cases}
      -\frac{1}{4y}+ O(y\log y)                  & \mbox{as $y\to0$},       \\
      \frac{y}{4} +O\bigl(\frac{\log y}{y}\bigr) & \mbox{as $y\to+\infty$}.
    \end{cases}
  \end{split}
\end{equation}
The inhomogeneous equation $\ham u = f$ admits the following explicit solution, which can be expressed using the formal inverse of $\ham$ by
\begin{equation}
  \label{eq:inverse-ham}
  u = \ham^{-1} f
  := \lphi\int_0^y f(x)\vg(x) x\dd x - \vg\int_0^y f(x)\lphi(x) x\dd x.
\end{equation}
The iterated resonances of $\ham$ are defined by
\begin{equation}
  \label{eq:reson}
  T_k = \ham^{-k}\lphi \quad\mbox{for}\quad k\ge 1.
\end{equation}
In particular, for $k=1$, $T_1$ takes the form
\begin{equation}
  \label{eq:t1-asymp}
  T_1(y) =
  \begin{cases}
    -\frac{y^3}{4} +O(y^5)                         & \mbox{as $y\to0$},       \\
    -y\log y +y +O\bigl(\frac{(\log y)^2}{y}\bigr) & \mbox{as $y\to+\infty$}.
  \end{cases}
\end{equation}
The leading part $\mhamp$ of the vectorial Hamiltonian $\mham$ is defined by
\[
  \mhamp\begin{pmatrix}
    \alpha \\ \beta \\ \gamma
  \end{pmatrix}
  = \begin{pmatrix}
    \ham\alpha \\ \ham\beta \\ 0
  \end{pmatrix}.
\]
The vectorial iterated resonances associated with $\mhamp$ are defined by
\begin{equation}
  \label{eq:vec-reson}
  \phia[k] = T_k\R^{k-1}\eone, \quad
  \phib[k] = T_k\R^{k-1}\etwo \quad\mbox{for}\quad k\ge 1.
\end{equation}
Since $\mhamp$ is diagonal,
\begin{equation}
  \label{eq:mham-r}
  \R\mhamp = \mhamp\R
  \quad\mbox{and}\quad
  \R\mham\R = -\mhamp.
\end{equation}
Thus by \eqref{eq:reson},
\[
  \mhamp\phia[k]
  = \begin{cases}
    \R\phia[k-1], & k\ge 2, \\
    \lphi e_1,    & k=1,
  \end{cases}
  \quad\mbox{and}\quad
  \mhamp\phib[k]
  = \begin{cases}
    \R\phib[k-1], & k\ge 2, \\
    \lphi e_2,    & k=1.
  \end{cases}
\]
In addition, from \eqref{eq:ham-factori} we have $\mhamp = \mas\ma$, where
\[
  \ma\begin{pmatrix} \alpha\\ \beta\\ \gamma\end{pmatrix}
  = \begin{pmatrix} \oa\alpha\\ \oa\beta\\ 0 \end{pmatrix},
  \quad \mas\begin{pmatrix} \alpha\\ \beta\\ \gamma\end{pmatrix}
  = \begin{pmatrix} \as\alpha\\ \as\beta\\ 0 \end{pmatrix}.
\]

\section{The approximate solution}
\label{sec:app-sol}
In this section, we construct an approximate solution to \eqref{eq:vect-form}. We begin with a formal computation based on \eqref{eq:constra}, then introduce the admissible/homogeneous function spaces to characterize the approximate profiles and bound the error. Finally, we localize the profiles and study the ODE system satisfied by the geometric parameters.

\subsection{A formal computation}
\label{sub:formal-comp}
We first present a formal computation to motivate the construction of the approximate solution $\hw = \wz$.

Let $[s_0,s_1]$ be a rescaled time interval, and $b_{1}^{\ast} \in (0,1)$ be a sufficiently small constant to be fixed. Let
\[
  a = (a_k)_{k=1}^{L}, \;
  b = (b_k)_{k=1}^{L}: [s_0,s_1] \to\mbR^L
\]
be $C^1$ functions
satisfying the a priori bounds
\[
  0 <b_1 \le b_{1}^{\ast} \andd
  |a_k| +|b_k| \les b_1^k \for 1\le k\le L.
\]
We define
\begin{equation}
  \label{eq:cb}
  \cb = \frac{4}{\int\chi_{\frac{S_0}{4}}(x)\lphi^2(x)x\dd x} = \frac{2}{\lgb}(1+o_{b_1\to 0}(1)).
\end{equation}
Then for $1\le k\le L$, we set
\begin{equation}
  \label{eq:msf-ab}
  \begin{aligned}
    A_k & = (2k-1+\cb)b_1a_k -a_1b_k -a_{k+1}, \\
    B_k & = (2k-1+\cb)b_1b_k +a_1a_k -b_{k+1},
  \end{aligned}
\end{equation}
where the convention $a_{L+1} = b_{L+1} = 0$ is used. We look for an approximate solution of the form
\begin{equation}
  \label{eq:wz-form}
  \wz = \vp +\x +\gamma_0 e_3,
  \where
  \vp = \sum_{k=1}^L \bigl(a_k\phia[k] +b_k\phib[k]\bigr).
\end{equation}
The functions $\x$ and $\gamma_0$ serve as corrections for the profile $\vp$, where $\x$ accounts for the $e_1,e_2$-components (we let $\x \cdot e_3 = 0$), and $\gamma_0$ accounts for the $e_3$-component. Note that we denote by $(\alpha_0,\beta_0,\gamma_0)$ the components of $\wz$ in the Frenet basis, and that we use the notation $\wzp = \vp +\x$.

Because of the constraint~\eqref{eq:constra}, the functions $\x$ and $\gamma_0$ should be chosen so that $e_3 + \wz$ lies approximately on $\mbS^2$. We define the corresponding $\rrr$ by
\begin{equation}
  \label{def:rrr}
  |\wzp|^2 +|1+\gamma_0|^2 = 1 + 2\rrr.
\end{equation}
Then the following holds.

\begin{lemma}[Formal computation]
  \label{le:form-comp}
  For any vector-valued function $\x=\x(a,b,y)$ satisfying $\x\cdot e_3=0$, there exists a scalar function $\gamma_0 = \gamma_0(\vp, \x)$ such that the approximate solution $\wz$ defined by~\eqref{eq:wz-form} satisfies
  \[
    \prs\wz -\frac{\la_s}{\la}\vla\wz +\vtha_s\z\R\wz +\hj[\wz](\mham\wz +\hp\lphi)
    = \mod +\psiz,
  \]
  where $\mod$ is the modulation vector defined by
  \begin{align*}
    \mod
    & = \sum_{k=1}^L (\prs a_k +A_k)\frac{\pr\wz}{\pr a_k}
    +\sum_{k=1}^L (\prs b_k +B_k)\frac{\pr\wz}{\pr b_k}                              \\
    & \quad +(\vtha_s +a_1)(\lphi(e_2-e_1\times\wz) +\z\R\wz)                       \\
    & \quad -\Bigl(\frac{\la_s}{\la} +b_1\Bigr) (\lphi(e_1+e_2\times\wz) +\vla\wz).
  \end{align*}
  Moreover, the error decomposes as $\psiz = \psizone + \psiztwo$ with
  \begin{align*}
    \psizone
    & = -\hj\R\si_{L+1}\Bigl\{ b_1\mcT\wzp +a_1\mcP\x -a_1(1+\z)\R\vp -\gamma_0\mcQ\vp +\mcQ\x \Bigr\}  \notag              \\
    & \quad -\hj\Bigl\{\zt\pry\gamma_0 e_1
    +\si_{L+1} \Bigl(\zt\Bigl(\pry+\tfrac{\z}{y}\Bigr)\alpha_0 -\Delta\gamma_0\Bigr)\wzp \Bigr\} + \hj\mhamp\x, \notag       \\
    \psiztwo
    & = \si_{L+1}(b_1\mcT\rrr +a_1\mcP\rrr +\mcQ\rrr) e_3   \notag                                                          \\
    & \quad +{(-\gamma_0)}^{L+2} \Bigl\{ b_1\mcT\wz +a_1\mcP\wz  -\gamma_0\mcQ\wz -\hj\R\mcQ\x +\mcQ\rrr e_3 \Bigr\} \notag \\
    & \quad +{(-\gamma_0)}^{L+2} \hj\Bigl(\zt\Bigl(\pry+\tfrac{\z}{y}\Bigr)\alpha_0 -\Delta\gamma_0\Bigr)e_3,
  \end{align*}
  and $\mcT$, $\mcP$, $\mcQ$ are linear operators defined by
  \begin{align*}
    \mcT & = \vla -\sum_{k=1}^L (2k-1+\cb)\biggl(a_k\frac{\pr}{\pr a_k} +b_k\frac{\pr}{\pr b_k}\biggr), \\
    \mcP & = -\z\R +\sum_{k=1}^L\biggl(b_k\frac{\pr}{\pr a_k} -a_k\frac{\pr}{\pr b_k}\biggr),\quad
    \mcQ  = \sum_{k=1}^L\biggl(a_{k+1}\frac{\pr}{\pr a_k} +b_{k+1}\frac{\pr}{\pr b_k}\biggr).
  \end{align*}
\end{lemma}

\begin{remark}
  \label{re:psiz-explain}
  (i) The term $\psizone$ will be  canceled, up to terms of sufficiently high parameter degree, by a suitable choice of $\x$. The term $\psiztwo$ is controllable, since each term in $\psiztwo$ involves either $\rrr$ or $(-\gamma_0)^{L+2}$. \\
  (ii) The component $\gamma_0$ is explicitly given in Step~2 of the following proof, but the choice of $\gamma_0$ is not unique. Indeed, any Taylor polynomial for $-1+\sqrt{1-|\cdot|^2}$ of sufficiently high degree also serves the purpose.

\end{remark}

\begin{proof}
  We first derive $\mod$ (and hence $\psiz$) for a general $\wz$, then choose a suitable $\gamma_0$, as explained previously, and finally use the complex structure $\hj[\wz]$ to obtain the explicit decomposition of $\psiz$. For brevity, we write $\hj=\hj[\wz]$ throughout this proof.

  \step[1, Deriving $\mod$]
  From \eqref{eq:msf-ab} we compute
  \begin{align*}
    \prs\wz -\frac{\la_s}{\la}\vla\wz +\vtha_s\z\R\wz
    & = \sum_{k=1}^L (\prs a_k +A_k)\frac{\pr\wz}{\pr a_k} +\sum_{k=1}^L (\prs b_k +B_k)\frac{\pr\wz}{\pr b_k} \\
    & \quad -\Bigl(\frac{\la_s}{\la} + b_1\Bigr)\vla\wz +(\vtha_s + a_1)\z\R\wz                                \\
    & \quad +b_1\mcT\wz +a_1\mcP\wz +\mcQ\wz,
  \end{align*}
  where the first and second lines correspond to $\mod$, and the last line belongs to $\psiz$. Recalling the definition of $\mham$ and $\hj$, and using the facts that $\mhamp\phia[k+1] = \R\phia[k]$ and $\mhamp\phib[k+1] = \R\phib[k]$, we have
  \begin{align*}
    \hj(\mham\wz +\hp\lphi)
    & = (\vtha_s+a_1)\lphi(e_2 -e_1\times\wz) -\Bigl(\frac{\la_s}{\la}+b_1\Bigr)\lphi(e_1 +e_2\times\wz)            \\
    & \quad +\hj \biggl(\mhamp\x -\Delta\gamma_0 e_3 -\zt e_2\times\pry\wz +\frac{\z\zt}{y}(\wz\cdot e_1)e_3\biggr) \\
    & \quad +\hj \sum_{k=1}^L\R(a_{k+1}\phia[k] +b_{k+1}\phib[k]),
  \end{align*}
  where the first line also corresponds to $\mod$, and the remaining lines correspond to $\psiz$. Collecting these identities, we obtain the desired $\mod$, and also
  \begin{equation}
    \begin{split}
      \label{pf:form-comp=psiz-formal}
      \psiz
      & = b_1\mcT\wz + a_1\mcP\wz + \mcQ\wz +\hj\R\mcQ \vp                                                                 \\
      & \quad +\hj\biggl( -\Delta\gamma_0 e_3 -\zt e_2\times\pry\wz +\frac{\z\zt}{y}(\wz\cdot e_1)e_3\biggr) +\hj\mhamp\x.
    \end{split}
  \end{equation}

  \step[2, Choice of $\gamma_0$]
  Let $f(x) = -1 + \sqrt{1-x}$. Denote by $f_n$ its Taylor polynomial of degree $n$ at the origin:
  \begin{equation}
    \label{pf:form-comp=f}
    f_n(x) = \sum_{k=1}^n c_k x^{k},
    \where c_k = (-1)^k\binom{1/2}{k}.
  \end{equation}
  We choose $\gamma_0$ to be
  \begin{equation}
    \label{pf:form-comp=gamma}
    \gamma_0 = f_{L+1}(|\wzp|^2) = \sum_{k=1}^{L+1} c_k |\wzp|^{2k}.
  \end{equation}
  In view of~\eqref{def:rrr}, this specifies $\rrr$.
  We claim that there exist constants $(d_k)_{k=L+2}^{2(L+1)}$ depending on $c_1,\dots,c_{L+1}$, such that
  \begin{equation}
    \label{pf:form-comp=rrr-expr}
    \rrr = \sum_{k=L+2}^{2(L+1)} d_k |\wzp|^{2k}.
  \end{equation}
  Indeed, on the one hand, we have that $\rrr$ is a polynomial in $|\wzp|^2$ of degree at most $2(L+1)$, since $f_{L+1}(|\wzp|^2)$ is a polynomial in $|\wzp|^2$ of degree $(L+1)$. On the other hand, if $\|\wzp\|_{L^{\infty}} \ll 1$, then by the Taylor series of $f$ at the origin,
  \begin{align*}
    |\rrr|
    & = \frac12 \Bigl| |\wzp|^2 +|1+f_{L+1}(|\wzp|^2)|^2 -1 \Bigr|        \\
    & = \frac12 \Bigl| (1+f_{L+1}(|\wzp|^2))^2 - (1+f(|\wzp|^2))^2 \Bigr| \\
    & \les |f_{L+1}(|\wzp|^2) - f(|\wzp|^2)|                              \\
    & \les |\wzp|^{2(L+2)}.
  \end{align*}
  Thus $\rrr$ has vanishing coefficients for monomials of degree less than $L+2$. This shows \eqref{pf:form-comp=rrr-expr} if $\|\wzp\|_{L^{\infty}} \ll 1$ holds. Since \eqref{pf:form-comp=rrr-expr} is merely an algebraic identity, it holds independently of the extra condition $\|\wzp\|_{L^{\infty}} \ll 1$.

  \step[3, Identities related to $\hj$]
  We denote
  \begin{equation}
    \label{pf:form-comp=si}
    \si_k = \sum_{i=0}^k (-\gamma_0)^i
    \for
    k\in\mbN.
  \end{equation}
  Then for all $k\in\mbN$ and any vector $v\in\mbR^3$ with $v\cdot e_3=0$, the following holds:
  \begin{align}
    v          & = -\hj\R\si_{k}v +\si_{k}(\wzp\cdot v)e_3 +(-\gamma_0)^{k+1}v,
    \label{pf:form-comp=hj-claim-1}                                                        \\
    v +\hj\R v & = \gamma_0\si_{k}\hj\R v +\si_{k+1} (\wzp\cdot v)e_3 +(-\gamma_0)^{k+2}v,
    \label{pf:form-comp=hj-claim-2}                                                        \\
    \hj e_3    & = -\si_{k}\hj\wzp +(-\gamma_0)^{k+1}\hj e_3.
    \label{pf:form-comp=hj-claim-3}
  \end{align}
  These are direct consequences of the complex structure $\hj$. We prove them by induction.
  By the double wedge formula,
  \[
    \hj(-\R v)
    = -(e_3 +\wz)\times(e_3\times v)
    = v -(\wzp\cdot v)e_3 +\gamma_0 v,
  \]
  which is the initial case of \eqref{pf:form-comp=hj-claim-1}. Then by induction, the result follows
  \begin{align*}
    v
    & = \hj(-\R\si_{k}v) +\si_{k}(\wzp\cdot v)e_3 +(-\gamma_0)^{k+1}(\hj(-\R v) +(\wzp\cdot v)e_3 -\gamma_0 v) \\
    & = \hj(-\R\si_{k+1}v) +\si_{k+1}(\wzp\cdot v)e_3 +(-\gamma_0)^{k+2}v.
  \end{align*}
  The identity~\eqref{pf:form-comp=hj-claim-2} follows by replacing $k$ with $k+1$ in \eqref{pf:form-comp=hj-claim-1} and using $\si_{k+1}=1+(-\gamma_0)\si_k$. To show~\eqref{pf:form-comp=hj-claim-3}, a direct computation gives the case $k=0$:
  \[
    \hj e_3
    = (e_3+\wz)\times e_3
    = -(e_3+\wz)\times\wz
    = -\hj\wz = -\hj\wzp +(-\gamma_0)\hj e_3,
  \]
  and then by induction, we obtain
  \begin{align*}
    \hj e_3
    & = -\si_{k}\hj\wzp +(-\gamma_0)^{k+1}(-\hj\wzp +(-\gamma_0)\hj e_3) \\
    & = -\si_{k+1}\hj\wzp +(-\gamma_0)^{k+2}\hj e_3.
  \end{align*}
  \step[4, Computing $\psiz$]
  From \eqref{def:rrr} and $(1+\gamma_0)\si_{L+1} = 1-(-\gamma_0)^{L+2}$ we have
  \begin{align*}
    \si_{L+1}(\wzp\cdot\vla\wzp)
    & = \si_{L+1} \vla\biggl(\frac{|\wzp|^2}{2}\biggr)                     \\
    & = \si_{L+1} \vla\biggl(\frac12+\rrr -\frac{|1+\gamma_0|^2}{2}\biggr) \\
    & = \si_{L+1}\vla\rrr - \vla\gamma_0 + (-\gamma_0)^{L+2}\vla\gamma_0.
  \end{align*}
  Similarly, for $1\le k\le L$,
  \begin{equation}
    \begin{split}
      \label{pf:form-comp=wzp-dot}
      \si_{L+1}\biggl(\wzp\cdot \frac{\pr\wzp}{\pr a_k}\biggr)
      & = \si_{L+1}\frac{\pr\rrr}{\pr a_k} -\frac{\pr\gamma_0}{\pr a_k} +(-\gamma_0)^{L+2}\frac{\pr\gamma_0}{\pr a_k}, \\
      \si_{L+1}\biggl(\wzp\cdot \frac{\pr\wzp}{\pr b_k}\biggr)
      & = \si_{L+1}\frac{\pr\rrr}{\pr b_k} -\frac{\pr\gamma_0}{\pr b_k} +(-\gamma_0)^{L+2}\frac{\pr\gamma_0}{\pr b_k}.
    \end{split}
  \end{equation}
  Therefore from \eqref{pf:form-comp=hj-claim-1},
  \begin{align*}
    \vla\wz
    & = \hj(-\R\si_{L+1} \vla\wzp) +\si_{L+1}(\wzp\cdot \vla\wzp)e_3 +(-\gamma_0)^{L+2}\vla\wzp +\vla\gamma_0 e_3 \\
    & = \hj(-\R\si_{L+1}\vla\wzp) +\si_{L+1}\vla\rrr e_3 +(-\gamma_0)^{L+2}\vla\wz.
  \end{align*}
  Similarly,
  \begin{align*}
    \frac{\pr\wz}{\pr a_k}
    & = \hj\biggl(-\R\si_{L+1}\frac{\pr\wzp}{\pr a_k}\biggr) +\si_{L+1}\frac{\pr\rrr}{\pr a_k} e_3 +(-\gamma_0)^{L+2}\frac{\pr\wz}{\pr a_k}, \\
    \frac{\pr\wz}{\pr b_k}
    & = \hj\biggl(-\R\si_{L+1}\frac{\pr\wzp}{\pr b_k}\biggr) +\si_{L+1}\frac{\pr\rrr}{\pr b_k} e_3 +(-\gamma_0)^{L+2}\frac{\pr\wz}{\pr b_k}.
  \end{align*}
  We apply these to compute each term of~\eqref{pf:form-comp=psiz-formal}.

  First, for the $\mcT$ term, a direct application of the above gives
  \[
    \mcT\wz
    = \hj( -\R\si_{L+1}\mcT\wzp )
    +\si_{L+1}\mcT\rrr e_3 +(-\gamma_0)^{L+2}\mcT\wz.
  \]
  Then, for the $\mcP$ term, using \eqref{pf:form-comp=hj-claim-1} again,
  \begin{align*}
    \mcP\wz
    & = \mcP\wzp +\mcP\gamma_0 e_3                                                                                \\
    & = \hj(-\R\si_{L+1}\mcP\wzp) +\si_{L+1}(\wzp\cdot\mcP\wzp)e_3 +(-\gamma_0)^{L+2}\mcP\wzp + \mcP\gamma_0 e_3.
  \end{align*}
  Note that from the definition of $\Phi$,
  \[
    \R\vp + \sum_{k=1}^{L}\biggl(b_k\frac{\pr\vp}{\pr a_k} - a_k\frac{\pr\vp}{\pr b_k}\biggr) = 0
    \quad\Longrightarrow\quad
    \mcP\wzp = -(1+\z)\R\vp + \mcP\x,
  \]
  and in addition, the term $\si_{L+1}(\wzp\cdot\mcP\wzp)$ is computed by using~\eqref{def:rrr} in the same fashion as \eqref{pf:form-comp=wzp-dot}, which yields
  \[
    \mcP\wz
    = \hj\bigl(-\R\si_{L+1}\bigl(-(1+\z)\R\vp + \mcP\x\bigr)\bigr) +\si_{L+1}\mcP\rrr e_3 +(-\gamma_0)^{L+2}\mcP\wz.
  \]
  Next, for the $\mcQ$ terms, by definition,
  \[
    \mcQ\wz +\hj\R\mcQ\vp = \sum_{k=1}^{L} a_{k+1}\biggl(\frac{\pr\wz}{\pr a_k} +\hj\R\phia[k]\biggr) + \sum_{k=1}^{L} b_{k+1}\biggl(\frac{\pr\wz}{\pr b_k} +\hj\R\phib[k]\biggr).
  \]
  We compute the first parenthesis; the second is treated analogously. From the definition of $\wz$,
  \[
    \frac{\pr\wz}{\pr a_k} +\hj\R\phia[k] = \phia[k] +\hj\R\phia[k] +\frac{\pr\x}{\pr a_k} + \frac{\pr\gamma_0}{\pr a_k}e_3.
  \]
  Applying~\eqref{pf:form-comp=hj-claim-2} to the terms involving $\phia[k]$ and~\eqref{pf:form-comp=hj-claim-1} to the $\x$ term, we obtain the combination
  \[
    \si_{L+2} (\wzp\cdot\phia[k]) + \si_{L+2}\biggl(\wzp\cdot\frac{\pr\x}{\pr a_k}\biggr)
    = \si_{L+2} \biggl(\wzp\cdot\frac{\pr\wzp}{\pr a_k}\biggr),
  \]
  to which we apply \eqref{pf:form-comp=wzp-dot} again, and then
  \begin{align*}
    \frac{\pr\wz}{\pr a_k} +\hj\R\phia[k]
    & = \gamma_0\si_{L+1}\hj\R\phia[k] -\si_{L+1}\hj\R\frac{\pr\x}{\pr a_k} +\si_{L+2}\frac{\pr\rrr}{\pr a_k}e_3 \\
    & \quad -(-\gamma_0)^{L+2}\hj\R\frac{\pr\x}{\pr a_k} +(-\gamma_0)^{L+3}\frac{\pr\wz}{\pr a_k}.
  \end{align*}
  Thus we have
  \begin{align*}
    \mcQ\wz +\hj\R\mcQ\vp
    & = \hj\bigl(-\R\si_{L+1}(-\gamma_0\mcQ\vp +\mcQ\x)\bigr) +\si_{L+2}\mcQ\rrr e_3 \\
    & \quad -{(-\gamma_0)}^{L+2}(\hj\R\mcQ\x +\gamma_0\mcQ\wz).
  \end{align*}
  Finally, for the second line of~\eqref{pf:form-comp=psiz-formal}, using \eqref{pf:form-comp=hj-claim-3},
  \begin{align*}
    & \quad \hj\biggl(-\Delta\gamma_0 e_3 -\zt e_2\times\pry\wz +\frac{\z\zt}{y}(\wz\cdot e_1)e_3 \biggr)                             \\
    & = \hj\biggl\{ -\zt\pry\gamma_0 e_1 -\si_{L+1}\Bigl(\zt\Bigl(\pry+\tfrac{\z}{y}\Bigr)\alpha_0 -\Delta\gamma_0\Bigr)\wzp \biggr\} \\
    & \quad +{(-\gamma_0)}^{L+2} \hj\Bigl(\zt\Bigl(\pry+\tfrac{\z}{y}\Bigr)\alpha_0 -\Delta\gamma_0\Bigr)e_3.
  \end{align*}
  In particular, the rotational contribution inside the first braces of
  $\psizone$ is $-a_1(1+\z)\R\vp$. Also, on setting
  $D_0=\zt(\pry+\z/y)\alpha_0-\Delta\gamma_0$, the geometric contribution is
  $-\hj\{\zt\pry\gamma_0 e_1+\si_{L+1}D_0\wzp\}$, with a plus sign
  inside the braces. Both signs follow directly from the identities above.
  \par
  Gathering these identities, we obtain the desired decomposition $\psiz = \psizone + \psiztwo$. This completes the proof.
\end{proof}

\begin{remark}
  \label{re:mhamp-x-preview}
  Observe in $\psiz$ that the expressions enclosed in braces constitute an elliptic equation of the form $\mhamp\x = f$, where $f$ is the source term given by products of the profiles in $\wz$. Thus $\psizone$ can be controlled if we let $\x$ be the solution of this equation. This is rigorously treated in Subsection~\ref{sub:app-prof}.

\end{remark}

\subsection{Admissible functions}
\label{sub:adm-func}
In this subsection, we introduce some function spaces needed to estimate $\psizone$ and $\psiztwo$: the admissible functions\footnote{Here, the subscript $b$ stands for the dependence on the variable $b_1$, and the subscript $g$ means ``generalized''} $\mcA$, $\mcAb$, and $\mcAg$. This is based on~\citep[Section 2]{raphael_quantized_2014}.

For any smooth scalar function $f(y)$, we define its suitable derivatives by
\begin{equation}
  \label{eq:adb-seq}
  f_{(0)} = f, \andd
  f\sd[k+1] =
  \begin{cases}
    \as f\sd[k] & \mbox{for $k$ odd},  \\
    \oa f\sd[k] & \mbox{for $k$ even},
  \end{cases}
  \quad k\ge 0.
\end{equation}
This notion also applies to vector-valued functions, where the suitable derivatives are defined in a componentwise manner.\footnote{Throughout this subsection, function classes are defined explicitly for scalar functions, but their vector-valued counterparts are also valid, by applying the definitions componentwise.}

\begin{definition}[Function space $\mcA$]
  \label{def:adb}
  Given $p_1, p_2\in\mbZ$, we define the function space $\mcA(p_1,p_2)$ as the set of all analytic functions $f\in C^{\infty}(\mbR_+;\mbR)$ such that
  \begin{enumerate}[label={\rm (\roman*)}]
    \item \label{it:def-adb=1}
      for $y\le 1$, $f$ admits a Taylor expansion at the origin to all orders:
      \[
        f(y) = y^{p_1}\Biggl(\sum_{k=0}^{\ell} c_k y^{2k} + O(y^{2\ell+2})\Biggr);
      \]
    \item \label{it:def-adb=2}
      for $y\ge1$, the suitable derivatives of $f$ satisfy, for all $k\in\mbN$,
      \[
        |f\sd[k](y)|\les
        \begin{cases}
          y^{p_2-k}\jbk[\log y]
          & \mbox{for $p_2-k\ge 0$}, \\
          y^{p_2-k} & \mbox{for $p_2-k < 0$}.
        \end{cases}
      \]
  \end{enumerate}
\end{definition}

Note that the collection of all scalar (or vector-valued) $\mcA$-functions of a certain degree forms a linear space.
The construction of the approximate profiles requires a variant of $\mcA$ with slower growth at infinity. For $j\in\mbZ$, we first introduce the following auxiliary functions:
\[
  \left\{\begin{aligned}
    F_j(y)       & = \bigl(\jbk[\log(\sqrt{b_1}y)] \one_{j\ge 0} + \one_{j<0}\bigr)\one_{y\le 2S_0}, \\
    \bar{F}_j(y) & = \frac{\jbk[\log y]}{\lgb} \one_{y\le 2S_0},
  \end{aligned}\right.
  \andd
  K_j(y) = \frac{\one_{j\ge 2}\one_{y\ge 2S_0}}{y^2 b_1}.
\]

\begin{definition}[Function space $\mcAb$]
  \label{def:badb}
  Given $p_1, p_2\in\mbZ$, we define the function space $\mcAb(p_1,p_2)$ as the set of all functions $f\in C^\infty(\mbR_+\times\mbR_+;\mbR)$ such that
  \begin{enumerate}[label={\rm (\roman*)}]
    \item \label{it:def-badb=1}
      for $y\le 1$, there exists a constant $I\in\mbN_+$ such that $f = \sum_{i=1}^{I}\wh{\delta}_{i}\wh{f}_{i}$ for some $\wh{f}_{i}(y)$ of the form specified in Definition~\itref{def:adb}{it:def-adb=1}, with size $y^{p_1}$, and some $\wh{\delta}_{i}(b_1)$ with
      \[
        |\wh{\delta}_{i}| \les 1
        \andd
        \bigg|\frac{\pr^{\ell} \wh{\delta}_{i}}{\pr b_1^{\ell}}\bigg| \les \frac{1}{b_1^{\ell}\lgb}
        \for \ell\ge 1;
      \]
    \item \label{it:def-badb=2}
      for $y\ge1$, there exists a constant $c_{p_2}\in\mbN$ such that the suitable derivatives of $f$ satisfy the following bounds: for $k\in\mbN$,
      \[
        |f\sd[k]|
        \les \frac{y^{p_2-k}}{\lgb}(F_{p_2-k} + K_{p_2-k}) +y^{p_2-k-2}\jbk[\log y]^{c_{p_2}},
      \]
      and for $\ell\ge 1$,
      \[
        \bigg|\frac{\pr^\ell f\sd[k]}{\pr b_1^\ell}\bigg|
        \les \frac{y^{p_2-k}}{b_1^\ell\lgb} (\bar{F}_{p_2-k} + K_{p_2-k}) + \frac{y^{p_2-k-2}\jbk[\log y]^{c_{p_2}}}{b_1^\ell\lgb}.
      \]
  \end{enumerate}
\end{definition}

\begin{remark}
  (i) Compared with $\mcA$, the functions in $\mcAb$ admit better spatial decay as $y\to\infty$, thanks to the extra $\lgb$ in the denominator. \\
  (ii) If $f\in\mcAb(p_1,p_2)$, then the following rough bounds hold for $y\ge1$:
  \begin{equation}
    \label{eq:badb-rough-bnd}
    |f\sd[k]| \les \jbk[y]^{p_2-k},
    \quad
    \bigg|\frac{\pr^\ell f\sd[k]}{\pr b_{1}^\ell}\bigg|
    \les \frac{\jbk[y]^{p_2-k}}{b_1^{\ell}\lgb},
    \quad k\ge 0, \; \ell\ge 1.
  \end{equation}
\end{remark}

We further introduce the following generalized admissible class $\mcAg$.

\begin{definition}[Function space $\mcAg$]
  \label{def:ag}
  Given $p_1, p_2\in\mbZ$, we define the function space $\mcAg(p_1,p_2)$ as the set of all functions $f\in C^{\infty}(\mbR_+\times\mbR_+;\mbR)$ such that
  \begin{enumerate}[label={\rm (\roman*)}]
    \item \label{it:def-ag=1}
      for $y\le 1$, the function $f$ admits a series expansion of the form specified in Definition~\itref{def:badb}{it:def-badb=1}.
    \item \label{it:def-ag=2}
      for $y\ge 1$, there exists $c_{p_2}\in\mbN$ such that the suitable derivatives of $f$ satisfy the following bounds: for $k\in\mbN$,
      \[
        |f\sd[k]| \les y^{p_2-k}\jbk[\log y]^{c_{p_2}},
      \]
      and for $\ell \ge 1$,
      \[
        \bigg|\frac{\pr^\ell f\sd[k]}{\pr b_1^\ell}\bigg|
        \les \frac{y^{p_2-k}\jbk[\log y]^{c_{p_2}}}{b_1^\ell\lgb}.
      \]
  \end{enumerate}
\end{definition}

We now collect the main closure properties of these spaces, which will be used constantly in the sequel. We first record the action of the linear Hamiltonian $\ham$ and of its formal inverse $\ham^{-1}$ on the function spaces $\mcA$ and $\mcAb$; the proofs are postponed to Appendix~\ref{ap:adm-lemma}.

\begin{lemma}[Action of $\ham$ on $\mcA$]
  \label{le:ham-act-adb}
  If $f\in\mcA(p_1,p_2)$, then
  \begin{enumerate}[label={\rm (\roman*)}]
    \item \label{it:ham-act-adb=ham}
      if $p_1\ge 0$, then for all $k\ge 1$,
      \[
        \ham^k f \in
        \begin{cases}
          \mcA(\max(p_1-2k,1),p_2-2k) & \mbox{for $p_1$ odd},  \\
          \mcA(p_1-2k,p_2-2k)         & \mbox{for $p_1$ even};
        \end{cases}
      \]
    \item \label{it:ham-act-adb=ham-inv}
      if $p_1\ge0$ and $p_2\ge-1$, then for all $k\ge 1$, $\ham^{-k}f\in\mcA(p_1+2k, p_2+2k)$.

  \end{enumerate}
\end{lemma}

\begin{lemma}[Action of $\ham, \vla$ on $\mcAb$]
  \label{le:ham-act-badb}
  If $f\in\mcAb(p_1, p_2)$, then
  \begin{enumerate}[label={\rm (\roman*)}]
    \item \label{it:ham-act-badb=ham}
      if $p_1\ge 0$, then for all $k\ge 1$,
      \[
        \ham^k f \in \begin{cases}
          \mcAb(\max(p_1-2k,1),p_2-2k) & \mbox{for $p_1$ odd},  \\
          \mcAb(p_1-2k,p_2-2k)         & \mbox{for $p_1$ even};
        \end{cases}
      \]
    \item \label{it:ham-act-badb=ham-inv}
      if $p_1\ge0$ and $p_2\ge1$, then for all $k\ge 0$, $\ham^{-k}f\in\mcAb(p_1+2k, p_2+2k)$;
    \item \label{it:ham-act-badb=vla}
      the function $\vla f\in\mcAb(p_1, p_2)$;
    \item \label{it:ham-act-badb=b1}
      the function $b_1 \frac{\pr f}{\pr b_1}\in\mcAb(p_1, p_2)$.

  \end{enumerate}
\end{lemma}

An important example of $\mcA$-functions is given by the iterated resonances of $\ham$ and the vectorial resonances of $\mhamp$, as defined in~\eqref{eq:reson}, \eqref{eq:vec-reson}.

\begin{corollary} \label{co:resona}
  For all $k\in \mbN_+$, the functions $T_k, \phia[k], \phib[k]\in\mcA(2k+1,2k-1)$.

\end{corollary}

\begin{proof}
  Recall from \eqref{eq:vec-reson} that $\phia[k] = T_k\R^{k-1}e_1$ and $\phib[k] = T_k\R^{k-1}e_2$. Recall from~\eqref{eq:reson} that $T_k = \ham^{-k}\lphi$. Observe from~\eqref{eq:lphi} that $\lphi\in\mcA(1,-1)$. Hence, applying Lemma~\itref{le:ham-act-adb}{it:ham-act-adb=ham-inv} to $f=\lphi$, we obtain $T_k\in\mcA(2k+1,2k-1)$.
\end{proof}

We next record the inclusion and product rules, which translate the degree calculus of the scalar spaces into the algebra of admissible functions.

\begin{lemma}[Inclusions involving $\mcAg$]
  \label{le:inclu-madb}
  The following inclusions hold:
  \begin{enumerate}[label={\rm (\roman*)}]
    \item \label{it:inclu-madb=A}
      $\mcA(p_1,p_2) \subset \mcAg(p_1,p_2)$;
    \item \label{it:inclu-madb=Ab}
      $\mcAb(p_1,p_2) \subset \mcAg(p_1,p_2) \subset \mcAb(p_1,p_2+2)$.

  \end{enumerate}
\end{lemma}

\begin{proof}
  (i) follows directly from the definitions of $\mcA$ and $\mcAg$. To prove~\ref{it:inclu-madb=Ab}, note that
  \begin{equation}
    \label{pf:inclu-madb=f-k-bnd}
    |F_j| \les \lgb, \quad |\bar{F}_j| + |K_j| \les 1,
    \quad y\ge1, \quad j\in\mbZ,
  \end{equation}
  and thus any $f\in\mcAb(p_1,p_2)$ admits
  \begin{align*}
    |f\sd[k]|
    & \les \frac{y^{p_2-k}}{\lgb} (\lgb + 1) + y^{p_2-k-2}\jbk[\log y]^{c_{p_2}} \les y^{p_2-k}\jbk[\log y]^{c_{p_2}}, \\
    \bigg|\frac{\pr^\ell f\sd[k]}{\pr b_1^\ell}\bigg|
    & \les \frac{y^{p_2-k}}{b_1^\ell\lgb} + \frac{y^{p_2-k-2}\jbk[\log y]^{c_{p_2}}}{b_1^\ell\lgb}
    \les \frac{y^{p_2-k}\jbk[\log y]^{c_{p_2}}}{b_1^\ell\lgb},
  \end{align*}
  for all $k\in\mbN$ and $\ell\ge 1$. This shows that $f\in\mcAg(p_1,p_2)$. The inclusion $\mcAg(p_1,p_2) \subset \mcAb(p_1,p_2+2)$ follows from the definitions of the two classes, since the $y\ge 1$ bound of $\mcAg(p_1,p_2)$ coincides with the last term in Definition~\itref{def:badb}{it:def-badb=2}.
\end{proof}

\begin{remark}
  \label{re:trivial-inclu}
  (i) All three classes satisfy the following trivial inclusion: for all $p_1, p_2\in\mbZ$,
  \begin{equation}
    \label{re:trivial-inclu-form}
    \mcA(p_1,p_2) \subset \mcA(p_1-2, p_2) \cap \mcA(p_1, p_2+2)
  \end{equation}
  and the same holds with $\mcA$ replaced by $\mcAb$ or $\mcAg$. Indeed, for $y\le 1$ it suffices to factor out $y^2$ in Definition~\itref{def:adb}{it:def-adb=1}, the constant term then vanishing; for $y\ge1$ the bounds in Definitions~\ref{def:adb}, \ref{def:badb} and \ref{def:ag} only improve as the second index increases. Iterating~\eqref{re:trivial-inclu-form} thus yields $\mcA(p_1,p_2)\subset\mcA(p_1-2k,p_2+2k)$ for every $k\in\mbN$, and similarly for $\mcAb$ and $\mcAg$.

  (ii) By Lemma~\itref{le:inclu-madb}{it:inclu-madb=Ab}, $\mcAg(p_1,p_2)\subset\mcAb(p_1,p_2+2)$, so that $\mcAg$ may be viewed as a subclass of $\mcAb$ up to a harmless shift of the second index. The same-index rules for $\ham$, $\vla$ and $b_1\partial_{b_1}$ on $\mcAg$ follow directly from its differentiated definition and the coefficient bounds for these operators; they are not inferred from the shifted inclusion.

\end{remark}

\begin{lemma}[Products in $\mcAg$]
  \label{le:produ-madb}
  Let $f\in\mcAg(p_1,p_2)$. The following hold:
  \begin{enumerate}[label={\rm (\roman*)}]
    \item \label{it:produ-madb=A}
      if $g\in\mcA(q_1, q_2)$, then $fg\in\mcAg(p_1+q_1, p_2+q_2)$;
    \item \label{it:produ-madb=Ab}
      if $g\in\mcAb(q_1, q_2)$, then $fg\in\mcAg(p_1+q_1, p_2+q_2)$;
    \item \label{it:produ-madb=Ag}
      if $g\in\mcAg(q_1, q_2)$, then $fg\in\mcAg(p_1+q_1, p_2+q_2)$.

  \end{enumerate}
\end{lemma}

\begin{proof}
  In view of Lemma~\ref{le:inclu-madb}, it suffices to prove the claim~\ref{it:produ-madb=Ag} and then~\ref{it:produ-madb=A} and~\ref{it:produ-madb=Ab} follow.
  For $y\le 1$, from Definition~\itref{def:badb}{it:def-badb=1}, the function $fg$ is a sum of terms of the form $\wh{\delta}_{i}\wh{\delta}'_{j}\wh{f}_{i}\wh{g}_{j}$, where by Leibniz,
  \[
    \bigg|\frac{\pr^{\ell}}{\pr b_1^{\ell}}(\wh{\delta}_{i}\wh{\delta}'_{j})\bigg|
    \les \sum_{\ell_1+\ell_2=\ell} \frac{1}{b_1^{\ell_1+\ell_2}\lgb}
    \les \frac{1}{b_1^{\ell}\lgb},
  \]
  and from Definition~\itref{def:adb}{it:def-adb=1}, the spatial factors satisfy $\wh{f}_{i}\wh{g}_{j} = y^{p_1+q_1}(C+O(y^2))$ for some real constant $C$. Thus $fg$ satisfies Definition~\itref{def:ag}{it:def-ag=1}. For $y\ge1$, we apply Leibniz for suitable derivatives (see \citep[Lemma D.1]{raphael_quantized_2014}) together with~\eqref{pf:inclu-madb=f-k-bnd}. For all $k\in\mbN$, this yields
  \[
    |(fg)\sd[k]|
    \les \sum_{0\le l\le k} \frac{1}{y^{k-l}} \sum_{l_1+l_2=l} |f\sd[l_1]|\,|g\sd[l_2]|
    \les y^{p_2+q_2-k}\jbk[\log y]^{c_{p_2}}.
  \]
  Similarly, for all $\ell \ge 1$,
  \[
    \bigg|\frac{\pr^\ell}{\pr b_1^\ell}(f g)\sd[k]\bigg|
    \les \sum_{0\le l\le k} \frac{1}{y^{k-l}} \sum_{\substack{l_1+l_2=l\\ \ell_1+\ell_2=\ell}} \bigg|\frac{\pr^{\ell_1} f\sd[l_1]}{\pr b_1^{\ell_1}}\bigg| \, \bigg|\frac{\pr^{\ell_2} g\sd[l_2]}{\pr b_1^{\ell_2}} \bigg|\les \frac{y^{p_2+q_2-k}\jbk[\log y]^{c_{p_2}}}{b_1^\ell \lgb}.
  \]
  Thus $fg$ satisfies Definition~\itref{def:ag}{it:def-ag=2}. This concludes the proof.
\end{proof}

\begin{lemma}[$L^2$ norm of $\mcA,\mcAb,\mcAg$]
  \label{le:l2-norm-adb}
  Let $p_1\ge 0, p_2\ge 1$. The following estimates hold:
  \begin{enumerate}[label={\rm (\roman*)}]
    \item \label{it:l2-norm-adb=A}
      if $f\in\mcA(p_1,p_2)\cup\mcAg(p_1,p_2)$, then
      \[ \|f\|_{L^2(y\le 2S_1)}^2 \les \frac{\lgb^{C}}{b_1^{p_2+1}}; \]
    \item \label{it:l2-norm-adb=Ab}
      if $f\in\mcAb(p_1,p_2)$, then
      \[ \|f\|_{L^2(y\le 2S_1)}^2 \les
        \frac{\lgb^{2p_2-4}}{b_1^{p_2+1}}.
      \]
  \end{enumerate}
\end{lemma}

\begin{proof}
  We first show~\ref{it:l2-norm-adb=Ab}. Using the bound in Definition~\itref{def:badb}{it:def-badb=2}, we obtain the following upper bound for $\|f\|_{L^2(y\le 2S_1)}^2$:
  \begin{align*}
    & 1+ \int_{1}^{2S_1} \biggl\{\frac{y^{2p_2}}{\lgb^2} \biggl[\jbk[\log(\sqrt{b_1}y)]^2\one_{y\le 2S_0} + \frac{\one_{y\ge 2S_0}\one_{p_2\ge 2}}{y^4b_1^2} \biggr] + y^{2p_2-4}\jbk[\log y]^{2c_{p_2}} \biggr\}y\dd y \\
    & \les \frac{1}{b_1^{p_2+1}\lgb^2} + \frac{\lgb^{2p_2-4}}{b_1^{p_2+1}}\one_{p_2\ge 2} + \frac{\lgb^{2c_{p_2}+2p_2-1}}{b_1^{p_2-1}}                                                                                 \\
    & \les \frac{\lgb^{2p_2-4}}{b_1^{p_2+1}}.
  \end{align*}
  To prove~\ref{it:l2-norm-adb=A}, it suffices to consider $\mcAg(p_1,p_2)$, since $\mcA(p_1,p_2)\subset\mcAg(p_1,p_2)$. Note that by Lemma~\ref{le:inclu-madb}, we have $\mcAg(p_1,p_2)\subset\mcAb(p_1,p_2+2)$. Thus the bound for $\|f\|_{L^2(y\le 2S_1)}^2$ reduces to the last part of the integral on the right-hand side above with $p_2\to p_2+2$, which yields $b_1^{-(p_2+1)}\lgb^{C}$, as desired.
\end{proof}

\subsection{Tail manipulation}

This subsection is devoted to the tail manipulation. More precisely, we introduce a radial profile $\vsig_{b_1}$ that cancels the slowly decaying part of $\vla T_1 - T_1$.

\begin{lemma}[Function $\vsig_{b_1}$]
  \label{le:vsig}
  There exists a radial function $\vsig_{b_1}(y)$ such that
  \[
    \vsig_{b_1} =
    \begin{cases}
      \cb T_1  & \mbox{for}\;\;y\le\frac{S_0}{4}, \\
      -4\Gamma & \mbox{for}\;\;y\ge 2S_0
    \end{cases},
  \]
  where $S_0 = b_1^{-\frac12}$. Furthermore, it holds that
  \begin{equation}
    \label{le:vsig-asymp}
    \vsig_{b_1}(y)
    = -y \frac{\int_{0}^{y}\chi_{\frac{S_0}{4}}\lphi^2(x)x\dd x}{\int \chi_{\frac{S_0}{4}}\lphi^2(x)x\dd x} + O\biggl(\frac{\jbk[y]}{\lgb}\biggr)
    \for 1\le y\le 2S_0.
  \end{equation}
\end{lemma}

\begin{proof}
  Recall the definition of $\cb$. We set
  \[
    d_{b_1} = \cb\int\chi_{\frac{S_0}{4}}\lphi(x)\Gamma(x)x\dd x = O\biggl(\frac{1}{b_1\lgb}\biggr),
  \]
  and define
  \[
    \vsig_{b_1} =\ham^{-1}\Bigl\{ \cb\chi_{\frac{S_0}{4}} \lphi -d_{b_1}\ham\bigl[(1-\chi_{S_0})\lphi\bigr] \Bigr\}.
  \]
  Let us verify the asymptotics in three regions. First for $y\le \frac{S_0}{4}$, we have $\chi_{\frac{S_0}{4}} = 1$ and $1-\chi_{S_0} =0$, which implies
  \[
    \vsig_{b_1} = \ham^{-1}(\cb\lphi) = \cb T_1.
  \]
  Then for $y\ge 2S_0$, we have $1-\chi_{S_0} = 1$. Thus from the definitions of $\cb, d_{b_1}$ and~\eqref{eq:inverse-ham},
  \begin{align*}
    \vsig_{b_1}
    & = \cb\lphi \int_{0}^{y}\chi_{\frac{S_0}{4}} \Gamma(x) \lphi(x) x\dd x - \cb\Gamma \int_{0}^{y} \chi_{\frac{S_0}{4}} \lphi^2(x) x\dd x - d_{b_1} (1-\chi_{S_0})\lphi \\
    & = d_{b_1} \lphi -4 \Gamma - d_{b_1} \lphi                                                                                                                         \\
    & = -4 \Gamma.
  \end{align*}
  Finally for $1\le y\le 2S_0$, note that
  \[
    \ham\bigl[(1-\chi_{S_0})\lphi\bigr] = O\biggl(\frac{\one_{S_0\le y\le 2S_0}}{y^3}\biggr).
  \]
  Using this together with~\eqref{eq:lphi} and~\eqref{eq:inverse-ham}, we have
  \begin{align*}
    \vsig_{b_1}
    & = \cb O\biggl(\frac{1}{y}\int_{0}^{y} \chi_{\frac{S_0}{4}} x\dd x\biggr) - \cb\biggl(\frac{y}{4}+O\biggl(\frac{\log y}{y}\biggr)\biggr) \int_{0}^{y}\chi_{\frac{S_0}{4}} \lphi^2 x\dd x \\
    & \quad + O\biggl( \frac{d_{b_1}}{y} \int_{0}^{y} \frac{\one_{S_0\le x\le 2S_0}}{x}\dd x \biggr) + O\biggl( d_{b_1} y \int_{0}^{y} \frac{\one_{S_0\le x\le 2S_0}}{x^3} \dd x \biggr)      \\
    & = -y \frac{\int_{0}^{y}\chi_{\frac{S_0}{4}}\lphi^2(x)x\dd x}{\int \chi_{\frac{S_0}{4}}\lphi^2(x)x\dd x} + O\biggl(\frac{\jbk[y]}{\lgb}\biggr).
  \end{align*}
  This completes the proof.
\end{proof}

We define for $k\ge 1$ the scalar profiles
\begin{equation}
  \label{def:upsilon}
  \up_k = \vla T_k -(2k-1)T_k -\vsig_k
  \with
  \vsig_k = \ham^{-(k-1)}\vsig_{b_1}.
\end{equation}
Denote their vectorial counterparts by
\[
  \begin{cases}
    \up_{k,0} = \up_k\R^{k-1}e_1, \\
    \up_{0,k} = \up_k\R^{k-1}e_2,
  \end{cases}
  \andd
  \begin{cases}
    \vsig_{k,0} = \vsig_k\R^{k-1}e_1, \\
    \vsig_{0,k} = \vsig_k\R^{k-1}e_2.
  \end{cases}
\]

\begin{lemma}[Tail manipulation]
  \label{le:upsilon-deg}
  For all $k\in\mbN_+$, the functions $\up_k,\up_{k,0},\up_{0,k}\in\mcAb(2k+1,2k-1)$.

\end{lemma}

\begin{proof}
  It suffices to prove $\up_k\in\mcAb(2k+1,2k-1)$ for all $k\ge 1$, since $\R^{k-1}e_1, \R^{k-1}e_2$ are fixed rotations, which do not affect the regularity. We proceed by induction.

  \emph{Case $k=1$}. First for $y\le 1$, by Lemma~\ref{le:vsig},
  \[
    \up_1 = \vla T_1 -(1+\cb)T_1.
  \]
  Since $T_1\in\mcA(3,1)$, the function $\up_1$ satisfies Definition~\itref{def:badb}{it:def-badb=1} with $I=2$, and $\up_1\in\mcAb(3,1)$ for $y\le 1$.
  Then for $1\le y\le 2S_0$, from~\eqref{eq:t1-asymp} we find
  \[
    \vla T_1 - T_1
    = - y + O\biggl(\frac{(\log y)^2}{y}\biggr),
  \]
  which together with~\eqref{le:vsig-asymp} yields
  \[
    \up_1(y)
    = -y +y\frac{\int_0^y\chi_{\frac{S_0}{4}}\lphi^2(x) x \dd x}{\int\chi_{\frac{S_0}{4}}\lphi^2(x) x \dd x} +O\biggl(\frac{\jbk[y]}{\lgb}\biggr) +O\biggl(\frac{(\log y)^2}{y}\biggr).
  \]
  Observe that
  \begin{equation}
    \label{pf:upsilon-deg=frac-cancel}
    1 - \frac{\int_0^y\chi_{\frac{S_0}{4}}\lphi^2(x) x \dd x}{\int\chi_{\frac{S_0}{4}}\lphi^2(x) x \dd x}
    = O\Biggl(\frac{\int_y^{S_0}\chi_{\frac{S_0}{4}}\lphi^2(x) x \dd x}{\lgb}\Biggr)
    = O\biggl(\frac{\left|\log (\sqrt{b_1}y)\right|}{\lgb} \biggr).
  \end{equation}
  Thus for $y\le 2S_0$,
  \[
    |\up_1(y)| \les \frac{y}{\lgb}\,F_1 + \frac{(\log y)^2}{y}.
  \]
  Now for $y\ge 2S_0$, notice that the left-hand side of~\eqref{pf:upsilon-deg=frac-cancel} vanishes, and hence the above bound remains valid.

  Next, we estimate the suitable derivatives. From the definition of $\vsig_{b_1}$, \eqref{pf:adm=au-u} and~\eqref{pf:upsilon-deg=frac-cancel}, we have for $1\le y\le 2S_0$ that
  \begin{align*}
    \oa \vsig_{b_1}
    & = \frac{1}{y\lphi} \int_0^y \Bigl\{ \cb\chi_{\frac{S_0}{4}}\lphi -d_{b_1}\ham\bigl[(1-\chi_{S_0})\lphi\bigr] \Bigr\} \lphi(x)x \dd x                                              \\
    & = \frac{4}{y\lphi} \frac{\int_0^y\chi_{\frac{S_0}{4}}\lphi^2(x) x \dd x}{\int \chi_{\frac{S_0}{4}}\lphi^2(x) x \dd x} + d_{b_1} O\biggl(\frac{\one_{S_0\le y\le 2S_0}}{y^2}\biggr) \\
    & = 2 +O\biggl(\frac{\left|\log (\sqrt{b_1}y)\right|}{\lgb}\biggr)+O\biggl(\frac{1}{y^2}\biggr).
  \end{align*}
  For $y\ge 2S_0$,
  \[
    \oa \vsig_{b_1} = \frac{4}{y\lphi} = 2 +O\biggl(\frac{1}{y^2}\biggr).
  \]
  Moreover, again by~\eqref{eq:t1-asymp},
  \[
    \oa(\vla T_1 - T_1)
    = 2+O\biggl(\frac{(\log y)^2}{y^2}\biggr)
    \for y> 1.
  \]
  Inserting these into~\eqref{def:upsilon}, we obtain
  \[
    |\oa \up_1|
    = |\oa(\vla T_1 - T_1) - \oa\vsig_{b_1}|
    \les \frac{1}{\lgb}\,F_0 + \frac{(\log y)^2}{y^2}
    \for y>1.
  \]
  The results for higher-order suitable derivatives are obtained similarly by a recursion argument.

  Finally, we treat the $b_1$-derivatives. It suffices to bound the $b_1$-derivatives of $\vsig_{b_1}$, since $\vsig_{b_1}$ is the only part of $\up_1$ that depends on $b_1$. For $y\ge 2S_0$, from the asymptotics in Lemma~\ref{le:vsig},
  \[
    \frac{\pr^\ell \up_1}{\pr b_1^\ell} = -\frac{\pr^\ell \vsig_{b_1}}{\pr b_1^\ell}
    = \frac{\pr^\ell}{\pr b_1^\ell} (-4\Gamma) = 0
    \for \ell\ge 1,
  \]
  and thus the desired bound holds. For $1\le y\le 2S_0$, by definition,
  \begin{align*}
    \frac{\pr^\ell \up_1}{\pr b_1^\ell} = -\frac{\pr^\ell \vsig_{b_1}}{\pr b_1^\ell}
    & = -\lphi\int_0^y \frac{\pr^\ell}{\pr b_1^\ell}\Bigl\{\cb\chi_{\frac{S_0}{4}}\Bigr\} \Gamma(x)\lphi(x) x\dd x                                                                          \\
    & \quad +\Gamma\int_0^y \frac{\pr^\ell}{\pr b_1^\ell} \Bigl\{\cb\chi_{\frac{S_0}{4}}\Bigr\} \lphi^2(x) x\dd x + \frac{\pr^\ell}{\pr b_1^\ell}\Bigl\{d_{b_1}(1-\chi_{S_0})\Bigr\} \lphi,
  \end{align*}
  where
  \begin{gather*}
    \biggl|\frac{\pr^\ell}{\pr b_1^\ell}\Bigl\{\cb\chi_{\frac{S_0}{4}}\Bigr\}\biggr|
    \les \frac{\one_{y\le 2S_0}}{b_1^\ell\lgb^2} +\frac{\one_{S_0/4 \le y\le 2S_0/2}}{b_1^{\ell}\lgb}, \\
    \biggl|\frac{\pr^\ell}{\pr b_1^\ell}\Bigl\{ d_{b_1}(1-\chi_{S_0}) \Bigr\}\biggr|
    \les \frac{\one_{S_0\le y\le 2S_0}}{b_1^{\ell+1}\lgb}.
  \end{gather*}
  Using~\eqref{eq:inverse-ham} and Leibniz, we have
  \[
    \biggl|\frac{\pr^\ell \up_1}{\pr b_1^\ell}\biggr|
    \les \frac{y\jbk[\log y]}{b_1^{\ell} \lgb^2} \one_{y\le 2S_0}
    \les \frac{y}{b_1^{\ell}\lgb}\bar{F}_1.
  \]
  The $b_1$-derivatives for $(\up_1)\sd[k]$ are obtained similarly by performing the estimates on separate regions, which gives
  \[
    \biggl|\frac{\pr^\ell (\up_1)\sd[k]}{\pr b_1^\ell}\biggr|
    \les \frac{y^{1-k}\jbk[\log y]}{b_1^{\ell} \lgb^2} \one_{y\le 2S_0}
    \les \frac{y^{1-k}}{b_1^{\ell}\lgb}\bar{F}_{1-k},
    \for k\ge 0.
  \]
  This yields the estimate for $y\ge1$ and proves that $\up_1\in\mcAb(3,1)$.

  \emph{Case $k \to k+1$}. For any smooth function $f$, we have
  \[
    \ham (\vla f) = \vla\ham f +2\ham f -\frac{\vla V}{y^2} f.
  \]
  This yields the recurrence formula:
  \begin{align*}
    \ham \up_{k+1}
    & = \ham(\vla T_{k+1}) -(2k+1)\ham T_{k+1} -\ham\vsig_{k+1}                   \\
    & = \vla\ham T_{k+1} -(2k-1)\ham T_{k+1} -\frac{\vla V}{y^2} T_{k+1} -\vsig_k \\
    & = \up_k -\frac{\vla V}{y^2}T_{k+1}.
  \end{align*}
  Then $\up_{k+1}$ is obtained via an application of~\eqref{eq:inverse-ham}. Thus, in view of Lemma~\itref{le:ham-act-badb}{it:ham-act-badb=ham-inv}, it suffices to show that
  \[
    \up_k,\; \frac{\vla V}{y^2}T_{k+1} \in \mcAb(2k+1,2k-1).
  \]
  The function $\up_k$ satisfies this by induction. For the other term, note that
  \[
    \frac{\vla V}{y^2}
    = -\frac{16(1-y^2)}{(1+y^2)^3}
    = \begin{cases}
      -16 + O(y^2)
      & \mbox{for $y\le 1$}, \\
      \frac{16}{y^4}\bigl(1+O\bigl(\frac{1}{y^2}\bigr)\bigr)
      & \mbox{for $y\ge1$}.
    \end{cases}
  \]
  Hence $y^{-2}\vla V\in\mcAg(0,-4)$. This, together with $T_{k+1}\in\mcA(2k+3, 2k+1)$ by Corollary~\ref{co:resona} and Lemma~\ref{le:produ-madb}, implies the desired result
  \[
    \frac{\vla V}{y^2}T_{k+1}
    \in\mcAg(2k+3, 2k-3)\subset\mcAb(2k+1, 2k-1).
  \]
  This concludes the proof.
\end{proof}

\subsection{Homogeneous functions}
\label{sub:homo-func}
In this subsection, we introduce the homogeneous function spaces $\mcH, \mcHb, \mcHg$, and also collect some basic tools for the error estimate in the construction of the approximate solution of Subsection~\ref{sub:app-prof}.

Let $k\in\mbZ_+$. For any $\rho\in\mbN^k$, we define the quantities
\[
  |\rho|_1 = \sum_{\ell=1}^{k} \rho_{\ell},
  \andd
  |\rho|_2 = \sum_{\ell=1}^{k} \ell\rho_{\ell}.
\]

\begin{definition}[Function spaces $\mcH, \mcHb, \mcHg$]
  \label{def:homo}
  \label{def:homo-b}
  \label{def:homo-g}
  Given $(p_1,p_2,p_3)\in\mbZ\times\mbZ\times\mbZ$, we define the function spaces $\mcH(p_1,p_2,p_3)$, $\mcHb(p_1,p_2,p_3)$ and $\mcHg(p_1,p_2,p_3)$ as the sets of all smooth radial functions $f(a, b, y)$ satisfying
  \[
    f(a, b, y) = \sum_{|\rho|_2+|\tau|_2=p_1} a^{\rho} b^{\tau} \wh{f}_{\rho,\tau}(b_1,y), \footnotemark
  \]
  where $\rho\in\mbN^L, \tau\in\mbZ\times\mbN^{L-1}$ and the coefficients $\wh{f}_{\rho,\tau}$ belong to $\mcA(p_2,p_3)$, $\mcAb(p_2,p_3)$ and $\mcAg(p_2,p_3)$, respectively, in the sense of Definitions~\ref{def:adb}, \ref{def:badb} and \ref{def:ag}. We call $p_1$ the parameter degree of $f$ and $(p_2,p_3)$ its spatial degree.
  \footnotetext{The sum contains only finitely many nonzero terms.}
\end{definition}

\begin{lemma}[Inclusion for $\mcH,\mcHb$ and $\mcHg$]
  \label{le:incl-homo}
  The following inclusions hold:
  \begin{enumerate}[label={\rm (\roman*)}]
    \item \label{it:incl-homo=triv}
      trivial inclusion (the same inclusion holds for $\mcHb$ and $\mcHg$):
      \[
        \mcH(p_1,p_2,p_3) \subset \mcH(p_1,p_2-2,p_3) \cap \mcH(p_1,p_2,p_3+2);
      \]
    \item \label{it:incl-homo=incl}
      inclusion between spaces:
      \begin{align*}
        \mcH(p_1,p_2,p_3)  & \subset \mcHg(p_1,p_2,p_3),                              \\
        \mcHb(p_1,p_2,p_3) & \subset \mcHg(p_1,p_2,p_3) \subset \mcHb(p_1,p_2,p_3+2).
      \end{align*}
  \end{enumerate}
\end{lemma}

\begin{lemma}[Products with $\mcHg$]
  \label{le:prod-homo}
  Let $f\in\mcHg(p_1,p_2,p_3)$. We have:
  \begin{enumerate}[label={\rm (\roman*)}]
    \item \label{it:prod-homo=H}
      if $g\in\mcH(q_1,q_2,q_3)$, then $fg\in\mcHg(p_1+q_1, p_2+q_2, p_3+q_3)$;
    \item \label{it:prod-homo=Hb}
      if $g\in\mcHb(q_1,q_2,q_3)$, then $fg\in\mcHg(p_1+q_1, p_2+q_2, p_3+q_3)$.

  \end{enumerate}
\end{lemma}

\begin{lemma}[$L^2$ norm of $\mcH,\mcHb,\mcHg$]
  \label{le:l2-norm-homo}
  Let $p_1, p_2\ge0$, and $p_3\ge1$. The following estimates hold:
  \begin{enumerate}[label={\rm (\roman*)}]
    \item \label{it:l2-norm-homo=H}
      if $f\in\mcH(p_1,p_2,p_3)\cup\mcHg(p_1,p_2,p_3)$, then
      \[
        \|f\|_{L^2(y\le 2S_1)}^2 \les b_1^{2p_1-p_3-1}\lgb^{C};
      \]
    \item \label{it:l2-norm-homo=Hb}
      if $f\in\mcHb(p_1,p_2,p_3)$, then
      \[
        \|f\|_{L^2(y\le 2S_1)}^2 \les
        b_1^{2p_1-p_3-1}\lgb^{2p_3-4}.
      \]
  \end{enumerate}
\end{lemma}

\subsection{Approximate profiles}
\label{sub:app-prof}
In this subsection, we construct the \homo\ profile $\x$, which completes the construction of the approximate solution~$\wz$. The resulting error is also estimated. This is based on the formal computation of Lemma~\ref{le:form-comp} and a systematic use of the admissible and homogeneous spaces introduced in Subsections~\ref{sub:adm-func}--\ref{sub:homo-func}.

\begin{lemma}[Approximate profiles]
  \label{le:app-sol}
  Given any $L\in\mbN_+$, let $M=M(L)>0$ be a large enough constant. There exist a constant $b_{1}^{\ast} = b_{1}^{\ast}(M)>0$ and a vector-valued function $\x = \sum_{k=2}^{L+2} \xx[k]$ with
  \[
    \xx[k] \in \mcHb(k,2k+1,2k-1)
    \for
    2\le k\le L+2,
  \]
  such that the vector $\wz = \vp + \x + \gamma_0 e_3$ is an approximate solution to \eqref{eq:vect-form} in the sense that
  \[
    \prs\wz -\frac{\la_s}{\la}\vla\wz +\vtha_s\z\R\wz +\hj[\wz](\mham\wz+\hp\lphi) = \mod + \psiz,
  \]
  where $\mod$ and $\psiz$ are given in Lemma~\ref{le:form-comp}. Furthermore, $\psiz$ admits:
  \begin{enumerate}[label={\rm (\roman*)}]
    \item \label{it:app-sol=sob}
      the Sobolev bounds:
      \begin{gather*}
        \int_{y\le 2S_1}|(\R\mham)^{L+1}\psiz|^2
        \les \frac{b_1^{2L+4}}{\lgb^2},  \\
        \int_{y\le 2S_1}|\ma\mhamp^L\psiz|^2
        \les \delta(b_{1}^{\ast})\frac{b_1^{2L+2}}{\lgb^2}, \\
        \int_{y\le 2S_1}|\mhamp^k\psiz|^2 \les b_1^{2k+2}\lgb^C
      \end{gather*}
      for $1\le k\le L$ and some universal constant $C>0$;
    \item \label{it:app-sol=weig}
      weighted bound:
      \[
        \sum_{\ell=0}^{k} \int_{y\le 2S_1} \frac{|\pry^{\ell}\psiz|^2}{\jbk[y]^{2(k-\ell)}}
        \les \begin{cases}
          b_1^{2L+4}               & \mbox{for}\;\; k= 2L+2, \\
          b_1^{k+2}\lgb^{4L+2-2k}, & 0\le k\le 2L+1;
        \end{cases}
      \]
    \item \label{it:app-sol=local}
      local bound: let $\xm_M$ and $\xm_{\ell,0},\xm_{0,\ell}$ be defined by \eqref{eq:local-reson} and \eqref{eq:vec-local-reson}. Then
      \begin{align*}
        |(\psiz, \xm_M\eone)| & + |(\psiz, \xm_M\etwo)|                                                      \\
        & + \sum_{\ell=1}^L \Bigl(|(\psiz,\xm_{\ell,0})| +|(\psiz,\xm_{0,\ell})|\Bigr)
        \les M^C b_1^{L+\frac32}.
      \end{align*}
  \end{enumerate}
\end{lemma}

\begin{proof}
  For brevity, we denote $\hj = \hj[\wz]$, $\sib = \si_{L+1}$ and
  \[
    \Phi_k = a_k\phia[k] +b_k\phib[k]
    \for 1\le k\le L.
  \]

  \step[1, Choice of $\x$]
  Reformulating $\psizone$ by~\eqref{def:upsilon}, we have
  \[
    \psizone
    = \hj(\mhamp\x-\tea) - \sum_{k=1}^L \hj\R\sib \Bigl\{ b_1a_k(\vsig_{k,0}-\cb\phia[k]) +b_1b_k(\vsig_{0,k} -\cb\phib[k]) \Bigr\},
  \]
  with
  \begin{align*}
    \tea
    & = \R\sib \Bigl\{ b_1\mcT\x +a_1\mcP\x +\mcQ\x -a_1(1+\z)\R\vp -\gamma_0\mcQ\vp \Bigr\}                                                                                        \\
    & \quad + \R\sib\sum_{k=1}^L\bigl(b_1a_k\up_{k,0} +b_1b_k\up_{0,k}\bigr) +\zt\pry\gamma_0 e_1 +\sib \Bigl(\zt\Bigl(\pry +\tfrac{\z}{y}\Bigr)\alpha_0 -\Delta\gamma_0\Bigr)\wzp.
  \end{align*}
  Formally, we let $\x$ be a linear combination of monomials of the form $a^{\rho}b^{\tau}\x_{\rho,\tau}$, where $\rho\in\mbN^L$, $\tau\in\mbZ\times\mbN^{L-1}$ and $\x_{\rho,\tau} = O(1)$ when $y\sim 1$. Thus, determining $\x$ amounts to finding the functions $\x_{\rho,\tau}$ for all $\rho,\tau$. We group $\x_{\rho, \tau}$ by the value $|\rho|_2 + |\tau|_2$, which we call the parameter degree of $\x_{\rho, \tau}$.

  Since the function $\gamma_0$ is a polynomial in $\wzp$, as in the proof of Lemma~\ref{le:form-comp}, the monomials in $\gamma_0$ can also be classified by parameter degrees. In view of~\eqref{pf:form-comp=si} and the expression of $\tea$, the same holds for $\sib, \tea$. Thus we let
  \[
    \x = \sum_{k=2}^{L+2} \xx[k], \quad
    \gamma_0 = \sum_{k=2}^{C_{\gamma_0}} \gamma_{0,k}, \quad
    \sib = \sum_{k=0}^{C_{\sib}} \sib_k, \quad
    \tea = \sum_{k=2}^{C_{\tea}} \tea_k,
  \]
  where $C_{\gamma_0} = 2(L+1)(L+2)$,
  $C_{\sib} = 2{(L+1)}^2(L+2)$, $C_{\tea} = (L+2)(2L^2+6L+5)$. All homogeneous components are understood to vanish outside their prescribed index ranges. We claim that there exists a sequence ${(\xx[k])}_{k=2}^{L+2}$ such that
  \begin{align}
    \xx[k]       & \in \mcHb(k, 2k+1, 2k-1)
    \for 2\le k\le L+2,
    \label{pf:app-sol=xk-deg}                \\
    \gamma_{0,k} & \in \mcHg(k, 2k+2, 2k-2)
    \for 2\le k\le C_{\gamma_0},
    \label{pf:app-sol=gamma-0-k-deg}
  \end{align}
  and
  \begin{equation}
    \label{pf:app-sol=tea-k-deg}
    \tea_k \in \begin{cases}
      \mcHb(k, 2k-1, 2k-3) &
      \mbox{for}\quad 2\le k\le L+3, \\
      \mcHg(k, 2k-1, 2k-5) &
      \mbox{for}\quad L+4 \le k\le C_{\tea}.
    \end{cases}
  \end{equation}
  From these estimates, we have
  \begin{equation}
    \label{pf:app-sol=wz-l-infty}
    \bigl\|\jbk[\log y]^{C}y^{k}\pry^{k}\wz\bigr\|_{L^\infty(y\le 2S_1)} \les b_1^{1/2}\lgb^C
    \for k\ge 0.
  \end{equation}
  The construction is recursive: at degree $k$, substitute the already
  constructed $\xx[2],\ldots,\xx[k-1]$ into the expression for
  $\tea$ above, extract its degree-$k$ source $\tea_k$, and set
  \[
    \xx[k]=\ham^{-1}\tea_k
    \for 2\le k\le L+2,
  \]
  where the inverse in~\eqref{eq:inverse-ham} acts componentwise on the
  first two components. Indeed, the linear terms involving $\x$ in
  $\tea$ raise the parameter degree by one, while its remaining
  $\x$-dependent terms contain at least one additional positive-degree
  factor. Thus $\tea_k$ does not involve $\xx[k]$.
  All subsequent occurrences of $\xx[k]$, $\gamma_0$, $\sib$, and
  $\wz$ refer to this recursive construction.

  \step[2, Proof of the claim]
  We prove \eqref{pf:app-sol=xk-deg}--\eqref{pf:app-sol=tea-k-deg} by induction.

  \emph{Case $k=2$}. By definition, $\sib_0 = 1$ and $\sib_1 = 0$. Expanding $\tea$ at parameter degree $2$ gives
  \begin{align*}
    \tea_2
    & = a_1(1+\z)(a_1\phia[1] +b_1\phib[1]) + \R(b_1a_1\up_{1,0} +b_1^2\up_{0,1})                            \\
    & \quad +\zt \pry \gamma_{0,2} e_1 +\zt\Bigl(\pry +\tfrac{\z}{y}\Bigr) a_1T_1(a_1\phia[1] +b_1\phib[1]).
  \end{align*}
  Since $\phib[1], \up_{0,1}$ belong to the same spaces as $\phia[1], \up_{1,0}$ respectively, it suffices to treat $\phia[1], \up_{1,0}$ and  $\gamma_{0,2}$. First, note that
  \[
    1+\z = \frac{\zt}{2}
    = \begin{cases}
      2 + O(y^2),
      & y\le 1, \\
      \frac{2}{y^2}\bigl(1+O\bigl(\frac{1}{y^2}\bigr)\bigr),
      & y\ge1,
    \end{cases}
    \quad\Longrightarrow\quad
    1+\z \in \mcAg(0,-2).
  \]
  Using Lemma~\ref{le:produ-madb} and Corollary~\ref{co:resona}, we find $(1+\z)\phia[1] \in \mcAg(3,-1)$. Together with Lemma~\ref{le:inclu-madb}, this implies
  \[
    a_1^2(1+\z)\phia[1] \in \mcHb(2,3,1).
  \]
  Then, directly from Lemma~\ref{le:upsilon-deg}, we have $b_1a_1\up_{1,0} \in \mcHb(2,3,1)$.
  Next, for the term involving $\gamma_{0,2}$, by the definition of $\gamma_0$,
  \[
    \gamma_{0,2}
    = -\frac{1}{2}(a_1^2 +b_1^2) T_1^2.
  \]
  In view of the product $T_1^2$, we take one $T_1\in\mcAg(3,1)$ and the other $T_1\in\mcA(3,1)$. Then by Lemma~\ref{le:produ-madb}, $\gamma_{0,2} \in \mcHg(2,6,2)$. Observe that
  \[
    \zt\pry = -\zt\oa + \frac{\z\zt}{y},
    \with
    \frac{\z\zt}{y}
    = \begin{cases}
      \frac{4}{y}(1+O(y^2)),
      & y\le 1, \\
      -\frac{4}{y^3}\bigl(1+O\bigl(\frac{1}{y^2}\bigr)\bigr),
      & y>1.
    \end{cases}
  \]
  Thus using Lemma~\itref{le:ham-act-adb}{it:ham-act-adb=ham} and Lemma~\ref{le:produ-madb} again, we find $\zt\pry\gamma_{0,2} \in \mcHg(2,3,-1)$, hence $\zt\pry\gamma_{0,2} \in \mcHb(2,3,1)$ by Lemma~\ref{le:inclu-madb}. Finally, for the last term, we note that $\zt(\pry + \z/y) = -\zt\oa + 2\z\zt/y$. Both terms on the right can be treated as above. The same argument shows that the last term belongs to $\mcHb(2,3,1)$. Collecting these estimates, we solve $\mhamp\xx[2] = \tea_2$ componentwise by~\eqref{eq:inverse-ham}. This gives
  \[
    \tea_2 \in \mcHb(2,3,1), \quad
    \xx[2] \in \mcHb(2,5,3).
  \]

  \emph{Case $2,\dots ,k\to k+1$}. Again from the expression of $\tea$,
  \begin{align*}
    \tea_{k+1}
    & = \sum_{k_1+k_2=k} \R\sib_{k_1} \tea_{k+1,k_2}  +\zt\pry\gamma_{0,k+1} e_1 \\
    & \quad +\sum_{\substack{k_1+k_2+k_3                                         \\=k+1}} \sib_{k_1} \Bigl[\zt\Bigl(\pry + \tfrac{\z}{y}\Bigr)(\Phi_{k_2}^{(1)} + \xx[k_2]^{(1)}) -\Delta\gamma_{0,k_2}\Bigr](\Phi_{k_3} +\xx[k_3]),
  \end{align*}
  where
  \begin{align*}
    \tea_{k+1,k_2}
    & = b_1\mcT\xx[k_2] +a_1\mcP\xx[k_2] +\mcQ\xx[k_2] -a_1(1+\z)\R\Phi_{k_2}                                              \\
    & \quad\, -\sum_{\ell_1+\ell_2=k_2} \gamma_{0,\ell_1}\mcQ\Phi_{\ell_2} + b_1a_{k_2}\up_{k_2,0} + b_1b_{k_2}\up_{0,k_2}.
  \end{align*}
  Here, the summation indices satisfy $k_1\ge 0$ and $k_2, k_3\ge 1$.

  Let us consider the case $2\le k\le L+1$. We begin with the estimate for $\tea_{k+1,k_2}$. First, by induction, $\xx[k_2] \in \mcHb(k_2,2k_2+1,2k_2-1)$. By Lemma~\itref{le:ham-act-badb}{it:ham-act-badb=vla}, \ref{it:ham-act-badb=b1}, the actions of $b_1\mcT, a_1\mcP, \mcQ$ only increase the parameter degree by one, and hence
  \[
    b_1\mcT\xx[k_2] +a_1\mcP\xx[k_2] +\mcQ\xx[k_2] \in
    \mcHb(k_2+1,2k_2+1,2k_2-1).
  \]
  Then, as in the previous case, we have
  \[
    -a_1(1+\z)\R\Phi_{k_2} + b_1a_{k_2}\up_{k_2,0} + b_1b_{k_2}\up_{0,k_2}\in \mcHb(k_2+1,2k_2+1,2k_2-1).
  \]
  Next, by~\eqref{pf:form-comp=gamma}, $\gamma_0$ is a polynomial in the pairings generated by the homogeneous components of $\wzp$. More precisely, every monomial in $\gamma_0$ is a product of factors of the form
  \[
    \Phi_{i}\cdot\Phi_{j}, \quad
    \Phi_{i}\cdot\xx[j], \quad
    \xx[i]\cdot\xx[j]
    \where
    1\le i,j\le L+2.
  \]
  A direct use of Lemma~\ref{le:produ-madb} implies
  \begin{equation}
    \label{pf:app-sol=gamma-0-si}
    \gamma_{0,\ell},\ \sib_{\ell}
    \in \mcHg(\ell, 2\ell+2, 2\ell-2)
    \for
    2\le \ell\le k.
  \end{equation}
  This shows~\eqref{pf:app-sol=gamma-0-k-deg} for $2,\dots,k$.
  Using again Lemma~\ref{le:produ-madb},
  \[
    \gamma_{0,\ell_1}\mcQ\Phi_{\ell_2}
    \in \mcHb(k_2+1, 2k_2+1, 2k_2-1).
  \]
  Therefore, we obtain
  \[
    \tea_{k+1,k_2}\in\mcHb(k_2+1, 2k_2+1, 2k_2-1).
  \]
  From Lemma~\ref{le:produ-madb} again,
  \[
    \sum_{k_1+k_2=k} \R\sib_{k_1} \tea_{k+1,k_2}\in\mcHb(k+1, 2k+1, 2k-1).
  \]
  Last, for the remaining terms of $\tea_{k+1}$, using the fact that
  \[
    \zt\pry = -\zt\oa + \frac{\z\zt}{y}
    \andd
    -\Delta = \ham - \frac{V}{y^2}
    \with
    \frac{V}{y^2} \in \mcAg(-2,-2),
  \]
  we obtain the same membership of $\mcHb(k+1, 2k+1, 2k-1)$, and hence the estimate for $\tea_{k+1}$.  Solving $\mhamp\xx[k+1] = \tea_{k+1}$ by \eqref{eq:inverse-ham}, we have
  \begin{align*}
    \tea_{k+1} & \in \mcHb(k+1, 2(k+1)-1, 2(k+1)-3), \\
    \xx[k+1]   & \in \mcHb(k+1, 2(k+1)+1, 2(k+1)-1).
  \end{align*}
  This shows~\eqref{pf:app-sol=xk-deg} and~\eqref{pf:app-sol=tea-k-deg} with $k\le L+2$. Moreover, given the choice of $\x$ and \eqref{pf:app-sol=gamma-0-si}, the estimates for $\gamma_{0,k}$, $\sib_k$ hold for $k\le C_{\gamma_0}$, $k\le C_{\sib}$ respectively. This proves~\eqref{pf:app-sol=gamma-0-k-deg}.

  We turn to the case $k\ge L+2$. First, for $k=L+2$, the estimate for $\tea_{k+1} = \tea_{L+3}$ is similar to the estimate above, resulting in the same estimate
  \begin{align*}
    \tea_{L+3}
    & \in \mcHb(L+3, 2L+5, 2L+3)        \\
    & = \mcHb(k+1, 2(k+1)-1, 2(k+1)-3),
  \end{align*}
  which is sharp, since in $\tea_{L+3}$,
  \[
    \R b_1 \mcT\xx[L+2] \in \mcHb(L+3, 2L+5, 2L+3).
  \]
  Then, for $k\ge L+3$, since $\xx[k_2]$, $\Phi_{k_2}$ are defined for $k_2\le L+2$, $k_2\le L$ respectively, the terms in the first sum and those involving $\gamma_0$ contain a positive-degree factor of the form
  \[
    \sib_{k_1},\,\gamma_{0,\ell_1},\,\gamma_{0,\ell_2}
    \in \mcHg(\cdot, 2\cdot+2, 2\cdot-2).
  \]
  The remaining terms in the last triple sum have the form
  \[
    \zt\Bigl(\pry+\frac{\z}{y}\Bigr)
    \bigl(\Phi_{k_2}^{(1)}+\xx[k_2]^{(1)}\bigr)
    \bigl(\Phi_{k_3}+\xx[k_3]\bigr),
    \quad k_2+k_3=k+1.
  \]
  Here the required far-field decay comes from the three powers of $y$ gained by $\zt(\pry+\z/y)$.
  With an application of Lemma~\ref{le:produ-madb}, we obtain
  \[
    \tea_{k+1} \in \mcHg(k+1, 2(k+1)-1, 2(k+1)-5)
    \for k\ge L+3,
  \]
  yielding a refined (compared with $k=L+3$) bound in \eqref{pf:app-sol=tea-k-deg} for $k\ge L+4$. This completes the proof of \eqref{pf:app-sol=xk-deg}--\eqref{pf:app-sol=tea-k-deg}.

  \step[3, Error estimation]
  The error $\psiz = \psizone + \psiztwo$ is now given by
  \begin{align*}
    \psizone
    & = -\sum_{k=L+3}^{C_{\tea}} \hj\tea_{k} - \sum_{k=1}^L \hj\R^k \sib (\vsig_{k}-\cb T_k)(b_1a_k e_1 +b_1b_k e_2), \\
    \psiztwo
    & = \sib\,\bigl\{ b_1\mcT\rrr +a_1\mcP\rrr +\mcQ\rrr \bigr\} e_3                                                  \\
    & \quad +{(-\gamma_0)}^{L+2} \Bigl\{ b_1\mcT\wz +a_1\mcP\wz -\gamma_0\mcQ\wz -\hj\R\mcQ\x +\mcQ\rrr e_3 \Bigr\}   \\
    & \quad +{(-\gamma_0)}^{L+2} \hj\Bigl(\zt\Bigl(\pry+\tfrac{\z}{y}\Bigr)\alpha_0 -\Delta\gamma_0\Bigr) e_3.
  \end{align*}
  For brevity, we define the norms
  \[
    \mcN_1[\eta] = \int_{y\le 2S_1}|(\R\mham)^{L+1}\eta|^2,
    \quad
    \mcN_{2,n}[\eta] = \sum_{\ell=0}^{n}\int_{y\le 2S_1}\frac{|\pry^{\ell}\eta|^2}{\jbk[y]^{2(n-\ell)}},
    \quad 0\le n\le 2L+2.
  \]
  In the rest of this proof, we shall prove for $\eta=\psizone,\psiztwo$ that
  \begin{equation}
    \label{pf:app-sol=mcN-norm}
    \mcN_1[\eta] \les \frac{b_1^{2L+4}}{\lgb^2}
    \quad
    \mcN_{2,n}[\eta] \les
    \begin{cases}
      b_1^{2L+4}               & \mbox{for}\;\; n= 2L+2, \\
      b_1^{n+2}\lgb^{4L+2-2n}, & 0\le n\le 2L+1.
    \end{cases}
  \end{equation}
  These give the first bound in~\ref{it:app-sol=sob} and the bound in~\ref{it:app-sol=weig}. The remaining two bounds in \ref{it:app-sol=sob} are direct consequences of \ref{it:app-sol=weig}. Moreover, by using ~\ref{it:app-sol=weig} and allocating the weight, we have
  \begin{align*}
    |(\psiz, \xm_M\eone)| + |(\psiz, \xm_M\etwo)|\les \biggl(\int_{y\le 2M}\frac{|\psiz|^2}{\jbk[y]^{4L+4}}\biggr)^{1/2} \|\jbk[y]^{2L+2}\xm_M\|_{L^2}
    \les M^C b_1^{L+\frac32}.
  \end{align*}
  The same holds for $\xm_{\ell,0}$ and $\xm_{0,\ell}$, which gives~\ref{it:app-sol=local}. We now estimate $\mcN_1$ and $\mcN_2$ for each part of $\psiz$ in turn.

  \step[4, $\psizone$ contribution, Part I]
  We consider the contribution of $\hj\tea_k$, which is the critical one in error estimation. By definition we have
  \[
    (\R\mham)^{L+1}\hj\tea_{k}
    = (\R\mham)^{L+1}\R\tea_{k} +(\R\mham)^{L+1}(\wz\times\tea_{k}).
  \]
  First, from \eqref{eq:mham-r}, the leading part is
  \[
    (\R\mham)^{L+1}\R\tea_{k} = \R^{L+2}\mhamp^{L+1}\tea_{k},
  \]
  where from \eqref{pf:app-sol=tea-k-deg}, Lemma~\itref{le:ham-act-badb}{it:ham-act-badb=ham},
  \[
    \mhamp^{L+1}\tea_{k} \in
    \begin{cases}
      \mcHb(L+3,3,1),             & k=L+3,    \\
      \mcHg(k,2(k-L)-3,2(k-L)-7), & k\ge L+4.
    \end{cases}
  \]
  Then, using Lemma~\ref{le:l2-norm-homo}, we obtain
  \[
    \int_{y\le 2S_1}|(\R\mham)^{L+1}\R\tea_{L+3}|^2
    \les \frac{b_1^{2L+4}}{\lgb^2}.
  \]
  Similarly, the $k\ge L+4$ part enjoys a better bound.
  Next, for the term involving $\wz\times\tea_k$, we split the estimate into two regions: $y\le 1$ and $y\ge1$. Note that $\wz\times\tea_k$ carries a factor $b_1^{k+1}$, and hence by $k\ge L+3$,
  \[
    \int_{y\le 1} |(\R\mham)^{L+1}(\wz\times\tea_k)|^2
    \les b_1^{2k+2}
    \les b_1^{2L+8}.
  \]
  Now for $y\ge1$, using~\eqref{pf:app-sol=wz-l-infty} and Leibniz,
  \begin{align*}
    & \int_{1\le y\le 2S_1} |(\R\mham)^{L+1}(\wz\times\tea_k)|^2                                  \\
    & \les \sum_{\ell=0}^{2L+2} \|\jbk[\log y]^C y^{\ell}\pry^{\ell}\wz\|_{L^\infty(y\le 2S_1)}^2
    \sum_{\ell=0}^{2L+2} \int_{1\le y\le 2S_1} \frac{|(\tea_k)\sd[\ell]|^2}{y^{2(2L+2-\ell)}}      \\
    & \les b_1^{2L+5}\lgb^C.
  \end{align*}
  Therefore, the first bound in~\ref{it:app-sol=sob} follows
  \[
    \mcN_1[\hj\tea_{k}] \les \frac{b_1^{2L+4}}{\lgb^2}.
  \]
  Last, we treat $\mcN_{2,n}[\hj\tea_k]$. For $y\le1$, from~\eqref{pf:app-sol=tea-k-deg}, we see $\hj\tea_k$ carries a factor $b_1^k$ for $k\ge L+3$, and thus a direct integration yields the bound $b_1^{2L+6}$. For $y\ge1$, using~\eqref{pf:app-sol=wz-l-infty} and Leibniz, we compute
  \[
    \sum_{\ell=0}^{n} \int_{1\le y\le 2S_1} \frac{|\pry^{\ell}(\hj\tea_k)|^2}{\jbk[y]^{2(n-\ell)}}
    \les \bigl(1+b_1\lgb^C\bigr) \sum_{\ell=0}^{n}\int_{1\le y\le 2S_1} \frac{|(\tea_k)\sd[\ell]|^2}{y^{2(n-\ell)}}.
  \]
  Using Lemma~\ref{le:l2-norm-homo} and Definition~\itref{def:badb}{it:def-badb=2}, we obtain~\ref{it:app-sol=weig}:
  \[
    \mcN_{2,n}[\hj\tea_{k}]\les
    \begin{cases}
      b_1^{2L+4}               & \mbox{for}\;\; n=2L+2, \\
      b_1^{n+2}\lgb^{4L+2-2n}, & 0\le n\le 2L+1.
    \end{cases}
  \]

  \step[5, $\psizone$ contribution, Part II]
  We now estimate the $\vsig_{k} -\cb T_k$ part of $\psizone$. To avoid duplication, we only treat the $b_1b_k e_2$ direction, the other direction being estimated in the same way using $|a_k|\les b_1^k$. First, note that from Lemma~\ref{le:vsig},
  \[
    \supp(\vsig_k -\cb T_k) \subset \biggl\{y\ge \frac{S_0}{4}\biggr\}.
  \]
  Moreover, we have
  \[
    \vsig_{b_1} = -4\Gamma\in\ker\ham \for y\ge 2S_0,
    \andd T_1\in\ker\ham^2.
  \]
  Therefore
  \[
    \bigl|\ham^{\ell}(\vsig_{k} -\cb T_{k})\bigr|
    \les \frac{y^{2(k-\ell)-1}}{\lgb} \one_{y\sim S_0},
    \quad \ell\ge k+1.
  \]
  Then, by $\sib_0=1$, the leading error is bounded by
  \begin{align*}
    & \sum_{k=1}^{L} \int_{y\le 2S_1} \bigl|(\R\mham)^{L+1}\R^{k+1}\bigl[b_1b_k(\vsig_{k} -\cb T_k)e_2\bigr]\bigr|^2 \\
    & \les \sum_{k=1}^{L} \frac{b_1^{2k+2}}{\lgb^2}\int_{y\sim S_0} \frac{y\dd y}{y^{4L-4k+6}}
    \les \frac{b_1^{2L+4}}{\lgb^2}.
  \end{align*}
  Next, the remaining error involving $\vsig_{b_1}$ contains a factor of either $\wz$ or $\sib_{\ell}$ with $\ell\ge 2$ (since $\sib_0=1$ has been considered and $\sib_1=0$), both of which are small, in view of \eqref{pf:app-sol=wz-l-infty} and
  \[
    \bigl\|\jbk[\log y]^{C}y^{k}\pry^{k}\sib_{\ell}\bigr\|_{L^{\infty}(y\le 2S_1)}^2 \les b_1^2\lgb^C,
    \quad k\ge 0,\;\; \ell\ge2.
  \]
  Combining these estimates with the rough bound
  \[
    \bigl|\pry^{\ell} (\vsig_{k}-\cb T_k) \bigr|
    \les y^{2k-\ell-1} \one_{y\ge \frac{S_0}{4}},
    \quad 0\le \ell\le 2L+2,\quad y\le 2S_1,
  \]
  we see the remaining error yields a better estimate than the leading one. Summing over $1\le k\le L$ yields the first part of \ref{it:app-sol=sob},
  \[
    \sum_{k=1}^{L} \mcN_1\bigl[\hj\R^{k} \sib (\vsig_{k}-\cb T_k) b_1b_k e_2\bigr] \les \frac{b_1^{2L+4}}{\lgb^2}.
  \]
  Last, the weighted bound follows by a similar computation,
  \begin{align*}
    & \sum_{\ell=0}^n\sum_{k=1}^{L}\int_{y\le 2S_1}\frac{\bigl|\pry^{\ell}\bigl[\hj\R^k \sib(\vsig_{k}-\cb T_{k})b_1b_k e_2\bigr]\bigr|^2}{\jbk[y]^{2(n-\ell)}} \\
    & \les \sum_{k=1}^{L} b_1^{2(k+1)} \int_{\frac{S_0}{4}\le y\le 2S_1} y^{2(2k-1-n)} y\dd y                                                                    \\
    & \les \begin{cases}
      b_1^{n+2}                & \mbox{for}\; n=2L+2, \\
      b_1^{n+2}\lgb^{4L-2n+2}, & 0\le n\le 2L+1.
    \end{cases}
  \end{align*}
  This yields~\ref{it:app-sol=weig}, and thus finishes the estimates for $\psizone$.

  \step[6, $\psiztwo$ contribution]
  Observe that each term of $\psiztwo$ carries a factor of $\rrr$ or of $(-\gamma_0)^{L+2}$. From the proof of Lemma~\ref{le:form-comp}, both of these are polynomials in $|\wzp|$ of degree at least $2(L+2)$,
  \[
    \rrr, \;(-\gamma_0)^{L+2} = O\bigl(|\wzp|^{2(L+2)}\bigr).
  \]
  Thus by the product rule in Lemma~\ref{le:prod-homo}, for some large constant $C$,
  \begin{align*}
    \sib\bigl\{ b_1\mcT\rrr +a_1\mcP\rrr +\mcQ\rrr \bigr\} e_3 +(-\gamma_0)^{L+2}\mcQ \rrr e_3
    & \in \sum_{k=2L+5}^C \mcHg(k,2k,2k-4), \\
    {(-\gamma_0)}^{L+2} \Bigl\{ b_1\mcT\wzp +a_1\mcP\wzp -\gamma_0\mcQ\wzp \Bigr\}
    & \in \sum_{k=2L+6}^C\mcHg(k,2k-1,2k-7), \\
    {(-\gamma_0)}^{L+2} \Bigl\{ b_1\mcT\gamma_0 +a_1\mcP\gamma_0 -\gamma_0\mcQ\gamma_0 \Bigr\}
    & \in \sum_{k=2L+7}^C \mcHg(k,2k+2,2k-6).
  \end{align*}
  In addition, using the estimates in Step~2 about $\zt\bigl(\pry+\tfrac{\z}{y}\bigr)$ and $-\Delta$, we have
  \begin{align*}
    & {(-\gamma_0)}^{L+2} \Bigl\{ \hj\Bigl(\zt\Bigl(\pry+\tfrac{\z}{y}\Bigr)\alpha_0 -\Delta\gamma_0\Bigr) e_3 -\hj\R\mcQ\x \Bigr\} \\
    & \in \sum_{i=1}^2\sum_{k=2L+6}^C\mcHg(k,2k+1,2k-5)e_i +\sum_{k=2L+8}^C\mcHg(k,2k+2,2k-6)e_3.
  \end{align*}
  Gathering these estimates, we use Lemma~\ref{le:l2-norm-homo} and obtain
  \[
    \mcN_1[\psiztwo] \les b_1^{2L+5}\lgb^C \andd
    \mcN_{2,n}[\psiztwo] \les b_1^{n+3}\lgb^C
  \]
  for all $0\le n\le 2L+2$.
  These estimates yield the first bound in~\ref{it:app-sol=sob} and the bound in~\ref{it:app-sol=weig}, and complete the proof.
\end{proof}

\subsection{Localization of the profiles}
\label{sub:local-prof}
In this subsection, we localize the approximate solution~$\wz$, restricting its support to the region $y\le 2S_1$. This is done via a direct multiplication of the profiles by a smooth cutoff $\chi_{S_1}(y) = \chi\bigl(\frac{y}{S_1}\bigr)$.

We define
\[
  \wzt := \chi_{S_1}\wz = \chi_{S_1}(\vp + \x +\gamma_0 e_3),
\]
and correspondingly,
\begin{equation}
  \label{eq:local-phi}
  \phiat[k] = \chi_{S_1}\phia[k]
  \quad\mbox{and}\quad
  \phibt[k] = \chi_{S_1}\phib[k].
\end{equation}
Then the following holds.

\begin{lemma}[Localized profiles]
  \label{le:local-app-sol}
  With the notations of Lemma~\ref{le:app-sol}, the localized vector $\wzt=\chi_{S_1}\wz$ is an approximate solution to~\eqref{eq:vect-form} in the sense that
  \begin{equation}
    \label{le:local-app-sol=wzt-eq}
    \prs\wzt -\frac{\la_s}{\la}\vla\wzt +\vtha_s\z\R\wzt +\hj[\wzt](\mham\wzt+\hp\lphi) = \modt + \psizt,
  \end{equation}
  where the localized modulation vector is given by
  \begin{equation}
    \begin{split}
      \label{le:local-app-sol=modt}
      \modt & =
      \sum_{k=1}^{L}(\prs a_k +A_k)\biggl(\phiat[k] +\chi_{S_1}\biggl(\frac{\pr\x}{\pr a_k} +\frac{\pr\gamma_0}{\pr a_k}e_3\biggr)\biggr)                                                                 \\
      & \quad +\sum_{k=1}^{L}(\prs b_k +B_k)\biggl(\phibt[k] +\chi_{S_1}\biggl(\frac{\pr\x}{\pr b_k} +\frac{\pr\gamma_0}{\pr b_k}e_3\biggr) + \delta_{1,k} \frac{\pr\chi_{S_1}}{\pr b_1}\wz \biggr) \\
      & \quad +(\vtha_s+a_{1})\bigl(\lphi\bigl(\etwo -\eone\times\wzt \bigr) +\z\R\wzt \bigr)                                                                                                       \\
      & \quad -\Bigl(\frac{\la_s}{\la} +b_1\Bigr)\bigl(\lphi\bigl(\eone +\etwo\times\wzt\bigr) +\vla\wzt \bigr),
    \end{split}
  \end{equation}
  and the corresponding error satisfies:
  \begin{enumerate}[label={\rm (\roman*)}]
    \item \label{it:local-app-sol=sob}
      the Sobolev bounds:
      \begin{gather*}
        \int |(\R\mham)^{L+1}\psizt|^2
        \les \frac{b_1^{2L+4}}{\lgb^2},  \\
        \int |\ma\mhamp^L\psizt|^2
        \les \delta(b_{1}^{\ast})\frac{b_1^{2L+2}}{\lgb^2}, \\
        \int |\mhamp^k\psizt|^2 \les b_1^{2k+2}\lgb^C
      \end{gather*}
      for $1\le k\le L$ and some universal constant $C>0$;
    \item \label{it:local-app-sol=weig}
      weighted bound:
      \[
        \sum_{\ell=0}^{k} \int \frac{|\pry^{\ell}\psizt|^2}{\jbk[y]^{2(k-\ell)}}
        \les \begin{cases}
          b_1^{2L+4}               & \mbox{for}\;\; k= 2L+2, \\
          b_1^{k+2}\lgb^{4L+2-2k}, & 1\le k\le 2L+1;
        \end{cases}
      \]
    \item \label{it:local-app-sol=local}
      local bound: let $\xm_M$ and $\xm_{\ell,0},\xm_{0,\ell}$ be defined by \eqref{eq:local-reson} and \eqref{eq:vec-local-reson}. Then
      \begin{align*}
        |(\psizt, \xm_M\eone)| & + |(\psizt, \xm_M\etwo)|                                                       \\
        & + \sum_{\ell=1}^L \Bigl(|(\psizt,\xm_{\ell,0})| +|(\psizt,\xm_{0,\ell})|\Bigr)
        \les M^C b_1^{L+\frac32}.
      \end{align*}
  \end{enumerate}
\end{lemma}

\begin{proof}
  By a direct computation,
  \begin{align*}
    & \prs\wzt -\frac{\la_s}{\la}\vla\wzt +\vtha_s\z\R\wzt +\hj[\wzt](\mham\wzt+\hp\lphi)                                     \\
    & = \chi_{S_1} (\mod + \psiz) +\prs\chi_{S_1}\wz -\frac{\la_s}{\la}(\vla\chi)_{S_1}\wz +(1-\chi_{S_1})\hj[\wzt](\hp\lphi) \\
    & \quad -(1-\chi_{S_1})\chi_{S_1} \wz\times(\mham\wz +\hp\lphi)                                                           \\
    & \quad -\hj[\wzt]\Bigl(2\pry\chi_{S_1}\pry\wz +\Delta\chi_{S_1}\wz +\zt\pry\chi_{S_1}\etwo\times\wz\Bigr).
  \end{align*}
  Since $\wzt = \chi_{S_1}\wz$,
  \begin{align*}
    & (1-\chi_{S_1})\hj[\wzt](\hp\lphi) - (1-\chi_{S_1})\chi_{S_1}\wz\times\hp\lphi
    = (1-\chi_{S_1})\R\hp\lphi \\
    & = (1-\chi_{S_1})\lphi \Bigl[(\vtha_s+a_1)e_2 -\Bigl(\frac{\la_s}{\la}+b_1\Bigr)e_1\Bigr] +(1-\chi_{S_1})\lphi(b_1 e_1-a_1 e_2).
  \end{align*}
  In addition,
  \[
    \prs\chi_{S_1}\wz = \frac{\pr\chi_{S_1}}{\pr b_1}(\prs b_1+B_1)\wz -\frac{\pr\chi_{S_1}}{\pr b_1} B_1\wz.
  \]
  We then include the terms containing $\prs b_1+B_1$, $\la_s/\la+b_1$ and $\vtha_s+a_1$ in $\modt$. Note that
  \begin{align*}
    & \chi_{S_1}\bigl(\lphi(\etwo-\eone\times\wz)+\z\R\wz\bigr) +(1-\chi_{S_1})\lphi\etwo \\
    & =\lphi(\etwo-\eone\times\wzt)+\z\R\wzt, \\
    & \chi_{S_1}\bigl[\lphi(\eone+\etwo\times\wz)+\vla\wz\bigr]
    +(1-\chi_{S_1})\lphi\eone +(\vla\chi)_{S_1}\wz \\
    & =\lphi(\eone+\etwo\times\wzt)+\vla\wzt,
  \end{align*}
  which gives $\modt$. The remaining error is now
  \[
    \psizt = \chi_{S_1}\psiz +(1-\chi_{S_1})(b_1\eone-a_1\etwo)\lphi +\psi,
  \]
  where $\psi$ denotes the part localized in $y\sim S_1$:
  \begin{align*}
    \psi
    & = \biggl(b_1\vla\chi_{S_1}-B_1\frac{\pr\chi_{S_1}}{\pr b_1}\biggr)\wz
    -(1-\chi_{S_1})\chi_{S_1}\wz\times\mham\wz \\
    & \quad -\hj[\wzt]\bigl(2\pry\chi_{S_1}\pry\wz +\Delta\chi_{S_1}\wz +\zt\pry\chi_{S_1}\etwo\times\wz\bigr).
  \end{align*}
  As in the proof of Lemma~\ref{le:app-sol}, it remains to prove the $\mcN_1$ bound in \eqref{pf:app-sol=mcN-norm} and the corresponding $\mcN_{2,n}$ bounds for $\eta=\psizt$ and $1\le n\le 2L+2$.
  Note that from Lemma~\itref{le:app-sol}{it:app-sol=sob}--\ref{it:app-sol=weig}, the leading error in $\psizt$ satisfies
  \begin{align*}
    & \int|(\R\mham)^{L+1}(\chi_{S_1}\psiz)|^2                                                                                                                                         \\
    & \les \int_{y\le 2S_1} \hs[-.2] |(\R\mham)^{L+1}\psiz|^2 + \sum_{k=1}^{2L+2}\frac{1}{S_1^{2k}} \hs[-.2] \sum_{\ell=0}^{2L+2-k} \hs[-.2] \int_{y\le 2S_1}\frac{|\pry^{\ell}\psiz|^2}{\jbk[y]^{2(2L+2-k-\ell)}}
    \les \frac{b_1^{2L+4}}{\lgb^2}.
  \end{align*}
  Moreover, the far-out error $(1-\chi_{S_1})(b_1 e_1-a_1 e_2)\lphi$ satisfies
  \[
    \sum_{\ell=0}^{n}\int_{y\ge S_1}\frac{|\pry^{\ell}(b_1\lphi)|^2}{\jbk[y]^{2(n-\ell)}} \les \frac{b_1^{n+2}}{\lgb^{2n}}
    \for 1\le n\le 2L+2,
  \]
  where the $a_1$ part also admits the above bound, since $|a_1|\les b_1$.
  For the remaining error $\psi$, using~\eqref{pf:form-comp=gamma} and \eqref{pf:app-sol=gamma-0-k-deg}, we have for $0\le\ell\le 2L+2$ that
  \begin{align*}
    \|y^{\ell}\pry^{\ell}\psi^\pp\|_{L^{\infty}(y\sim S_1)}^2
    & \les b_1^3\lgb^{4L}, \\
    \|y^{\ell}\pry^{\ell}(\psi\cdot e_3)\|_{L^{\infty}(y\sim S_1)}^2
    & \les b_1^4\lgb^{8L},
  \end{align*}
  and hence
  \[
    \sum_{\ell=0}^{n}\int\frac{|\pry^\ell\psi|^2}{\jbk[y]^{2(n-\ell)}}
    \les b_1^{n+2}\lgb^{4L+2-2n}
    \for 1\le n\le 2L+2.
  \]
  In particular, taking $n=2L+2$ in the above bound gives $b_1^{2L+4}\lgb^{-2}$, and thus these errors do not deteriorate the estimates in Lemma~\ref{le:app-sol}. Consequently, the bounds~\ref{it:local-app-sol=sob}--\ref{it:local-app-sol=local} follow from Lemma~\itref{le:app-sol}{it:app-sol=sob}--\ref{it:app-sol=local}. This completes the proof.
\end{proof}

\subsection{The model system for \texorpdfstring{$(a,b)$}{(a,b)}}
\label{sub:ab-sys}
In this subsection, we study the model ODE system for the geometric parameters $(a,b)$ on the time interval $[s_{0},s_{1}]$ where $s_{1}>s_{0}\gg 1$. More precisely, we consider
\begin{equation}
  \label{eq:ode}
  \left\{\begin{aligned}
    & \begin{aligned}
      & \prs a_\ell +\left(2\ell-1+\frac{2}{\log s}\right)b_1a_\ell -a_1b_\ell -a_{\ell+1} =0,   \\
      & \prs b_\ell +\left(2\ell-1+\frac{2}{\log s}\right)b_1b_\ell \,+a_1a_\ell -b_{\ell+1} =0,
    \end{aligned}\quad \  \mbox{for} \ \ 1\le \ell\le L,
  \end{aligned}\right.
\end{equation}
where the convention $a_{L+1}=b_{L+1}=0$ is adopted. This system serves as an approximation to the modulation equations for the parameters $(a,b)$ (see Lemma~\ref{le:mod-eq}).

We define the sequences $(c_{\ell})_{\ell=1}^{L}, (d_{\ell})_{\ell=1}^{L}$ by
\begin{equation}
  \label{eq:c1}
  \left\{\begin{aligned}
    & c_1 = \frac{L}{2L-1},                       \\
    & c_{\ell} = \frac{\ell-1-L}{2L-1}c_{\ell-1},
  \end{aligned}\right.
  \ \ \mbox{and}\ \
  \left\{\begin{aligned}
    & d_1 = -\frac{2L}{(2L-1)^2},                                                       \\
    & d_{\ell} =\frac{4L(L-\ell+1)}{(2L-1)^2}c_{\ell-1}+\frac{\ell-1-L}{2L-1}d_{\ell-1}
  \end{aligned}\right.
\end{equation}
for $2\le \ell\le L$,
and the vector-valued functions $a^{e},b^{e}\in C^1([s_0,s_1];\mbR^L)$ by
\begin{equation}
  \label{eq:ae-be}
  a_{\ell}^{e}(s)\equiv 0
  \andd
  b_{\ell}^{e}(s) = \frac{c_\ell}{s^\ell} +\frac{d_\ell}{s^\ell \log s}
  \for
  1\le \ell\le L.
\end{equation}
A direct computation gives the following.

\begin{lemma}
  \label{le:ode-app-sol}
  The map $(a^e, b^e)$ is an approximate solution to the system \eqref{eq:ode} in the sense that
  \[
    \left\{
    \begin{aligned}
      \prs a_{\ell}^{e} +\left(2\ell-1+\frac{2}{\log s}\right) b_{1}^{e}a_{\ell}^{e} -a_{1}^{e}b_{\ell}^{e} -a_{\ell+1}^{e} & = 0,                                            \\
      \prs b_{\ell}^{e} +\left(2\ell-1+\frac{2}{\log s}\right) b_{1}^{e}b_{\ell}^{e} +a_{1}^{e}a_{\ell}^{e} -b_{\ell+1}^{e} & = O\biggl(\frac{1}{s^{\ell+1}(\log s)^2}\biggr)
    \end{aligned}
    \right.
  \]
  for all $1\le \ell\le L$.

\end{lemma}

We linearize the system around $(a^e,b^e)$. Denote
\[
  U =(U_\ell)_{\ell=1}^{L}, \;\; V =(V_\ell)_{\ell=1}^{L}\in C^{1}([s_{0},s_{1}];\mbR^{L}).\footnotemark
\]
We let
\begin{equation}\label{eq:ab-abe}
  b_\ell = b^e_\ell+\frac{U_\ell}{s^\ell(\log s)^{\frac54}},
  \quad
  a_\ell = a^e_\ell +\frac{V_\ell}{s^\ell(\log s)^{\frac54}},
  \quad 1\le \ell\le L.
\end{equation}
Then we have the following lemmas.
\footnotetext{This $V$ should not be confused with the function $V$ in the definition of the Hamiltonian $\ham = -\Delta + \frac{V}{y^2}$.}

\begin{lemma}[\cite{raphael_quantized_2014}]
  \label{le:b-sys-lin}
  The following estimates hold.
  \begin{enumerate}[label={\rm (\roman*)}]
    \item \emph{Linearization.}
      There exists a matrix $M_U \in\mbR^{L\times L}$ such that
      \label{it:b-sys-lin=lin}
      \[
        \begin{aligned}
          & \prs b_\ell +\biggl(2\ell-1+\frac{2}{\log s}\biggr)b_1b_\ell +a_1a_\ell -b_{\ell+1} \\
          & = \frac{1}{s^{\ell+1}(\log s)^{\frac54}}
          \left(s\prs U_\ell -(M_U U)_\ell+
          O\biggl(\frac{|U| +|U|^{2}+|V|^{2}}{\log s}\biggr)
          +O\biggl(\frac{1}{(\log s)^{\frac34}}\biggr)
          \right).
        \end{aligned}
      \]
    \item \emph{Diagonalization.}
      The matrix $M_U$ is diagonalizable. More precisely, there exists $P_U\in \GL(L,\mbR)$ such that $M_U = P_U^{-1}D_U P_U$ with
      \label{it:b-sys-lin=diag}
      \[
        D_U = \diag\left\{-1,\frac{2}{2L-1},\frac{3}{2L-1},\cdots,\frac{L}{2L-1}\right\}.
      \]
  \end{enumerate}
\end{lemma}

\begin{proof}
  This proof is based on an argument similar to that used in the proof of Lemma~\ref{le:a-sys-lin}.
  We refer to~\cite[Lemma 2.15]{raphael_quantized_2014} for more details.
\end{proof}

\begin{lemma}
  \label{le:a-sys-lin}
  The following estimates hold.
  \begin{enumerate}[label={\rm (\roman*)}]
    \item \emph{Linearization.}
      There exists a matrix $M_V \in\mbR^{L\times L}$ such that
      \label{it:a-sys-lin=lin}
      \[
        \begin{aligned}
          & \prs a_\ell +\biggl(2\ell-1+\frac{2}{\log s}\biggr)b_1a_\ell -a_1b_\ell -a_{\ell+1} \\
          & = \frac{1}{s^{\ell+1}(\log s)^{\frac54}}
          \left(s\prs V_\ell -(M_V V)_\ell+
          O\biggl(\frac{|V| +|U||V|}{\log s}\biggr) \right).
        \end{aligned}
      \]
    \item \emph{Diagonalization}.
      The matrix $M_V$ is diagonalizable. More precisely, there exists $P_V\in \GL(L,\mbR)$ such that $M_V = P_V^{-1}D_V P_V$ with
      \label{it:a-sys-lin=diag}
      \[
        D_V = \diag\biggl\{\frac{1}{2L-1},\frac{2}{2L-1},\dots,\frac{L}{2L-1}\biggr\}.
      \]
  \end{enumerate}
\end{lemma}

\begin{proof}
  Proof of~\ref{it:a-sys-lin=lin}. Using the definition of ${(c_{\ell})}_{\ell=1}^{L}, {(d_{\ell})}_{\ell=1}^{L}$, we compute:
  \begin{align*}
    & \quad \ell c_{\ell}-(2\ell-1)c_{1}c_{\ell}+c_{\ell+1}                            \\
    & = \ell d_{\ell}-2c_{1}c_{\ell}-(2\ell-1)(c_{1}d_{\ell}+c_{\ell}d_{1})+d_{\ell+1} \\
    & =0,
  \end{align*}
  where the convention $c_{L+1}=d_{L+1}=0$ is used.
  Together with Lemma~\ref{le:ode-app-sol}, this implies
  \begin{align*}
    & s^{\ell+1}(\log s)^{\frac54}
    \left(\prs a_\ell +\biggl(2\ell-1+\frac{2}{\log s}\biggr)b_1a_\ell -a_1b_\ell -a_{\ell+1} \right) \\
    & =
    s\prs V_\ell  -\frac{L-\ell}{2L-1} V_\ell -c_\ell V_1 -V_{\ell+1} +O\biggl(\frac{|V|+|U||V|}{\log s}\biggr).
  \end{align*}
  It follows directly that the matrix $
    M_V =
    (M_{i,j})_{L\times L}$ is given by $M_{1,1}=1$ and
  \[
    \left\{\begin{aligned}
      & M_{i,1} = c_i,               & \quad & 2\le i\le L,      \\
      & M_{i,i} = \tfrac{L-i}{2L-1}, & \quad & 2\le i\le L,      \\
      & M_{i-1,i} = 1,               & \quad & 2\le i\le L,      \\
      & M_{i,j} = 0,                 & \quad & \mbox{otherwise}.
    \end{aligned}\right.
  \]
  Proof of~\ref{it:a-sys-lin=diag}. We compute the characteristic polynomial. The cases $L=1,2,3$ follow by direct inspection. Let us assume $L\ge 4$.
  For any $1\le\ell\le L$, we denote by $D_\ell(x)$ the determinant of the submatrix of $M_V-x\id$ obtained from the first $\ell$ rows and the first $\ell$ columns.
  In particular, we have $D_1(x)=1-x$ and $D_L(x) = \det(M_V-x\id)$. From~\eqref{eq:c1}, for any $1\le \ell\le L-2$, we compute
  \[
    (-1)^{\ell+2}c_{\ell+2}+xD_{\ell+1}(x)=\left(\frac{L-\ell-1}{2L-1}-x\right)((-1)^{\ell+1}c_{\ell+1}+xD_{\ell}(x)).
  \]
  It follows that
  \[
    \begin{aligned}
      (-1)^{L}c_{L}+xD_{L-1}(x)= \left(c_{2}+xD_{1}(x)\right)\prod_{\ell=1}^{L-2}\left(\frac{\ell}{2L-1}-x\right)=
      -\prod_{\ell=1}^{L}\left(\frac{\ell}{2L-1}-x\right).
    \end{aligned}
  \]
  Expanding along the last row of $M_V-x\id$, we obtain
  \[
    D_{L}(x)=(-1)^{L+1}c_{L}-xD_{L-1}(x)=\prod_{\ell=1}^{L}\left(\frac{\ell}{2L-1}-x\right).
  \]
  This completes the proof.
\end{proof}

\section{Trapped regime}
\label{sec:trap-reg}
In this section, we introduce the main dynamical tools for our analysis, namely modulation theory and the bootstrap regime (Proposition~\ref{pro:boot}). These are the key components of the energy method, and of the proof of Theorem~\ref{th:main-thm}. The equation for the corrector $\w$ and the modulation equations are also derived.

\subsection{Localized resonances}
\label{sub:local-reson}
In this subsection, we introduce a localized version of $\lphi$, which is useful for the modulation analysis below in Subsection~\ref{sub:mod-eq}.

For any $L\in\mbN_+$ and any sufficiently large $M$, we define
\begin{equation}
  \label{eq:local-reson}
  \xm_M(y) = \sum_{k=0}^{L}\rho_k \ham^k(\chi_M\lphi),
\end{equation}
with the coefficients ${(\rho_k)}_{k=0}^{L}$ given by
\[
  \rho_0 = 1 \andd
  \rho_k = -\sum_{i=0}^{k-1}\frac{\rho_i(\ham^{i}(\chi_M\lphi),T_k)}{(\chi_M\lphi,\lphi)}
  \quad\mbox{for}\quad 1\le k\le L.
\]
We further define
\begin{equation}
  \label{eq:vec-local-reson}
  \xm_{k,0} = \ham^{k}\xm_{M}\R^{k-1}e_1, \quad
  \xm_{0,k} = \ham^{k}\xm_{M}\R^{k-1}e_2
  \quad\mbox{for}\quad 1\le k\le L.
\end{equation}

\begin{lemma}[Localized resonance]
  \label{le:local-reson}
  The following identities hold:
  \begin{enumerate}[label={\rm (\roman*)}]
    \item non-vanishing: $\|\xm_M\|_{L^2}^2 =4(1+o(1))\log M$;
      \label{it:local-reson=non-van}
    \item non-degeneracy: $(\lphi, \xm_M) =4(1+o(1))\log M$;
      \label{it:local-reson=non-deg}
    \item orthogonality: $(T_j, \xm_M) =0$ for $1\le j\le L$.
      \label{it:local-reson=orth}
  \end{enumerate}
\end{lemma}

\begin{lemma}[Vectorial localized resonance]
  \label{le:vec-local-reson}
  For $1\le j,k \le L$, we have
  \begin{enumerate}[label={\rm (\roman*)}]
    \item non-degeneracy: $(\phia[j], \xm_{k,0}) = (\phib[j], \xm_{0,k}) = 4\delta_{j,k}\,(1+o(1))\log M$;
      \label{it:vec-local-reson=non-deg}
    \item orthogonality: $(\phia[j], \xm_{0,k}) = (\phib[j], \xm_{k,0}) = 0$.
      \label{it:vec-local-reson=orth}
  \end{enumerate}
\end{lemma}

Indeed, Lemma~\ref{le:vec-local-reson} is a direct consequence of Lemma~\ref{le:local-reson}, since by the definition in~\eqref{eq:vec-reson} and the self-adjointness of $\ham$, we have
\[
  (\phia[j], \xm_{k,0})
  = \begin{cases}
    (\ham^{k-j}\lphi e_1, \xm_M e_1) = 0,           & j<k, \\
    (\lphi e_1, \xm_M e_1) = (\lphi, \xm_M) \neq 0, & j=k, \\
    (\phia[j-k], \xm_M e_1) = (T_{j-k}, \xm_M) = 0, & j>k.
  \end{cases}
\]
An analogous result holds for $(\phib[j], \xm_{0,k})$, which yields Lemma~\itref{le:vec-local-reson}{it:vec-local-reson=non-deg}. Moreover, Lemma~\itref{le:vec-local-reson}{it:vec-local-reson=orth} follows from the orthogonality of the unit vectors:
\[
  (\phia[j], \xm_{0,j}) = (\phib[j], \xm_{j,0}) = \int \lphi \xm_M (e_1 \cdot e_2) = 0.
\]
Now we turn to proving Lemma~\ref{le:local-reson}.

\begin{proof}[Proof of Lemma~\ref{le:local-reson}]
  We define
  \[
    f_0 = \chi_M\lphi,
    \andd
    f_k = f_{k-1} +\rho_{k}\ham^{k}(\chi_M\lphi)
    \quad\mbox{for}\quad k\ge 1,
  \]
  and claim that
  \[
    (f_k,\lphi) = 4(1+o(1))\log M,
    \andd
    (f_k, T_\ell) = 0
    \quad\mbox{for}\quad 1\le \ell\le k.
  \]
  Then Lemma~\itref{le:local-reson}{it:local-reson=non-deg}, \ref{it:local-reson=orth} follow immediately by taking $\xm_{M} = f_{L}$. We now prove the claim by induction.

  \emph{Case $k=1$}. From $\ham\lphi = 0$ and the asymptotics~\eqref{eq:lphi} of $\lphi$,
  \begin{align*}
    (f_1,\lphi)
    & = (\chi_M\lphi,\lphi) +\rho_1(\chi_M\lphi,\ham\lphi) \\
    & = \int \chi_M \lphi^2
    = 4(1+o(1))\log M.
  \end{align*}
  Similarly,
  \[
    (f_1,T_1)
    = (\chi_M\lphi,T_1) -\frac{(\chi_M\lphi,T_1)}{(\chi_M\lphi,\lphi)}(\chi_M\lphi,\lphi)
    = 0.
  \]
  \emph{Case $k\to k+1$}. By induction and $\ham\lphi=0$,
  \begin{align*}
    (f_{k+1},\lphi)
    & = (f_k,\lphi) +\rho_{k+1}(\ham^{k}(\chi_M\lphi),\ham\lphi) \\
    & = 4(1+o(1))\log M.
  \end{align*}
  For $1\le \ell\le k$, since $\ham^{k+1}T_\ell = \ham^{k+1}(\ham^{-\ell}\lphi) = 0$, we have
  \begin{align*}
    (f_{k+1}, T_\ell)
    & = (f_k,T_\ell) +\rho_{k+1}(\ham^{k+1}(\chi_M\lphi),T_\ell) \\
    & = \rho_{k+1}(\chi_M\lphi,\ham^{k+1}T_\ell)
    = 0.
  \end{align*}
  In addition, for $\ell = k+1$, since $\ham^{k+1}T_{k+1} = \lphi$, from the definition of $\rho_{k+1}$ and ${(f_k)}_k$, we compute:
  \begin{align*}
    (f_{k+1},T_{k+1})
    & = (f_k,T_{k+1}) +\rho_{k+1}(\ham^{k+1}(\chi_M\lphi),T_{k+1})                                                                   \\
    & = (f_k,T_{k+1}) -\sum_{i=0}^{k} \frac{\rho_i(\ham^i(\chi_M\lphi),T_{k+1})}{(\chi_M\lphi,\ham^{k+1}T_{k+1})}(\chi_M\lphi,\lphi) \\
    & = \left(f_k-\sum_{i=0}^{k} \rho_i\ham^i(\chi_M\lphi),\, T_{k+1}\right)                                                         \\
    & = 0,
  \end{align*}
  which completes the proof of the claim.
  It remains to show~\ref{it:local-reson=non-van} for $\xm_M = f_L$. From the asymptotics~\eqref{eq:lphi} and Corollary~\ref{co:resona}, it holds that
  \[
    |(\ham^i(\chi_M\lphi), T_k)|  = |(\chi_M\lphi, T_{k-i})|
    \les M^{2(k-i)}\log M,
    \quad 0\le i\le k-1.
  \]
  Then, together with $\rho_0=1$, a direct induction yields
  \[
    |\rho_{k}| \les M^{2k}
    \for k\ge 0
    \quad\Longrightarrow\quad
    |f_{L}(y) - f_0(y)|
    \les \frac{\one_{y\sim M}}{\jbk[y]},
  \]
  where we use $\ham(\chi_M \lphi) = 0$ for $y\le M$ in the last estimate. Consequently,
  \begin{align*}
    \|\xm_M\|_{L^2}^2 = \|f_L\|_{L^2}^2
    & = (f_0,f_0) + 2(f_0,f_L-f_0) + (f_L-f_0,f_L-f_0)                             \\
    & = \int\chi_M^2 \lphi^2 + O\left(\int_{y\sim M}\frac{y\dd y}{\jbk[y]^2}\right) \\
    & = 4(1+o(1))\log M.
  \end{align*}
  This completes the proof of Lemma~\ref{le:local-reson}.
\end{proof}

\subsection{Modulation and the bootstrap setup}
\label{sub:mod-boot}
In this subsection, we develop modulation theory and the bootstrap setup. Recall from Lemma~\ref{le:local-app-sol} that $\wzt$ is an approximate solution to the equation~\eqref{eq:vect-form}. To obtain an actual solution, we add to $\wzt$ an extra corrector $\w$, which is commonly referred to as the radiation. Specifically, we let
\begin{equation}
  \label{eq:w}
  \hw = \wzt +\w
  \with
  \w = \begin{pmatrix}
    \alpha \\
    \beta  \\
    \gamma
  \end{pmatrix},
\end{equation}
and thus
\begin{equation}
  \label{eq:u-decomp}
  u(t,r) = e^{\vtha\R}\Bigl(Q+(e_r,e_\tau,Q)(\wzt+\w)\Bigr)_{\la}.
\end{equation}
Then, by standard modulation theory, the decomposition \eqref{eq:u-decomp} is unique on a sufficiently small time interval $[0,t_1]$, with the geometric parameters $\la,b,\vtha,a$ being $C^1$ maps and the radiation satisfying the orthogonality conditions
\begin{equation}
  \label{eq:ortho}
  (\alpha,\ham^{k}\xm_M) = (\beta,\ham^{k}\xm_M) = 0, \quad 0\le k\le L,
\end{equation}
and the smallness condition
\[
  \|\nabla\w\|_{L^2} +\Big\|\frac{\w}{y}\Big\|_{L^2} +|b_1| +|a_1| \ll 1.
\]
The existence of such a decomposition is ensured by the nonvanishing of the determinant (we denote $T_0 = \lphi$):
\[
  \det\Bigl(\bigl\{(T_j, \ham^{k}\xm_M)\bigr\}_{0\le j,k\le L}\Bigr)
  = {(\lphi,\xm_M)}^{L+1} \neq 0,
\]
the matrix being diagonal with non-zero diagonal entries $(\lphi, \xm_M)$, by Lemma~\ref{le:local-reson}.
Thus the decomposition~\eqref{eq:u-decomp} and the orthogonality conditions~\eqref{eq:ortho} are preserved as long as $\w$ remains small in the $\dot{H}^1$ topology. Thus the original radiation $\w$ need not satisfy the orthogonality conditions; see~\citep[Remark~4.1]{merle_blowup_2013} for the corresponding reparametrization argument.

We measure $\w$ by the energy norm $\E_1$ and the high-order Sobolev norms $\E_{2k}$ ($1\le k\le L+1$) adapted to the Hamiltonian $\mham$, which are defined as follows:
\begin{equation}
  \begin{gathered}
    \label{eq:e2k}
    \E_1 =\|\nabla\wp\|_{L^2}^2 +\Big\|\frac{\wp}{y}\Big\|_{L^2}^2, \\
    \E_{2k} = \|\mhamp^{k}\w\|_{L^2}^2
    \quad\mbox{for}\quad 1\le k\le L+1.
  \end{gathered}
\end{equation}
With this setup, we now state the bootstrap argument to control the radiation $\w$ and the parameters $(\la, b, \vtha, a)$. Recall from~\eqref{eq:ss-trans} the rescaled variables:
\[
  y =\frac{r}{\la(t)}, \quad
  s = s_0 +\int_0^{t}\frac{d\tau}{\la^2(\tau)},
\]
where the map $t \mapsto s$ is a bijection satisfying
\[
  s|_{t=0} =s_0,
  \quad s|_{t\uparrow T} \to +\infty.
\]
Here $T\in\mbR_+$ is the expected blowup time. The existence of such radiation $\w$ on the actual time interval $[0, T)$ is now equivalent to the long-time existence on the rescaled time interval $[s_0,\infty)$ \footnote{For convenience we sometimes use the actual time $t$, especially for the initial time $t=0$. The reader should distinguish it from the rescaled time $s$.}. Recall that in Subsection~\ref{sub:ab-sys} we have imposed
\[
  b_k = b_k^e +\frac{U_k}{s^k(\log s)^{\frac54}}, \quad
  a_k = a_k^e +\frac{V_k}{s^k(\log s)^{\frac54}}, \quad
  1\le k\le L.
\]
Define the transformed versions of $(U, V)$ by
\[
  U' = P_U U, \quad V' = P_V V,
\]
where $P_U, P_V\in\GL(L,\mbR)$ are given in Lemma~\itref{le:b-sys-lin}{it:b-sys-lin=diag} and~\itref{le:a-sys-lin}{it:a-sys-lin=diag}. Then the existence of $(a,b)$ reduces to that of $(U', V')$. We now state the assumptions made on the initial data:
\begin{enumerate}[label={\rm (\roman*)}]
  \item control of the modes:
    \begin{equation}
      \label{ini:modes-bnd}
      |U'_1|(0) \le \frac14, \quad
      \bigl(U'_2, \dots, U'_L, V'_1, \dots, V'_L\bigr)(0) \in \bar{B}(0,1).
    \end{equation}
  \item initial normalization:
    \begin{equation}
      \label{ini:fixed-scaling}
      \vtha(0)=0, \quad \la(0)=1.
    \end{equation}
  \item control of the adapted norms:
    \begin{align}
      \E_1(0)    & \le b_1(0)^{1/2},
      \label{ini:e1-bnd}                                  \\
      \E_{2k}(0) & \le b_1(0)^{10L}, \quad 1\le k\le L+1.
      \label{ini:e2k-bnd}
    \end{align}
\end{enumerate}

\begin{remark}
  \label{re:fixed-scaling}
  (i) $\bar{B}(0,1)$ denotes the closed unit ball in $\mbR^{2L-1}$. \\
  (ii) The initial normalization~\eqref{ini:fixed-scaling} is a legitimate assumption. Indeed, since the \sch\ map equation~\eqref{eq:sm} admits the scaling and rotational symmetries, any fixed transformation
  \[
    u \to u_{\vtha,\la}
    = e^{\vtha\R} u\biggl(\frac{t}{\la^2}, \frac{r}{\la}\biggr),
    \with (\vtha,\la) \in\mbR\times\mbR_+
  \]
  preserves solutions. In other words, if a solution $u$ blows up in finite time, then so does the transformed one $u_{\vtha, \la}$. Thus an a priori transformation ensures~\eqref{ini:fixed-scaling}.
\end{remark}

We claim that the radiation $\w$ is trapped in a regime where $\E_1, \E_{2k}$ ($1\le k\le L+1$) are bounded, and hence remains in the associated Sobolev space. More precisely, we have the following.

\begin{proposition}[Bootstrap regime]
  \label{pro:boot}
  Given any $L\in \mbN_+$, there exist sufficiently large universal constants $M, K, s_0>1$ depending on $L$ and an initial datum
  \[
    (U'_2, \dots, U'_L, V'_1, \dots, V'_L)(0) \in \bar{B}(0,1),
  \]
  such that the following bounds hold for $s\in[s_0,+\infty)$:
  \begin{enumerate}[label={\rm (\roman*)}]
    \item control of the adapted norms:
      \begin{align}
        \E_1(s)      & \le 10b_1(0)^{1/4},
        \label{pro:boot-e1-bnd}                                                                      \\
        \E_{2k}(s)   & \le \bigl\{ b_1^{(2k-1)\frac{2L}{2L-1}}\lgb^{K}\bigr\}(s), \quad 1\le k\le L,
        \label{pro:boot-e2k-bnd}                                                                     \\
        \E_{2L+2}(s) & \le \frac{K b_1^{2L+2}}{\lgb^2}(s),
        \label{pro:boot-e2l2-bnd}
      \end{align}
    \item control of the stable mode:
      \begin{equation}
        \label{pro:boot-u1-bnd}
        |U'_1|(s) \le 1,
      \end{equation}
    \item control of the unstable modes:
      \begin{equation}
        \label{pro:boot-u2-bnd}
        \bigl(U'_2, \dots, U'_L, V'_1, \dots, V'_L\bigr)(s) \in \bar{B}(0,1).
      \end{equation}
  \end{enumerate}
\end{proposition}

Indeed, for any small enough initial data $(\w,U',V')(0)$, the local well-posedness of the \sch\ flow guarantees that the bounds~\eqref{pro:boot-e1-bnd}--\eqref{pro:boot-u2-bnd} hold on a small time interval $[s_0, s_1]$, possibly with $s_1=s_0$ for boundary initial data. We therefore define the \emph{exit time} by
\begin{equation}
  \label{eq:exit-time}
  s^{\ast}(u_0)
  = \sup\left\{s\ge s_0: \mbox{\eqref{pro:boot-e1-bnd}--\eqref{pro:boot-u2-bnd} hold on $[s_0, s]$}\right\},
\end{equation}
which describes in terms of rescaled variables the moment $s^{\ast}$ when the solution $u$ leaves the bootstrap regime. Then Proposition~\ref{pro:boot} asserts the existence of an initial datum $u_0$ such that $s^{\ast}(u_0) =+\infty$. Together with the coercivity of $\E_1$ (Lemma~\ref{le:coer-ham}) and the orthogonality~\eqref{eq:ortho}, this ensures the existence of the radiation $\w$. Hence, in view of~\eqref{eq:u-decomp}, the existence of the solution follows.

\begin{remark}
  Proposition~\ref{pro:boot} is crucial in proving Theorem~\ref{th:main-thm}. Proposition~\ref{pro:boot} is proved in Subsection~\ref{sub:clos-boot}, and Theorem~\ref{th:main-thm} is proved in Subsection~\ref{sub:proof-main-thm}.
\end{remark}

\subsection{Equation for the radiation}
\label{sub:w-eq}
This subsection is devoted to the equation satisfied by the radiation $\w$ and to the associated high-order derivatives.

From~\eqref{eq:vect-form} and \eqref{le:local-app-sol=wzt-eq}, the equation for $\w$ is
\begin{equation}
  \label{eq:w-eq}
  \begin{cases}
    \prs\w -\frac{\la_s}{\la}\vla\w +\hj\mham\w +\mF = 0, \\
    \mF = \psizt +\modt +\vtha_s\z\R\w +\w\times(\mham\wzt +\hp\lphi),
  \end{cases}
\end{equation}
where, in the rest of this article, we use the notation
\[
  \hj = \hj[\hw].
\]
We let $\W(t,r) = \w(s,y)$. Then in view of~\eqref{eq:w}, its components read
\begin{equation}
  \label{eq:scale-w-w}
  \W(t,r) = \begin{pmatrix}
    \alpha_\la(t,r) \\
    \beta_\la(t,r)  \\
    \gamma_\la(t,r)
  \end{pmatrix},
  \with
  \begin{cases}
    \alpha_\la(t,r)=\alpha(s,y), \\
    \beta_\la(t,r)=\beta(s,y),   \\
    \gamma_\la(t,r)=\gamma(s,y),
  \end{cases}
\end{equation}
where we use the subscript notation $(\cdot)_{\la}$ as in~\eqref{eq:la-trans}. From~\eqref{eq:w-eq},
\begin{equation}
  \label{eq:W-eq}
  \prt\W +\hj\mhaml\W +\frac{1}{\la^2}\mF_{\la} = 0,
\end{equation}
where the renormalized vectorial Hamiltonian $\mhaml$ is
\begin{align*}
  \mhaml\begin{pmatrix*} \alpha_{\la} \\ \beta_{\la} \\ \gamma_{\la} \end{pmatrix*}
  & = \begin{pmatrix*}[l]
    \ham_{\la}\alpha_{\la} -\frac{\zt_{\la}}{\la}\prr\gamma_{\la}\\
    \ham_{\la}\beta_{\la} \\
    -\Delta_{r}\gamma_{\la} +\frac{\zt_{\la}}{\la}\bigl(\prr +\tfrac{\z_{\la}}{r}\bigr)\alpha_{\la}
  \end{pmatrix*},
\end{align*}
with the scalar Hamiltonian
\[
  \ham_{\la} = -\Delta_{r} +\frac{V_{\la}(r)}{r^2}.
\]
We define the leading part of $\mhaml$ by
\[
  \mhamp_{\la}\begin{pmatrix*} \alpha_{\la} \\ \beta_{\la} \\ \gamma_{\la} \end{pmatrix*}
  = \begin{pmatrix*}
    \ham_{\la}\alpha_{\la} \\
    \ham_{\la}\beta_{\la} \\
    0
  \end{pmatrix*}.
\]
From~\eqref{eq:ham-factori}, it admits the decomposition $\mhamp_{\la} = \ma_{\la}^{\ast} \ma_{\la}$ with
\[
  \ma_{\la} \begin{pmatrix} \alpha_{\la}\\ \beta_{\la}\\ \gamma_{\la}\end{pmatrix}
  = \begin{pmatrix} \oa_{\la}\alpha_{\la}\\ \oa_{\la}\beta_{\la}\\ 0 \end{pmatrix},
  \quad
  \ma_{\la}^{\ast} \begin{pmatrix} \alpha_{\la}\\ \beta_{\la}\\ \gamma_{\la}\end{pmatrix}
  = \begin{pmatrix} \as_{\la}\alpha_{\la}\\ \as_{\la}\beta_{\la}\\ 0 \end{pmatrix}.
\]
Now we define the derivatives of $\w$ associated with $\hj\mham$ and $\R\mham$. Let
\[
  \w_{2k} = (\hj\mham)^{k}\w
  \andd
  \w_{2k}^{0} = (\R\mham)^{k}\wp,
\]
for all $k\in\mbN$. Accordingly,
\begin{equation}
  \label{eq:scaling-w-w}
  \W_{2k} = (\hj\mhaml)^k\W
  \andd
  \W_{2k}^0 = (\R\mhaml)^k\Wp.
\end{equation}
Since $\hj\mham = \R\mham + \hw\times\mham$, where $\hw$ is expected to be small, the operator $\R\mham$ serves as the dominating part of $\hj\mham$. Thus $\w_{2k}^0$ ($\W_{2k}^0$ resp.) is the leading term of $\w_{2k}$ ($\W_{2k}$ resp.). We denote the corresponding remainders by
\begin{equation}
  \label{def:w2k1}
  \w_{2k}^{1} = \w_{2k} -\w_{2k}^{0}
  \andd
  \W_{2k}^1 = \W_{2k} -\W_{2k}^0.
\end{equation}
From the definition of $\W$, we have
\[
  \W_{2k}^{0} = \frac{\w_{2k}^{0}}{\la^{2k}}, \quad
  \W_{2k}^{1} = \frac{\w_{2k}^{1}}{\la^{2k}}.
\]
In view of~\eqref{def:w2k1} and~\eqref{eq:mham-r}, the adapted energy $\E_{2k}$ can be rewritten as
\begin{equation}
  \label{eq:w2k0-e2k}
  \E_{2k} = \int|(\R\mham)^{k}\wp|^2 =\|\w_{2k}^0\|_{L^2}^2.
\end{equation}
Moreover, using \eqref{eq:W-eq}, we have
\begin{equation}
  \label{eq:w2k-eq}
  \begin{split}
    \prt\W_{2k} +\W_{2k+2}
    & = [\prt, (\hj\mhaml)^k]\W -(\hj\mhaml)^{k}\Bigl(\frac{1}{\la^2}\mF_{\la}\Bigr),                                 \\
    \prt\W_{2k}^0 +\W_{2k+2}^0
    & = [\prt, (\R\mhaml)^k]\Wp + (\R\mhaml)^k(\R^2\W_2^1) -(\R\mhaml)^{k}\Bigl(\frac{1}{\la^2}\mF^{\pp}_{\la}\Bigr).
  \end{split}
\end{equation}

\subsection{Modulation equations}
\label{sub:mod-eq}
In this subsection, we derive the dynamical system of the geometric parameters $(\vtha,\la, a,b)$. This is done by projecting~\eqref{eq:w-eq} on fixed directions, especially $\xm_{k,0}, \xm_{0,k}$ defined by~\eqref{eq:vec-local-reson}.

We let
\[
  D(s)
  = |\vtha_s+a_1| +\Big|\frac{\la_s}{\la}+b_1\Big|
  +\sum_{k=1}^L \Bigl(|\prs a_k +A_k| + |\prs b_k +B_k|\Bigr),
\]
where $A_k, B_k$ are given in~\eqref{eq:msf-ab}.

\begin{lemma}[Modulation equations]
  \label{le:mod-eq}
  For $s\in[s_0,s^{\ast})$, we have
  \begin{gather}
    |\vtha_s+a_1| +\Big|\frac{\la_s}{\la}+b_1\Big| +\sum_{k=1}^{L-1} \Bigl(|\prs a_k +A_k| + |\prs b_k +B_k|\Bigr)
    \les \delta(b_{1}^{\ast})\frac{b_1^{L+1}}{\lgb},
    \label{le:mod-eq-1}\\
    |\prs a_L +A_L| +|\prs b_L +B_L|
    \les \frac{1}{\sqrt{\log M}}\biggl(\sqrt{\E_{2L+2}} +\frac{b_1^{L+1}}{\lgb}\biggr).
    \label{le:mod-eq-2}
  \end{gather}
\end{lemma}

\begin{remark}
  (i) The ODE system~\eqref{eq:ode} is an approximation of the $(a,b)$ part of the above inequalities. \\
  (ii) The proof of this lemma relies on some definitions and bounds established in Subsection~\ref{sub:eng-func} and Appendix~\ref{ap:interp-bnd}, which are independent of this lemma.

\end{remark}

\begin{proof}
  We carry out the proof in detail only for $(\vtha,a)$.  The case of $(\la,b)$ follows by a symmetric argument.

  \step[1, Laws for $a_L$]
  Choose the bootstrap threshold small enough that $2M<S_1$ throughout the interval. Taking the inner product of~\eqref{eq:w-eq} with $\xm_{L,0}$,
  \begin{align*}
    -(\modt, \xm_{L,0})
    & = -\frac{\la_s}{\la}(\vla\w, \xm_{L,0}) +\vtha_s(\z\R\w, \xm_{L,0}) +(\hj\mham\w, \xm_{L,0}) \\
    & \quad +(\psizt, \xm_{L,0}) +(\w\times(\mham\wzt+\hp\lphi), \xm_{L,0}).
  \end{align*}
  Here, the time derivative term vanishes due to the orthogonality~\eqref{eq:ortho}. For the left-hand side, from Lemma~\ref{le:vec-local-reson} and the fact that $\supp \xm_{M} \subset \{y\le 2M\} \subset \{y<S_1\}$,
  \begin{align*}
    (\xm_{L,0}, \phiat[L]) & = 4(1+o(1))\log M,
    \;\;\mbox{and}\;\;
    \begin{cases}
      (\xm_{L,0}, \phiat[k]) = 0, \; 1\le k\le L-1, \\
      (\xm_{L,0}, \phibt[k]) = 0, \; 1\le k\le L
    \end{cases}.
  \end{align*}
  By $\ham\lphi = 0$, we have
  \[
    (\xm_{L,0}, \lphi e_i) = 0
    \for i=1,2.
  \]
  These, together with the explicit expression~\eqref{le:local-app-sol=modt} of $\modt$, give
  \[
    (\modt, \xm_{L,0})
    = 4(1+o(1))\log M \, (\prs a_L +A_L) +O(M^C b_1 D(s)).
  \]
  Let us treat the right-hand side. First, by Cauchy-Schwarz,
  \[
    \left|\frac{\la_s}{\la}(\vla\w,\xm_{L,0})\right|
    \les (b_1 +D(s)) M^C\biggl( \int\frac{|\mhamp^L(\vla\w)|^2}{\jbk[y]^4 \jbk[\log y]^2} \biggr)^{1/2}.
  \]
  We split the integration domain into $y\ge1$ and $y\le 1$.
  For $y\ge1$, from Lemma~\ref{le:coer-e-2k2}, the integral is bounded by
  \[
    \sum_{j=0}^{2L+2}\int_{y\ge 1} \frac{|\pry^{j}\wp|^2}{\jbk[y]^{2(2L+2-j)}\jbk[\log y]^2}
    \les C(M)\E_{2L+2}.
  \]
  By Corollary~\itref{co:dore}{it:dore=w}, $\wp\in\mcD_{L+1}$ with factor $b_1^{L+1}\lgb^{-1}$ for $y\le1$. We control the polynomial part by Lemma~\itref{le:preserv}{it:preserv=op}. Since $\vla=y\pry$, we use only derivatives up to order $2L+1$ of the remainder. Hence
  \[
    |\mhamp^L(\vla\wp)|
    \les C(M,K)\frac{b_1^{L+1}}{\lgb}y\jbk[\log y],
  \]
  from which we obtain the squared $L^2$ bound
  \[
    \int_{y\le 1} \frac{|\mhamp^L(\vla\w)|^2}{\jbk[y]^2\jbk[\log y]^2}
    \le \|\mhamp^L(\vla\wp)\|_{L^2(y\le1)}^2
    \les C(M,K)\frac{b_1^{2L+2}}{\lgb^2}.
  \]
  These, together with the bootstrap bound~\eqref{pro:boot-e2l2-bnd}, yield
  \begin{align*}
    \Big|\frac{\la_s}{\la}(\vla\w,\xm_{L,0})\Big|
    & \les C(M,K)(b_1 +D(s))b_1^{L+1}\lgb^{-1} \\
    & \les b_1 D(s) +C(M)\frac{b_1^{L+2}}{\lgb}.
  \end{align*}
  Note that the $\vtha_s$ term yields the identical bound. Then, for the quasilinear term $\hj\mham\w$, from~\eqref{def:w2k1}, $\hj\mham\w = \w_2^0 +\w_2^1$. For the $\w_2^0$ part, from Lemma~\itref{le:local-reson}{it:local-reson=non-van},
  \begin{align*}
    |(\w_2^0, \xm_{L,0})|
    \les \|\mhamp^{L+1}\wp\|_{L^2} \|\xm_M\|_{L^2}
    \les \sqrt{\log M}\,\sqrt{\E_{2L+2}}.
  \end{align*}
  Note that this bound leads to the final bound in~\eqref{pf:mod-eq=mod-compute-main-1}. For the $\w_2^1$ part, by the definition in~\eqref{def:w2k1}, $\w_2^1 = \R\mham(\gamma\ethr) +\hw\times\mham\w$. From the proof of Lemma~\ref{le:w2l2-e},
  \[
    |(\w_{2}^1, \xm_{L,0})|
    \les \delta(b_{1}^{\ast}) \sqrt{\E_{2L+2}}+ b_1^{\frac{1}{4L}} \frac{b_1^{L+1}}{\lgb}.
  \]
  Next, the cross terms involving $\wzt, \w$ are bounded in the same fashion. Last, from~Lemma~\itref{le:local-app-sol}{it:local-app-sol=local}, we have
  \[
    |(\psizt, \xm_{L,0})| \les M^C b_1^{L+\frac32}.
  \]
  Inserting these back, we obtain
  \begin{equation}
    \label{pf:mod-eq=mod-compute-main-1}
    |\prs a_L +A_L|
    \les \frac{\sqrt{\E_{2L+2}}}{\sqrt{\log M}} +M^C \biggl(b_1 D(s) +b_1^{\frac{1}{4L}}\frac{b_1^{L+1}}{\lgb}\biggr).
  \end{equation}

  \step[2, Laws for $a_k$ {\rm (}$1\le k\le L-1${\rm )}] Taking the inner product of \eqref{eq:w-eq} with $\xm_{k,0}$, we have
  \begin{align*}
    -(\modt, \xm_{k,0})
    & = -\frac{\la_s}{\la}(\vla\w, \xm_{k,0}) +\vtha_s(\z\R\w, \xm_{k,0}) +(\hj\mham\w, \xm_{k,0}) \\
    & \quad +(\psizt, \xm_{k,0}) +(\w\times(\mham\wzt+\hp\lphi), \xm_{k,0}),
  \end{align*}
  where from Lemma~\ref{le:vec-local-reson},
  \[
    (\modt, \xm_{k,0})
    = (\prs a_k +A_k)(\lphi, \xm_M) +O(M^C b_1 D(s)).
  \]
  We apply Lemma~\ref*{le:coer-e-2k2} and \eqref{pro:boot-e2l2-bnd} to estimate:
  \[
    \Big|\frac{\la_s}{\la}(\vla\w,\xm_{k,0})\Big| +|\vtha_s(\z\R\w,\xm_{k,0})|
    \les C(M) D(s) \sqrt{\E_{2L+2}} +b_1^{L+2}.
  \]
  From the orthogonality~\eqref{eq:ortho},
  \[
    |(\w_{2}^0, \xm_{k,0})| = 0.
  \]
  Note that this cancellation yields a sharper bound than \eqref{pf:mod-eq=mod-compute-main-1}.
  Other terms are small perturbations, as illustrated previously. Consequently,
  \begin{equation}
    \label{pf:mod-eq=mod-compute-main-2}
    |\prs a_k +A_k|
    \les \delta(b_{1}^{\ast})\sqrt{\E_{2L+2}} +M^C \biggl(b_1 D(s) +b_1^{\frac{1}{4L}}\frac{b_1^{L+1}}{\lgb}\biggr).
  \end{equation}

  \step[3, Laws for $\vtha$] Taking the inner product of~\eqref{eq:w-eq} with $\xm_M e_2$, we obtain a similar identity with the left-hand side
  \[
    (\modt, \xm_{M}e_2)
    = (\vtha_s +a_1)(\lphi, \xm_M) +O(M^C b_1 D(s)).
  \]
  We estimate the right-hand side in turn, which gives
  \begin{equation}
    |\vtha_s +a_1|
    \les \delta(b_{1}^{\ast})\sqrt{\E_{2L+2}} +M^C \biggl(b_1 D(s) +b_1^{\frac{1}{4L}}\frac{b_1^{L+1}}{\lgb}\biggr).
    \label{pf:mod-eq=mod-compute-main-3}
  \end{equation}

  \step[4, Conclusion]
  We collect~\eqref{pf:mod-eq=mod-compute-main-1}, \eqref{pf:mod-eq=mod-compute-main-2}, \eqref{pf:mod-eq=mod-compute-main-3}, and also the corresponding estimates for $(\la,b)$. The $M^C b_1 D(s)$ on the right-hand side is absorbed by the left-hand side, as we will choose $s_0\gg M^C$ so that $M^C b_1 \ll 1$. Therefore,
  \begin{equation}
    \label{pf:mod-eq=d}
    |D(s)| \les \frac{1}{\sqrt{\log M}}\biggl(\sqrt{\E_{2L+2}} +\frac{b_1^{L+1}}{\lgb}\biggr).
  \end{equation}
  This gives~\eqref{le:mod-eq-2}. We insert \eqref{pf:mod-eq=d} into~\eqref{pf:mod-eq=mod-compute-main-2}--\eqref{pf:mod-eq=mod-compute-main-3}, and then use the bootstrap bound~\eqref{pro:boot-e2l2-bnd}. The constant $K$ is absorbed by the infinitesimal $\dbs$. This yields~\eqref{le:mod-eq-1}, and hence completes the proof.
\end{proof}

We slightly modify the definitions of $a_L, b_L$ to derive a refined modulation equation for $a_L, b_L$, which will be used in Subsection~\ref{sub:clos-boot}. For any $\eta\in(0,\tfrac14)$, we let
\[
  B_{\eta} = b_1^{-\eta}.
\]
We define
\[
  \ti{\xm}_{L,0} = \ham^{L}(\chi_{ B_{\eta}}\lphi)\R^{L-1}e_1, \quad
  \ti{\xm}_{0,L} = \ham^{L}(\chi_{ B_{\eta}}\lphi)\R^{L-1}e_2.
\]
Then,  a direct computation yields
\begin{align*}
  (\ti{\xm}_{L,0}, \phia[k])
  & = (\ti{\xm}_{0,L}, \phib[k]) = \delta_{L,k}\bigl(4\eta\lgb +O(1)\bigr), \\
  (\ti{\xm}_{L,0}, \phib[k])
  & = (\ti{\xm}_{0,L}, \phia[k]) = 0.
\end{align*}
Define the modified $a_L, b_L$ by
\begin{equation}
  \label{eq:ti-b}
  \bar{a}_L = a_L +\frac{(\w,\ti{\xm}_{L,0})}{4\eta\lgb},\quad
  \bar{b}_L = b_L +\frac{(\w,\ti{\xm}_{0,L})}{4\eta\lgb}.
\end{equation}

\begin{lemma}[Refined modulation equations]
  \label{le:refined-mod-eq}
  Let $0<\eta<\frac14$ be sufficiently small. For all $s\in[s_0,s^{\ast})$, the following hold:
  \begin{enumerate}[label={\rm (\roman*)}]
    \item almost equivalences: \label{it:refined-mod-eq=diff}
      \[
        |\bar{a}_L -a_L| +|\bar{b}_L -b_L| \les b_1^{L+\frac12},
      \]
    \item refined equations: \label{it:refined-mod-eq=eq}
      \begin{align*}
        & \bigl|\prs\bar{a}_L  + (2L -  1 + \cb)b_1\bar{a}_L  -  a_1\bar{b}_L\bigr|                                                                                         \\
        & +\bigl|\prs\bar{b}_L  + (2L -  1 + \cb)b_1\bar{b}_L  +  a_1\bar{a}_L\bigr| \les  \frac{1}{\sqrt{\lgb}}\biggl(  \sqrt{\E_{2L+2}}  + \frac{b_1^{L+1}}{\lgb}\biggr).
      \end{align*}
  \end{enumerate}
\end{lemma}

\begin{proof}
  The strategy of this proof is parallel to that of Lemma~\ref{le:mod-eq}. Let us sketch the proof for $a_L$.

  Proof of~\ref{it:refined-mod-eq=diff}. From Lemma~\ref{le:coer-e-2k2},
  \begin{align*}
    |(\w,\ti{\xm}_{L,0})|
    & \les \Biggl(\int\frac{|\mhamp^L\w|^2}{\jbk[y]^4\jbk[\log y]^2}\Biggr)^{1/2} \bigl\|\jbk[y]^2\jbk[\log y]\chi_{ B_{\eta}}\lphi\bigr\|_{L^2} \\
    & \les C(M)\sqrt{\E_{2L+2}} \,b_1^{-C\eta}\lgb                                                                                                   \\
    & \les \bigl[ C(M)b_1^{\frac12-C\eta} \bigr] b_1^{L+\frac12}
  \end{align*}
  for some constant $C>0$ independent of $\eta$. Taking $\eta$ sufficiently small, we observe that the factor in the bracket above is of size $O(1)$. This yields~\ref{it:refined-mod-eq=diff}.

  Proof of~\ref{it:refined-mod-eq=eq}. Taking the inner product of~\eqref{eq:w-eq} with $\ti{\xm}_{L,0}$,
  \begin{align*}
    -(\modt, \ti{\xm}_{L,0})
    & = (\prs\w, \ti{\xm}_{L,0}) -\frac{\la_s}{\la}(\vla\w, \ti{\xm}_{L,0}) +(\hj\mham\w, \ti{\xm}_{L,0})                         \\
    & \quad +(\psizt, \ti{\xm}_{L,0}) +\vtha_s(\z\R\w, \ti{\xm}_{L,0}) +\bigl(\w\times(\mham\wzt+\hp\lphi), \ti{\xm}_{L,0}\bigr).
  \end{align*}
  First, using the properties of $\ti{\xm}_{L,0}$ and~\eqref{le:mod-eq-2},
  \begin{align*}
    (\modt, \ti{\xm}_{L,0})
    & = (4\eta\lgb +O(1)) (\prs a_L +A_L)
    +O\bigl(b_1^{L+\frac32-C\eta}\bigr)                                     \\
    & = 4\eta\lgb (\prs a_L +A_L) + O\biggl(\frac{b_1^{L+1}}{\lgb}\biggr).
  \end{align*}
  Then, commuting $\prs$ with $\ham$,
  \[
    (\prs\w, \ti{\xm}_{L,0})
    = \frac{\dd}{\dd s}(\w,\ti{\xm}_{L,0}) -\bigl(\mhamp^{L}\w, \prs\chi_{ B_{\eta}}\lphi\R^{L-1} e_1\bigr).
  \]
  From~\eqref{le:mod-eq-1}, we have
  \[
    |\prs\chi_{ B_{\eta}}|
    \les b_1 \one_{y\sim B_{\eta}}.
  \]
  Thus by setting $\eta$ sufficiently small, we use Lemma~\ref{le:coer-e-2k2} and obtain
  \[
    \bigl|\bigl(\mhamp^L\w,\prs\chi_{B_{\eta}}\lphi\R^{L-1}e_1\bigr)\bigr|
    \les C(M)\sqrt{\E_{2L+2}}\, b_1^{1-C\eta}\lgb
    \les b_1^{L+\frac32}.
  \]
  The terms involving $\la_s/\la, \vtha_s$ share the same bounds.
  Next, for the linear term, in view of the definition in~\eqref{eq:e2k},
  \[
    |(\w_2^0, \ti{\xm}_{L,0})|
    \les \|\mhamp^{L+1}\w\|_{L^2} \|\chi_{B_{\eta}}\lphi\|_{L^2}
    \les \sqrt{\E_{2L+2}} \sqrt{\lgb}.
  \]
  Note that this bound leads to the dominant bound in~\ref{it:refined-mod-eq=eq}. The remaining terms are estimated in the same fashion. Consequently,
  \begin{align*}
    & \left|\frac{\dd}{\dd s}(\w,\ti{\xm}_{L,0})  +  \bigl(4\eta\lgb  +  O(1)\bigr)(\prs a_L  +  A_L)\right| \\
    & \les  \sqrt{\lgb}\biggl( \sqrt{\E_{2L+2}}  + \frac{b_1^{L+1}}{\lgb}\biggr).
  \end{align*}
  The equation~\ref{it:refined-mod-eq=eq} now follows from a straightforward reformulation of the above using~\eqref{eq:ti-b}. This completes the proof.
\end{proof}

\section{Energy estimates}
\label{sec:energy-est}
In this section, we establish the energy estimates needed for the proof of Proposition~\ref{pro:boot}, which is postponed to Section~\ref{sec:compl-constr}. We first give an equivalent form of the high-order energy $\E_{2L+2}$, and then derive the mixed energy estimate for $\E_{2L+2}$, in which the Virial and Morawetz functionals play a central role. Throughout this section, we work on the rescaled time interval $[s_0, s^{\ast})$, where $s^{\ast}$ is the exit time given by~\eqref{eq:exit-time}, on which the bootstrap bounds~\eqref{pro:boot-e1-bnd}--\eqref{pro:boot-u2-bnd} and Lemmas~\ref{le:mod-eq} and~\ref{le:refined-mod-eq} are in force. The estimates of Appendices~\ref{ap:coer-bnd} and~\ref{ap:interp-bnd} are also used.

\subsection{The energy functional}
\label{sub:eng-func}
In this subsection, we present an equivalent form of the energy functional $\E_{2L+2}$.

\begin{lemma}[Almost energy functional]
  \label{le:w2l2-e}
  For $s\in[s_0,s^{\ast})$, we have
  \[
    \|\w_{2L+2}^1\|_{L^2}^2 \les \dbs \E_{2L+2} +b_1^{ \frac{1}{4L}} \frac{b_1^{2L+2}}{\lgb^2}.
  \]
\end{lemma}

\begin{remark}\label{re:w2l2-e}
  In view of \eqref{eq:w2k0-e2k} and \eqref{def:w2k1}, we use H\"older and find
  \[
    \Bigl|\|\w_{2L+2}\|_{L^2}^2 -\E_{2L+2}\Bigr|
    \les \frac{1}{K}\|\w_{2L+2}^0\|_{L^2}^2 +K\|\w_{2L+2}^1\|_{L^2}^2,
  \]
  where the constant $K$ is the one given in the bootstrap bound~\ref{pro:boot-e2l2-bnd}.
  Together with Lemma~\ref{le:w2l2-e} and \ref{pro:boot-e2l2-bnd}, this yields for some constant $C$ independent of $K$ that
  \begin{align*}
    \Bigl| \|\w_{2L+2}\|_{L^2}^2 - \E_{2L+2} \Bigr|
    \le \frac{Cb_1^{2L+2}}{\lgb^2}.
  \end{align*}
\end{remark}

Before proving Lemma~\ref{le:w2l2-e}, let us introduce two notions useful in estimating cross terms.

\begin{definition}[Dore series $\mcD_p$]
  \label{def:dore}
  For any $p\in\mbN_+$, we define the function space $\mcD_p$ as the set of all vector-valued functions $\smw\in C^\infty(\mbR_+\times\mbR_+;\mbR^3)$ such that there exist an infinitesimal $\delta(b_1)>0$ and coefficients $(c_{1,k})_{k=1}^{p}$, $(c_{2,k})_{k=1}^{p}$, $(c_{3,k})_{k=1}^{p}$ with
  \[
    |c_{1,k}| + |c_{2,k}| + |c_{3,k}|
    \le C\delta(b_1)
    \for 1\le k\le p,
  \]
  such that $\smw$ admits the decomposition for $0<y\le 1$:
  \[
    \smw = \sum_{k=1}^{p}
    \begin{pmatrix}
      c_{1,k} \hs[.1] y^{2k-1} \\
      c_{2,k} \hs[.1] y^{2k-1} \\
      c_{3,k} \hs[.1] y^{2k}
    \end{pmatrix}
    + \begin{pmatrix}
      r_{1,p} \\
      r_{2,p} \\
      r_{3,p}
    \end{pmatrix},
  \]
  with the remainders satisfying
  \[
    |\pry^k r_{1,p}|+|\pry^k r_{2,p}|+\frac{|\pry^k r_{3,p}|}{y}
    \le C\delta(b_1)y^{2p-1-k}\jbk[\log y] \for 0\le k\le 2p-1.
  \]
  The constant $C$ may depend on $k,p$, but not on $b_1$.
  By truncation, we have $\mcD_{p+1}\subset\mcD_p$ with the same factor.
  For brevity, we say that $\smw\in\mcD_p$ is of Dore type of order $p$ with factor $\delta(b_1)$.
\end{definition}

The Dore series preserve their structure under the action of various operators, with a loss of order for differential operators.

\begin{lemma}[Action of $\mham$, $\times$ on $\mcD_p$]
  \label{le:preserv}
  For $p\in\mbN_+$, we let $\smw\in\mcD_{p+1}$. Then the following hold:
  \begin{enumerate}[label={\rm (\roman*)}]
    \item \label{it:preserv=mult}
      the functions $\R\smw, \z\smw, \zt\smw\in\mcD_{p+1}$;
    \item \label{it:preserv=op}
      the functions  $\R\mham\smw, \mhamp\smw, \vla\smw\in\mcD_p$;
    \item \label{it:preserv=cross}
      if $\vv\in\mcD_{p+1}$, then $\vv\times\smw,\vv\times\mham\smw\in\mcD_p$.
  \end{enumerate}
\end{lemma}

\begin{proof}
  Proof of~\ref{it:preserv=mult}. Since $\R=e_3\times$, we have $\R\smw\in\mcD_{p+1}$ by permuting the first two components of $\smw$ up to sign and setting the third component to zero. Multiplying the Dore expansion of $\smw$ by the even series of $\z,\zt$ and applying Leibniz to the remainders, we obtain $\z\smw,\zt\smw\in\mcD_{p+1}$.

  Proof of~\ref{it:preserv=op}. Let $\smw=\alpha e_1+\beta e_2+\gamma e_3$. Then
  \[
    \mham\smw
    = \bigl(\ham\alpha -\zt\pry\gamma  \bigr)e_1
    +\ham\beta   e_2
    +\biggl(-\Delta\gamma +\zt\biggl(\pry+\frac{\z}{y}\biggr)\alpha \biggr)e_3.
  \]
  Since $V,\z,\zt$ are even for $y\le1$, we have polynomial parts of the same parity under $\ham,\Delta$ and of the opposite parity under $\zt\pry,\zt(\pry+\z/y)$. Moreover, for $i\ge0$,
  \[
    \ham y^i=(1-i^2)y^{i-2}+O(y^i)
    =\begin{cases}
      O(y^{i-2}), & \mbox{$i\neq1$},\\
      O(y^i), & \mbox{$i=1$}.
    \end{cases}
  \]
  Thus we have odd polynomial parts vanishing at least to order one in the first two components. We remove the possible constant term in the third component by applying $\R$. For $k\le2p-1$, we bound the derivatives of order $k$ of the first two remainders by $C\delta(b_1)y^{2p-1-k}\jbk[\log y]$, so $\R\mham\smw\in\mcD_p$. From $\mhamp=-\R\mham\R$, we have the same conclusion for $\mhamp\smw$. Finally, for $i=1,2,3$, we have
  \[
    \pry^k(\vla r_{i,p+1})
    =y\pry^{k+1}r_{i,p+1}+k\pry^kr_{i,p+1}, \quad 0\le k\le2p-1,
  \]
  from which we obtain $\vla\smw\in\mcD_p$.

  Proof of~\ref{it:preserv=cross}. For both cross products, we have odd polynomial parts in the first two components and an even polynomial part vanishing at least to order two in the third. By Leibniz, we bound the remainders through order $2p+1$ for $\vv\times\smw$ and order $2p-1$ for $\vv\times\mham\smw$. We control products of two remainders using $y^{2p}\jbk[\log y]\les1$. Thus we obtain $\vv\times\smw\in\mcD_{p+1}$ and $\vv\times\mham\smw\in\mcD_p$.
\end{proof}

\begin{remark}
  \label{re:preserv}
  (i) More generally, by the argument in the proof of Lemma~\itref{le:preserv}{it:preserv=mult}, we can multiply a function in $\mcD_p$ by any smooth even function (independent of $b_1$), retaining the same Dore order and factor.\\
  (ii) If $\smw\in\mcD_{p+1}$ has factor $\delta(b_1)$, then the functions in Lemma~\itref{le:preserv}{it:preserv=mult} belong to $\mcD_{p+1}$ with factor $\delta(b_1)$, and those in~\ref{it:preserv=op} belong to $\mcD_p$ with the same factor. For the two cross products in~\ref{it:preserv=cross}, both the coefficient bounds and the remainder bounds contain the product of the factors of $\vv$ and $\smw$.
\end{remark}

\begin{corollary}
  \label{co:dore}
  The following hold:
  \begin{enumerate}[label={\rm (\roman*)}]
    \item  $\w\in\mcD_{L+1}$ with factor $b_1^{L+1}\lgb^{-1}$;
      \label{it:dore=w}
    \item $\wzt\in\mcD_{L+2}$ with factor $b_1$;
      \label{it:dore=wzt}
    \item $\hw\in\mcD_{L+1}$ with factor $b_1$.
      \label{it:dore=hw}
  \end{enumerate}
\end{corollary}

\begin{proof}
  We have~\ref{it:dore=w} directly from Lemma~\itref{le:interp-wp}{it:interp-wp=taylor-orig}. For~\ref{it:dore=wzt}, we take the Dore expansion of the profile constructed in Lemma~\ref{le:app-sol} to order $L+2$, since the cutoff is identically one for $y\le1$. From these two claims, $\mcD_{L+2}\subset\mcD_{L+1}$ and $\hw=\wzt+\w$, we obtain~\ref{it:dore=hw}.
\end{proof}

Next, we introduce a notation that arises in the iterated operators $(\hj[\vv]\mham)^k$. Recall from the definition in~\eqref{eq:hj} that $\hj[\vv] = (e_3+\vv)\times = \R +\vv\times$. The vector $\vv$ here is always a small perturbation in our setting, and thus the dominating part of $\hj[\vv]\mham$ is just $\R\mham$. Therefore we separate the leading part of $(\hj[\vv]\mham)^k$ by defining
\[
  \Omega^{k}_+[\vv] = (\hj[\vv]\mham)^k - (\R\mham)^k.
\]
Letting $\rho\in\{0,1\}^k$, we further define $\Omega^{k}_{\rho}$ by
\begin{equation}
  \label{def:omega}
  \Omega^{k}_{\rho}[\vv]
  = \prod_{1\le i\le k} (\R\mham)^{1-\rho_{i}}(\vv\times\mham)^{\rho_{i}}.
\end{equation}
Then we have
\[
  (\R\mham)^k = \Omega^{k}_{(0,\dots,0)}[\vv]
  \andd
  \Omega^{k}_+[\vv] =  \sum_{|\rho|_1\ge 1}\Omega^{k}_{\rho}[\vv].
\]
The following pointwise bound holds for $\Omega^{k}_{\rho}$, whose proof is postponed to Appendix~\ref{ap:proof-omega-pw}.

\begin{lemma}[Pointwise bound for $\Omega^{k}_{\rho}$ for $y\ge 1$]
  \label{le:omega-pw-bnd}
  Let $j\in \mbN_+$, $\rho\in\{0,1\}^j$ and $n=|\rho|_1$. For any $\xi = (\xi_1,\dots,\xi_{n+2})\in\mbN^{n+2}$ we denote
  \[
    |\xi|_1 = \sum_{k=1}^{n+2}\xi_{k}
    \andd
    |\xi|_+ = \xi_{n+1} +\xi_{n+2}.
  \]
  Then for any $i\ge 0$ and any smooth vector-valued functions $\vv,\w$, there exists a constant $C(i,j)\ge 1$ such that for all $y\ge 1$,
  \[
    |\pry^{i}\Omega^{j}_{\rho}[\vv]\w|
    \le C(i,j) \sum_{\substack{|\xi|_1=i+2j\\|\xi|_{+}\ge 2}}
    \Biggl(\prod_{1\le k\le n}|\pry^{\xi_{k}}\vv|\Biggr)
    \frac{|\pry^{\xi_{n+1}}\w|}{y^{\xi_{n+2}}}.
  \]
  In addition, the constant $C(i,j)$, viewed as a map from $\mbN\times\mbN_+$ to $\mbR_+$, is monotone increasing in $i,j$.

\end{lemma}

With these notions, we now prove Lemma~\ref{le:w2l2-e}. Note that the lemmas in Appendix~\ref{ap:interp-bnd} are also used.

\begin{proof}[Proof of Lemma~\ref{le:w2l2-e}]
  \label{pf:w2l2-e=pf}
  From the definition in~\eqref{def:w2k1}, we expand
  \[
    \w_{2L+2}^1 = (\R\mham)^{L+1}(\gamma e_3) +\Omega^{L+1}_{+}\w.
  \]
  We estimate the right-hand side for $y\le 1$ and $y\ge1$ respectively.

  \step[1, Estimate for $y\le 1$]
  By Lemma~\itref{le:interp-wp}{it:interp-wp=taylor-orig}, $\gamma e_3\in\mcD_{L+1}$ with factor $b_1^{L+2}\lgb^{-1}$. When applying $(\R\mham)^{L+1}$ to $\gamma e_3$, we use only derivatives up to order $2L+1$ of $\gamma$. By the remainder bound and Lemma~\itref{le:preserv}{it:preserv=op} for the polynomial part, we have
  \[
    |(\R\mham)^{L+1}(\gamma e_3)|
    \les C(M,K)\frac{b_1^{L+2}}{\lgb}y\jbk[\log y],
  \]
  and, after integration, we obtain
  \[
    \|(\R\mham)^{L+1}(\gamma e_3)\|_{L^2(y\le1)}^2
    =\|\mhamp^L(\zt\pry\gamma e_1)\|_{L^2(y\le1)}^2
    \les C(M,K)\frac{b_1^{2L+4}}{\lgb^2}.
  \]
  Next, by Corollary~\itref{co:dore}{it:dore=w}, \ref{it:dore=hw}, we have $\w,\hw\in\mcD_{L+1}$ with factors $b_1^{L+1}\lgb^{-1}$ and $b_1$, respectively. For $|\rho|_1\ge1$, we isolate the derivative of order $2L+2$. By Lemma~\itref{le:preserv}{it:preserv=op}--\ref{it:preserv=cross} and Remark~\ref{re:preserv}, we have
  \[
    |\Omega_\rho^{L+1}\w|
    \les b_1y|\pry^{2L+2}\w|
    +\frac{b_1^{L+2}}{\lgb}\jbk[\log y].
  \]
  Indeed, we use only derivatives up to order $2L$ of $\hw$, while the coefficient of $\pry^{2L+2}\w$ contains $\hw=O(b_1y)$. Applying the local elliptic estimate for $\ham^{L+1}$ and using Lemma~\itref{le:interp-gamma}{it:interp-gamma=sobolev} and~\eqref{pro:boot-e2l2-bnd}, we have
  \begin{align*}
    \|\Omega_\rho^{L+1}\w\|_{L^2(y\le1)}^2
    & \les b_1^2\|\pry^{2L+2}\w\|_{L^2(y\le1)}^2
    +C(M,K)b_1^{2L+4}\lgb^{-2} \\
    & \les C(M,K)\frac{b_1^{2L+4}}{\lgb^2}.
  \end{align*}
  Summing over $\rho$, we obtain the desired estimate for $y\le1$.

  \step[2, Estimate for $y\ge1$]
  By~\eqref{eq:mham-r}, the asymptotics $|\pry^\ell\zt|\les y^{-2-\ell}$ for $y\ge1$ and Leibniz,
  \begin{equation}
    \label{pf:w2l2-e=small-gamma-int}
    |(\R\mham)^{L+1}(\gamma e_3)|
    =|\mhamp^L(\zt\pry\gamma e_1)|
    \les \sum_{i=0}^{2L+2}\frac{|\pry^i\gamma|}{y^{2L+2-i}}.
  \end{equation}
  Thus, directly by Lemma~\ref{le:interp-gamma},
  \begin{align*}
    \|(\R\mham)^{L+1}(\gamma e_3)\|_{L^2(y\ge1)}^2
    \les\sum_{i=0}^{2L+2}\int_{y\ge1}\frac{|\pry^i\gamma|^2}{y^{4L+4-2i}}
    \les\dbs\E_{2L+2}+b_1^{\frac{1}{4L}}\frac{b_1^{2L+2}}{\lgb^2}.
  \end{align*}
  We now estimate $\Omega^{L+1}_{+}\w$. First, from the construction of $\wzt$ in Lemma~\ref{le:local-app-sol},
  \begin{equation}
    \label{pf:w2l2-e=wzt-l-infty}
    \bigl\|\jbk[\log y]^Cy^\ell\pry^\ell\wzt\bigr\|_{L^\infty(y\ge1)}
    \les b_1^{1/2}\lgb^C
    \for \ell\ge0.
  \end{equation}
  In view of Lemma~\ref{le:omega-pw-bnd}, for $y\ge1$,
  \begin{equation}
    \label{pf:w2l2-e=omega-l1-w}
    |\Omega^{L+1}_{+}\w|
    \les \sum_{n,\xi}
    \,\prod_{\ell=1}^{n}|\pry^{\xi_{\ell}}\hw|\,
    \frac{|\pry^{\xi_{n+1}}\w|}{y^{\xi_{n+2}}},
  \end{equation}
  where the sum is taken over
  \[
    1\le n\le L+1,\quad \xi\in\mbN^{n+2},\quad |\xi|_1=2L+2,\quad |\xi|_+\ge2.
  \]
  By the decomposition $\hw=\wzt+\w$, we split the right-hand side into various monomials. Each resulting monomial contains $m$ copies of the $\wzt$ factor for some $0\le m\le n$, and hence
  \[
    |\Omega^{L+1}_{+}\w| \les \sum_{n,m,\xi} \mcJ_{n,m,\xi},
  \]
  which together with~\eqref{pf:w2l2-e=wzt-l-infty} yields
  \[
    \mcJ_{n,m,\xi}
    = b_1^{\frac{m}{2}}\lgb^C
    \prod_{\ell=m+1}^{n}|\pry^{\xi_{\ell}}\w|
    \frac{|\pry^{\xi_{n+1}}\w|}{y^{\xi_1+\dots+\xi_{m}+\xi_{n+2}}\jbk[\log y]^{2m}}.
  \]
  The indices $(n,m,\xi)$ fall into the following cases:
  \begin{enumerate}[label={\rm (\roman*)}]
    \item $m\ge 1$; \label{pf:w2l2-e=case-1}
    \item $m=0$, $\xi_{n+2}\ge 1$; \label{pf:w2l2-e=case-2}
    \item $m=0$, $\xi_{n+2}=0$, $\xi_1+\dots+\xi_n=0$; \label{pf:w2l2-e=case-3}
    \item $m=0$, $\xi_{n+2}=0$, $\xi_1+\dots+\xi_n\ge 1$. \label{pf:w2l2-e=case-4}
  \end{enumerate}
  We estimate $\mcJ_{n,m,\xi}$ for each case.

  \emph{Case~\ref{pf:w2l2-e=case-1}}.
  Since $m\ge1$, we estimate using Lemma~\itref{le:interp-wp}{it:interp-wp=pw} and \ref{it:interp-wp=gen-lossy},
  \begin{align*}
    & \|\mcJ_{n,m,\xi}\|_{L^2(y\ge1)}^2                                                      \\
    & \les b_1^{m}\lgb^C
    \prod_{\ell=m+1}^{n} \|\pry^{\xi_{\ell}}\w\|_{L^\infty(y\ge1)}^2
    \int_{y\ge1}\frac{|\pry^{\xi_{n+1}}\w|^2}{y^{2(\xi_1+\dots+\xi_{m}+\xi_{n+2})}\jbk[\log y]^2} \\
    & \les b_1^{d_{n,m,\xi}} \lgb^C.
  \end{align*}
  The quantity $d_{n,m,\xi}$ is determined as
  \begin{itemize}
    \item for $\xi_{m+1}+\dots+\xi_n=0$:
      \[
        d_{n,m,\xi} = 2L+2+m \ge 2L+3;
      \]
    \item for $\xi_{m+1}+\dots+\xi_n=1$:
      \[
        d_{n,m,\xi} = \frac{2L}{2L-1} + 2L+1 +m\ge 2L+3;
      \]
    \item for $2\le \xi_{m+1}+\dots+\xi_n\le 2L-1$:
      \[
        d_{n,m,\xi} = (2L+1)\frac{2L}{2L-1} +m\ge 2L+3;
      \]
    \item for $\xi_{m+1}+\dots+\xi_n=2L$:
      \[
        d_{n,m,\xi} = \min\biggl(2L+1, \frac{(2L)^2}{2L-1}\biggr) + \frac{2L}{2L-1} +m\ge 2L+3.
      \]
  \end{itemize}

  \emph{Case~\ref{pf:w2l2-e=case-2}}.
  The $\wzt$ part is absent. For $|\xi|_+\le 2L+1$, Lemma~\itref{le:interp-wp}{it:interp-wp=pw} and \ref{it:interp-wp=gen-lossy} give
  \begin{align*}
    \|\mcJ_{n,m,\xi}\|^2_{L^2(y\ge1)}
    & \les \prod_{\ell=1}^{n}\|\pry^{\xi_{\ell}}\w\|^2_{L^\infty(y\ge1)}
    \int_{y\ge1}\frac{|\pry^{\xi_{n+1}}\w|^2}{y^{2\xi_{n+2}}}                                   \\
    & \les b_1^{2L+2+\frac{1}{2L-1}}\lgb^C.
  \end{align*}
  For the borderline case $|\xi|_{+}=2L+2$, we have $\xi_1 = \dots = \xi_n = 0$. Writing one undifferentiated factor as $y\frac{\w}{y}$, we obtain
  \begin{align*}
    \|\mcJ_{n,m,\xi}\|_{L^2(y\ge1)}^2
    & \les \|\w\|_{L^\infty(y\ge1)}^{2(n-1)}
    \biggl\|\frac{\w}{y}\biggr\|_{L^\infty(y\ge1)}^2
    \int_{y\ge1}\frac{|\pry^{\xi_{n+1}}\w|^2}{y^{2(2L+1-\xi_{n+1})}} \\
    & \les b_1^{2L+2+\frac{1}{2L-1}}\lgb^C.
  \end{align*}
  Here the unweighted endpoint $\xi_{n+1}=2L+1$ follows from the even-order expansion and interpolation in Case~\ref{pf:w2l2-e=case-4}.

  \emph{Case~\ref{pf:w2l2-e=case-3}}.
  In this case, $\xi_{n+1}=2L+2$. By the definition of $\w_{2L+2}^0$, we rewrite the integral as
  \begin{align*}
    & \|\mcJ_{n,m,\xi}\|^2_{L^2(y\ge1)}         \\
    & \les \prod_{\ell=1}^{n}\|\w\|^2_{L^\infty(y\ge1)}
    \int_{y\ge1}|\pry^{2L+2}\w|^2                        \\
    & \les \prod_{\ell=1}^{n}\|\w\|^2_{L^\infty(y\ge1)}
    \Biggl(\int|\w_{2L+2}^0|^2
    +\sum_{\ell=0}^{2L+1}\int_{y\ge1}\frac{|\pry^\ell\w|^2}{y^{2(2L+2-\ell)}}
    +\sum_{\ell=0}^{2L+2}\int_{y\ge1}\frac{|\pry^\ell\gamma|^2}{y^{2(2L+2-\ell)}}\Biggr),
  \end{align*}
  where $\gamma$ is the third component of $\w$.
  From \eqref{eq:w2k0-e2k}, the first integral is exactly $\E_{2L+2}$, and then by the smallness $\|\wp\|_{L^\infty} \les \dbs$, this part is bounded by $\dbs\E_{2L+2}$. The second integral is estimated as in Case~\ref{pf:w2l2-e=case-2}. The third integral yields the desired bound directly by Lemma~\ref{le:interp-gamma}.

  \emph{Case~\ref{pf:w2l2-e=case-4}}.
  Since $\xi_{n+2}=0$, there is no $y$-decay in the denominator in $\mcJ_{n,m,\xi}$,
  \[
    \|\mcJ_{n,m,\xi}\|^2_{L^2(y\ge1)}
    \les \prod_{\ell=1}^{n}\|\pry^{\xi_{\ell}}\w\|^2_{L^\infty(y\ge1)}
    \int_{y\ge1}|\pry^{\xi_{n+1}}\w|^2.
  \]
  From $|\xi|_+\ge 2$ and $\xi_1+\dots+\xi_n\ge 1$, we have $2\le \xi_{n+1} \le 2L+1$.
  If $\xi_{n+1}=2k$ is even, we expand $\ham^k$ to obtain
  \begin{equation}\label{pf:w2l2-e=even-expansion}
    \int_{y\ge1}|\pry^{2k}\wp|^2
    \les \E_{2k}+\sum_{i=0}^{2k-1}\int_{y\ge1}\frac{|\pry^i\wp|^2}{y^{2(2k-i)}}.
  \end{equation}
  This expansion holds also for $k=L+1$. Combining it with Lemma~\itref{le:interp-wp}{it:interp-wp=gen-lossy} and Lemma~\itref{le:interp-gamma}{it:interp-gamma=unweighted}, we obtain
  \[
    \prod_{\ell=1}^{n}\|\pry^{\xi_{\ell}}\w\|^2_{L^\infty(y\ge1)}
    \int_{y\ge1}|\pry^{2k}\w|^2
    \les b_1^{2L+2+\frac{1}{2L-1}}\lgb^C.
  \]
  If $\xi_{n+1}$ is odd, choose a fixed smooth cutoff $\chi_{\ge 1}$ with $\chi_{\ge 1}=0$ for $y\le1/2$ and $\chi_{\ge 1}=1$ for $y\ge1$. The interpolation
  \[
    \int_{y\ge1}|\pry^{2k+1}\w|^2
    \le \biggl(\int|\pry^{2k}(\chi_{\ge 1}\w)|^2\biggr)^{1/2}
    \biggl(\int|\pry^{2k+2}(\chi_{\ge 1}\w)|^2\biggr)^{1/2}
  \]
  reduces it to the preceding even orders. The cutoff errors on $1/2\le y\le1$ are controlled by Lemma~\itref{le:interp-wp}{it:interp-wp=taylor-orig}. To sum up, in all cases we have
  \begin{equation} \label{pf:w2l2-e=omega-l1-w-ygt1}
    \|\Omega^{L+1}_{+}\w\|^2_{L^2(y\ge1)}
    \les \dbs\E_{2L+2} + b_1^{\frac{1}{4L}}\frac{b_1^{2L+2}}{\lgb^2}.
  \end{equation}
  This completes the proof.
\end{proof}

\subsection{Direct energy inequality}
\label{sub:dire-eng-ineq}
This subsection is devoted to a preliminary energy estimate for $\E_{2L+2}$. It is the foundation of Proposition~\ref{pro:mix-eng}.

\begin{lemma}[Direct energy inequality]
  \label{le:eng}
  For $s\in[s_0,s^{\ast})$, there exists a constant $C > 1$ independent of $M$ such that
  \begin{equation}
    \begin{split}
      \label{le:eng=dire-eng}
      & \frac{\dd}{\dd t}\biggl\{\frac{1}{\la^{4L+2}}\biggl[\E_{2L+2} +O\biggl(\frac{b_1^{2L+2}}{\lgb^2}\biggr)\biggr]\biggr\}                         \\
      & \le
      \frac{2b_1}{\la^{4L+4}}\Biggl\{\int_{y\le R}|\w_{2L+2}^0|^2 +\sum_{j=0}^{L}\int\w_{2L+2}^0\cdot\R^{L+1}\mhamp^j\bigl(g_2\mhamp^{L-j}\wp\bigr)\Biggr\} \\
      & \quad +2\QQ_{\E}
      -\frac{2}{\la^{4L+4}}\int\w_{2L+2}\cdot(\hj\mham)^{L+1}\mF +\frac{C b_1}{\la^{4L+4}}\frac{b_1^{2L+2}}{\lgb^2},
    \end{split}
  \end{equation}
  where $\QQ_{\E}$ is the quadratic remainder given in~\eqref{pf:eng=q-expr}.

\end{lemma}

\begin{proof}
  From the equation~\eqref{eq:w2k-eq} for $\W_{2k}$, we first derive the following energy identity:
  \begin{align}
    \frac{\dd}{\dd t}\int|\W_{2L+2}|^2
    & = -2\int\W_{2L+2}\cdot\hj\mhaml\W_{2L+2}
    \label{pf:eng=lead-quad-term}                                                 \\
    & \quad +2\int\W_{2L+2}\cdot\bigl[\prt, (\hj\mhaml)^{L+1}\bigr]\W
    \label{pf:eng=commu-term}                                                     \\
    & \quad -\frac{2}{\la^2}\int\W_{2L+2}\cdot(\hj\mhaml)^{L+1}\mF_{\la}. \notag
  \end{align}
  By the scaling law~\eqref{eq:scaling-w-w}, the integral on the left-hand side is
  \[
    \int|\W_{2L+2}|^2
    = \frac{\|\w_{2L+2}\|_{L^2}^2}{\la^{4L+2}}.
  \]
  From Lemma~\ref{le:w2l2-e} (and Remark~\ref{re:w2l2-e}), this yields the left-hand side of~\eqref{le:eng=dire-eng}. Similarly, by~\eqref{eq:scaling-w-w}, the last integral involving $\mF_{\la}$ above is just the $\mF$ term in~\eqref{le:eng=dire-eng}. It remains to treat~\eqref{pf:eng=lead-quad-term} and \eqref{pf:eng=commu-term}.

  \step[1, Leading quadratic term]
  By~\eqref{eq:scaling-w-w}, we have
  \[
    -\int\W_{2L+2}\cdot\hj\mhaml\W_{2L+2}
    = -\frac{1}{\la^{4L+4}}\int\w_{2L+2}\cdot\hj\mham\w_{2L+2}.
  \]
  We use the following weighted form of the gain-of-two-derivatives estimate in \citep[Lemma~4.7 and Appendix~C]{merle_blowup_2013}.

  \begin{lemma*}[Gain of two derivatives \citep{merle_blowup_2013}]
    For any vector $\smw$ with respect to the Frenet basis, we have
    \[
      -\int\hj\mham\smw\cdot(\hj\mham)^2\smw
      \le b_1\int|f|\,|\hj\mham\smw|^2 +b_1\dbs \|\mham\smw\|_{L^2}^2,
    \]
    where the function $f = (1+\z)\oa T_1$ satisfies $0\le f(y)<1$ for all $y\ge 0$ and $f(y)\to0$ as $y\to +\infty$.
  \end{lemma*}
  The properties of $f$ follow from \citep[Appendix~C, Step~3]{merle_blowup_2013}. To apply the same argument here, we check the bounds on the profile-dependent remainder terms. The leading part of $\wzt$ is $b_1\phibt[1] = b_1\chi_{S_1}T_1 e_2$. From Lemmas~\ref{le:local-app-sol} and~\ref{le:interp-wp}, together with the support of the cutoff error on $y\ge S_1$, we have
  \[
    \|\pry\hw\|_{L^\infty}^2
    +\biggl\|\frac{\pry(\hw-b_1\phibt[1])}{\jbk[y]^2}\biggr\|_{L^\infty}
    +\biggl\|\frac{\hw-b_1\phibt[1]}{\jbk[y]^4}\biggr\|_{L^\infty}
    \les b_1\dbs.
  \]
  These bounds control the remainder terms in the computation of \citep[Appendix~C]{merle_blowup_2013} by $b_1\dbs\|\mham\smw\|_{L^2}^2$, yielding the above lemma for the present profile.

  We now continue the proof of Lemma~\ref{le:eng}. From the estimate for $f$, there exists $R>1$ such that
  \[
    \|f\|_{L^{\infty}(y\ge R)} \le \frac{1}{K}.
  \]
  Here, $K$ is the constant in Proposition~\ref{pro:boot}. Taking $\smw = \w_{2L}$ in the above lemma, we have
  \[
    -\int\w_{2L+2}\cdot\hj\mham\w_{2L+2}
    \le b_1\int_{y\le R} |\w_{2L+2}|^2 + \frac{b_1}{K}\int_{y\ge R} |\w_{2L+2}|^2 + b_1\dbs \int|\mham\w_{2L}|^2.
  \]
  We claim that
  \begin{equation}
    \label{pf:eng=claim-w2l}
    \int|\mham\w_{2L}|^2
    \les C(M)\biggl(\E_{2L+2}+\frac{b_1^{2L+2}}{\lgb^2}\biggr).
  \end{equation}
  This, together with Lemma~\ref{le:w2l2-e}, Remark~\ref{re:w2l2-e} and the bootstrap bound~\eqref{pro:boot-e2l2-bnd}, yields
  \begin{align*}
    & -\int\w_{2L+2}\cdot\hj\mham\w_{2L+2}                                                                                                                  \\
    & \le b_1\int_{y\le R}|\w_{2L+2}^0|^2 +\frac{Cb_1^{2L+3}}{\lgb^2} + \biggl(\frac{b_1}{K}+C(M)b_1\dbs\biggr)\biggl(\E_{2L+2}+\frac{Cb_1^{2L+2}}{\lgb^2}\biggr) \\
    & \le b_1\int_{y\le R}|\w_{2L+2}^0|^2 +\frac{Cb_1^{2L+3}}{\lgb^2}.
  \end{align*}
  Here the constant $C$ is independent of $M$, since $\bs>0$ will be chosen sufficiently small so that $C(M)$ is absorbed by $\dbs$. Thus the estimate for~\eqref{pf:eng=lead-quad-term} follows:
  \[
    -\int\W_{2L+2}\cdot\hj\mhaml\W_{2L+2}
    \le \frac{b_1}{\la^{4L+4}}
    \Biggl\{ \int_{y\le R}|\w_{2L+2}^0|^2 +C\frac{b_1^{2L+2}}{\lgb^2} \Biggr\}.
  \]

  \emph{Proof of Claim~\eqref{pf:eng=claim-w2l}}. Using~\eqref{def:w2k1}, we expand:
  \begin{align*}
    \mham\w_{2L}
    & = -\R\w_{2L+2}^0 +\zt\bigl(\pry+\tfrac{\z}{y}\bigr)(\w_{2L}^{0}\cdot e_1)e_3 +\mham(\Omega_{+}^{L}\w)                       \\
    & \quad -\R(\R\mham)^{L+1}(\gamma e_3) +\zt\bigl(\pry+\tfrac{\z}{y}\bigr)((\R\mham)^L(\gamma e_3)\cdot e_1)             e_3 .
  \end{align*}
  We estimate each term on the right-hand side separately. For the first term, from~\eqref{eq:w2k0-e2k}, we note that its squared $L^2$ norm is exactly $\E_{2L+2}$. From~Lemma~\itref{le:interp-wp}{it:interp-wp=sobolev}, the second term admits the bound
  \[
    \int\bigl|\zt\bigl(\pry+\tfrac{\z}{y}\bigr)(\w_{2L}^{0}\cdot e_1)\bigr|^2
    \les \int\frac{|\wp_{(2L+1)}|^2 +\frac{|\wp_{(2L)}|^2}{y^2}}{y^2\jbk[\log y]^2}
    \les C(M) \E_{2L+2},
  \]
  where we use $(\partial_y+Z/y)f=-Af+2Zf/y$.
  The third term is estimated as in \eqref{pf:w2l2-e=omega-l1-w}, by applying Lemma~\ref{le:omega-pw-bnd} to $\Omega^L_+$ followed by a crude estimate for $\mham$.
  The fourth term has been treated in the proof of Lemma~\ref{le:w2l2-e}.
  Finally, by Lemma~\itref{le:interp-wp}{it:interp-wp=taylor-orig}, we have $\gamma e_3\in\mcD_{L+1}$ with factor $b_1^{L+2}\lgb^{-1}$. By Lemma~\itref{le:preserv}{it:preserv=op}, $(\R\mham)^L(\gamma e_3)\in\mcD_1$ with the same factor. Thus, for $y\le1$,
  \[
    |\pry(\R\mham)^L(\gamma e_3)|
    +\frac{|(\R\mham)^L(\gamma e_3)|}{y}
    \les C(M,K)\frac{b_1^{L+2}}{\lgb}\jbk[\log y].
  \]
  We thus obtain the squared $L^2$ bound
  \[
    \|\zt\bigl(\pry+\tfrac{\z}{y}\bigr)((\R\mham)^L(\gamma e_3)\cdot e_1)\|_{L^2(y\le 1)}^2
    \les C(M,K)\frac{b_1^{2L+4}}{\lgb^2}.
  \]
  For $y\ge1$, using Lemma~\ref{le:interp-gamma},
  \begin{align*}
    & \|\zt\bigl(\pry+\tfrac{\z}{y}\bigr)((\R\mham)^L(\gamma e_3)\cdot e_1)\|_{L^2(y\ge1)}^2 \\
    & \les \sum_{\ell=0}^{2L+1}  \int_{y\ge1} \frac{|\pry^{\ell}\gamma|^2}{y^{2(2L+2-\ell)}}
    \les \dbs \Eplus.
  \end{align*}
  Collecting these bounds proves the claim~\eqref{pf:eng=claim-w2l}.

  \step[2, Commutator term]
  For ease of notation, we define the functions:
  \begin{equation}
    \label{pf:eng=g2}
    g_1(y)=\frac{\vla \z}{y},
    \andd
    g_2(y)=\frac{\vla V}{y^2},
  \end{equation}
  and the corresponding rescaled forms
  \[
    G_1(r) =\frac{g_1(\tfrac{r}{\la})}{\la^3},
    \andd
    G_2(r) = \frac{g_2(\tfrac{r}{\la})}{\la^4}.
  \]
  In fact, the commutators between $\prt$ and linear operators $\oa_{\la}, \ham_{\la}$ can be written as
  \[
    \left[\prt, \oa_\la\right]
    = -\frac{\la_s}{\la}G_1, \quad
    \left[\prt, \ham_\la\right]
    = -\frac{\la_s}{\la}G_2.
  \]
  Hence, in view of~\eqref{eq:scale-w-w},
  \begin{equation}
    \label{pf:eng=ma-mham-comm}
    \left[\prt, \ma_\la\right]\W
    = -\frac{\la_s}{\la}G_1 \Wp, \quad
    \left[\prt, \mhaml\right]\W
    = -\frac{\la_s}{\la} \frac{\mfg\w}{\la^4},
  \end{equation}
  where the operator $\mfg$ is defined by
  \begin{align*}
    \mfg\w
    & = \bigl(g_2\alpha -(1+\vla)\zt\pry\gamma\bigr)e_1 +g_2\beta e_2               \\
    & \quad +\biggl((1+\vla)\zt\pry +\frac{1}{y}(1+\vla)(\zt \z)\biggr) \alpha e_3.
  \end{align*}
  Using the equation~\eqref{eq:vect-form} for $\hw$, we compute:
  \begin{equation}\begin{split}
      \label{pf:eng=hj-mhaml-comm}
      \left[\prt, \hj\mhaml\right]\W
      & = \R[\prt, \mhaml]\W +\hW\times[\prt, \mhaml]\W +\prt\hW\times\mhaml\W                                                          \\
      & = -\frac{\la_s}{\la} \frac{\hj \mfg\w}{\la^4} -\frac{1}{\la^4}\Bigl(\vtha_s\z\R\hw +\hj(\mham\hw +\hp\lphi)\Bigr)\times\mham\w.
    \end{split}\end{equation}
  Moreover, by induction, the high-order commutator is
  \begin{align*}
    \left[\prt, (\hj\mhaml)^{L+1}\right]\W
    & = \sum_{\ell=0}^L (\hj\mhaml)^\ell[\prt,\hj\mhaml](\hj\mhaml)^{L-\ell}\W                                                         \\
    & = \sum_{\ell=0}^L (\hj\mhaml)^\ell\Bigl(\R[\prt, \mhaml] + \hW\times [\prt, \mhaml]  + \prt\hW \times\mhaml\Bigr)\W_{2(L-\ell)}.
  \end{align*}
  With this, we compute using the commutator~\eqref{pf:eng=ma-mham-comm}, the decomposition~\eqref{def:w2k1}, and the diagonality~\eqref{eq:mham-r}, which yields the expansion of \eqref{pf:eng=commu-term}:
  \begin{align*}
    \int\W_{2L+2}\cdot\left[\prt,(\hj\mhaml)^{L+1}\right]\W
    & = \sum_{\ell=0}^{L} b_1\int\W_{2L+2}^0\cdot(\R\mhaml)^\ell\R \bigl(G_2\W_{2(L-\ell)}^{0}\bigr) +\QQ_{\E}                       \\
    & = \sum_{\ell=0}^{L} \frac{b_1}{\la^{4L+4}}\int\w_{2L+2}^0\cdot \R^{L+1}\mhamp^\ell\bigl(g_2\mhamp^{L-\ell}\wp\bigr) +\QQ_{\E},
  \end{align*}
  where the quadratic remainder $\QQ_{\E}$ is
  \begin{equation}
    \label{pf:eng=q-expr}
    \begin{aligned}
      \hs[-.5]\QQ_{\E}
      & = \sum_{\ell=0}^L \Biggl\{
      -\biggl(b_1+\frac{\la_s}{\la}\biggr)\int\W_{2L+2}^0\cdot (\R\mhaml)^\ell\R (G_2\W_{2(L-\ell)}^0)                                                                                     \\
      & \qquad +\int\W_{2L+2}^1\cdot(\hj\mhaml)^\ell[\prt, \hj\mhaml]\W_{2(L-\ell)}                                                                                                       \\
      & \qquad +\int\W_{2L+2}^0\cdot(\hj\mhaml)^\ell[\prt,\hj\mhaml]\W_{2(L-\ell)}^1                                                                                                      \\
      & \qquad +\int\W_{2L+2}^0\cdot\bigl((\hj\mhaml)^\ell-(\R\mhaml)^\ell\bigr)[\prt,\hj\mhaml]\W_{2(L-\ell)}^0                                                                          \\
      & \qquad +\int\W_{2L+2}^0 \cdot \hs[-.1] (\R\mhaml)^\ell\Bigl( \hW\hs[-.2]\times[\prt,\mhaml]\W_{2(L-\ell)}^0+ \prt\hW\hs[-.2]\times\mhaml\W_{2(L-\ell)}^0 \Bigr) \hs[-.1]\Biggr\}.
    \end{aligned}
  \end{equation}
  Inserting the estimates for~\eqref{pf:eng=lead-quad-term} and the expansion of~\eqref{pf:eng=commu-term} into the right-hand side of the energy identity, we obtain~\eqref{le:eng=dire-eng}. This completes the proof.
\end{proof}

\subsection{Virial correction}
\label{sub:vr}
In this subsection, we introduce the Virial functional $\vr$. Its time derivative eliminates the leading quadratic term in the direct energy inequality (see Lemma~\ref{le:eng}).

We begin by constructing a spatial weight $q(y)$ that dominates the dynamics of $\vr$. For any smooth function $f$, we define the operators
\begin{equation}
  \label{eq:t-p}
  \T f = \pry f -\frac{1+2\z}{y}f, \quad
  \PP f = \pry f +\frac{1+2\z}{y}f.
\end{equation}

\begin{lemma}[Virial weight]
  \label{le:vr-weig}
  For any $c \in (0, 1)$ and any sufficiently large $R > 1$, there exists a scalar function $q(y) = q_{c,R}(y) \in C^{3,1}(0, +\infty)$ such that:
  \begin{enumerate}[label={\rm (\roman*)}]
    \item $q(y) = \frac{1}{2}y^2$ for all $y \le R$; \label{le:vr-weig=approx-ym}
    \item $q''(y)\ge-c$, $\frac{1}{y}q'(y)\ge-c$ for all $y > 0$; \label{le:vr-weig=convex-ym}
    \item there exist $R_{1}>0$ and $C_1,C_2\in\mbR$ depending on $c, R$ such that $q(y)=C_1 + C_2\log y$ for all $y \ge R_{1}$; \label{le:vr-weig=support-ym}
    \item $(\T\pry)^2 q(y) \le c \jbk[y]^{-2}\jbk[\log y]^{-2}$ for a.e. $y>0$. \label{le:vr-weig=bilapl-ym}

  \end{enumerate}
\end{lemma}

\begin{proof}
  We seek a weight function of the form
  \[
    q(y) = c_1 + c_2 \log y + c_3 y^2 + c_4 y^2 \log y,
  \]
  where the radial functions $1,\log y,y^2,y^2\log y$ belong to $\ker\Delta_{\mathrm{rad}}^2$, and $(c_1,c_2,c_3,c_4)$ are perturbed coefficients close to $(0,0,\tfrac12,0)$. We first specify these coefficients, and then verify the properties of $q(y)$.

  \step[1, Choice of coefficients] Let $\alpha_1 > 0$ be a small constant specified later. We set $X=\log R$ and consider the auxiliary function:
  \[
    f(x) = \frac 1X \biggl(\frac{1}{2X+1} - \frac{1}{x+1}\biggr)^{-1}\log\left(\frac{1+x}{1+2X}\right) -\frac{3X+2}{2X}.
  \]
  Given $X$ large enough, we have $f(4X+1) < 10$.
  Since $f(x)\to +\infty$ as $x\to+\infty$, there exists $X_{1} > 4X + 1$ such that $f(X_{1}) = \frac{1}{2\alpha_1}$. This implies
  \[
    2\biggl(1+O\biggl(\Bigl|\frac{1}{X}\Bigr| +\Bigl|\frac{X}{X_1}\Bigr| \biggr) \biggr) \log\biggl(\frac{1+X_1}{1+2X}\biggr) -\frac32 + O\biggl(\frac{1}{X}\biggr) = \frac{1}{2\alpha_1},
  \]
  and thus
  \[
    \log\left(\frac{1+X_{1}}{1+2X}\right) \simeq \frac{1}{4\alpha_1},
    \andd
    X_{1} \gg X.
  \]
  Let $\alpha_2$ be such that
  \[
    \frac{\alpha_1}{X} = \alpha_2\biggl(\frac{1}{2X+1} - \frac{1}{X_{1}+1}\biggr).
  \]
  Then from $f(X_1) = \frac{1}{2\alpha_1}$,
  \[
    \alpha_2\log\left(\frac{1+X_{1}}{1+2X}\right)  = \frac{1}{2X} \alpha_1(3X+2)+\frac 12.
  \]
  Note that $\alpha_2\les\alpha_1$. We define the function
  \[
    F(x) = \begin{cases}
      -\alpha_1X^{-2},    & x \in (X, 2X),     \\
      \alpha_2(1+x)^{-2}, & x \in (2X, X_{1}), \\
      0,                  & \mbox{elsewhere}.
    \end{cases}
  \]
  The following hold:
  \begin{align}
    \label{pf:vr-weig=int-f-1}
    \int_{-\infty}^\infty F(x) \dd x
    & = -\frac{\alpha_1}{X} + \alpha_2\biggl(\frac{1}{2X+1}-\frac{1}{X_{1}+1}\biggr) = 0,       \\
    \int_{-\infty}^\infty (1+x)F(x) \dd x
    & = -\frac{\alpha_1}{2X}(3X+2) + \alpha_2\log\left(\frac{1+X_{1}}{1+2X}\right)  = \frac 12.
    \label{pf:vr-weig=int-f-2}
  \end{align}
  Now we set
  \begin{align*}
    & c_1(x) = {-}\int_{-\infty}^x (z-1)e^{2z}F(z) \dd z,  &  &
    c_2(x) = \int_{-\infty}^x e^{2z} F(z) \dd z,                 \\
    & c_3(x) = \frac 12 -\int_{-\infty}^x (1+z)F(z) \dd z, &  &
    c_4(x) = \int_{-\infty}^x F(z) \dd z.
  \end{align*}
  These functions are smooth except at points $X$, $2X$, and $X_{1}$. More precisely, they are Lipschitz continuous. Note that since $F(x)=0$ for $x\le X$,
  \[
    (c_1, c_2, c_3, c_4)(x)
    = \bigl(0, 0, \tfrac12, 0\bigr)
    \for x\le X.
  \]
  Also, by definition, we have
  \begin{equation}
    \label{pf:vr-weig=bound-cj}
    |c_1(x)| \les \frac{\alpha_1e^{2x}}{\jbk[x]},\quad
    |c_2(x)| \les \frac{\alpha_1e^{2x}}{\jbk[x]^2}, \quad
    |c_4(x)| \les \frac{\alpha_1}{\jbk[x]}
    \for x\ge X.
  \end{equation}
  In addition, the function $c_3(x)$ is monotone increasing on $[X, 2X]$, reaching its maximum at $2X$:
  \[
    c_3(2X) = \frac{1}{2} + \frac{\alpha_1(3X+2)}{2X}.
  \]
  It then decreases until $c_3(X_{1}) = 0$ (by~\eqref{pf:vr-weig=int-f-2}).
  We then choose $\alpha_1$ sufficiently small, depending on $c$, so that $0\le c_3\le 1$ for all $x$.

  \step[2, Construction of $q(y)$] Consider the function
  \[
    p(x) = c_1(x)+c_2(x) x + c_3(x) e^{2x} + c_4(x) xe^{2x}.
  \]
  It is Lipschitz continuous. We compute:
  \begin{align*}
    p'(x) & = c_1'(x) + c_2'(x) x + c_3'(x) e^{2x} + c_4'(x) xe^{2x} \\
    & \quad + c_2(x) + c_3(x)2e^{2x} + c_4(x)(1+2x)e^{2x}.
  \end{align*}
  By the definitions of $c_j(x)$, the terms on the first line cancel:
  \[
    -(x-1)e^{2x}F(x) + x e^{2x}F(x) - (1+x)e^{2x}F(x) + xe^{2x}F(x) = 0,
  \]
  and thus
  \[
    p'(x) = c_2(x) + (2c_3(x)+c_4(x))e^{2x} + 2c_4(x)xe^{2x}.
  \]
  This function is Lipschitz. Repeating this procedure,
  \begin{align*}
    & \quad c_2'(x) + (2c_3'(x)+c_4'(x))e^{2x} + 2c_4'(x)xe^{2x} \\
    & = e^{2x}F(x) - (1+2x)e^{2x}F(x) + 2x e^{2x}F(x) = 0,
  \end{align*}
  and hence
  \[
    p''(x) = (4c_3(x)+4c_4(x))e^{2x} + 4c_4(x)x e^{2x}.
  \]
  It is again Lipschitz. Again,
  \[
    (4c_3'(x) + 4c_4'(x))e^{2x} + 4c_4'(x)xe^{2x}
    = -4xe^{2x}F(x) + 4xe^{2x}F(x) = 0,
  \]
  so that
  \[
    p^{(3)}(x) = (8c_3(x) + 12c_4(x))e^{2x} + 8c_4(x)xe^{2x},
  \]
  which is also Lipschitz. Finally,
  \begin{align*}
    p^{(4)}(x) & = (8c_3'(x) + 12c_4'(x))e^{2x} + 8c_4'(x) xe^{2x}    \\
    & \quad + (16c_3(x)+32c_4(x))e^{2x} + 16c_4(x)xe^{2x}.
  \end{align*}
  Set $x = \log y$ and $q(y) = p(x) = p(\log y)$. We now prove that $q(y)$ satisfies the required properties.

  \step[3, Verification of the properties]
  Proof of~\ref{le:vr-weig=approx-ym}. Since $y \le R$ if and only if $x \le X$, we have, by the definition of $p(x)$, that
  \[
    q(y) = p(x) = \frac{e^{2x}}{2} = \frac{y^2}{2}
    \for y\le R.
  \]

  Proof of~\ref{le:vr-weig=convex-ym}. The expression $q(y) = \frac{1}{2}y^2$ already satisfies \ref{le:vr-weig=convex-ym} on $x\le X$. It suffices to consider the case $x\ge X$. Note that
  \begin{align*}
    q''(y)         & = e^{-2x}(p''(x) - p'(x))
    = -c_2(x)e^{-2x} + 2c_3(x) + c_4(x)(3+2x),                                \\
    \frac 1y q'(y) & = e^{-2x}p'(x) = c_2(x)e^{-2x} + 2c_3(x) + c_4(x)(1+2x).
  \end{align*}
  Then, the bounds~\eqref{pf:vr-weig=bound-cj} and the non-negativity of $c_3$ yield~\ref{le:vr-weig=convex-ym}.

  Proof of~\ref{le:vr-weig=support-ym}. Let $R_{1} = e^{X_1}$. Then, the region $y\ge R_{1}$ is just $x\ge X_1$. In view of the definition of $c_3, c_4$ and~\eqref{pf:vr-weig=int-f-1}, \eqref{pf:vr-weig=int-f-2}, we have
  \[
    c_3(x) = c_3(X_1) = 0
    \andd
    c_4(x) = c_4(X_1) = 0
    \for x\ge X_1,
  \]
  and hence
  \[
    q(y) = p(x) = c_1(X_1) + c_2(X_1)x = c_1(X_1) + c_2(X_1) \log y.
  \]
  This yields~\ref{le:vr-weig=support-ym} upon setting $C_1 = c_1(X_1)$ and $C_2 = c_2(X_1)$.

  Proof of~\ref{le:vr-weig=bilapl-ym}. Since $(\T\pry)^2 (\frac12 y^2) = -16(1+y^2)^{-2}$ for $y\le R$, it suffices to consider the case $y\ge R$. An explicit computation gives
  \begin{align*}
    (\T\pry)^2 q(y)
    & = \frac{p^{(4)}(x)-4p^{(3)}(x) + 4p''(x)}{(1+y^2)^2} + O\biggl(\frac{1}{y^6}\sum_{j=1}^{4}|p^{(j)}(x)|\biggr) \\
    & = \frac{4y^2}{(1+y^2)^2}\Bigl(2c_3'(x) + 3c_4'(x) + 2c_4'(x)x\Bigr)+O\biggl(\frac{\log y}{y^4}\biggr),
  \end{align*}
  where the $4^{\rm th}$ derivative $p^{(4)}$ is understood in the distributional sense. Now, by the definitions of $c_j(x)$ and the fact that $\alpha_1\ll c$,
  \[
    2c_3'(x) + 3c_4'(x) + 2c_4'(x)x = F(x) \les \frac{\alpha_1}{\jbk[x]^2} \les \frac{c}{\jbk[\log y]^2}.
  \]
  Thus the leading term is bounded by $c\jbk[y]^{-2}\jbk[\log y]^{-2}$. Note that the remainder $O(\log y/y^4)$ can be absorbed by the leading term, thanks to the largeness of $R$. This yields~\ref{le:vr-weig=bilapl-ym}, and hence completes the proof.
\end{proof}

For the rescaled weight $q_{\la}(r) = q(\frac{r}{\la})$, we define the corresponding Virial functional by
\[
  \vr(t) = b_1 \int\ma_{\la}\W_{2L}^0\cdot \prr (q_{\la})\W_{2L+2}^0.
\]
Then we have the following estimate.

\begin{lemma}[Virial Correction]
  \label{le:vr-est}
  For any $c_0>0$ and any sufficiently large $R>1$, there exists a constant $c>0$ such that the Virial functional $\vr = \vr(t)$, defined with the weight $q=q_{c,R}$ of Lemma~\ref{le:vr-weig}, satisfies
  \begin{enumerate}[label={\rm (\roman*)}]
    \item boundedness of $\vr$:
      \label{le:vr-est=bdd}
      \[
        |\vr| \les \frac{b_1^{1/4}}{\la^{4L+2}}\frac{b_1^{2L+2}}{\lgb^2},
      \]
    \item cancellation of the leading quadratic term:
      \label{le:vr-est=eq}
      \[
        \frac{\dd\vr}{\dd t}
        \le \frac{b_1}{\la^{4L+4}} \biggl\{
        -2\int_{y\le R}|\w_{2L+2}^0|^2 +c_0\int|\w_{2L+2}^0|^2 \biggr\}
        +\sum_{j=1}^{3} \QQ_{\vr,j} +\mF_{\vr},
      \]
      where
      \begin{align}
        \begin{split}
          \label{le:vr-est=qq1}
          \QQ_{\vr,1}
          & = \prt b_1 \int\ma_{\la}\W_{2L}^0\cdot\prr(q_{\la})\W_{2L+2}^0        \\
          & \quad + b_1 \int\ma_{\la}\W_{2L}^0\cdot \prt\prr(q_{\la})\W_{2L+2}^0,
        \end{split}                                                                         \\
        \begin{split}
          \label{le:vr-est=qq2}
          \QQ_{\vr,2}
          & = b_1\int[\prt,\ma_\la]\W_{2L}^0\cdot\prr(q_{\la})\W_{2L+2}^0                      \\
          & \quad +b_1\int\ma_\la[\prt,(\R\mhaml)^L]\Wp\cdot \prr (q_{\la})\W_{2L+2}^0         \\
          & \quad +b_1\int\ma_{\la}\W_{2L}^0\cdot \prr (q_{\la})[\prt,(\R\mhaml)^{L+1}]\W^\pp,
        \end{split} \\
        \begin{split}
          \label{le:vr-est=qq3}
          \QQ_{\vr,3}
          & = b_1\int\ma_\la(\R\mhaml)^L(\R^2\W_2^1)\cdot \prr (q_{\la})\W_{2L+2}^0          \\
          & \quad +b_1\int\ma_{\la}\W_{2L}^0\cdot\prr (q_{\la})(\R\mhaml)^{L+1}(\R^2\W_2^1),
        \end{split}                                                              \\
        \begin{split}
          \label{le:vr-est=f}
          \mF_{\vr}
          & = -\frac{b_1}{\la^2} \int\ma_{\la}(\R\mhaml)^L\mF_{\la}^\pp\cdot \prr (q_{\la})\W_{2L+2}^0        \\
          & \quad -\frac{b_1}{\la^2} \int\ma_{\la}\W_{2L}^0\cdot \prr (q_{\la})(\R\mhaml)^{L+1}\mF_{\la}^\pp.
        \end{split}
      \end{align}
  \end{enumerate}
\end{lemma}

\begin{proof}
  Proof of~\ref{le:vr-est=bdd}. By the scaling law~\eqref{eq:scale-w-w} and Cauchy-Schwarz,
  \[
    |\vr|
    \les \frac{b_1}{\la^{4L+2}} \|q' \ma\w_{2L}^0\|_{L^2} \sqrt{\E_{2L+2}}.
  \]
  By Lemma~\itref{le:vr-weig}{le:vr-weig=support-ym}, the scalar function $q'(y)$ has polynomial growth for $y\le R_1$ and decays as $1/y$ for $y\ge R_1$. Hence there exists $C>0$ such that
  \[
    \|q' \ma\w_{2L}^0\|_{L^2}^2
    \les R_{1}^C \|\ma\w_{2L}^0\|_{L^2}^2
    \les R_{1}^C \sqrt{\E_{2L+2}} \sqrt{\E_{2L}}.
  \]
  Using~\eqref{pro:boot-e2k-bnd}, \eqref{pro:boot-e2l2-bnd}, we obtain~\ref{le:vr-est=bdd}:
  \[
    |\vr|
    \les \frac{b_1}{\la^{4L+2}} R_{1}^C \E_{2L+2}^{3/4} \E_{2L}^{1/4}
    \les \frac{b_1^{1/4}}{\la^{4L+2}} \frac{ b_1^{2L+2}}{\lgb^2}.
  \]
  Proof of~\ref{le:vr-est=eq}. First, we recall from the definition of $\w_{2k}^0$ that $\w_{2k+2}^0 = \R\mhamp\w_{2k}^0$ for $k\ge 0$, where we use the notation $\w_0^0 =\wp$. We further define
  \begin{equation}
    \label{pf:vr-est=w2k1}
    \w_{2k+1}^0 = \ma\w_{2k}^0
    \for k\ge 0,
  \end{equation}
  which gives
  \[
    \w_{2k+2}^0 = \R\mas\w_{2k+1}^0.
  \]
  Then, a direct computation using~\eqref{eq:w2k-eq} yields the equation for $\vr$:
  \begin{align*}
    \frac{\dd\vr}{\dd t}
    & = -\frac{b_1}{\la^{4L+4}} \mcJ +\sum_{j=1}^{3} \QQ_{\vr,j} +\mF_\vr,      \\
    \mcJ & = \int\w_{2L+3}^0\cdot q'\w_{2L+2}^0 +\int\w_{2L+1}^0\cdot q'\w_{2L+4}^0.
  \end{align*}
  From the definition~\eqref{eq:t-p}, for any smooth vector $\smw$ with vanishing third component and any smooth scalar $f(y)$, we have
  \begin{equation}
    \label{pf:vr-est=a-t-p}
    \ma(f\smw) +f\mas\smw = -\T f \,\smw, \quad
    \mas (f\smw) +f\ma\smw = \PP f \,\smw.
  \end{equation}
  Next, from the adjointness of $\mas, \ma$ and the first identity of~\eqref{pf:vr-est=a-t-p},
  \begin{align*}
    \int \w_{2L+1}^0  \cdot  q'\w_{2L+4}^0
    & = \int \w_{2L+1}^0  \cdot  q' \R\mas\ma\w_{2L+2} ^0                                                    \\
    & = -\int \w_{2L+2}^0  \cdot  \R\mas\ma(q'\w_{2L+1}^0)                                                   \\
    & = \int\w_{2L+2}^0\cdot \R\mas \left(\T q'\w_{2L+1}^0 +q'\mas\w_{2L+1}^0 \right)                        \\
    & = \int\w_{2L+2}^0\cdot \left(\R\pry \T q'\w_{2L+1}^0 + \T q'\w_{2L+2}^0 + \mas(q'\w_{2L+2}^0) \right).
  \end{align*}
  Thus, by the second identity of~\eqref{pf:vr-est=a-t-p} and the fact that $\T + \PP = 2\pry$,
  \begin{align*}
    \mcJ
    & = \int\w_{2L+2}^0\cdot \left(q'\ma\w_{2L+2}^0 + \R\pry \T q'\w_{2L+1}^0 + \T q'\w_{2L+2}^0 + \mas(q'\w_{2L+2}^0) \right) \\
    & = \int\w_{2L+2}^0\cdot \left(\PP q'\w_{2L+2}^0 +\R\pry \T q'\w_{2L+1}^0 +\T q'\w_{2L+2}^0 \right)                        \\
    & = 2\int q''|\w_{2L+2}^0|^2 + \int \w_{2L+2}^0 \cdot \R\pry \T q'\w_{2L+1}^0.
  \end{align*}
  In view of \eqref{eq:t-p}, the last term above is
  \begin{align*}
    \int \w_{2L+2}^0 \cdot \R\pry \T q'\w_{2L+1}^0
    & =  \int  \w_{2L+1}^0  \cdot  \ma(\pry  \T q'\w_{2L+1}^0)                                \\
    & = \int\w_{2L+1}^0\cdot \left(-\pry \T q'\mas\w_{2L+1}^0 -\T\pry\T q'\w_{2L+1}^0 \right) \\
    & = -\int\ma(\pry \T q' \w_{2L+1}^0)\cdot\w_{2L+1}^0 -\int \T\pry\T q'|\w_{2L+1}^0|^2.
  \end{align*}
  Observe that the first term on the last line is a rewriting of the left-hand side, and thus
  \[
    \int  \w_{2L+2}^0 \cdot \R\pry \T q'\w_{2L+1}^0 = -\frac{1}{2}\int (\T\pry)^2q|\w_{2L+1}^0|^2.
  \]
  Inserting this into $\mcJ$, we choose a constant $R$ large enough, which yields
  \begin{align*}
    \mcJ
    & = 2\int q''|\w_{2L+2}^0|^2 -\frac{1}{2}\int(\T\pry)^2q|\w_{2L+1}^0|^2                                        \\
    & = 2\int_{y\le R}q''|\w_{2L+2}^0|^2 +2\int_{y>R}q''|\w_{2L+2}^0|^2 -\frac{1}{2}\int(\T\pry)^2q|\w_{2L+1}^0|^2.
  \end{align*}
  Finally, we choose a small constant $c>0$ such that $c \le \frac{c_0}{2}\bigl(1+\frac{1}{4C}\bigr)^{-1}$, where $C>0$ is the constant in Lemma~\ref{le:coer-ti-ham}, and then apply Lemma~\ref{le:vr-weig} with $c,R$ given as above. From Lemma~\itref{le:vr-weig}{le:vr-weig=approx-ym}, \ref{le:vr-weig=convex-ym}, \ref{le:vr-weig=bilapl-ym} and Lemma~\ref{le:coer-ti-ham}, we obtain
  \begin{align*}
    \mcJ
    & \ge 2\int_{y\le R}\Bigl(\frac{y^2}{2}\Bigr)''|\w_{2L+2}^0|^2 -2c\int_{y>R}|\w_{2L+2}^0|^2 -\frac{c}{2}\int\frac{|\w_{2L+1}^0|^2}{y^2\jbk[\log y]^2} \\
    & \ge 2\int_{y\le R}|\w_{2L+2}^0|^2 -2c \Bigl(1+\frac{1}{4C}\Bigr) \int|\w_{2L+2}^0|^2                                                                \\
    & \ge 2\int_{y\le R}|\w_{2L+2}^0|^2 -c_0\int|\w_{2L+2}^0|^2.
  \end{align*}
  Inserting this into the equation for $\vr$ yields~\ref{le:vr-est=eq}. This completes the proof.
\end{proof}

\subsection{Morawetz correction}
\label{sub:mora}
In this subsection, we introduce the Morawetz functional $\mora_L(t)$. Its time derivative cancels the commutator term in the direct energy inequality (see Lemma~\ref{le:eng}).

For any $0\le k\le L$, let the sequences $(C^{k}_{\ell})_{\ell=0}^{k}$ be such that
\[
  C^{k}_{\ell} = (-1)^{k-\ell}(\ell+1)
  \for 0\le \ell\le k.
\]
For ease of notation, we also use $C^{k}_{-1} = 0$.
Now we define the following Morawetz functional:
\begin{equation}
  \label{def:mora-func}
  \mora_{k}(t)
  = -C^{k}_{k}b_1\int\W_{2k}^0\cdot\R G_1\ma_{\la}\W_{2k}^0 +\sum_{\ell=0}^{k-1} C^{k}_{\ell} b_1\int\W_{2(2k-\ell)}^0\cdot\R G_2\W_{2\ell}^0.
\end{equation}
Here, the functions $G_1, G_2$ are the rescaled $g_1, g_2$ defined by~\eqref{pf:eng=g2}, and the above summation vanishes when $k=0$. The following holds for $k=L$.

\begin{lemma}[Morawetz Correction]
  \label{le:mora-est}
  Let $\mora = \mora_{L}(t)$. The following hold:
  \begin{enumerate}[label={\rm (\roman*)}]
    \item boundedness:
      \label{le:mora-est=bdd}
      \[
        |\mora| \les \frac{b_1^{1/2}}{\la^{4L+2}}\frac{b_1^{2L+2}}{\lgb^2}.
      \]
    \item cancellation of the commutator term:
      \label{le:mora-est=eq}
      \[
        \frac{\dd\mora}{\dd t}
        \le -\frac{b_1}{\la^{4L+4}} \sum_{\ell=0}^{L}\int\w_{2L+2}^0\cdot \R^{L+1}\mhamp^\ell\bigl(g_2\mhamp^{L-\ell}\wp\bigr) + \sum_{j=1}^{3} \QQ_{\mora,j} + \mF_{\mora},
      \]
      where
      \begin{align}
        \begin{split}
          \label{le:mora-est=qq1}
          \QQ_{\mora,1}
          & = -C^{L}_{L} \prt b_1\int\W_{2L}^0\cdot\R G_1\ma_{\la}\W_{2L}^0                              \\
          & \quad -C^{L}_{L} b_1\int\W_{2L}^0\cdot\R \prt G_1\ma_{\la}\W_{2L}^0                           \\
          & \quad +\sum_{\ell=0}^{L-1} C^{L}_{\ell} \prt b_1\int\W_{2(2L-\ell)}^0\cdot\R G_2\W_{2\ell}^0 \\
          & \quad +\sum_{\ell=0}^{L-1} C^{L}_{\ell} b_1\int\W_{2(2L-\ell)}^0\cdot\R\prt G_2\W_{2\ell}^0,
        \end{split}                                                       \\
        \begin{split}
          \label{le:mora-est=qq2}
          \QQ_{\mora,2}
          & = -C^{L}_{L}b_1\int\W_{2L}^0\cdot\R G_1[\prt, \ma_{\la}]\W_{2L}^0                                             \\
          & \quad -C^{L}_{L} b_1\int[\prt, (\R\mhaml)^{L}]\W^{\pp}\cdot\R G_1\ma_{\la}\W_{2L}^0                           \\
          & \quad -C^{L}_{L} b_1\int\W_{2L}^0\cdot\R G_1\ma_{\la}[\prt, (\R\mhaml)^{L}]\W^{\pp}                           \\
          & \quad +\sum_{\ell=0}^{L-1} C^{L}_{\ell} b_1\int[\prt, (\R\mhaml)^{2L-\ell}]\W^{\pp}\cdot\R G_2\W_{2\ell}^0    \\
          & \quad +\sum_{\ell=0}^{L-1} C^{L}_{\ell} b_1\int\W_{2(2L-\ell)}^0\cdot\R G_2[\prt, (\R\mhaml)^{\ell}]\W^{\pp},
        \end{split} \\
        \begin{split}
          \label{le:mora-est=qq3}
          \QQ_{\mora,3}
          & = -C^{L}_{L} b_1\int(\R\mhaml)^{L}(\R^2\W_2^1)\cdot\R G_1\ma_{\la}\W_{2L}^0                               \\
          & \quad -C^{L}_{L} b_1\int\W_{2L}^0\cdot\R G_1\ma_{\la}(\R\mhaml)^{L}(\R^2\W_2^1)                           \\
          & \quad +\sum_{\ell=0}^{L-1} C^{L}_{\ell} b_1\int(\R\mhaml)^{2L-\ell}(\R^2\W_2^1)\cdot\R G_2\W_{2\ell}^0    \\
          & \quad +\sum_{\ell=0}^{L-1} C^{L}_{\ell} b_1\int\W_{2(2L-\ell)}^0\cdot\R G_2(\R\mhaml)^{\ell}(\R^2\W_2^1),
        \end{split}                                          \\
        \begin{split}
          \label{le:mora-est=f}
          \mF_{\mora}
          & = \frac{C^{L}_{L} b_1}{\la^2} \int(\R\mhaml)^L\mF^{\pp}_{\la}\cdot\R G_1\ma_{\la}\W_{2L}^0                                \\
          & \quad +\frac{C^{L}_{L} b_1}{\la^2} \int\W_{2L}^0\cdot\R G_1\ma_{\la}(\R\mhaml)^L\mF^{\pp}_{\la}                           \\
          & \quad -\sum_{\ell=0}^{L-1} \frac{C^{L}_{\ell} b_1}{\la^2} \int(\R\mhaml)^{2L-\ell}\mF_\la^{\pp}\cdot\R G_2\W_{2\ell}^0    \\
          & \quad -\sum_{\ell=0}^{L-1} \frac{C^{L}_{\ell} b_1}{\la^2} \int\W_{2(2L-\ell)}^0\cdot\R G_2(\R\mhaml)^{\ell}\mF_\la^{\pp}.
        \end{split}
      \end{align}
  \end{enumerate}
\end{lemma}

\begin{proof}
  Proof of~\ref{le:mora-est=bdd}. Using Cauchy-Schwarz together with Lemma~\itref{le:interp-wp}{it:interp-wp=lossy} and Lemma~\ref{le:coer-e-2k2},
  \begin{align*}
    \left|b_1\int\W_{2L}^{0}\cdot\R G_1\ma_{\la}\W_{2L}^0\right|
    & \les \frac{b_1}{\la^{4L+2}} \biggl\|\frac{\jbk[\log y]}{\jbk[y]^2}\w_{2L}^0\biggr\|_{L^2}
    \biggl\|\frac{\ma\w_{2L}^0}{\jbk[y]\jbk[\log y]}\biggr\|_{L^2}                               \\
    & \les \frac{b_1^{1/2}}{\la^{4L+2}}
    \biggl(\E_{2L+2} +\frac{b_1^{2L+2}}{\lgb^2}\biggr).
  \end{align*}
  From the self-adjointness of $\mhamp$ and the estimate~\eqref{pf:mix-eng=rep-ham-claim} with $\ell_1+\ell_2=L-1$, which we prove independently in Proposition~\ref{pro:mix-eng},
  \begin{align*}
    \left|b_1\int\W_{2(2L-j)}^0\cdot\R G_2\W_{2j}^0\right|
    & \les \frac{b_1}{\la^{4L+2}} \|\w_{2L+2}^0\|_{L^2} \bigl\|\mhamp^{L-1-j}\bigl(g_2\mhamp^j\wp\bigr)\bigr\|_{L^2} \\
    & \les \frac{b_1^{1/2}}{\la^{4L+2}}
    \biggl(\E_{2L+2} +\frac{b_1^{2L+2}}{\lgb^2}\biggr).
  \end{align*}
  Inserting these into $\mora$, we obtain~\ref{le:mora-est=bdd}.

  Proof of~\ref{le:mora-est=eq}.
  First, a direct computation using \eqref{eq:w2k-eq} and the definition in~\eqref{pf:vr-est=w2k1} yields the equation for $\mora$:
  \begin{align*}
    \frac{\dd\mora}{\dd t}
    & = \frac{b_1}{\la^{4L+4}} ( \mcJ_1 -\mcJ_2 )
    +\sum_{j=1}^{3}\QQ_{\mora,j} +\mF_{\mora},                                                                                                                \\
    \mcJ_1 & = C^{L}_{L} \left\{ \int\w_{2L+2}^0\cdot\R g_1\w_{2L+1}^0 +\int\w_{2L}^0\cdot\R g_1\w_{2L+3}^0 \right\},                                             \\
    \mcJ_2 & = \sum_{\ell=0}^{L-1}C^{L}_{\ell} \left\{ \int\w_{2(2L+1-\ell)}^0\cdot\R g_2\w_{2\ell}^0 +\int\w_{2(2L-\ell)}^0\cdot\R g_2\w_{2(\ell+1)}^0 \right\}.
  \end{align*}
  We now estimate $\mcJ_1$ and $\mcJ_2$. We begin by computing $\T g_1$ and $\PP g_1$. Using~\eqref{eq:vla-phi-z}, with $y$ in place of $r$, we have
  \[
    \vla \z=\z^2-1 =-4y^2(1+y^2)^{-2},
    \andd
    V=2\vla\z+1,
  \]
  which implies
  \[
    -\T g_1 = -\frac{\vla(\vla\z -2\z -\z^2)}{y^2} = \frac{2\vla\z}{y^2}\le 0, \quad
    \PP g_1 = \frac{\vla(\vla \z +\z^2)}{y^2} = \frac{\vla V}{y^2} = g_2.
  \]
  Then, for the first integral in $\mcJ_1$, from the adjointness of $\mas, \ma$ and the first identity of~\eqref{pf:vr-est=a-t-p},
  \begin{align*}
    \int \w_{2L+2}^0 \cdot \R g_1 \w_{2L+1}^{0}
    & = \int \w_{2L+1}^{0} \cdot \ma(g_1\w_{2L+1}^{0})                                            \\
    & = -\int \w_{2L+1}^{0} \cdot \T g_1\w_{2L+1}^{0} -\int \w_{2L+1}^{0}\cdot g_1\mas\w_{2L+1}^0 \\
    & = -\int \w_{2L+1}^{0}\cdot \T g_1\w_{2L+1}^{0} -\int\R g_1\w_{2L+1}^{0} \cdot \w_{2L+2}^0.
  \end{align*}
  Note that the second term on the last line above is the negative of the left-hand side.
  Moving this term to the left-hand side yields:
  \begin{align*}
    \int \w_{2L+2}^0 \cdot\R g_1 \w_{2L+1}^{0}
    = -\frac{1}{2} \int \w_{2L+1}^{0} \cdot \T g_1\w_{2L+1}^{0}
    = \int\frac{\vla \z}{y^2} |\w_{2L+1}^{0}|^2
    \le 0.
  \end{align*}
  Using this and the second identity of~\eqref{pf:vr-est=a-t-p}, we estimate the second integral in $\mcJ_1$:
  \begin{align*}
    & \quad \int \w_{2L}^0 \cdot \R g_1 \w_{2L+3}^0                                                                            \\
    & = -\int\w_{2L+2}^0 \cdot \R \mas(g_1\w_{2L}^0)                                                                           \\
    & = -\int\w_{2L+2}^0 \cdot \R\bigl(\mas(g_1\w_{2L}^0) +g_1\ma\w_{2L}^{0}\bigr) + \int\w_{2L+2}^0 \cdot \R g_1\w_{2L+1}^{0} \\
    & \le -\int\w_{2L+2}^0 \cdot \R\PP g_1\w_{2L}^0
    = -\int\w_{2L+2}^0\cdot\R g_2\w_{2L}^0.
  \end{align*}
  Therefore
  \[
    \mcJ_1 \le -C^{L}_{L}\int\w_{2L+2}^0\cdot\R g_2\w_{2L}^0
    = -C^{L}_{L}\int\w_{2L+2}^0\cdot\R^{L+1} g_2\mhamp^L\wp.
  \]
  Next, we treat $\mcJ_2$. By the definition of $C^{L}_{\ell}$, we have $C^{L}_{-1} = 0$ and
  \begin{align*}
    C^{L}_{L-\ell} +C^{L}_{L-1-\ell}
    = (-1)^{\ell}(L-\ell+1) +(-1)^{\ell+1}(L-\ell)
    = (-1)^{\ell}.
  \end{align*}
  Using these, the fact that $\R^2 = -\id$ for vectors whose third component vanishes, and also a change of summation index, we compute
  \begin{align*}
    \mcJ_2 & = \sum_{\ell=1}^{L} C^{L}_{L-\ell} \int\w_{2(L+1+\ell)}^{0} \cdot\R g_2\w_{2(L-\ell)}^{0} + \sum_{\ell=0}^{L-1} C^{L}_{L-1-\ell} \int\w_{2(L+1+\ell)}^{0} \cdot\R g_2\w_{2(L-\ell)}^{0} \\
    & = \sum_{\ell=1}^L (C^{L}_{L-\ell} + C^{L}_{L-1-\ell}) \hs[-.2]\int\hs[-.15] \w_{2(L+1+\ell)}^{0} \cdot \R g_2\w_{2(L-\ell)}^{0} + C^{L}_{L-1} \hs[-.2]\int\hs[-.15] \w_{2(L+1)}^{0} \cdot \R g_2\w_{2L}^{0}        \\
    & = \sum_{\ell=1}^{L} \int\mhamp^{\ell} \w_{2L+2}^{0} \cdot \R^{\ell+1} \bigl(g_2\w_{2(L-\ell)}^{0}\bigr) + C^{L}_{L-1} \int\w_{2L+2}^{0} \cdot\R g_2\w_{2L}^{0}                     \\
    & = \sum_{\ell=0}^{L} \int\w_{2L+2}^{0} \cdot \R^{L+1} \mhamp^\ell \bigl(g_2\mhamp^{L-\ell}\wp\bigr) -C^{L}_{L} \int\w_{2L+2}^{0} \cdot \R^{L+1} g_2\mhamp^{L}\wp.
  \end{align*}
  Here, we note that in the third line above, the sign $(-1)^\ell$ is canceled using the identity $(\R^\ell)^\ast=(-1)^\ell\R^\ell$. Finally, inserting the estimates of $\mcJ_1$, $\mcJ_2$ into the equation for $\mora$ yields~\ref{le:mora-est=eq}. This completes the proof.
\end{proof}

\subsection{Mixed energy estimate}
\label{sub:mix-eng}
In this subsection, we establish the mixed energy estimate based on Lemmas~\ref{le:eng}, \ref{le:vr-est} and~\ref{le:mora-est}. This estimate is the core of our analysis. It is crucial in the proof of Proposition~\ref{pro:boot}.

We first record the following elementary estimates, which are used repeatedly below.

\begin{lemma}
  \label{le:est-before-mix}
  The following estimates hold:
  \begin{gather*}
    |\prt b_1| \les \frac{b_1^2}{\la^2}, \quad
    |\prt G_1| \les \frac{b_1}{\la^5\jbk[y]^5}, \quad
    |\prt G_2| \les \frac{b_1}{\la^6\jbk[y]^6}, \\
    |[\prt,\ma_\la]| \les \frac{b_1}{\la^3\jbk[y]^3}, \quad
    |\prr(q_\la)| \les \frac{R_{1}^C}{\la\jbk[y]}, \quad
    |\prt\prr(q_\la)| \les \frac{b_1 R_{1}^C}{\la^3\jbk[y]},
  \end{gather*}
  where $C>0$ is a constant independent of $K, M$.

\end{lemma}

\begin{proof}
  First, by the time-rescaling relation $\prt = \frac{1}{\la^2}\prs$, \eqref{le:mod-eq-1} and \eqref{le:mod-eq-2},
  \[
    |\prt b_1| \les \frac{|\prs b_1|}{\la^2} \les \frac{1}{\la^2}\biggl(b_1^2 +\frac{b_1^{L+1}}{\lgb}\biggr) \les \frac{b_1^2}{\la^2}.
  \]
  Then, we recall the definitions of $g_1$ and $g_2$ from~\eqref{pf:eng=g2}:
  \begin{align*}
    g_1 & = -\frac{4y}{(1+y^2)^2}
    = \begin{cases}
      -4y(1+O(y^2)) & \mbox{for $y\le 1$}, \\
      -\frac{4}{y^3}\bigl(1+O\bigl(\frac{1}{y^2}\bigr)\bigr)
      & \mbox{for $y\ge1$},
    \end{cases}, \\
    g_2 & = \frac{16(y^2-1)}{(1+y^2)^3}
    = \begin{cases}
      -16(1+O(y^2)) & \mbox{for $y\le 1$}, \\
      \frac{16}{y^4}\bigl(1+O\bigl(\frac{1}{y^2}\bigr)\bigr)
      & \mbox{for $y\ge1$}
    \end{cases}.
  \end{align*}
  A direct computation yields
  \[
    \prt G_1 = -\frac{\la_s}{\la^6} g_3
    \andd
    \prt G_2 = -\frac{\la_s}{\la^7} g_4,
  \]
  where
  \begin{align*}
    g_3 & = 3g_1 +\vla g_1 = \frac{-16y}{(1+y^2)^3}
    = \begin{cases}
      -16y(1+O(y^2)) & \mbox{for $y\le 1$}, \\
      -\frac{16}{y^5}\bigl(1+O\bigl(\frac{1}{y^2}\bigr)\bigr)
      & \mbox{for $y\ge1$},
    \end{cases} \\
    g_4 & = 4g_2 +\vla g_2 = \frac{64(2y^2-1)}{(1+y^2)^4}
    = \begin{cases}
      -64(1+O(y^2)) & \mbox{for $y\le 1$}, \\
      \frac{128}{y^6}\bigl(1+O\bigl(\frac{1}{y^2}\bigr)\bigr)
      & \mbox{for $y\ge1$}.
    \end{cases}
  \end{align*}
  This, together with~\eqref{le:mod-eq-1}, yields the bounds for $\prt G_1$ and $\prt G_2$. Next, from the definition of $\ma_{\la}$, we have $[\prt,\ma_\la]\W = \sum_{i=1,2} [\prt,\oa_\la]\W^{(i)}e_i$ for any vector $\W$, with
  \begin{align*}
    [\prt,\oa_\la]
    = \prt\biggl(\frac{\z_{\la}(r)}{r}\biggr)
    = O\Bigl(\frac{\la_s}{\la^4}\Bigr)\,
    \frac{\vla\z}{y}.
  \end{align*}
  The asymptotics of $\z$ yield the fourth estimate. Finally, note that
  \[
    |\prr(q_\la)| = \frac{|q'(y)|}{\la}
    \andd
    |\prt\prr(q_{\la})| = \Bigl|\frac{\la_s}{\la^4}\Bigr|\,|q'+\vla(q')|(y).
  \]
  From Lemma~\ref{le:vr-weig}, the function $q'(y)$ has polynomial growth for $y\le R_1$ and decays like $1/y$ for $y>R_1$. This yields the last two estimates, and hence completes the proof.
\end{proof}

\begin{proposition}[Mixed energy estimate]
  \label{pro:mix-eng}
  For any $c_0\in(0,1)$, there exists a constant $C > 1$ independent of $K,M$, such that
  \begin{align*}
    & \frac{\dd}{\dd t} \biggl\{\frac{1}{\la^{4L+2}}\biggl[\E_{2L+2} +O\biggl(\frac{b_1^{2L+2}}{\lgb^2}\biggr)\biggr] + \vr +2\mora \biggr\} \\
    & \le
    \frac{b_1}{\la^{4L+4}}\Biggl\{ \biggl(c_0 +\frac{C}{\sqrt{\log M}}\biggr)\E_{2L+2} +C(M)\frac{b_1^{2L+2}}{\lgb^2}\Biggr\}.
  \end{align*}
\end{proposition}

\begin{proof}
  A direct application of Lemma~\ref{le:eng}, Lemma~\itref{le:vr-est}{le:vr-est=eq}, and Lemma~\itref{le:mora-est}{le:mora-est=eq} gives
  \begin{equation}\begin{split}
      \label{pf:mix-eng=eng-expand}
      & \quad \frac{\dd}{\dd t}\biggl\{\frac{1}{\la^{4L+2}}\biggl[\E_{2L+2} +O\biggl(\frac{b_1^{2L+2}}{\lgb^2}\biggr)\biggr] + \vr +2\mora \biggr\} \\
      & \le
      \frac{b_1}{\la^{4L+4}} \Biggl\{ c_0 \int|\w_{2L+2}^0|^2 +\frac{Cb_1^{2L+2}}{\lgb^2} \Biggr\}
      +\sum_{j=1}^3 (\QQ_{\vr,j} +2\QQ_{\mora,j}) +2\QQ_{\E}                                                                                       \\
      & \quad +\mF_{\vr} +2\mF_{\mora} -\frac{2}{\la^{4L+4}}\int\w_{2L+2}\cdot(\hj\mham)^{L+1}\mF.
    \end{split}\end{equation}
  Here, the $\QQ$ terms are the small quadratic errors, and the $\mF$ terms are the errors coming from the force $\mF$. From Lemma~\itref{le:vr-est}{le:vr-est=bdd}, Lemma~\itref{le:mora-est}{le:mora-est=bdd}, and the smallness of $b_1$, we have
  \[
    |\vr| +|\mora|
    \les \frac{b_1^{1/4}}{\la^{4L+2}}\frac{b_1^{2L+2}}{\lgb^2}
    \les \frac{b_1^{\frac{1}{4L}}}{\la^{4L+2}}\frac{b_1^{2L+2}}{\lgb^2}.
  \]
  This gives the left-hand side of the mixed energy inequality. For the right-hand side, we recall from \eqref{eq:w2k0-e2k} that the integral of the squared $\w_{2L+2}^0$ gives $\E_{2L+2}$ directly. Thus it suffices to show
  \begin{align*}
    & \sum_{j=1}^3 \bigl(|\QQ_{\vr,j}| +|\QQ_{\mora,j}|\bigr) +|\QQ_{\E}|
    + |\mF_{\vr}| +|\mF_{\mora}|                                                      \\
    & \qquad +\left|\frac{1}{\la^{4L+4}}\int\w_{2L+2}\cdot(\hj\mham)^{L+1}\mF\right|
    \les \frac{b_1}{\la^{4L+4}} \biggl(\frac{\E_{2L+2}}{\sqrt{\log M}} +C(M)\frac{b_1^{2L+2}}{\lgb^2}\biggr),
  \end{align*}
  where the implicit constant is independent of $K,M$, while $C(M)$ may depend on $M$ but is independent of $K$. The bootstrap bounds~\eqref{pro:boot-e2k-bnd}, \eqref{pro:boot-e2l2-bnd} are extensively used below; we do not quote them explicitly.

  \step[1, Control of $\QQ_{\vr,1}, \QQ_{\mora,1}$]
  Recall the definitions~\eqref{le:vr-est=qq1} and \eqref{le:mora-est=qq1}. We first estimate $\QQ_{\vr,1}$. Using the bound for $\prr(q_{\la})$ in Lemma~\ref{le:est-before-mix} and Lemma~\itref{le:vr-est}{le:vr-est=bdd}, we obtain
  \[
    \left|\prt b_1\int\ma_{\la}\W_{2L}^0\cdot\prr(q_\la)\W_{2L+2}^0\right|
    \les \frac{b_1}{\la^2}|\vr|
    \les \frac{b_1\dbs}{\la^{4L+4}}\Eplus.
  \]
  By Lemma~\ref{le:est-before-mix} and Lemma~\itref{le:interp-wp}{it:interp-wp=sobolev},
  \[
    \left|b_1\int\ma_{\la}W_{2L}^0\cdot\prt\prr(q_\la)\W_{2L+2}^0\right|
    \les \frac{b_1^2 R_{1}^{C}}{\la^{4L+4}} \biggl\|\frac{\ma\w_{2L}^0}{\jbk[y]^2}\biggr\|_{L^2} \|\w_{2L+2}^0\|_{L^2} \les \frac{b_1\dbs}{\la^{4L+4}}\E_{2L+2}.
  \]
  We turn to $\QQ_{\mora,1}$. The first and the third terms in $\QQ_{\mora,1}$ are estimated as above by the smallness of $\prt b_1$ from Lemma~\ref{le:est-before-mix} and Lemma~\itref{le:mora-est}{le:mora-est=bdd}. The second term is controlled by the bound for $\prt G_1$ in Lemma~\ref{le:est-before-mix} together with Lemma~\itref{le:interp-wp}{it:interp-wp=sobolev}. For the last term, from the self-adjointness of $\mhamp$ and the bound for $\prt G_2$ in Lemma~\ref{le:est-before-mix},
  \begin{align*}
    & \left| b_1\int\W_{2(2L-\ell)}^0\cdot\R\prt G_2\W_{2\ell}^0 \right|
    \les \frac{b_1^2}{\la^{4L+4}} \sqrt{\E_{2L+2}}\,
    \Bigl\|\mhamp^{L-\ell-1}\bigl(g_4 \mhamp^\ell\wp\bigr)\Bigr\|_{L^2}.
  \end{align*}
  Here, the function $g_4$ is given in the proof of Lemma~\ref{le:est-before-mix}. By Corollary~\itref{co:dore}{it:dore=w}, $\wp\in\mcD_{L+1}$ with factor $b_1^{L+1}\lgb^{-1}$. Since $g_4$ admits an even series expansion for $y\le1$, by Lemma~\itref{le:preserv}{it:preserv=op} and Remark~\ref{re:preserv}, we have $\mhamp^{L-\ell-1}(g_4\mhamp^\ell\wp)\in\mcD_2$ with the same factor. We bound its remainder by $C(M,K)b_1^{L+1}\lgb^{-1}y^3\jbk[\log y]$, and hence
  \[
    |\mhamp^{L-\ell-1}(g_4\mhamp^\ell\wp)|
    \les C(M,K)\frac{b_1^{L+1}}{\lgb}y.
  \]
  We thus obtain the squared $L^2$ bound
  \[
    \Bigl\|\mhamp^{L-\ell-1}\bigl(g_4\mhamp^\ell\wp\bigr)\Bigr\|_{L^2(y\le 1)}^2
    \les C(M,K)\frac{b_1^{2L+2}}{\lgb^2}.
  \]
  For $y\ge1$, applying Lemma~\ref{le:omega-pw-bnd} twice and using Lemma~\itref{le:interp-wp}{it:interp-wp=sobolev},
  \[
    \Bigl\|\mhamp^{L-\ell-1}\bigl(g_4\mhamp^\ell\wp\bigr)\Bigr\|_{L^2(y\ge1)}^2
    \les \sum_{\ell=0}^{2L-2}\int_{y\ge1}\frac{|\pry^{\ell}\wp|^2}{y^{2(2L+4-\ell)}}
    \les C(M)\E_{2L+2}.
  \]
  Therefore,
  \begin{equation}
    \label{pf:mix-eng=vm-qq1-est}
    |\QQ_{\vr,1}| + |\QQ_{\mora,1}|
    \les \frac{b_1\dbs}{\la^{4L+4}} \Eplus.
  \end{equation}

  \step[2, Control of $\QQ_{\vr,2}, \QQ_{\mora,2}$]
  Recall the definitions~\eqref{le:vr-est=qq2}, \eqref{le:mora-est=qq2}. We first treat the terms involving $[\prt,\ma_{\la}]$. From Lemma~\ref{le:est-before-mix} and Lemma~\itref{le:interp-wp}{it:interp-wp=sobolev}, the one in $\QQ_{\vr,2}$ is bounded by
  \[
    \left|b_1\int[\prt,\ma_\la]\W_{2L}^0\cdot\prr(q_\la)\W_{2L+2}^0\right|
    \les \frac{R_{1}^C b_1^2}{\la^{4L+4}} \biggl\|\frac{\w_{2L}^0}{\jbk[y]^3}\biggr\|_{L^2} \|\w_{2L+2}^0\|_{L^2}
    \les \frac{b_1\dbs }{\la^{4L+4}}\E_{2L+2}.
  \]
  The one in $\QQ_{\mora,2}$ vanishes, due to the skew-symmetry of $\R$,
  \[
    b_1\int\W_{2L}^0\cdot\R G_1[\prt, \ma_{\la}]\W_{2L}^0
    = O(b_1^2)\int G_1^2 (\W_{2L}^0\cdot\R\W_{2L}^0)
    = 0.
  \]
  We now estimate the high-order commutator terms involving $[\prt, (\R\mhaml)^k]$. By \eqref{pf:eng=ma-mham-comm}, an induction argument yields
  \[
    \bigl[\prt, (\R\mhaml)^k\bigr]\Wp
    = \frac{O(b_1)}{\la^{2k+2}} \sum_{\ell_1+\ell_2=k-1} \R^{k}\mhamp^{\ell_1}\bigl(g_2\mhamp^{\ell_2}\wp\bigr)
    \for k\ge 1.
  \]
  We claim that there exists $C=C(M,K)>0$ such that
  \begin{equation}
    \label{pf:mix-eng=rep-ham-claim}
    \begin{split}
      \sum_{\ell_1+\ell_2 =L-1,L} \Bigl\|\mhamp^{\ell_1}\bigl(g_2\mhamp^{\ell_2}\wp\bigr) \Bigr\|_{L^2}^2
      & + \sum_{\ell_1+\ell_2=L-1} \Bigl\|\ma\mhamp^{\ell_1}\bigl(g_2\mhamp^{\ell_2}\wp\bigr) \Bigr\|_{L^2}^2  \\
      & + \sum_{\substack{\ell_1+\ell_2+\ell_3                                                        \\=L-2}} \Bigl\| \mhamp^{\ell_1} \bigl( g_2\mhamp^{\ell_2} \bigl( g_2\mhamp^{\ell_3}\wp \bigr) \bigr) \Bigr\|_{L^2}^2 \\
      & \les C(M,K) \Eplus.
    \end{split}
  \end{equation}
  Here, the first summation is taken as either $\ell_1+\ell_2=L$ or $\ell_1+\ell_2=L-1$.
  Then by \eqref{pf:mix-eng=rep-ham-claim} and Lemma~\ref{le:est-before-mix}, the second term of $\QQ_{\vr,2}$ is bounded by
  \begin{align*}
    & \left| b_1\int\ma_\la\bigl[\prt, (\R\mhaml)^L\bigr]\W^\pp\cdot\prr(q_\la)\W_{2L+2}^0 \right|                                                     \\
    & \les \frac{b_1^2}{\la^{4L+4}} \sum_{\ell_1+\ell_2=L-1} \Bigl\|\ma\mhamp^{\ell_1}\bigl(g_2\mhamp^{\ell_2}\wp\bigr)\Bigr\|_{L^2} \|q'\w_{2L+2}^0\|_{L^2} \\
    & \les \frac{b_1\dbs}{\la^{4L+4}} \Eplus.
  \end{align*}
  By \eqref{pf:mix-eng=rep-ham-claim} and Lemma~\itref{le:interp-wp}{it:interp-wp=sobolev}, similar estimates hold for the third term in $\QQ_{\vr,2}$ and those in $\QQ_{\mora,2}$. The fourth term in $\QQ_{\mora,2}$ requires an application of the self-adjointness of $\mhamp$, which splits the bound into two controllable ones:
  \begin{align*}
    & \left| b_1\int\bigl[\prt, (\R\mhaml)^{2L-\ell}\bigr]\W^{\pp}\cdot\R G_2\W_{2\ell}^0 \right|                                                                                                     \\
    & \les \frac{b_1^2}{\la^{4L+4}} \sum_{k=L+1}^{2L-\ell-1} \bigl\|\mhamp^{L+1}\wp\bigr\|_{L^2} \Bigl\|\mhamp^{k-(L+1)}\bigl( g_2\mhamp^{2L-\ell-1-k} \bigl(g_2\mhamp^{\ell}\wp\bigr) \bigr) \Bigr\|_{L^2} \\
    & \quad +\frac{b_1^2}{\la^{4L+4}} \sum_{k=0}^{L} \,\Bigl\|\mhamp^{L-k} \bigl( g_2\mhamp^{k}\wp \bigr) \Bigr\|_{L^2}  \Bigl\|\mhamp^{L-1-\ell} \bigl( g_2\mhamp^{\ell}\wp \bigr) \Bigr\|_{L^2}.
  \end{align*}
  The right-hand side is then bounded using \eqref{pf:mix-eng=rep-ham-claim} as before. Therefore, we obtain
  \begin{equation}
    \label{pf:mix-eng=vm-qq2-est}
    |\QQ_{\vr,2}| + |\QQ_{\mora,2}|
    \les \frac{b_1\dbs }{\la^{4L+4}} \biggl(\E_{2L+2}+\frac{b_1^{2L+2}}{\lgb^2}\biggr).
  \end{equation}
  \emph{Proof of the Claim~\eqref{pf:mix-eng=rep-ham-claim}}. We treat the first summation in~\eqref{pf:mix-eng=rep-ham-claim} with $\ell_1+\ell_2=L$. For $y\le1$, by Corollary~\itref{co:dore}{it:dore=w} and Lemma~\itref{le:preserv}{it:preserv=op}, we have $\mhamp^{\ell_1}(g_2\mhamp^{\ell_2}\wp)\in\mcD_1$ with factor $b_1^{L+1}\lgb^{-1}$, since $g_2$ admits an even series expansion. Hence
  \[
    \Bigl\| \mhamp^{\ell_1} \bigl( g_2\mhamp^{\ell_2} \wp \bigr) \Bigr\|_{L^2(y\le 1)}^2
    \le C(M,K)\frac{b_1^{2L+2}}{\lgb^2}.
  \]
  For $y \ge 1$, from the definition of $\Omega^K_{\rho}$,
  \[
    \bigl|\mhamp^{\ell_1}\bigl( g_2\mhamp^{\ell_2}\wp \bigr)\bigr|
    = \bigl|\Omega^{\ell_1}_{(0,\dots,0)} \bigl(g_2\,\Omega^{\ell_2}_{(0,\dots,0)}\wp\bigr)\bigr|.
  \]
  We apply Lemma~\ref{le:omega-pw-bnd} twice together with Leibniz, the asymptotics of $g_2$ and Lemma~\itref{le:interp-wp}{it:interp-wp=sobolev},
  \[
    \Bigl\| \mhamp^{\ell_1} \bigl( g_2\mhamp^{\ell_2}\wp\bigr) \Bigr\|_{L^2(y\ge1)}^2
    \les \sum_{\xi_1+\xi_2 = 2L} \biggl\|\frac{\pry^{\xi_1}\wp}{y^{\xi_2+4}}\biggr\|^2_{L^2(y\ge1)}
    \les C(M)\E_{2L+2}.
  \]
  We have the same squared $L^2$ bound for the remaining summations: by multiplying by functions admitting even series expansions, we retain the Dore order; for the extra $\ma$, we use only the available $1^{\rm st}$ derivative and quotient by $y$. We estimate the terms for $y\ge1$ as above and conclude~\eqref{pf:mix-eng=rep-ham-claim}.

  \step[3, Control of $\QQ_{\vr,3}, \QQ_{\mora,3}$]
  Recall the definitions \eqref{le:vr-est=qq3}, \eqref{le:mora-est=qq3}. From the definition of $\w_{2k}^0$,
  \begin{equation}
    \label{pf:mix-eng=vr-qq3-expand}
    \QQ_{\vr,3}
    = \frac{b_1}{\la^{4L+4}} \left\{ \int\R\ma\mhamp^L\w_2^1\cdot q'\mhamp^{L+1}\wp
    + \int\R\ma\mhamp^{L}\wp\cdot q'\mhamp^{L+1}\w_2^1 \right\}.
  \end{equation}
  Using the adjointness of $\mas, \ma$ and the first identity of~\eqref{pf:vr-est=a-t-p},
  \begin{align*}
    & \quad \int\R\ma(q'\ma\mhamp^L\wp)\cdot \ma\mhamp^L\w_2^1                                                    \\
    & = - \int \R q'\mas\ma\mhamp^L\wp \cdot \ma\mhamp^L\w_2^1 -\int\R\T q'\ma\mhamp^L\wp \cdot \ma\mhamp^L\w_2^1 \\
    & = \int q'\mhamp^{L+1}\wp \cdot \R\ma\mhamp^L\w_2^1
    -\int\R \T q'\mhamp^{L+1}\wp\cdot \mhamp^L\w_2^1                                                               \\
    & \quad -\int\R (\T q')'\ma\mhamp^{L}\wp \cdot \mhamp^L\w_2^1.
  \end{align*}
  The first integral on the right-hand side of the final equality coincides with the first integral in~\eqref{pf:mix-eng=vr-qq3-expand}. It is the most involved one to estimate. We treat it later on. First, we claim for $\ell_1+\ell_2 = L-1$ that
  \begin{equation}\begin{split}
      \label{pf:mix-eng=w21-claim}
      \bigl\| \mhamp^L\w_2^1 \bigr\|_{L^2}^2
      +\bigl\| \mhamp^{\ell_1} \bigl( g_2\mhamp^{\ell_2}\w_2^1 \bigr)\bigr\|_{L^2}^2
      \les \dbs\Eplus.
    \end{split}\end{equation}
  Then, other terms in $\QQ_{\vr,3}$ are bounded using the construction of $q(y)$:
  \[
    |\pry^\ell q| \les \frac{R_1^C}{\jbk[y]^\ell}, \quad 1\le\ell\le 3
    \quad\Longrightarrow\quad
    \frac{|\T q'|}{\jbk[y]} + |(\T q')'| \les \frac{R_1^C}{\jbk[y]^3},
  \]
  together with \eqref{pf:mix-eng=w21-claim}, Lemma~\itref{le:interp-wp}{it:interp-wp=sobolev}. For instance,
  \begin{align*}
    \left|\int\R (\T q')'\ma\mhamp^{L}\wp \cdot \mhamp^L\w_2^1\right|
    & \les R_{1}^C \biggl\|\frac{\ma\w_{2L}^0}{\jbk[y]^3}\biggr\|_{L^2} \|\mhamp^L\w_2^1\|_{L^2} \\
    & \les \dbs \Eplus.
  \end{align*}
  The terms in $\QQ_{\mora,3}$ are estimated similarly using \eqref{pf:mix-eng=w21-claim}, \eqref{pf:mix-eng=rep-ham-claim}, and the self-adjointness of $\mhamp$, yielding the same bound as above. For instance, the third term in $\QQ_{\mora,3}$ is bounded by
  \begin{align*}
    & \quad \left| b_1\int(\R\mhaml)^{2L-\ell}(\R^2\W_2^1)\cdot\R G_2\W_{2\ell}^0 \right|                                                   \\
    & = \frac{b_1}{\la^{4L+4}} \left| \int\R^{2L-\ell}\mhamp^{2L-\ell}\w_2^1 \cdot \R^{\ell+1}\bigl( g_2\mhamp^\ell\wp \bigr) \right|       \\
    & \les \frac{b_1}{\la^{4L+4}} \bigl\|\mhamp^L\w_2^1\bigr\|_{L^2} \Bigl\| \mhamp^{L-\ell} \bigl(g_2\mhamp^{\ell}\wp \bigr) \Bigr\|_{L^2}
    \les \frac{b_1\dbs }{\la^{4L+4}} \Eplus.
  \end{align*}

  Now we turn to the involved integral:
  \[
    \int\R\ma\mhamp^L\w_2^1\cdot q'\mhamp^{L+1}\wp.
  \]
  From the definition in~\eqref{def:w2k1}, we expand $\w_2^1$ into three parts:
  \[
    \w_2^1 = \R\mham(\gamma e_3) + \hw\times\mham\w
    = -\zt\pry\gamma e_2 + \hw\times(\mham\w-\mhamp\wp) +\hw\times\mhamp\wp.
  \]
  Set $\vv=\mhamp^{L+1}\wp$ and recall that $\wh\gamma=\ti\gamma_0+\gamma$. The first part involves derivatives up to order $2L+2$ of $\gamma$ and is controlled by Lemma~\ref{le:interp-gamma}. The difficulty in the other two parts is the occurrence of derivatives of order $2L+3$, which cannot be estimated directly by these lemmas. Using the constraint~\eqref{eq:constra}, we express the highest derivative of $\gamma$ in terms of the highest derivatives of $\wp$, with coefficient $-\hw^\pp/(1+\wh\gamma)$. Keeping only these highest-order terms gives
  \begin{align*}
    \int\R\ma\mhamp^L\w_2^1\cdot q'\vv
    &=\int\frac{q'}{1+\wh\gamma}
    (\hw^\pp\cdot\vv)(\hw^\pp\cdot\pry\vv)
    -\int q'\wh\gamma\ma\vv\cdot\vv \\
    &\quad+O\biggl(\dbs\Eplus\biggr).
  \end{align*}
  We control the remainder, which contains derivatives up to order $2L+2$, by Lemmas~\ref{le:interp-wp}, \ref{le:interp-gamma}. For $y\le1$, we use Corollary~\itref{co:dore}{it:dore=w}, \ref{it:dore=hw} for derivatives up to order $2L+1$ and keep the highest derivative in $L^2$, as in the proof of Lemma~\ref{le:w2l2-e}. For $y\ge1$, we use the weighted and pointwise bounds as in~\eqref{pf:w2l2-e=omega-l1-w-ygt1}. We control the contributions of the profiles by~\eqref{pf:form-comp=rrr-expr} and~\eqref{pf:app-sol=wz-l-infty}.

  For the first integral, an integration by parts yields
  \begin{align*}
    &\int\frac{q'}{1+\wh\gamma}(\hw^\pp\cdot\vv)(\hw^\pp\cdot\pry\vv) \\
    &=-\frac12\int\frac1y\pry\biggl(\frac{yq'}{1+\wh\gamma}\biggr)(\hw^\pp\cdot\vv)^2
    -\int\frac{q'}{1+\wh\gamma}(\hw^\pp\cdot\vv)(\pry\hw^\pp\cdot\vv).
  \end{align*}
  Note that no derivative now falls on $\vv$. By Corollary~\itref{co:dore}{it:dore=hw}, Lemma~\itref{le:interp-wp}{it:interp-wp=pw} and the profile bounds,
  \[
    |\hw^\pp\|_{L^\infty}\les\sqrt{\E_1}+b_1^{1/2}\lgb^C\ll1.
  \]
  Together with the bounds for $q$, this gives a bound of $\dbs\|\vv\|_{L^2}^2=\dbs\E_{2L+2}$. The second integral has the same bound after integration by parts, since $\ma=-\pry+\z/y$.
  Inserting this into \eqref{pf:mix-eng=vr-qq3-expand} finishes the estimate for $\QQ_{\vr,3}$. Therefore,
  \begin{equation}
    \label{pf:mix-eng=vm-qq3-est}
    |\QQ_{\vr,3}| +|\QQ_{\mora,3}|
    \les \frac{b_1\dbs }{\la^{4L+4}} \biggl(\E_{2L+2}+\frac{b_1^{2L+2}}{\lgb^2}\biggr).
  \end{equation}
  \emph{Proof of Claim~\eqref{pf:mix-eng=w21-claim}}. By~\eqref{def:w2k1}, $\w_2^1=\R\mham(\gamma e_3)+\hw\times\mham\w$. To use Lemma~\ref{le:w2l2-e}, we apply~\eqref{eq:mham-r}:
  \[
    \R^L\mhamp^L(\hw\times\mham\w)
    =\Omega_{(0,\dots,0,1)}^{L+1}\w
    +\R^{L-1}\mhamp^{L-1}
    \bigl[\zt\pry\bigl((\hw\times\mham\w)\cdot e_3\bigr)e_2\bigr].
  \]
  Thus the first part of the left-hand side is bounded by
  \begin{align*}
    |\mhamp^L\w_2^1|
    & \le |(\R\mham)^{L+1}(\gamma e_3)|+\bigl|\Omega_{(0,\dots,0,1)}^{L+1}\w\bigr| \\*
    & \quad+\bigl|\mhamp^{L-1}\bigl[\zt\pry\bigl((\hw\times\mham\w)\cdot e_3\bigr)e_2\bigr]\bigr|.
  \end{align*}
  The first two terms have been estimated in Lemma~\ref{le:w2l2-e}. For the last term, $(\hw\times\mham\w)\cdot e_3=\wh\alpha\ham\beta-\wh\beta(\ham\alpha-\zt\pry\gamma)$. By Corollary~\itref{co:dore}{it:dore=w}, \ref{it:dore=hw} and Lemma~\itref{le:preserv}{it:preserv=cross}, we have $\hw\times\mham\w\in\mcD_L$ with factor $b_1^{L+2}\lgb^{-1}$. Since we use only derivatives up to order $2L-1$ of its third component when applying $\mhamp^{L-1}\zt\pry$, we bound the squared $L^2$ norm on $y\le1$ by $C(M,K)b_1^{2L+4}\lgb^{-2}$. For $y\ge1$, by Leibniz,
  \[
    \bigl|\mhamp^{L-1}\bigl[\zt\pry\bigl((\hw\times\mham\w)\cdot e_3\bigr)e_2\bigr]\bigr|
    \les\sum_{0\le k\le 2L+1}^{} \sum_{\substack{k_1+k_2=k\\0\le k_1\le2L-1\\0\le k_2\le 2L+1}}
    \frac{|\pry^{k_1}\hw|\,|\pry^{k_2}\w|}{y^{2L+3-k}}.
  \]
  We split $\hw=\wzt+\w$, and use~\eqref{pf:w2l2-e=wzt-l-infty} and Lemma~\itref{le:interp-wp}{it:interp-wp=gen-lossy},~\ref{it:interp-wp=pw} (putting $\w/y$ in $L^\infty$ when $k_1=0$). For $k_2=2L,2L+1$, taking $k=L$ in~\ref{it:interp-wp=lossy} and $i+j=1$ in~\ref{it:interp-wp=pw} gives a squared $L^2$ bound of $b_1^{2L+3}\lgb^C$. Thus
  \begin{align*}
    \bigl\|\mhamp^{L-1}\bigl[\zt\pry\bigl((\hw\times\mham\w)\cdot e_3\bigr)e_2\bigr]\bigr\|_{L^2}^2
    & \les C(M,K)\frac{b_1^{2L+4}}{\lgb^2}+b_1^{2L+3}\lgb^C \\
    & \les\dbs\Eplus.
  \end{align*}
  Since $\ell_1+\ell_2=L-1$, the second part is estimated directly, keeping the cross product intact:
  \[
    \bigl| \mhamp^{\ell_1}\bigl(g_2\mhamp^{\ell_2}\w_2^1 \bigr)\bigr|
    \le \bigl|\mhamp^{\ell_1} \bigl(g_2\mhamp^{\ell_2}\R\mham(\gamma e_3 ) \bigr)\bigr| +\bigl|\mhamp^{\ell_1} \bigl(g_2\mhamp^{\ell_2}(\hw\times\mham\w) \bigr)\bigr|.
  \]
  For $y\le1$, by Corollary~\itref{co:dore}{it:dore=w}, \ref{it:dore=hw}, Lemma~\itref{le:preserv}{it:preserv=op}--\ref{it:preserv=cross} and the bound for $\gamma e_3$ in Lemma~\itref{le:interp-wp}{it:interp-wp=taylor-orig}, we have $\R\mham(\gamma e_3),\hw\times\mham\w\in\mcD_L$ with factor $b_1^{L+2}\lgb^{-1}$. Since $\ell_1+\ell_2=L-1$ and multiplication by $g_2$ preserves the Dore order, both terms belong to $\mcD_1$ with the same factor. Hence
  \[
    \bigl\|\mhamp^{\ell_1}\bigl(g_2\mhamp^{\ell_2}\w_2^1\bigr)\bigr\|_{L^2(y\le 1)}^2
    \les C(M,K)\frac{b_1^{2L+4}}{\lgb^2}.
  \]
  For $y\ge1$, by Lemma~\ref{le:omega-pw-bnd},
  \[
    \bigl|\mhamp^{\ell_1} \bigl( g_2\mhamp^{\ell_2}\R\mham(\gamma e_3) \bigr) \bigr|
    = \bigl| \mhamp^{\ell_1} \bigl( g_2\mhamp^{\ell_2} \bigl( \zt\pry\gamma e_2 \bigr) \bigr) \bigr|
    \les \sum_{\substack{k_1+k_2=2L+2\\1\le k_1\le 2L-1}} \frac{|\pry^{k_1}\gamma|}{y^{k_2+2}},
  \]
  which reduces to the estimate of \eqref{pf:w2l2-e=small-gamma-int}. Similarly,
  \[
    \bigl|\mhamp^{\ell_1} \bigl( g_2\mhamp^{\ell_2}(\hw\times\mham\w) \bigr) \bigr|
    \les \sum_{\substack{k_1+k_2+k_3 \\=2L+2}}
    \biggl(\bigl|y^{k_1}\pry^{k_1}\wzt\bigr| \frac{|\pry^{k_2}\w|}{y^{k_1+k_3}}
    +\bigl|\pry^{k_1}\w\bigr| \frac{|\pry^{k_2}\w|}{y^{k_3}} \biggr),
  \]
  which is estimated as in~\eqref{pf:w2l2-e=omega-l1-w}. Therefore we obtain
  \[
    \bigl\| \mhamp^{\ell_1} \bigl(g_2\mhamp^{\ell_2}\w_2^1 \bigr) \bigr\|_{L^2(y\ge1)}^2
    \les \dbs\Eplus,
  \]
  which finishes the proof of~\eqref{pf:mix-eng=w21-claim}.

  \step[4, Control of $\QQ_{\E}$]
  Recall the definition~\eqref{pf:eng=q-expr}. Using Cauchy-Schwarz, the modulation equation~\eqref{le:mod-eq-1} and the claim~\eqref{pf:mix-eng=rep-ham-claim} with $\ell_1+\ell_2=L$, we bound the first term in $\QQ_{\E}$ by
  \begin{align*}
    & \left|\Bigl(b_1+\frac{\la_s}{\la}\Bigr)\int\W_{2L+2}^0\cdot(\R\mhaml)^\ell\R (G_2\W_{2(L-\ell)}^0)\right|         \\
    & \les \frac{D(s)}{\la^{4L+4}}\, \|\w_{2L+2}^0\|_{L^2} \bigl\|\mhamp^\ell \bigl(g_2\mhamp^{L-\ell}\wp \bigr)\bigr\|_{L^2} \\
    & \les \frac{b_1\dbs }{\la^{4L+4}}\biggl(\E_{2L+2}+\frac{b_1^{2L+2}}{\lgb^2}\biggr).
  \end{align*}
  For the remaining terms in~\eqref{pf:eng=q-expr}, we repeat the preceding estimates separately for $y \le 1$ and $y \ge 1$. These yield the same bound as above. The required smallness of the second, third, and fourth terms follows from the bounds for $\w_{2L+2}^1$ and $\w_{2(L-\ell)}^1$, together with the estimates for $\Omega_+^\ell$ established above. Now let us consider the last term in \eqref{pf:eng=q-expr},
  \[
    \int\W_{2L+2}^0 \cdot(\R\mhaml)^\ell\Bigl(\hW\times[\prt,\mhaml]\W_{2(L-\ell)}^0 +\prt\hW\times\mhaml\W_{2(L-\ell)}^0\Bigr).
  \]
  Recalling \eqref{pf:eng=hj-mhaml-comm} and the scaling, we see from Cauchy-Schwarz that it suffices to show
  \begin{align*}
    & \Bigl|\frac{\la_s}{\la}\Bigr|^2\,\bigl\|(\R\mham)^\ell(\hw\times\mfg\w_{2(L-\ell)}^0)\bigr\|^2_{L^2}
    +|\vtha_s|^2 \bigl\|(\R\mham)^\ell[(\z\R\hw)\times\mham\w_{2(L-\ell)}^0]\bigr\|^2_{L^2}                 \\
    & \qquad\quad
    +\Bigl\|(\R\mham)^\ell\Bigl(\hj(\mham\wzt+\hp\lphi)\times\mham\w_{2(L-\ell)}^0\Bigr)\Bigr\|^2_{L^2}     \\
    & \qquad\qquad\qquad
    + \bigl\|(\R\mham)^\ell\bigl(\hj\mham\w\times\mham\w_{2(L-\ell)}^0\bigr)\bigr\|^2_{L^2}
    \les \dbs b_1^2 \Eplus.
  \end{align*}

  We provide details only for the first term. By Corollary~\itref{co:dore}{it:dore=w}, \ref{it:dore=hw}, $\w,\hw\in\mcD_{L+1}$ with factors $b_1^{L+1}\lgb^{-1}$ and $b_1$. The possible constant third component of $\mfg\w_{2(L-\ell)}^0$ causes no loss after taking the cross product with $\hw$, as in the proof of Lemma~\itref{le:preserv}{it:preserv=cross}. Since only derivatives up to order $2L+1$ occur, we have a pointwise bound of $C(M,K)b_1^{L+2}\lgb^{-1}\jbk[\log y]$ for $y\le1$, and hence
  \[
    \bigl\|(\R\mham)^\ell\bigl(\hw\times\mfg\w_{2(L-\ell)}^0\bigr)\bigr\|_{L^2(y\le1)}^2
    \les C(M,K)\frac{b_1^{2L+4}}{\lgb^2}.
  \]
  For $y\ge1$, from the definition of $\mfg$, Lemma~\ref{le:omega-pw-bnd} and Leibniz,
  \[
    \bigl|(\R\mham)^\ell\bigl(\hw\times\mfg\w_{2(L-\ell)}^0\bigr)\bigr|
    \les \sum_{\substack{\xi_1+\xi_2+\xi_3\\ =2\ell+1}} \hs[-.3] \frac{|\pry^{\xi_1}\hw|\,|\pry^{\xi_2}\w_{2(L-\ell)}^0|}{y^{\xi_3+2}}
    \les \hs[-.3] \sum_{\substack{\xi_1+\xi_2+\xi_3=2L+2\\\xi_2+\xi_3\ge 2}} \hs[-.8]
    \frac{|\pry^{\xi_1}\hw|\,|\pry^{\xi_2}\w|}{y^{\xi_3}},
  \]
  whose $L^2$ norm is estimated as~\eqref{pf:w2l2-e=omega-l1-w} with $n=1$. Therefore, it admits the same bound as in~\eqref{pf:w2l2-e=omega-l1-w-ygt1}. Together with the estimate for $y\le 1$, and also $|\la_s/\la| \les b_1$ (by the modulation equation~\eqref{le:mod-eq-1}), we have
  \[
    \Bigl|\frac{\la_s}{\la}\Bigr|^2 \bigl\|(\R\mham)^\ell \bigl(\hw\times\mfg\w_{2(L-\ell)}^0\bigr)\bigr\|^2_{L^2}
    \les b_1^2\dbs \E_{2L+2}.
  \]
  Therefore,
  \begin{equation}
    \label{pf:mix-eng=vm-qqe-est}
    \QQ_{\E} \les \frac{b_1\dbs}{\la^{4L+4}} \Eplus.
  \end{equation}

  \step[5, Control of $\mF$ terms]
  Recall the definitions of $\mF_{\vr}$ and $\mF_{\mora}$ in~\eqref{le:vr-est=f} and~\eqref{le:mora-est=f}, respectively. We claim that
  \begin{gather}
    \begin{split}
      \label{pf:mix-eng=mf-claim-1}
      \bigl\|\ma\mhamp^L\mF^\pp\bigr\|_{L^2}^2
      \les \dbs \Eplus,
    \end{split} \\
    \begin{split}
      \label{pf:mix-eng=mf-claim-2}
      \|(\R\mham)^{L+1}\mF\|_{L^2}^2 +\|(\hj\mham)^{L+1}\mF\|_{L^2}^2
      \les b_1^2 \biggl(\frac{\E_{2L+2}}{\log M} + \frac{b_1^{2L+2}}{\lgb^2}\biggr),
    \end{split} \\
    \begin{split}
      \label{pf:mix-eng=mf-claim-3}
      \bigl\|\mhamp^{L-j-1}\bigl( g_2\mhamp^j\mF^\pp \bigr)\bigr\|_{L^2}^2
      +\biggl\|\frac{\jbk[\log y]}{\jbk[y]^2}\mhamp^L\mF^\pp\biggr\|_{L^2}^2
      \les \frac{\E_{2L+2}}{\log M} + \frac{b_1^{2L+2}}{\lgb^2},
    \end{split}
  \end{gather}
  where $0\le j\le L-1$, and the implicit constants are independent of $M$. The first term in $\mF_{\vr}$ is directly bounded using~\eqref{pf:mix-eng=mf-claim-1} by
  \begin{align*}
    \left|\frac{b_1}{\la^2} \int\ma_{\la}(\R\mhaml)^L\mF_{\la}^\pp \cdot \prr (q_{\la})\W_{2L+2}^0\right|
    & \les \frac{b_1}{\la^{4L+4}} \|q'\w_{2L+2}^0\|_{L^2} \bigl\|\ma\mhamp^L\mF^\pp\bigr\|_{L^2} \\
    & \les \frac{b_1\dbs }{\la^{4L+4}} \Eplus.
  \end{align*}
  For the second integral, we use the identities in~\eqref{pf:vr-est=a-t-p}, Lemma~\ref{le:interp-wp}, and~\eqref{pf:mix-eng=mf-claim-1} to obtain
  \begin{align*}
    & \left|\frac{b_1}{\la^2} \int\ma_{\la}\W_{2L}^0\cdot \prr (q_{\la})(\R\mhaml)^{L+1}\mF_{\la}^\pp \right|                                               \\
    & = \frac{b_1}{\la^{4L+4}} \left|\int(q'\mas\w_{2L+1}^0 +\T q'\w_{2L+1}^0) \cdot\R^{L+1}\ma\mhamp^L\mF^\pp\right|                                             \\
    & \les \frac{b_1}{\la^{4L+4}} R_{1}^C \biggl(\|\w_{2L+2}^0\|_{L^2} +\biggl\|\frac{\ma\w_{2L}^0}{\jbk[y]^2}\biggr\|_{L^2} \biggr) \|\ma\mhamp^L\mF^\pp\|_{L^2} \\
    & \les \frac{b_1\dbs }{\la^{4L+4}} \biggl(\E_{2L+2} +\frac{b_1^{2L+2}}{\lgb^2}\biggr).
  \end{align*}
  Then for $\mF_{\mora}$, by~\eqref{pf:mix-eng=mf-claim-3} and Lemma~\ref{le:coer-ti-ham}, the first term is bounded by
  \[
    \left|\frac{b_1}{\la^2}\int(\R\mhaml)^L\mF_{\la}^\pp\cdot \R G_1\ma_\la\W_{2L}^0\right|
    \les \frac{b_1}{\la^{4L+4}} \EplusM.
  \]
  The second, third and fourth terms in $\mF_{\mora}$ are controlled in the same fashion by using the claims~\eqref{pf:mix-eng=mf-claim-1}, \eqref{pf:mix-eng=mf-claim-2} and \eqref{pf:mix-eng=mf-claim-3} respectively. Note that for the application of these claims, in the estimates for the third and fourth terms we also use the self-adjointness of the Hamiltonian $\mhamp$ and \eqref{pf:mix-eng=rep-ham-claim}. Finally, for the last integral in~\eqref{pf:mix-eng=eng-expand}, we apply
  \[
    \|\w_{2L+2}\|_{L^2}^2 \le \E_{2L+2} + \frac{Cb_1^{2L+2}}{\lgb^2},
  \]
  by Remark~\ref{re:w2l2-e}, together with \eqref{pf:mix-eng=mf-claim-2}, which gives
  \[
    |\mF_{\vr}| + |\mF_{\mora}| + \left|\frac{1}{\la^{4L+4}} \int\w_{2L+2}\cdot (\hj\mham)^{L+1}\mF\right|
    \les \frac{b_1}{\la^{4L+4}} \EplusM.
  \]

  Now, inserting \eqref{pf:mix-eng=vm-qq1-est}, \eqref{pf:mix-eng=vm-qq2-est}, \eqref{pf:mix-eng=vm-qq3-est}, \eqref{pf:mix-eng=vm-qqe-est}, and the above inequality into~\eqref{pf:mix-eng=eng-expand}, we obtain the desired mixed energy inequality. It remains to prove the claims~\eqref{pf:mix-eng=mf-claim-1}--\eqref{pf:mix-eng=mf-claim-3}. Recall from~\eqref{eq:w-eq} that
  \[
    \mF = \psizt +\modt +\vtha_s \z\R\w +\w\times(\mham\wzt +\hp\lphi).
  \]
  We prove~\eqref{pf:mix-eng=mf-claim-1}--\eqref{pf:mix-eng=mf-claim-3} for each term on the right-hand side above. The modulation terms determine the energy-dependent parts of these bounds; $\widetilde\Psi_0$ contributes the separate source term in the second claim.

  \step[6, Contribution of $\psizt$]
  Note that the claims follow directly from the bounds given in Lemma~\itref{le:local-app-sol}{it:local-app-sol=sob}--\ref{it:local-app-sol=weig}, except for the $(\hj\mham)^{L+1}$ part of~\eqref{pf:mix-eng=mf-claim-2}, since the complex structure
  \[
    \hj = (e_3 +\wzt +\w) \times
  \]
  contains the actions of $\wzt$ and $\w$. Thus it suffices to control $(\hj\mham)^{L+1}\psizt$. By the definition in~\eqref{def:omega}, we have
  \[
    \int|(\hj\mham)^{L+1}\psizt|^2
    \les \int|(\R\mham)^{L+1}\psizt|^2 + \int|\Omega^{L+1}_+[\wzt+\w]\psizt|^2,
  \]
  where we have used Lemma~\itref{le:local-app-sol}{it:local-app-sol=sob} for the first part. For $y\le1$, from the construction in Lemma~\ref{le:app-sol} and Corollary~\itref{co:dore}{it:dore=hw}, we have $\psiz\in\mcD_{L+2}$ and $\hj\mham\psiz\in\mcD_{L+1}$ with factor $b_1^{L+3}$. Then by Lemma~\itref{le:preserv}{it:preserv=op}--\ref{it:preserv=cross}, we have
  \[
    \|\Omega_+^{L+1}[\hw]\psizt\|_{L^2(y\le1)}^2
    \les Cb_1^{2L+6}.
  \]
  For $y\ge1$, from Lemma~\ref{le:omega-pw-bnd} and~\eqref{pf:app-sol=wz-l-infty},
  \begin{align*}
    & \int|\Omega^{L+1}_+[\wzt+\w]\psizt|^2 \\
    & \les \sum_{n,m,\xi} b_1^{m}\lgb^C \hs[-.4] \prod_{\ell=m+1}^{n} \hs[-.4] \|\pry^{\xi_{\ell}}\w\|_{L^\infty(y\ge1)}^2 \int_{y\ge1}\frac{|\pry^{\xi_{n+1}}\psizt|^2}{y^{2(\xi_1+\dots+\xi_{m}+\xi_{n+2})}},
  \end{align*}
  where the sum is taken over
  \[
    1\le n\le L+1,\quad 0\le m\le n,\quad \xi\in\mbN^{n+2},\quad |\xi|_1=2L+2,\quad |\xi|_+\ge2.
  \]
  If $m=0$ and $\xi_1+\dots+\xi_n=\xi_{n+2}=0$, then $\xi_{n+1}=2L+2$. By Lemma~\itref{le:local-app-sol}{it:local-app-sol=weig} and the explicit far-out error $O(b_1\lphi)$ in the proof of that lemma, we have
  \begin{align*}
    \int_{y\ge1}y^2|\pry^{2L+2}\psizt|^2
    & \les S_1^2\int_{1\le y\le2S_1}|\pry^{2L+2}\psizt|^2
    +b_1^2\int_{2S_1}^{\infty}\frac{y\dd y}{y^{4L+4}} \\
    & \les b_1^{2L+3}\lgb^2.
  \end{align*}
  Together with Lemma~\itref{le:interp-wp}{it:interp-wp=pw}, this gives the bound
  \[
    \|\w\|_{L^\infty}^{2(n-1)}
    \Bigl\|\frac{\w}{y}\Bigr\|_{L^\infty(y\ge1)}^2
    \int_{y\ge1}y^2|\pry^{2L+2}\psizt|^2
    \les b_1^{2L+4+\frac{1}{2L-1}}\lgb^C.
  \]
  For the remaining terms, either $m\ge1$, a derivative falls on a factor $\w$, or $\xi_{n+2}\ge1$. In the last case, we estimate one factor as $\w/y$. Lemma~\itref{le:local-app-sol}{it:local-app-sol=weig} and Lemma~\itref{le:interp-wp}{it:interp-wp=pw} then give
  \[
    \int|\Omega^{L+1}_+[\wzt+\w]\psizt|^2
    \les b_1^{2L+4+\frac{1}{2L-1}}\lgb^C,
  \]
  which yields \eqref{pf:mix-eng=mf-claim-2} for $\psizt$.

  \step[7, Contribution of $\modt$]
  Recall from \eqref{le:local-app-sol=modt} that
  \begin{align*}
    \modt & =
    \sum_{k=1}^{L}(\prs a_k +A_k)\biggl(\phiat[k] +\chi_{S_1}\biggl(\frac{\pr\x}{\pr a_k} +\frac{\pr\gamma_0}{\pr a_k}e_3\biggr)\biggr)                                                                 \\
    & \quad +\sum_{k=1}^{L}(\prs b_k +B_k)\biggl(\phibt[k] +\chi_{S_1}\biggl(\frac{\pr\x}{\pr b_k} +\frac{\pr\gamma_0}{\pr b_k}e_3\biggr) + \delta_{1,k}\frac{\pr\chi_{S_1}}{\pr b_1}\wz \biggr) \\
    & \quad +(\vtha_s+a_{1})\bigl(\lphi\bigl(\etwo -\eone\times\wzt \bigr) +\z\R\wzt \bigr)                                                                                                       \\
    & \quad -\Bigl(\frac{\la_s}{\la} +b_1\Bigr)\bigl(\lphi\bigl(\eone +\etwo\times\wzt\bigr) +\vla\wzt \bigr).
  \end{align*}
  Let us first show~\eqref{pf:mix-eng=mf-claim-1}. From the cancellation $\oa\ham^{L}T_{k}=0$ for $1\le k\le L$,
  \[
    \int|\ma\mhamp^{L}\phiat[k]|^2
    \les \int_{y\sim S_1} \biggl|\frac{y^{2k-1}\jbk[\log y]^{C}}{y^{2L+1}} \biggr|^2
    \les b_1\lgb^C.
  \]
  From the properties of the functions in $\mcHg$ and the fact that $\mhamp$ vanishes on the third component,
  \begin{align*}
    & \int \biggl| \ma\mhamp^L \biggl[ \chi_{S_1} \biggl(\frac{\pr\x}{\pr a_k}+ \frac{\pr\gamma_0}{\pr a_k} e_3 \biggr)\biggr] \biggr|^2 \\
    & \qquad \les \lgb^C\sum_{\ell=k+1}^{L+2}b_1^{2(\ell-k)} \Biggl(1+\int_{1}^{2S_1} \biggl|\frac{y^{2\ell-1}}{y^{2L+1}}\biggr|^2 \Biggr) \\
    & \qquad \les b_1\lgb^{C+6}.
  \end{align*}
  The same estimate holds for the correction involving the $b_k$ derivative. The additional cutoff term is supported on $y\sim S_1$ and satisfies
  \[
    b_1\frac{\pr\chi_{S_1}}{\pr b_1}
    =\Bigl(\frac12+\frac1{\lgb}\Bigr)\vla\chi_{S_1}
    = O(1) \one_{y\sim S_1}.
  \]
  Hence, by the same profile estimates,
  \[
    \int\biggl|\ma\mhamp^L\biggl(\frac{\pr\chi_{S_1}}{\pr b_1}\wz\biggr)\biggr|^2
    \les b_1\lgb^C, \quad
    \int\biggl|(\R\mham)^{L+1}\biggl(\frac{\pr\chi_{S_1}}{\pr b_1}\wz\biggr)\biggr|^2
    \les b_1^2.
  \]
  Moreover, since $\ham\lphi = 0$,
  \[
    \int|\ma\mhamp^L(\lphi\, e_i)|^2 = 0
    \for i=1,2.
  \]
  From the explicit construction of $\wzt$ given by Lemma~\ref{le:app-sol} and \ref{le:local-app-sol},
  \begin{align*}
    & \sum_{i=1}^2\bigl\|\ma\mhamp^L\bigl(\lphi(e_i\times\wzt)\bigr)\bigr\|_{L^2}^2 \\
    & +\|\ma\mhamp^L(\z\R\wzt)\|_{L^2}^2
    +\|\ma\mhamp^L(\vla\wzt)\|_{L^2}^2
    \les b_1\lgb^C.
  \end{align*}
  These together with \eqref{pf:mod-eq=d} yield \eqref{pf:mix-eng=mf-claim-1}:
  \[
    \|\ma\mhamp^L\modt\|_{L^2}^2 \les \dbs \Eplus.
  \]
  For~\eqref{pf:mix-eng=mf-claim-2}, from the cancellation $\ham^{L+1} T_{k} = 0$,  the function $(\R\mhamp)^{L+1}\phiat[k]$ is localized inside the region $y\sim S_1$, which yields
  \[
    \int|(\R\mham)^{L+1}\phiat[k]|^2
    \les \int_{y\sim S_1} \biggl|\frac{y^{2k-1}\jbk[\log y]}{y^{2L+2}} \biggr|^2
    \les b_1^2.
  \]
  By \eqref{pf:app-sol=xk-deg}, \eqref{pf:app-sol=gamma-0-k-deg} and Definition~\ref{def:badb}, using $\zt=O(y^{-2})$ for $y\ge1$
  and including the cutoff commutators on $y\sim S_1$,
  \[
    \int \biggl| (\R\mham)^{L+1}
    \biggl[ \chi_{S_1}
      \biggl(\frac{\pr\x}{\pr a_k}
      +\frac{\pr\gamma_0}{\pr a_k}e_3\biggr)\biggr]
    \biggr|^2
    \les \sum_{\ell=k+1}^{L+2} b_1^{2(\ell-k)}
    \biggl(1+\frac{\one_{\ell=L+2}}{b_1^2\lgb^2}\biggr)
    \les b_1^2.
  \]
  The same bound holds for the correction involving the $b_k$ derivative, while the cutoff derivative is covered by the second estimate above.
  At order $L+1$, the same argument gives
  \begin{align*}
    & \sum_{i=1}^2\bigl\|(\R\mham)^{L+1}\bigl(\lphi(e_i\times\wzt)\bigr)\bigr\|_{L^2}^2 \\
    & +\|(\R\mham)^{L+1}(\z\R\wzt)\|_{L^2}^2
    +\|(\R\mham)^{L+1}(\vla\wzt)\|_{L^2}^2
    \les b_1^2.
  \end{align*}
  Moreover, by the decomposition
  \[
    (\hj\mham)^{L+1} = (\R\mham)^{L+1} +\Omega^{L+1}_+,
  \]
  the $\Omega^{L+1}_+$ part involves $\hw$, which contributes an additional $b_1^2$ by Lemma~\ref{le:interp-wp}. Together with \eqref{pf:mod-eq=d}, these yield \eqref{pf:mix-eng=mf-claim-2}.
  Finally, to show~\eqref{pf:mix-eng=mf-claim-3}, we divide the integration domain into $y\le 1$ and $y\ge1$, and then apply Corollary~\itref{co:dore}{it:dore=wzt} and Lemma~\itref{le:preserv}{it:preserv=mult}--\ref{it:preserv=op} for $\wzt$ in the region $y\le 1$. For $y\ge1$, note that the cancellation $\ham^{L+1}T_{k} =0$ is no longer applicable, and thus we perform a crude computation, which produces no extra infinitesimal $\dbs$ or $b_1^2$. Hence, the resulting bound is merely due to the smallness of $D(s)$, as desired.

  \step[8, Contribution of $\vtha_s\z\R\w$]
  By Corollary~\itref{co:dore}{it:dore=w} and Lemma~\itref{le:preserv}{it:preserv=mult}--\ref{it:preserv=op}, we have $\mhamp^L(\z\R\w)\in\mcD_1$ with factor $b_1^{L+1}\lgb^{-1}$ for $y\le1$. Estimating its $1^{\rm st}$ derivative and quotient by $y$ and using~\eqref{le:mod-eq-1}, we obtain
  \begin{align*}
    \|\ma\mhamp^L(\vtha_s\z\R\w)\|_{L^2(y\le 1)}^2
    & \les C(M,K)(a_1^2+D(s)^2) b_1^{2L+2}\lgb^{-2} \\
    & \les C(M,K)\frac{b_1^{2L+4}}{\lgb^2}.
  \end{align*}
  For $y\ge1$, by Lemma~\itref{le:interp-wp}{it:interp-wp=gen-lossy},
  \begin{align*}
    & \|\ma\mhamp^L(\vtha_s\z\R\w)\|_{L^2(y\ge1)}^2 \\
    & \les (a_1^2+D(s)^2)
    \sum_{\xi_1+\xi_2 = 2L+1}\int_{y\ge1}\frac{|\pry^{\xi_1}\wp|^2}{y^{2\xi_2}} \\
    & \les \dbs\frac{b_1^{2L+2}}{\lgb^2} + (a_1^2+D(s)^2)
    \int_{y\ge1}|\pry^{2L+1}\wp|^2.
  \end{align*}
  For the last part, we use the standard interpolation together with~\eqref{pf:w2l2-e=even-expansion}, to recover the $\ham^L$, $\ham^{L+1}$ from even-order derivatives, which gives
  \begin{equation}
    \label{pf:mix-eng=odd-interpolation}
    \int_{y\ge1}|\pry^{2L+1}\wp|^2
    \les \bigl(b_1^{2L}\lgb^C\bigr)^{1/2}
    \bigl(b_1^{2L+2}\lgb^C\bigr)^{1/2}
    =b_1^{2L+1}\lgb^C.
  \end{equation}
  The weighted terms and cutoff errors are controlled by Lemma~\itref{le:interp-wp}{it:interp-wp=taylor-orig} and \ref{it:interp-wp=gen-lossy}. Together with the smallness $|a_1|^2 \les b_1^2\dbs$, this does not perturb the previous bound. These yield \eqref{pf:mix-eng=mf-claim-1}.

  The claim~\eqref{pf:mix-eng=mf-claim-3} is proved in the same way. It remains to show~\eqref{pf:mix-eng=mf-claim-2}. For $y\le1$, by Corollary~\itref{co:dore}{it:dore=w}, \ref{it:dore=hw} and the argument in the proof of Lemma~\ref{le:w2l2-e},
  \begin{align*}
    \|(\hj\mham)^{L+1}(\vtha_s\z\R\w)\|_{L^2(y\le1)}^2
    & \les C(M,K)(a_1^2+D(s)^2)\Eplus \\
    & \les b_1^2\dbs\Eplus.
  \end{align*}
  For $y\ge1$, we estimate as in~\eqref{pf:w2l2-e=omega-l1-w}:
  \begin{align*}
    & \|(\hj\mham)^{L+1}(\vtha_s \z\R\w)\|_{L^2(y\ge1)}^2 \\
    & \les (a_1^2+D(s)^2) \hs[-.4]\sum_{n,\xi}\prod_{\ell=1}^{n}
    \|\pry^{\xi_{\ell}}\hw\|_{L^\infty(y\ge1)}^2
    \int_{y\ge1} \frac{|\pry^{\xi_{n+1}}\w|^2}{y^{2\xi_{n+2}}} \\
    & \les b_1^2 \dbs \biggl(\E_{2L+2} +\frac{b_1^{2L+2}}{\lgb^2}\biggr),
  \end{align*}
  where the sum is taken over
  \[
    0\le n\le L+1,\quad \xi\in\mbN^{n+2},\quad |\xi|_1=2L+2,\quad |\xi|_+\ge2.
  \]
  Since $(\R\mham)^{L+1} =\Omega_{(0,\dots,0)}^{L+1}$ is merely a special case of general $\Omega_{\rho}^{L+1}$, this estimate also works for $(\R\mham)^{L+1}$. This proves~\eqref{pf:mix-eng=mf-claim-2} for $\vtha_s\z\R\w$.

  \step[9, Contribution of $\w\times(\mham\wzt+\hp\lphi)$]
  Recall the definition in~\eqref{eq:wz-form} of the approximate solution $\wz$ and Lemma~\ref{le:local-app-sol} for the localized profile. We decompose $\wzt$ as
  \begin{equation}
    \label{pf:mix-eng=wzto-wztt}
    \wzt = \wzto + \wztt
    \where
    \wz^1=a_1\phia[1]+b_1\phib[1], \quad
    \wzto = \chi_{S_1}\wz^1.
  \end{equation}
  From the definition of $\wzt$, the leading part of $\mham\wzt+\hp\lphi$ is
  \begin{align*}
    \mham\wzto +\hp\lphi
    & = O(D(s))\lphi e_1+ O(D(s))\lphi e_2 + a_1\zt\Bigl(\pry+\frac{\z}{y}\Bigr)(\chi_{S_1}T_1) e_3 \\
    & \quad +(\chi_{S_1}-1)\mhamp\wz^1 -2\pry\chi_{S_1}\pry\wz^1 -\Delta\chi_{S_1}\wz^1,
  \end{align*}
  where the second line vanishes for $y\le1$, due to the cutoff. The first line belongs to $\mcD_{L+1}$ with factor $D(s)+|a_1|$. By Corollary~\itref{co:dore}{it:dore=w} and the proof of Lemma~\itref{le:preserv}{it:preserv=cross}, its cross product with $\w$ belongs to $\mcD_{L+1}$ with factor $(D(s)+|a_1|)b_1^{L+1}\lgb^{-1}$. By Lemma~\itref{le:preserv}{it:preserv=op}, after applying $\ma\mhamp^L$, their squared $L^2$ norms on $y\le1$ are bounded by
  \[
    C(M,K)(D(s)^2+a_1^2)\frac{b_1^{2L+2}}{\lgb^2}
    \les \dbs\frac{b_1^{2L+4}}{\lgb^2}.
  \]
  For $y\ge1$, by Lemma~\itref{le:interp-wp}{it:interp-wp=gen-lossy},
  \begin{align*}
    & \left\| \ma\mhamp^L\Bigl[ \w\times\Bigl( O(D(s))\lphi e_1+ O(D(s))\lphi e_2 \Bigr)\Bigr] \right\|_{L^2(y\ge 1)}^2 \\
    & \qquad \les C(M)D(s)^2b_1^{2L+2}\lgb^C
    \les b_1^2\dbs \Eplus.
  \end{align*}
  Moreover, since
  \[
    \left|\pry^{\ell}\biggl[\zt\biggl(\pry+\frac{\z}{y}\biggr)(\chi_{S_1}T_1)\biggr]\right| \les
    \frac{\jbk[\log y]}{y^{2+\ell}},
    \quad 1\le y\le 2S_1,
  \]
  we have by Lemma~\itref{le:interp-wp}{it:interp-wp=sobolev} that
  \[
    \left\| \ma\mhamp^L\biggl[\w\times\biggl(a_1\zt\Bigl(\pry+\frac{\z}{y}\Bigr)(\chi_{S_1}T_1) e_3\biggr)\biggr]\right\|_{L^2(y\ge 1)}^2
    \les b_1^2\dbs \Eplus.
  \]
  For other terms, a straightforward estimate using the properties of $\mcHg$ yields the same bounds, which gives \eqref{pf:mix-eng=mf-claim-1}.

  The estimate~\eqref{pf:mix-eng=mf-claim-3} is obtained in the same fashion. To show~\eqref{pf:mix-eng=mf-claim-2}, we first consider $y\le1$. By the profile construction and the argument in the proof of Lemma~\ref{le:w2l2-e}, we keep the derivatives of order $2L+2$ of $\w$ in $L^2$ and use the Dore bounds for the remaining terms. Thus the squared $L^2$ norm on $y\le 1$ is bounded by
  \[
    C(M,K)(D(s)^2+a_1^2+b_1^4)\Eplus
    \les b_1^2\dbs\Eplus.
  \]
  For $y\ge1$, by Lemma~\ref{le:omega-pw-bnd},
  \begin{align*}
    & \bigl|(\hj\mham)^{L+1} \bigl(\w\times \bigl(\mham\wzt+\hp\lphi \bigr)\bigr)\bigr| \\*
    & \les \sum_{n,\xi}
    \biggl(\prod_{\ell=1}^{n-1}|\pry^{\xi_{\ell}}\hw|\biggr) |\pry^{\xi_n}\w|
    \frac{\bigl| \pry^{\xi_{n+1}} \bigl( \mham\wzt+\hp\lphi \bigr) \bigr|}{y^{\xi_{n+2}}},
  \end{align*}
  where the sum is taken over
  \[
    1\le n\le L+2,\quad \xi\in\mbN^{n+2},\quad |\xi|_1=2L+2.
  \]
  From the property of $\mcHg$, the numerator on the right-hand side admits
  \begin{equation}
    \begin{split}
      \label{pf:mix-eng=wzt-hp-lphi}
      \bigl|\pry^\ell \bigl( \mham\wzt+\hp\lphi \bigr)\bigr|
      & \les \frac{b_1 \jbk[\log y]^C}{y^{1+\ell}}\one_{y\gtr S_1}
      + \frac{|a_1|\jbk[\log y]^C}{y^{2+\ell}}\one_{y\les S_1} \\
      & \quad +\sum_{k=2}^{C_L} b_1^{k} \jbk[y]^{2k-3-\ell} \jbk[\log y]^C \one_{y\les S_1}.
    \end{split}
  \end{equation}
  We split the estimate into three parts. We decompose each $\pry^{\xi_{\ell}}\hw$ in the above product by $\hw=\wzt+\w$. Each resulting monomial contains $m$ copies of the $\wzt$ factor for some $0\le m\le n-1$. Choosing $\xi_n\ge\xi_\ell$ for $m+1\le\ell\le n-1$, we have $\xi_1+\dots+\xi_m+|\xi|_++\xi_n\ge2$. We use Lemma~\itref{le:interp-wp}{it:interp-wp=pw}, \ref{it:interp-wp=gen-lossy} and~\eqref{pf:w2l2-e=wzt-l-infty}. Note that
  \[
    \Bigl\|\frac{\jbk[\log y]^{C+2}}{y^2}\Bigr\|_{L^{\infty}(y\gtr S_1)}\les b_1\lgb^C,
  \]
  and the first part is bounded by
  \begin{align*}
    & b_1^{m+3}\lgb^C \prod_{\ell=m+1}^{n-1} \|\pry^{\xi_{\ell}}\w\|_{L^\infty(y\ge1)}^2
    \int_{y\ge 1} \frac{|\pry^{\xi_{n}}\w|^2}{y^{2(\xi_1+\dots+\xi_{m}+|\xi|_{+})}\jbk[\log y]^2} \\
    & \les b_1^{m+3} b_1^{\xi_{m+1}+\dots+\xi_{n-1}} b_1^{(\xi_1+\dots+\xi_{m}+\xi_{n}+|\xi|_{+}-1)+\frac{1}{2L-1}} \lgb^C \\
    & \les b_1^2 \dbs \frac{b_1^{2L+2}}{\lgb^2}.
  \end{align*}
  Similarly, for the second part, if $\xi_{m+1}+\dots+\xi_{n-1}\ge1$, then $\xi_1+\dots+\xi_{m}+\xi_n+|\xi|_{+}+1\le 2L+2$, so
  \begin{align*}
    & b_1^{m}\lgb^C
    \prod_{\ell=m+1}^{n-1} \|\pry^{\xi_{\ell}}\w\|_{L^\infty(y\ge1)}^2
    a_1^2 \int_{1\le y\les S_1} \frac{\jbk[\log y]^C\, |\pry^{\xi_n}\w|^2}{y^{2(2+\xi_1+\dots+\xi_{m}+|\xi|_{+})}}        \\
    & \les b_1^{m} b_1^{\xi_{m+1}+\dots+\xi_{n-1}} a_1^2 b_1^{\xi_1+\dots+\xi_{m}+\xi_{n}+|\xi|_{+}+1} \lgb^C \\
    & \les b_1^2 \dbs \frac{b_1^{2L+2}}{\lgb^2}.
  \end{align*}
  If $\xi_{m+1}+\dots+\xi_{n-1}=0$, then the integral part is controlled by $C(M)(\E_{2L+2}+b_1^{2L+2}\lgb^{-2})$ from Lemma~\itref{le:interp-wp}{it:interp-wp=sobolev} and Lemma~\itref{le:interp-gamma}{it:interp-gamma=sobolev}, which yields the bound
  \[
    a_1^2 C(M)\Eplus \les b_1^2 \dbs \frac{b_1^{2L+2}}{\lgb^2}.
  \]
  For the last part, note that
  \[
    \|b_1^k y^{2k-3}\|_{L^{\infty}(y\les S_1)} \les b_1^{\frac32}\lgb^C \for k\ge 2,
  \]
  then again from Lemma~\itref{le:interp-wp}{it:interp-wp=pw}, \ref{it:interp-wp=gen-lossy},
  \begin{align*}
    & b_1^{m+3}\lgb^C
    \prod_{\ell=m+1}^{n-1} \|\pry^{\xi_{\ell}}\w\|_{L^\infty(y\ge1)}^2
    \int_{1\le y\les S_1} \frac{\jbk[\log y]^C\, |\pry^{\xi_{n}}\w|^2}{y^{2(\xi_1+\dots+\xi_{m}+|\xi|_{+})}}        \\
    & \les b_1^{m+3} b_1^{\xi_{m+1}+\dots+\xi_{n-1}} b_1^{\xi_1+\dots+\xi_{m}+\xi_{n}+|\xi|_{+}-1+\frac{1}{2L-1}} \lgb^C \\
    & \les b_1^2 \dbs \frac{b_1^{2L+2}}{\lgb^2},
  \end{align*}
  which yields
  \[
    \|(\hj\mham)^{L+1}\bigl(\w\times(\mham\wzt+\hp\lphi)\bigr)\|_{L^2}^2
    \les b_1^2\dbs \Eplus.
  \]
  This gives~\eqref{pf:mix-eng=mf-claim-2}, and hence completes the proof.
\end{proof}

\section{Completion of the construction}
\label{sec:compl-constr}
\subsection{Closing the bootstrap}
\label{sub:clos-boot}
In this subsection, we prove Proposition~\ref{pro:boot}. This establishes the existence of $\w\in C([s_0,+\infty);\dot{H}^1)$.

\begin{proof}[Proof of Proposition~\ref{pro:boot}]
  We argue by contradiction, as in \citep{merle_blowup_2013,raphael_stable_2013,collot_type_2018,ghoul_construction_2018,xu_blowup_2025}. Recall from the last paragraph in Subsection~\ref{sub:mod-boot} that Proposition~\ref{pro:boot} is equivalent to the claim:
  \[
    \exists \bigl(U'_2, \dots, U'_L, V'_1, \dots, V'_L\bigr)(0) \in \bar{B}(0,1)
    \quad\mbox{such that}\quad
    s^{\ast}  = +\infty.
  \]
  Here, $\bar{B}(0,1)$ denotes the closed unit ball in $\mbR^{2L-1}$, and $s^{\ast}$, defined by \eqref{eq:exit-time}, is the maximal time up to which~\eqref{pro:boot-e1-bnd}--\eqref{pro:boot-u2-bnd} remain valid. We now assume for contradiction that
  \begin{equation}
    \label{pf:pro-boot=contra-ass}
    \forall \bigl(U'_2, \dots, U'_L, V'_1, \dots, V'_L\bigr)(0) \in \bar{B}(0,1),
    \quad
    s^{\ast}<+\infty.
  \end{equation}
  In what follows, we show that~\eqref{pro:boot-e1-bnd}--\eqref{pro:boot-u1-bnd} can be improved on $[s_0,s^{\ast})$, and thus the finiteness of $s^{\ast}$ is due to the saturation of \eqref{pro:boot-u2-bnd}. We will see that this leads to a contradiction.

  \step[1, Refined bound for $\E_1$]
  Recall from \eqref{eq:sol-form} that the solution is of the form
  \[
    u = \mcS z
    \with z = Q+\hv = (e_r, e_{\tau}, Q) ( e_3  +\hw), \quad |z|=1,
  \]
  with the Dirichlet energy
  \begin{align*}
    E(u)
    & = \frac12 \int |\nabla z|^2 \dd x
    = -\pi\int z\cdot\left(\Delta z+|\nabla Q|^2 z\right) y\dd y +\pi\int|\nabla Q|^2 y\dd y \\
    & = E(Q) +\pi(\hw, \mham\hw).
  \end{align*}
  From the explicit construction of $\wzt$ in Lemma~\ref{le:local-app-sol}, the quadratic term satisfies
  \begin{align*}
    (\hw, \mham\hw)
    & = (\wzt, \mham\wzt) +(\wzt, \mham\w) +(\w, \mham\wzt) +(\w, \mham\w) \\
    & = (\w, \mham\w) +O(\sqrt{b_1}).
  \end{align*}
  Thus from the conservation of energy and the asymptotics of $b_1$:
  \begin{equation}
    \label{pf:pro-boot=b1-fin}
    b_1 = b_1^e + \frac{U_1}{s(\log s)^{5/4}} = \frac{c_1}{s} + O\biggl(\frac{1}{s\log s}\biggr),
  \end{equation}
  we have
  \[
    \bigl|(\w, \mham\w)(s) - (\w, \mham\w)(s_0)\bigr| \les b_1^{1/2}(s_0).
  \]
  From the definition of $\mham,\mhamp$, we compute using integration by parts:
  \begin{align*}
    (\w, \mham\w) - (\w, \mhamp\w)
    & = (\gamma, -\Delta\gamma) + \int \zt\biggl( -\alpha \pry\gamma + \gamma \biggl(\pry+\frac{\z}{y}\biggr) \alpha\biggr) \\
    & = \|\pry\gamma\|_{L^2}^2 - 2\int\zt\alpha\pry \gamma.
  \end{align*}
  By Lemma~\itref{le:interp-gamma}{it:interp-gamma=h1} and Lemma~\itref{le:interp-wp}{it:interp-wp=gen-lossy}, the right-hand side is bounded by \(\dbs\E_1 +b_1^{1/2}\). Thus by the bootstrap bound~\eqref{pro:boot-e1-bnd}, we have
  \begin{equation}
    \label{pf:pro-boot=w-mhamp-w-prod}
    \bigl|(\w, \mhamp\w)(s) - (\w, \mhamp\w)(s_0)\bigr| \les b_1^{1/2}(s_0) + \dbs b_1^{1/4}(s_0).
  \end{equation}
  This implies a refined bootstrap bound for $\E_1$.
  Indeed, on the one hand, a straightforward computation yields
  \[
    (\w, \mhamp\w)
    \les \|\nabla\wp\|_{L^2}^2 +\bigg\|\frac{\wp}{y}\bigg\|_{L^2}^2
    \les \E_1.
  \]
  On the other hand, Lemma~\ref{le:coer-ham} and~\eqref{eq:ortho} provide the reverse:
  \[
    \E_1 \les C(M)\bigl[ (\alpha,\ham\alpha) + (\beta,\ham\beta) \bigr]
    = C(M) (\w, \mhamp\w).
  \]
  Hence there exists $C=C(M)>1$ such that
  \[
    \frac1C (\w, \mhamp\w)
    \le \E_1 \le C(\w, \mhamp\w).
  \]
  This, together with \eqref{pf:pro-boot=w-mhamp-w-prod} and the initial bound~\eqref{ini:e1-bnd}, implies for sufficiently large $s_0$ the refined bound:
  \begin{align*}
    \E_1(s)
    & \le C\bigl|(\w, \mhamp\w)(s) - (\w, \mhamp\w)(s_0)\bigr| +C^2 \E_1(s_0) \\
    & \le C^2 b_1^{1/2}(s_0) +(1+o(1))b_1^{1/4}(s_0)
    \le 2 b_1^{1/4}(s_0).
  \end{align*}

  \step[2, The scaling law]
  We now compute the scaling parameter $\la(s)$.
  From the modulation equation~\eqref{le:mod-eq-1}, the definitions of $c_1, d_1$ in \eqref{eq:c1} and of $U_1$ in \eqref{eq:ab-abe}, together with the bootstrap bound~\eqref{pro:boot-u1-bnd} for $U_1$, we find
  \[
    \frac{\la_s}{\la} = -b_1 +O\biggl(\frac{b_1^{L+1}}{\lgb}\biggr)
    = -\frac{c_1}{s} +\frac{|d_1|}{s\log s} +O\biggl(\frac{1}{s(\log s)^{5/4}}\biggr).
  \]
  By a straightforward computation, we obtain
  \[
    \left|\frac{\dd}{\dd s}\log \biggl(\frac{\la s^{c_1}}{(\log s)^{|d_1|}}\biggr)\right|
    \les \frac{1}{s (\log s)^{5/4}}.
  \]
  Integrating this bound from $s$ to $s^{\ast}$ gives
  \begin{equation}
    \label{pf:pro-boot=la-s}
    \la(s) = \kappa_0\bigl(1+o(1)\bigr)\frac{(\log s)^{|d_1|}}{s^{c_1}}
    \with
    \kappa_0 = \kappa_0(s_0) \in\mbR_+.
  \end{equation}
  Here, $o(1)$ is indeed $O((\log s)^{-1/4})$, which tends to zero uniformly for $s_0$ sufficiently large. In view of \eqref{eq:ab-abe}, we have the following:
  \begin{equation}
    \label{pf:pro-boot=la-b-rel}
    \biggl(\frac{b_1^{c_1}\lgb^{|d_1|}}{\la}\biggr)(s)
    \simeq \biggl(\frac{b_1^{c_1}\lgb^{|d_1|}}{\la}\biggr)(s_0).
  \end{equation}

  \step[3, Refined bound for $\E_{2L+2}$]
  We integrate the mixed energy inequality in Proposition~\ref{pro:mix-eng} using the bootstrap bound~\eqref{pro:boot-e2l2-bnd} for $\E_{2L+2}$,
  \begin{align*}
    & \biggl\{\E_{2L+2} +O\biggl(\frac{b_1^{2L+2}}{\lgb^2}\biggr)\biggr\}(s)                                              \\
    & \le \biggl(\frac{\la(s)}{\la(s_0)}\biggr)^{4L+2} \biggl\{\E_{2L+2} +O\biggl(\frac{b_1^{2L+2}}{\lgb^2}\biggr)\biggr\}(s_0) \\
    & \quad
    +\biggl(c_0+\frac{C}{\sqrt{\log M}}+\frac{C(M)}{K}\biggr)K \la^{4L+2}(s) \int_{s_0}^s \frac{b_1^{2L+3}}{\la^{4L+2}\lgb^2} \dd\sigma.
  \end{align*}
  Let us estimate the right-hand side. For the initial data part, using $\la(s_0)=1$ by~\eqref{ini:fixed-scaling}, the initial bound~\eqref{ini:e2k-bnd}, the relation~\eqref{pf:pro-boot=la-b-rel} and the largeness of $s_0$, we compute:
  \begin{align*}
    & \biggl(\frac{\la(s)}{\la(s_0)}\biggr)^{4L+2} \biggl\{b_1^{10L} +O\biggl(\frac{b_1^{2L+2}}{\lgb^2}\biggr)\biggr\}(s_0)                                                                                     \\
    & \le C \Biggl\{ \frac{\bigl(b_1^{c_1}\lgb^{|d_1|}\bigr)(s)}{\bigl(b_1^{c_1}\lgb^{|d_1|}\bigr) (s_0)}\Biggr\}^{4L+2} \frac{b_1^{2L+2}}{\lgb^2}(s_0)                                                               \\
    & \le C \Biggl\{ \frac{\bigl(b_1^{c_1-\frac{L+1}{2L+1}}\lgb^{|d_1|+\frac{1}{2L+1}}\bigr)(s)}{\bigl(b_1^{c_1-\frac{L+1}{2L+1}}\lgb^{|d_1|+\frac{1}{2L+1}}\bigr)(s_0)}\Biggr\}^{4L+2}\,\frac{b_1^{2L+2}}{\lgb^2}(s) \\
    & \le C\frac{b_1^{2L+2}}{\lgb^2}(s),
  \end{align*}
  where we used the identity:
  \[
    (4L+2)c_1 - (2L+2) = \frac{2}{2L-1}>0
    \quad\implies\quad
    c_1 > \frac{L+1}{2L+1}.
  \]
  For the integral part, we use again the relation~\eqref{pf:pro-boot=la-b-rel}, $b_1\simeq c_1/s$ from~\eqref{pf:pro-boot=b1-fin}, and $c_1=\frac{L}{2L-1}$ from~\eqref{eq:c1}, together with the above algebra ($(4L+2)c_1 - (2L+3)>-1$, more precisely) and find
  \begin{align*}
    & \la^{4L+2}(s) \int_{s_0}^s \frac{b_1^{2L+3}}{\la^{4L+2}\lgb^2} \dd\sigma                                       \\
    & \le (1+o(1)) \frac{(\log s)^{(4L+2)|d_1|}}{s^{(4L+2)c_1}} \int_{s_0}^s \frac{c_1^{2L+3} \sigma^{(4L+2)c_1-(2L+3)}}{(\log\sigma)^{(4L+2)|d_1|+2}} \dd\sigma \\
    & \le (1+o(1)) \frac{c_1^{2L+3}}{(4L+2)c_1-(2L+2)} \frac{1}{s^{2L+2}(\log s)^{2}}                                                                          \\
    & \le (1+o(1)) \frac{L}{2} \frac{b_1^{2L+2}}{\lgb^2}(s).
  \end{align*}
  Taking the coefficient into account, we choose \(c_0\) (given by Lemma~\ref{le:vr-est}) sufficiently small, then choose \(M\) sufficiently large first and \(K\) sufficiently large afterwards, so that
  \[
    \biggl(c_0+\frac{C}{\sqrt{\log M}}+\frac{C(M)}{K}\biggr)\frac{L}{2} \ll \frac14.
  \]
  Inserting these estimates back, we obtain the refined bootstrap bound for $\E_{2L+2}$:
  \[
    \E_{2L+2} \le \biggl(C+\frac{K}{4}\biggr) \frac{b_1^{2L+2}}{\lgb^2}
    \le \frac{K}{2} \frac{b_1^{2L+2}}{\lgb^2}.
  \]

  \step[4, Refined bound for $\E_{2k+2}$]
  We shall use the following energy estimate for the intermediate energies $\E_{2k+2}$, $0\le k\le L-1$, which is proved in Appendix~\ref{ap:proof-mix-low-eng}.

  \begin{lemma}[Mixed energy estimate for \texorpdfstring{$\E_{2k+2}$}{E_{2k+2}}]
    \label{le:mix-low-eng}
    For any $0\le k\le L-1$, there exists a constant $C > 1$ independent of $K, M$, such that
    \[
      \frac{\dd}{\dd t}\biggl\{\frac{1}{\la^{4k+2}}\biggl[\E_{2k+2} + O\Bigl(b_1^{\frac12+(2k+1)\frac{2L}{2L-1}}\Bigr)\biggr]\biggr\} \\
      \le \frac{C(M,K)}{\la^{4k+4}} b_1^{1+(2k+1)\frac{2L}{2L-1}}\lgb^{K/2+C}.
    \]
  \end{lemma}
  A direct integration yields
  \begin{align*}
    & \left\{\E_{2k+2} + O\Bigl(b_1^{\frac12+(2k+1)\frac{2L}{2L-1}}\Bigr) \right\}(s)                                              \\
    & \le \biggl(\frac{\la(s)}{\la(s_0)}\biggr)^{4k+2} \left\{\E_{2k+2} + O\Bigl(b_1^{\frac12+(2k+1)\frac{2L}{2L-1}}\Bigr) \right\}(s_0) \\
    & \quad +C(M,K)\la(s)^{4k+2} \int_{s_0}^s \frac{b_1^{1+(2k+1)\frac{2L}{2L-1}}\lgb^{K/2+C}}{\la^{4k+2}} \dd\sigma.
  \end{align*}
  Let us estimate the right-hand side. For the initial data part, we use the initial bound $\E_{2k+2}(s_0)\le b_1^{10L}(s_0)$ by~\eqref{ini:e2k-bnd}, the scaling relation~\eqref{pf:pro-boot=la-b-rel}, and choose $K>\max((4L-2)|d_1|+2,2C+4)$. This gives
  \begin{align*}
    & \biggl(\frac{\la(s)}{\la(s_0)}\biggr)^{4k+2} \left\{b_1^{10L} +O\Bigl(b_1^{\frac12+(2k+1)\frac{2L}{2L-1}}\Bigr)\right\}(s_0) \\
    & \le C\Biggl\{\frac{\bigl(b_1^{c_1}\lgb^{|d_1|}\bigr)(s)}{\bigl(b_1^{c_1}\lgb^{|d_1|}\bigr)(s_0)}\Biggr\}^{4k+2}
    b_1(s_0)^{\frac12+(2k+1)\frac{2L}{2L-1}} \\
    & \le C\biggl\{\frac{b_1^{\frac12}}{\lgb^{(4k+2)|d_1|}}\biggr\}(s_0)
    \biggl\{b_1^{(2k+1)\frac{2L}{2L-1}}\lgb^{(4k+2)|d_1|}\biggr\}(s) \\
    & \le \frac18 \biggl\{ b_1^{(2k+1)\frac{2L}{2L-1}} \lgb^{K} \biggr\}(s),
  \end{align*}
  where we used the identity
  \[
    (4k+2)c_1-(2k+1)\frac{2L}{2L-1}=0.
  \]
  For the integral part, using~\eqref{pf:pro-boot=la-b-rel}, $b_1\simeq c_1/s$, and the above identity,
  \begin{align*}
    & C(M,K)\la(s)^{4k+2} \int_{s_0}^s \la^{-(4k+2)} b_1^{1+(2k+1)\frac{2L}{2L-1}}\lgb^{K/2+C} \dd\sigma \\
    & \le C(M,K)\frac{(\log s)^{(4k+2)|d_1|}}{s^{(4k+2)c_1}} \int_{s_0}^s \frac{(\log\sigma)^{K/2+C-(4k+2)|d_1|}}{\sigma}\dd\sigma \\
    & \le C(M,K)b_1^{(2k+1)\frac{2L}{2L-1}}\Bigl(\lgb^{K/2+C+1}+\lgb^{(4k+2)|d_1|+1}\Bigr) \\
    & \le \frac18 \biggl\{b_1^{(2k+1)\frac{2L}{2L-1}}\lgb^{K}\biggr\}(s).
  \end{align*}
  Inserting these estimates back, we obtain the refined bootstrap bound for $\E_{2k+2}$:
  \[
    \E_{2k+2} \le \frac12 b_1^{(2k+1)\frac{2L}{2L-1}}\lgb^{K}.
  \]

  \step[5, Refined bound for the stable mode]
  Recall from Lemma~\ref{le:refined-mod-eq} that we defined the modified versions of the parameters $a_L, b_L$, namely $\bar{a}_L, \bar{b}_L$. We now let
  \[
    \bar{a} = (\bar{a}_k)_{k=1}^L, \quad \bar{b} = (\bar{b}_k)_{k=1}^L \in C^1([s_0,s^{\ast});\mbR^L),
  \]
  where $\bar{a}_L, \bar{b}_L$ are given by \eqref{eq:ti-b}, and
  \[
    \bar{a}_k = a_k, \quad
    \bar{b}_k = b_k, \quad 1\le k\le L-1.
  \]
  Linearizing $(\bar{a},\bar{b})$ near the approximate solution $(a^e, b^e)$ (see~\eqref{eq:ae-be}) to the ODE system~\eqref{eq:ode}, we define the vector-valued functions $\bar{U}, \bar{V}$ as
  \[
    \bar{b}_k = b^e_k +\frac{\bar{U}_k(s)}{s^k(\log s)^{5/4}}, \quad
    \bar{a}_k = a^e_k +\frac{\bar{V}_k(s)}{s^k(\log s)^{5/4}},
    \quad 1\le k\le L.
  \]
  We further let
  \[
    \bar{U}' = P_U \bar{U}, \quad \bar{V}' = P_V \bar{V},
  \]
  where $P_U, P_V\in \GL(L,\mbR)$ are given in Lemma~\itref{le:b-sys-lin}{it:b-sys-lin=diag} and \itref{le:a-sys-lin}{it:a-sys-lin=diag}. Now from Lemma~\itref{le:refined-mod-eq}{it:refined-mod-eq=diff} and $b_1\simeq c_1/s$ by \eqref{pf:pro-boot=b1-fin}, we have
  \begin{equation}
    \label{pf:mix-low-eng=diff}
    |\bar{U}'-U'| +|\bar{V}'-V'|
    \les s^{L}(\log s)^{\frac54} \left(|\bar{a}-a| + |\bar{b}-b|\right) \les s^{-\frac14}.
  \end{equation}
  Thus choosing $s_0$ sufficiently large and using the initial bound~\eqref{ini:modes-bnd}, we have
  \[
    |\bar{U}'_1|(s_0)
    \le |U'_1|(s_0) + |\bar{U}'-U'|(s_0)
    \le \frac14 + \frac14
    \le \frac12.
  \]
  On the other hand, we apply Lemma~\ref{le:b-sys-lin} for $\bar{b}$ and $\bar{U}$, and use Lemma~\ref{le:mod-eq} and Lemma~\itref{le:refined-mod-eq}{it:refined-mod-eq=eq}, which give the equation for $\bar{U}$:
  \begin{align*}
    & \left|s\prs\bar{U}_k - (M_U \bar{U})_k +O\biggl(\frac{1}{(\log s)^{1/4}}\biggr)\right|                                                                                                                                                                                          \\
    & \les s^{k+1}(\log s)^{5/4}\biggl\{ \bigl|\prs\bar{b}_k \mhso+\mhso (2k \mhso-\mhso 1\mhso+\mhso\cb)\bar{b}_1\bar{b}_k \mhso+\mhso \bar{a}_1\bar{a}_k \mhso-\mhso \bar{b}_{k+1}\bigr| +\biggl|\bar{b}_1\bar{b}_k \biggl( \cb \mhso-\mhso \frac{2}{\log s} \biggr)\biggr|\biggr\} \\
    & \les s^{k+1}(\log s)^{5/4} \biggl\{ b_1^{\frac{1}{4L-2}}b_1^{L+1} +(C+K)\frac{b_1^{L+1}}{\lgb^{3/2}} +\frac{b_1^{k+1}}{(\log s)^2} \biggr\}                                                                                                                                     \\
    & \les \frac{1}{(\log s)^{1/4}}.
  \end{align*}
  By the linear transformation $P_U$, this yields the equation for $\bar{U}'$; and also the equation for $\bar{V}'$, by a similar argument:
  \begin{equation}
    \label{pf:mix-low-eng=eq}
    s\prs\bar{U}' = D_U\bar{U}' +O\biggl(\frac{1}{(\log s)^{1/4}}\biggr),
    \quad
    s\prs\bar{V}' = D_V\bar{V}' +O\biggl(\frac{1}{(\log s)^{1/4}}\biggr),
  \end{equation}
  where $D_U, D_V$ are the diagonal matrices given in Lemma~\ref{le:b-sys-lin} and~\ref{le:a-sys-lin}. Note that the $(1,1)$ entry of $D_U$ is $-1$, and thus by a direct computation on \eqref{pf:mix-low-eng=eq}, we have the dynamics for the scalar $\bar{U}'_1$:
  \[
    \left|\prs(s\bar{U}'_1)\right| \les \frac{1}{(\log s)^{1/4}}.
  \]
  Since $s_0> 1$ is chosen sufficiently large,
  \[
    |\bar{U}'_1|(s) \le \frac{s_0}{s} |\bar{U}'_1|(s_0) +\frac{C}{s}\int_{s_0}^s \frac{\dd\sigma}{(\log\sigma)^{1/4}}
    \le \frac12 +\frac14
    =\frac34.
  \]
  Together with~\eqref{pf:mix-low-eng=diff}, this implies the refined bootstrap bound for $U'_1$:
  \[
    |U'_1(s)| < \frac78.
  \]

  \step[6, Reduction to a contradiction]
  Denote the associated unstable modes corresponding to $(\bar{a}, \bar{b})$ by
  \[
    X = \left(\bar{U}'_2,\dots,\bar{U}'_L,\bar{V}'_1,\dots,\bar{V}'_L\right) \in \mbR^{2L-1}.
  \]
  We take the initial data for $X$ to be
  \[
    X(s_0) \in \bar{B}(0,\tfrac12).
  \]
  Fixing $s_0, U'_1(s_0), \wp(s_0)$, we compute from \eqref{eq:ti-b} that
  \[
    |\pr_{b_1}(\bar a_L-a_L)|+|\pr_{b_1}(\bar b_L-b_L)| \les C(M,\eta)b_1^{-1-2\eta}\sqrt{\E_{2L+2}} \ll 1.
  \]
  Together with~\eqref{ini:e2k-bnd} and~\eqref{pf:mix-low-eng=diff}, this implies the map
  \[
    (U'_2,\dots,U'_L,V'_1,\dots,V'_L)\mapsto X
  \]
  is a small $C^1$ perturbation of the identity for large $s_0$, so each initial datum $X(s_0)$ corresponds to a unique choice of $(U'_2,\dots,U'_L,V'_1,\dots,V'_L)$ in Proposition~\ref{pro:boot}.
  Let the bootstrap bound for $X$ be
  \begin{equation}
    \label{pf:mix-low-eng=x-boot-bnd}
    X(s) \in \bar{B}(0,\tfrac12),
  \end{equation}
  and the associated exit time be
  \[
    \bar{s}^{\ast}
    = \sup\left\{s\ge s_0: \mbox{\eqref{pro:boot-e1-bnd}--\eqref{pro:boot-u1-bnd} and \eqref{pf:mix-low-eng=x-boot-bnd} hold on $[s_0, s]$}\right\}.
  \]
  Note that, from \eqref{pf:mix-low-eng=diff}, the bound~\eqref{pf:mix-low-eng=x-boot-bnd} implies the original bootstrap bound~\eqref{pro:boot-u2-bnd}. Hence $\bar{s}^{\ast} \le s^{\ast}$. Together with~\eqref{pf:pro-boot=contra-ass}, this implies
  \[
    \forall X(s_0) \in \bar{B}(0,\tfrac12),
    \quad \bar{s}^{\ast} <+\infty.
  \]
  Since \eqref{pro:boot-e1-bnd}--\eqref{pro:boot-u1-bnd} are improved in the previous steps, the finiteness of $\bar{s}^{\ast}$ indicates the saturation of \eqref{pf:mix-low-eng=x-boot-bnd}; in other words, we have
  \[
    |X|(\bar{s}^{\ast}) = \frac{1}{2}.
  \]
  Using this and~\eqref{pf:mix-low-eng=eq}, we obtain a strictly outward flow on the boundary $\pr\bar{B}(0,\tfrac12)$:
  \begin{align*}
    \frac{\dd |X|^2}{\dd s}(\bar{s}^{\ast})
    & = \frac{2}{\bar{s}^{\ast}} \left\{ \sum_{k=2}^{L} \frac{k}{2L-1}\bigl|\bar{U}'_k\bigr|^2(\bar{s}^{\ast}) + \sum_{k=1}^{L} \frac{k}{2L-1}\bigl|\bar{V}'_k\bigr|^2(\bar{s}^{\ast}) +O\biggl(\frac{1}{(\log \bar{s}^{\ast})^{1/4}}\biggr) \right\} \\
    & \ge \frac{2}{\bar{s}^{\ast}} \left\{\frac{1}{4L}\biggl(\frac{1}{2}\biggr)^2 +O\biggl(\frac{1}{(\log \bar{s}^{\ast})^{1/4}}\biggr)\right\}  >0                                                                                            .
  \end{align*}
  This, together with the continuous dependence of \eqref{pf:mix-low-eng=eq} on the initial data, yields a continuous retraction
  \[
    \bar{B}(0,\tfrac12) \to \pr\bar{B}(0,\tfrac12), \quad
    X(s_0) \mapsto X(\bar{s}^{\ast}),
  \]
  which, restricted to the boundary $\pr\bar{B}(0,\tfrac12)$, is the identity map. However, this contradicts the Brouwer fixed point theorem. Consequently, the claim~\eqref{pf:pro-boot=contra-ass} does not hold, and thus there exists an initial data $(U'_2, \dots, U'_L, V'_1, \dots, V'_L)(0)$ such that the corresponding $s^{\ast} = +\infty$. This concludes the proof of Proposition~\ref{pro:boot}.
\end{proof}

\begin{remark}[Choices of $M,K,s_0$]
  \label{re:k-m-choi}
  Note that in the previous proof and other contexts of Section~\ref{sec:energy-est}, we have several times used the largeness of the constants $M, K, s_0 > 1$. Although their values are not explicitly given, we point out that there is no circular dependence in the choice of these constants. In fact, the constants are chosen in the following order:
  \begin{enumerate}[label={\rm (\roman*)}]
    \item we fix the universal constant $L\in\mbN_+$ and also the explicit ($C>0,c_0>0$) or implicit ($\les$) constants that are independent of $M,K$;
    \item we choose $M$ sufficiently large first and then $K$ sufficiently large depending on $M$;
    \item we choose $s_0$ sufficiently large, and hence $\bs \ge \sup_{[s_0,s^{\ast})}|b_1(s)|=b_1(s_0)$ sufficiently small, so that the corresponding small constant $\delta(\bs)$ absorbs the constants in Steps~(i) and~(ii).
  \end{enumerate}
\end{remark}

\subsection{Refined bound for \texorpdfstring{$\E_2$}{E2}}
In this subsection, we derive a refined bound for $\E_2$, which is used in the proof of Theorem~\ref{th:main-thm}.

\begin{lemma}[Refined bound for $\E_2$]
  \label{pro:refined-e2}
  For $s\in[s_0,s^{\ast})$, we have
  \[
    \E_2\les C(u_0)\la^2.
  \]
\end{lemma}

\begin{proof}
  \step[1, Estimates for $\w_2^1,\w_4^1$]
  By~\eqref{def:w2k1}, \eqref{eq:mham} and the definition of $\Omega_+^k$, we consider the remainders in $\w_2,\w_4$:
  \begin{equation}
    \label{pf:refined-e2=w2w4-expand}
    \begin{aligned}
      \w_2^1
      &=\R\mham(\gamma e_3)+\Omega_+^1[\hw]\w
      =-\zt\pry\gamma e_2+\hw\times\mham\w,\\
      \w_4^1
      &=(\R\mham)^2(\gamma e_3)+\Omega_+^2[\hw]\w
      =\ham(\zt\pry\gamma)e_1+\Omega_+^2[\hw]\w.
    \end{aligned}
  \end{equation}
  For $y\le1$ and $L\ge2$, from Corollary~\itref{co:dore}{it:dore=w}, \ref{it:dore=hw}, Lemma~\itref{le:preserv}{it:preserv=op}--\ref{it:preserv=cross}, and the bound for $\gamma e_3$ used in Lemma~\ref{le:w2l2-e}, we have $\w_2^1\in\mcD_L$ and $\w_4^1\in\mcD_{L-1}$ with factor $b_1^{L+2}\lgb^{-1}$. Hence we obtain
  \[
    \|\w_2^1\|_{L^2(y\le1)}^2+\|\w_4^1\|_{L^2(y\le1)}^2
    \les \frac{b_1^{2L+4}}{\lgb^2}.
  \]
  For $L=1$, the $4^{\rm th}$ derivative is controlled by $\E_4$ and Lemma~\ref{le:interp-gamma}, as in the proof of Lemma~\ref{le:w2l2-e}.
  For $y\ge1$, by Lemma~\itref{le:interp-wp}{it:interp-wp=gen-lossy}, and also Lemma~\ref{le:interp-gamma} when $L=1$, we have
  \begin{align*}
    \|\R\mham(\gamma e_3)\|_{L^2(y\ge1)}^2
    &\les\int_{y\ge1}\frac{|\pry\gamma|^2}{y^4}
    \les\begin{cases}
      b_1^3\lgb^C,&L=1,\\
      b_1^{4c_1}\lgb^C,&L\ge2,
    \end{cases}\\
    \|(\R\mham)^2(\gamma e_3)\|_{L^2(y\ge1)}^2
    &\les\sum_{i=1}^3\int_{y\ge1}\frac{|\pry^i\gamma|^2}{y^{2(5-i)}}
    \les\begin{cases}
      \dbs\E_4+b_1^{4+\frac14}\lgb^{-2},&L=1,\\
      b_1^{4+\frac{1}{4L}}\lgb^C,&L\ge2.
    \end{cases}
  \end{align*}
  For the $\Omega_+^k$ terms, from Lemma~\ref{le:omega-pw-bnd} and~\eqref{pf:w2l2-e=wzt-l-infty}, we have
  \begin{align*}
    |\Omega_+^1[\hw]\w|
    &\les|\hw|\sum_{i=0}^2\frac{|\pry^i\w|}{y^{2-i}},\\
    |\Omega_+^2[\hw]\w|
    &\les\sum_{n,m,\xi}b_1^{\frac{m}{2}}\lgb^C
    \prod_{\ell=m+1}^{n}|\pry^{\xi_\ell}\w|
    \frac{|\pry^{\xi_{n+1}}\w|}
    {y^{\xi_1+\dots+\xi_{m}+\xi_{n+2}}\jbk[\log y]^{2m}},
  \end{align*}
  where the sum is taken over
  \[
    1\le n\le2,\quad 0\le m\le n,\quad \xi\in\mbN^{n+2}, \quad
    |\xi|_1=4, \quad |\xi|_+\ge2.
  \]
  The terms containing $\wzt$ are estimated by Lemma~\itref{le:interp-wp}{it:interp-wp=gen-lossy}, \ref{it:interp-wp=pw}. For instance,
  \[
    \|\jbk[\log y]^C\wzt\|_{L^\infty}^2
    \sum_{i=0}^2\int_{y\ge1}
    \frac{|\pry^i\w|^2}{y^{2(2-i)}\jbk[\log y]^2}
    \les b_1^{1+2c_1}\lgb^C.
  \]
  Similarly, the squared $L^2$ norm of the terms with $m\ge1$ in the second line is bounded by $b_1^5\lgb^C$ for $L=1$, and by $b_1^{1+6c_1}\lgb^C$ for $L\ge2$.
  For the pure radiation terms, we interpolate between the small $L^\infty$ norm and the adapted energies to obtain $\dbs\E_2$ and $\dbs\E_4$. We first treat $\wp$. Since $\ham_0=\ham+\zt^2/2$, by radial integration by parts and Lemma~\itref{le:interp-wp}{it:interp-wp=sobolev}, we have, for $k=1,2$,
  \[
    \|\pry^{2k}[(1-\chi_{\frac12})\wp]\|_{L^2}^2
    \le\|\ham_0^k[(1-\chi_{\frac12})\wp]\|_{L^2}^2
    \les C(M)\E_{2k}.
  \]
  Here the cutoff terms are supported on a fixed annulus, and the lower-order terms are controlled by the same lemma using
  \[
    \ham_0^2=\ham^2+\zt^2\ham-(\pry\zt^2)\pry
    -\frac12\Delta(\zt^2)+\frac14\zt^4.
  \]
  From the radial Gagliardo--Nirenberg inequality, we have
  \begin{align*}
    \|\pry\wp\|_{L^4(y\ge1)}
    &\les C(M)\|\wp\|_{L^\infty}^{1/2}\E_2^{1/4},\\
    \|\pry^i\wp\|_{L^{8/i}(y\ge1)}
    &\les C(M)\|\wp\|_{L^\infty}^{1-i/4}\E_4^{i/8},
    \qquad i=1,2,3.
  \end{align*}
  Indeed, we apply the one-dimensional inequality to $(1-\chi_{\frac12})\wp$ on $(\eta,\infty)$, raise to the corresponding powers, and integrate in $\eta>0$. From Hardy's inequality applied to the same cutoff with exponents $4,8,8/3>2$, we also have the corresponding bounds for $\pry^i\wp/y^j$ of total order $i+j$.

  For example, taking all derivative norms below on $y\ge1$, we obtain from H\"older and the smallness of $\|\wp\|_{L^\infty}$ that
  \begin{align*}
    \bigl\||\pry\wp|\,|\pry^3\wp|\bigr\|_{L^2}^2
    &\le\|\pry\wp\|_{L^8}^2\|\pry^3\wp\|_{L^{8/3}}^2\\
    &\les C(M)\bigl(\|\wp\|_{L^\infty}^{3/2}\E_4^{1/4}\bigr)
    \bigl(\|\wp\|_{L^\infty}^{1/2}\E_4^{3/4}\bigr)
    \les\dbs\E_4,\\
    \bigl\||\pry^2\wp|^2\bigr\|_{L^2}^2
    &=\|\pry^2\wp\|_{L^4}^4
    \les C(M)\|\wp\|_{L^\infty}^2\E_4
    \les\dbs\E_4.
  \end{align*}
  From the first interpolation inequality, we similarly obtain $\dbs\E_2$ for products of total order two. For the remaining products, distribute the powers of $y^{-1}$ among the factors and apply H\"older with reciprocal exponents summing to $1/2$. If all derivatives fall on one factor, use its $L^2$ bound and put the remaining factors in $L^\infty$.
  To include $\gamma$, we use the constraint~\eqref{eq:constra} to reduce its derivatives to the preceding products of $\wp$ and terms containing the profiles. In particular, we have
  \[
    \Delta\gamma
    =-\frac{|\pry\hw^\perp|^2+\hw^\perp\cdot\Delta\hw^\perp}{1+\wh\gamma}
    -\frac{(\hw^\perp\cdot\pry\hw^\perp)^2}{(1+\wh\gamma)^3}
    -\Delta\widetilde\gamma_0.
  \]
  Using $\Delta\wp=-\ham_0\wp+\wp/y^2$, we estimate the terms containing only $\wp$ by the preceding interpolation, and those containing $\wztp$ by the profile bound and Lemma~\ref{le:interp-wp}. The pure profile difference consists of the cutoff discrepancy and the Taylor remainder; its derivative of order $m$ has squared $L^2$ norm bounded by $b_1^2S_1^{2-2m}\lgb^C\les b_1^{m+1}\lgb^C$ for $m=2,4$. Thus we obtain
  \begin{equation}
    \label{pf:refined-e2=gamma}
    \|\pry^2\gamma\|_{L^2}^2+\|\Delta\gamma\|_{L^2}^2
    +\|\zt\pry\gamma\|_{L^2}^2+\|\zt^2\gamma\|_{L^2}^2
    \les\dbs\E_2+b_1^{1+2c_1}\lgb^C.
  \end{equation}
  Treating $4^{\rm th}$ derivatives similarly, we obtain
  \[
    \|\pry^4[(1-\chi_{\frac12})\gamma]\|_{L^2}^2
    \les\begin{cases}
      \dbs\E_4+b_1^{4+\frac14}\lgb^{-2},&L=1,\\
      \dbs\E_4+b_1^{1+6c_1}\lgb^C,&L\ge2,
    \end{cases}
  \]
  where we have used Lemma~\ref{le:interp-gamma} for $L=1$. Using these derivative bounds in the same interpolation and H\"older estimates, together with the bounds for the profile terms, we obtain
  \begin{align*}
    \|\Omega_+^1[\hw]\w\|_{L^2}^2
    &\les\dbs\E_2+b_1^{1+2c_1}\lgb^C,\\
    \|\Omega_+^2[\hw]\w\|_{L^2}^2
    &\les\begin{cases}
      \dbs\E_4+b_1^{4+\frac14}\lgb^{-2},&L=1,\\
      \dbs\E_4+b_1^{1+6c_1}\lgb^C,&L\ge2.
    \end{cases}
  \end{align*}
  Therefore, from~\eqref{eq:w2k0-e2k} and~\eqref{pro:boot-e2k-bnd}, we obtain
  \begin{equation}
    \begin{split}
      \label{pf:e2-la2=w21-w41}
      \|\w_2^1\|_{L^2}^2
      & \les \dbs\E_2 + b_1^{1+2c_1}\lgb^C, \\
      \|\w_4^1\|_{L^2}^2
      & \les \dbs\E_4 + b_1^{4+\frac{1}{4L}}\lgb^C.
    \end{split}
  \end{equation}

  \step[2, Estimates for $\wzt$]
  Note that the identity $\ham T_k=T_{k-1}$, the $\mcHb,\mcHg$ classes of Lemma~\ref{le:app-sol}, and the cutoff commutators on $y\sim S_1$ give
  \begin{equation}
    \label{pf:e2-la2=jh-wzt}
    \|\hj\mham\wzt\|_{L^2}^2
    \les \|\mham\wzt\|_{L^2}^2
    \les b_1^2\lgb^{4L-2}.
  \end{equation}
  Similarly, we have
  \[
    \|(\R\mham)^2\wzt\|_{L^2}^2
    \les b_1^4\lgb^{\max\{4L-6,0\}}.
  \]
  For $y\le1$, by Corollary~\itref{co:dore}{it:dore=wzt}, \ref{it:dore=hw} and Lemma~\itref{le:preserv}{it:preserv=op}--\ref{it:preserv=cross}, we have $(\hj\mham)^2\wzt\in\mcD_L$ with factor $b_1^2$. Here the leading linear term in $(\R\mham)^2\wzt$ vanishes since $\ham^2T_1=\ham\lphi=0$. Thus the squared $L^2$ norm on $y\le1$ is bounded by $Cb_1^4$. For $y\ge1$, from Lemma~\ref{le:omega-pw-bnd}, Lemma~\itref{le:interp-wp}{it:interp-wp=pw} and Lemma~\ref{le:app-sol},
  \[
    \Bigl\|\frac{\w}{y}\Bigr\|_{L^\infty(y\ge1)}^4\les b_1^{4c_1}\lgb^C,
    \quad
    \|y^2\mham^2\wzt\|_{L^2(y\ge1)}^2\les b_1^2\lgb^C.
  \]
  Their product yields $b_1^{2+4c_1}\lgb^C$, which is smaller than $b_1^4$. The remaining radiation terms have a strictly positive gain in the power of $b_1$ by~\ref{it:interp-wp=gen-lossy}, and the pure profile terms are read from the same classes. Thus
  \begin{equation}
    \label{pf:e2-la2=jh2-wzt}
    \|(\hj\mham)^2\wzt\|_{L^2}^2
    \les b_1^4\lgb^{\max\{4L-6,0\}}.
  \end{equation}

  \step[3, Energy estimate for $\E_2$]
  As in~\citep[Step~3 of Section 5.1]{merle_blowup_2013}, we have
  \begin{align*}
    &\frac12\frac{\dd}{\dd t}\biggl\{\frac{1}{\la^2}\int|\hj\mham\hw|^2\biggr\}
    =-\frac{1}{\la^4}\int\hj\mham\hw\cdot(\hj\mham)^2\hw \\
    &=-\frac{1}{\la^4}\biggl\{ \int\w_2\cdot\w_4+\int\w_2\cdot(\hj\mham)^2\wzt
    + \int\hj\mham\wzt\cdot\w_4 +\int\hj\mham\wzt\cdot(\hj\mham)^2\wzt  \biggr\},
  \end{align*}
  where we use $\hj\mham\hw = \w_2 + \hj\mham\wzt$. Then by Cauchy-Schwarz, \eqref{pro:boot-e2k-bnd}, \eqref{pro:boot-e2l2-bnd} and the previous estimates \eqref{pf:e2-la2=w21-w41}--\eqref{pf:e2-la2=jh2-wzt},
  \begin{equation}
    \begin{split}
      \label{pf:e2-la2=e2-energy}
      & \frac12\frac{\dd}{\dd t}\biggl\{\frac{1}{\la^2}\int|\hj\mham\hw|^2\biggr\} \\
      & \les \frac{1}{\la^4}
      \Bigl(\|\w_2\|_{L^2}+\|\hj\mham\wzt\|_{L^2}\Bigr)
      \Bigl(\|\w_4\|_{L^2} + \|(\hj\mham)^2\wzt\|_{L^2}\Bigr) \\
      & \les \frac{1}{\la^4} \begin{cases}
        b_1^2\sqrt{\E_2} + b_1^3\lgb, & L=1,\\
        b_1^{4c_1}\lgb^C, & L\ge2.
      \end{cases}
    \end{split}
  \end{equation}
  In addition, we expand $\hw$ and get
  \begin{align*}
    \int|\hj\mham\hw|^2
    & = \int|\hj\mham\w|^2 +\int|\hj\mham\wzt|^2 +2\int\hj\mham\w\cdot\hj\mham\wzt \\
    & = \|\w_2^0\|_{L^2}^2 + \|\w_2^1\|_{L^2}^2 + 2\int\w_2^0\cdot\w_2^1 + \|\hj\mham\wzt\|_{L^2}^2 + 2\int \w_2\cdot \hj\mham\wzt.
  \end{align*}
  Thus by Cauchy-Schwarz, \eqref{pro:boot-e2l2-bnd}, and \eqref{pf:e2-la2=w21-w41}--\eqref{pf:e2-la2=jh2-wzt}, we have
  \[
    \left|\int|\hj\mham\hw|^2-\E_2\right|
    \les \dbs\E_{2} +b_1^2\lgb^{4L-2}.
  \]
  After direct integration of \eqref{pf:e2-la2=e2-energy}, the estimate splits into two cases. For $L\ge2$, we have
  \begin{align*}
    & \biggl\{(1+\dbs)\E_2+O(b_1^2\lgb^{4L-2})\biggr\} \\
    & \leq \biggl(\frac{\la(s)}{\la(s_0)}\biggr)^2\biggl\{(1+\dbs)\E_2+O(b_1^2\lgb^{4L-2})\biggr\}(s_0)
    + \la^{2}(s) \int_{s_0}^s \frac{b_1^{4c_1}\lgb^{C}}{\la^2} \dd\sigma.
  \end{align*}
  The initial value part is small by \eqref{ini:e2k-bnd}. For the integral part, from $b_1\simeq c_1/s$, the asymptotics~\eqref{pf:pro-boot=la-s} for $\la$, we have
  \[
    \int_{s_0}^s \frac{b_1^{4c_1}\lgb^{C}}{\la^2} \dd\sigma
    \les \int_{s_0}^{\infty} \frac{(\log \sigma)^{C}}{\sigma^{2c_1}} \dd\sigma
    \les 1,
  \]
  since $2c_1>1$. Thus we obtain $\E_2 \les \la^2$, where the implicit constant depends on $u_0$. For $L=1$, we have $c_1=1$, $d_1=-2$. The initial value part remains small by \eqref{ini:e2k-bnd}, but the integral part becomes
  \[
    \int_{s_0}^s \frac{1}{\la^2}\Bigl(
    b_1^2\sqrt{\E_2}+b_1^3\lgb\Bigr) \dd\sigma
    \les \int_{s_0}^s \frac{(\log\sigma)^{\frac{K}{2}-4}}{\sigma} \dd\sigma
    \les \begin{cases}
      (\log s)^{\frac{K}{2}-3}, &K>6,\\
      \log\log s, &K=6,\\
      1, &K<6.
    \end{cases}
  \]
  Since $K$ is chosen large enough, using \eqref{pf:pro-boot=la-b-rel} we have the improved bound
  \[
    \E_2 \les \la^2(1+ (\log s)^{\frac{K}{2}-3}) \les b_1^{2}\lgb^{\frac{K}{2}+1}.
  \]
  As $K/2+1<K$ for $K>2$, inserting this bound back into the preceding inequality strictly decreases the logarithmic exponent. After finitely many iterations, if the resulting $K$ satisfies $K<6$, then the desired bound follows
  \[
    \E_2 \les \la^2(1+1) \les \la^2;
  \]
  if $K=6$ after finitely many steps, then using $\log \log s \les \log s$ we have
  \[
    \E_2\les \la^2(1+\log s) \les b_1^{2}\lgb^{5}.
  \]
  Then one more iteration ensures the desired bound. This completes the proof.
\end{proof}

\subsection{Proof of the main theorem}
\label{sub:proof-main-thm}
We are now in a position to prove Theorem~\ref{th:main-thm}, and show the corresponding blowup behavior.

\begin{proof}[Proof of Theorem~\ref{th:main-thm}]
  \step[1, Finite time blowup] Using the asymptotic formula~\eqref{pf:pro-boot=la-s} for $\la$, the modulation equation~\eqref{le:mod-eq-1}, and $b_1\simeq c_1/s$ by \eqref{pf:pro-boot=b1-fin}, we have
  \begin{align*}
    -2L \frac{\dd}{\dd t} \bigl(\la^{\frac{1}{2L}}\bigr)
    & = \la^{\frac{1}{2L}-2}\biggl(b_1 + O\biggl(\frac{b_1^{L+1}}{\lgb}\biggr)\biggr)                                                                \\
    & = C(u_0)(1+o(1)) \biggl(\frac{(\log s)^{|d_1|}}{s^{c_1}}\biggr)^{\frac{1}{2L}-2} \biggl(\frac{c_1}{s} +O\biggl(\frac{1}{s\log s}\biggr)\biggr) \\
    & = C(u_0)(1+o(1)) s^{\frac{1}{4L-2}} (\log s)^{(\frac{1}{2L}-2)|d_1|}                                                                           \\
    & \ge C(u_0),\footnotemark
  \end{align*}
  which yields the uniform upper bound for the time derivative: $\frac{\dd}{\dd t}\la^{\frac{1}{2L}} \le -\frac{C(u_0)}{2L} <0$. Note that we used
  \[
    -1 -c_1\biggl(\frac{1}{2L} -2\biggr) = \frac{1}{4L-2}.
  \]
  Hence, there exists $T\in(0,+\infty)$ such that
  \[
    \la\downarrow 0 \quad\mbox{as}\quad t\uparrow T.
  \]
  Together with the decomposition~\eqref{eq:u-decomp} and the existence of $\w\in C([s_0,+\infty);\dot{H}^1)$, this shows that the constructed solution $u(t)$ to~\eqref{eq:sm} forms a concentrating energy bubble as $t\uparrow T$, and hence blows up in finite time.

  \footnotetext{$C(u_0)$ denotes constants depending on $u_0$, which may change from line to line.}

  \step[2, Quantized blowup rate] From~\eqref{pf:pro-boot=la-s} we have $\la\downarrow 0$ only if $s\uparrow\infty$, since $\la(s)>0$ for all $s\ge s_0$. Thus the rescaled time transformation
  \[
    [0,T) \to [s_0,\infty), \quad t\mapsto s(t) = s_0 +\int_{0}^{t}\frac{\dd\tau}{\la^2(\tau)}
  \]
  is well-defined and satisfies $s(t)\to\infty$ as $t\uparrow T$. Then from \eqref{pf:pro-boot=la-s},
  \[
    T-t
    = \kappa_0^2(1+o(1)) \int_{s}^{\infty} \frac{(\log\sigma)^{2|d_1|}}{\sigma^{2c_1}} \dd\sigma
    = C(u_0)(1+o(1)) \frac{(\log s)^{2|d_1|}}{s^{\frac{1}{2L-1}}},
  \]
  and hence
  \[
    \left|\log (T-t)\right| = C(u_0)(1+o(1)) \log s.
  \]
  Consequently, applying the above identities together with $c_1=\frac{L}{2L-1}$, $d_1=-\frac{2L}{(2L-1)^2}$ by~\eqref{eq:c1}, we obtain the following blowup rate:
  \begin{align*}
    \la(t)
    & = \kappa_0              (1+o(1)) \frac{(\log s)^{|d_1|}}{s^{c_1}}               \\
    & = C(u_0)(1+o(1)) \biggl(\frac{T-t}{(\log s)^{2|d_1|}}\biggr)^L (\log s)^{|d_1|} \\
    & = C(u_0)(1+o(1)) \frac{(T-t)^L}{\left|\log (T-t)\right|^{\frac{2L}{2L-1}}}.
  \end{align*}

  \step[3, Convergence of the phase] Let $0<t_1<t_2<T$. We denote the corresponding rescaled times by $s_1=s(t_1)$, $s_2=s(t_2)$. Now using the modulation equation~\eqref{le:mod-eq-1} for $\vtha$, $b_1\simeq c_1/s$ by~\eqref{pf:pro-boot=b1-fin}, and \eqref{pf:pro-boot=la-s},
  \begin{align*}
    |\vtha(t_2)-\vtha(t_1)|
    & =\left|\int_{s_1}^{s_2}
    \biggl(-a_1+O\biggl(\frac{b_1^{L+1}}{\lgb}\biggr)\biggr) \dd\sigma\right| \\
    & \les\int_{s_1}^{\infty}
    \biggl(\frac{1}{\sigma^{L+1}\log \sigma} +\frac{1}{\sigma(\log\sigma)^{5/4}} \biggr)\dd\sigma
    \les \frac{1}{(\log s_1)^{1/4}}
    \to 0
  \end{align*}
  as $s_1\uparrow \infty$, equivalently, as $t_1\uparrow T$. This implies that $\vtha(t)\to\vtha(T)\in\mbR$ as $t\uparrow T$.

  \step[4, Strong convergence of the excess energy]
  Recalling the decomposition~\eqref{eq:w}, we have
  \[
    u  = e^{\vtha \R} (Q+\hv)_{\la} =: e^{\vtha \R} Q_{\la} + \bar{u}.
  \]
  By scaling invariance and \eqref{pro:boot-e1-bnd}, we have
  \[
    \|\nabla \bar{u}\|_{L^2}^2 \les 1.
  \]
  In addition, note that
  \[
    \Delta\hv = (e_r,e_{\tau},Q)\Bigl(-\mham\hw-\frac{\zt^2}{2}\hw\Bigr),
  \]
  and thus from Lemma~\ref{le:coer-e-2k2}, Lemma~\ref{pro:refined-e2} and the explicit construction of $\wzt$,
  \begin{align*}
    \|\Delta \bar{u}\|_{L^2}^2
    \les \frac{\|\Delta \hv\|_{L^2}^2}{\la^2}
    & \les \frac{1}{\la^2} \Bigl(\|\mham\w\|_{L^2}^2 + \|\zt^2\w\|_{L^2}^2 + \|\mham\wzt\|_{L^2}^2 + \|\zt^2\wzt\|_{L^2}^2 \Bigr) \\
    & \les \frac{1}{\la^2}\bigl(\E_2 +b_1^{1+2c_1}\lgb^C +b_1^2\lgb^{4L-2}\bigr)
    \les 1.
  \end{align*}
  Then for any $R>0$, by Gagliardo--Nirenberg,
  \[
    \|\nabla\bar{u}\|_{L^2(r\le R)}^2
    \les \|\nabla\bar{u}\|_{L^4(r\le R)}^2 R
    \les \|\nabla\bar{u}\|_{L^2} \|\Delta\bar{u}\|_{L^2} R
    \les R.
  \]
  For the far-out region, since
  \[
    \|\Delta (Q_{\la})\|_{L^2(r\ge R)}^2
    \sim \la^4 \to 0
    \quad\mbox{as}\quad \la\downarrow 0,
  \]
  we have the uniform bound
  \[
    \sup_{0<t<T}\|\Delta u\|_{L^2(r\ge R)}^2 \les 1.
  \]
  Moreover, using the equation~\eqref{eq:sm} and $|u|=1$, we have for $0<t_1<t_2<T$ that
  \[
    \|u(t_2)-u(t_1)\|_{L^2(r\ge R)}
    \les \int_{t_1}^{t_2} \|\Delta u(\sigma)\|_{L^2(r\ge R)}\dd\sigma
    \les |t_2-t_1|.
  \]
  Then a standard interpolation yields for $t_1\uparrow T$ that
  \begin{align*}
    & \|\nabla[(1-\chi_{R/2})(u(t_2)-u(t_1))]\|_{L^2}^2 \\
    & \les \|(1-\chi_{R/2})(u(t_2)-u(t_1))\|_{L^2}
    \|\Delta[(1-\chi_{R/2})(u(t_2)-u(t_1))]\|_{L^2} \\
    & \les |t_2-t_1| \to 0.
  \end{align*}
  Hence
  \begin{align*}
    & \|\nabla\bar{u}(t_2) -\nabla\bar{u}(t_1)\|_{L^2(r\ge R)} \\
    & \les \|\nabla u(t_2)-\nabla u(t_1)\|_{L^2(r\ge R)} + \sum_{i=1,2}\|\nabla(Q_{\la(t_i)})\|_{L^2(r\ge R)} \to 0.
  \end{align*}
  By taking $R\to 0$, we also have $\|\nabla\bar{u}\|_{L^2(r\le R)}\to 0$, and thus
  \[
    \|\nabla\bar{u}(t_2) -\nabla\bar{u}(t_1)\|_{L^2} \to 0 \quad\mbox{as}\quad t_1\uparrow T.
  \]
  This yields the existence of $u^{\ast} = \lim_{t\uparrow T} \bar{u}(t)\in\dot{H}^1$, and hence completes the proof of Theorem~\ref{th:main-thm}.
\end{proof}

\appendix

\section{Geometric properties}
\label{ap:1-equivariance-geom}
This appendix is devoted to two basic geometric properties of the $1$-equivariant \sch\ maps.

\begin{lemma}
  \label{le:geo-prop}
  Let $u$ be a solution to~\eqref{eq:sm} with initial data $u_0$. We have:
  \begin{enumerate}[label={\rm (\roman*)}]
    \item if $u_0$ takes values in $\mbS^2$, then so does $u$;~\label{it:geo-prop=1}
    \item if $u_0$ is $1$-equivariant, then $u$ is also $1$-equivariant.~\label{it:geo-prop=2}
  \end{enumerate}
\end{lemma}

\begin{proof}
  The property~\ref{it:geo-prop=1} is a direct result of~\eqref{eq:sm}:
  \[
    \prt|u|^2 = 2u\cdot (u\times\Delta u) =0
    \quad\Longrightarrow\quad
    |u|^2(t,x) \equiv |u_0|^2(x) = 1.
  \]
  To show \ref{it:geo-prop=2}, we let $R_\tha\in\SO(2)$ be the rotation by $\tha\in\mbR$ in the domain $\mbR^2$, and $\R_\tha = e^{\tha\R}\in\SO(3)$ be the corresponding rotation in the target $\mbS^2$ given by the standard inclusion $\SO(2)\hookrightarrow\SO(3)$. Then the $1$-equivariance of the initial data reads
  \[
    u_0(x) = \R_{-\tha} u_0(R_\tha x).
  \]
  By the invariance of the Laplacian under rotations of the domain and of the cross product under rotations of the target,
  \[
    \Delta (u\circ R_{\tha}) = (\Delta u)\circ R_{\tha},
    \andd
    \R_{\tha}\vv\times\R_{\tha}\w = \R_{\tha}(\vv\times\w)
    \;\;\mbox{for}\;\; \vv,\w\in\mbR^3,
  \]
  we find that $\R_{-\tha} u(t,R_\tha x)$ solves~\eqref{eq:sm} with the initial data $\R_{-\tha} u_0(R_\tha x)$. By the local well-posedness of the \sch\ map flow, this implies
  \[
    u(t,x) = \R_{-\tha} u(t,R_\tha x),
  \]
  and hence, by definition, the $1$-equivariance of the solution $u$.
\end{proof}

For consistency with the notation in Section~\ref{sec:app-sol}, we use $y$ rather than $r$ for the radial variable of equivariant maps.

\begin{lemma}[Regularity of $1$-equivariant maps]
  \label{le:reg}
  Let $u = Q +\vv$ be a smooth $1$-equivariant solution of~\eqref{eq:sm}, where $\vv = (e_r,e_{\tau},Q)(\alpha,\beta,\gamma)^{\rm T}$. If $u$ satisfies
  \[
    \sum_{\ell=1}^{N}\int|(-\Delta)^{\frac{\ell}{2}}u|^2 < +\infty
  \]
  for some $ N\gg L$, then the following estimates hold:
  \begin{enumerate}[label={\rm (\roman*)}]
    \item for every integer $1\le p\le 10L$, the components of $\vv$ admit the following Taylor expansions at the origin. More precisely, there exist sequences $(c_{\alpha,k})_{k=1}^{p}$, $(c_{\beta ,k})_{k=1}^{p}$, $(c_{\gamma,k})_{k=1}^{p} \subset \mbR$ such that \label{le:reg=app-taylor}
      \[
        \begin{pmatrix}
          \alpha \\ \beta\\ \gamma
        \end{pmatrix} = \sum_{k=1}^{p}\begin{pmatrix}
          c_{\alpha,k} \hs[.1] y^{2k-1} \\
          c_{\beta, k} \hs[.1] y^{2k-1} \\
          c_{\gamma,k} \hs[.1] y^{2k}
        \end{pmatrix}
        + O\hs[-.15]\begin{pmatrix}
          y^{2p+1} \\ y^{2p+1} \\ y^{2p+2}
        \end{pmatrix};
      \]
    \item let $\varepsilon$ be either $\alpha$ or $\beta$. If
      \[
        \|\nabla \vv\|_{L^2} + \Bigl\|\frac{\vv}{y}\Bigr\|_{L^2} \ll 1,
      \]
      then for $1 \le k\le L$,
      \label{le:reg=app-wei-bnd-finite}
      \begin{align*}
        \int|\varepsilon_{(2k+2)}|^2 +\int\frac{|\varepsilon_{(2k+1)}|^2}{y^2\jbk[y]^2}
        & +\sum_{\ell=1}^{k}\int
        \frac{|\varepsilon_{(2\ell-1)}|^2}{y^6\jbk[\log y]^2\jbk[y]^{4(k-\ell)}}                         \\
        & + \sum_{\ell=0}^{k}\int\frac{|\varepsilon_{(2\ell)}|^2}{y^4\jbk[\log y]^2\jbk[y]^{4(k-\ell)}}
        < +\infty.
      \end{align*}
  \end{enumerate}
\end{lemma}

\begin{proof}
  Proof of~(i). We compute $\nabla u$ using Lemma~\ref{le:frenet-basis},
  \begin{align*}
    \nabla u = (\pry u, \pr_{\tau} u)
    & = \biggl(
    \biggl( \pry\alpha +\frac{\zt}{2}(1+\gamma)\biggr) e_r
    +\pry\beta e_\tau
    + \biggl( \pry\gamma -\frac{\zt}{2}\alpha\biggr) Q, \\
    & \qquad\qquad\qquad
    -\frac{\z}{y}\beta e_r
    +\biggl( \frac{\z}{y}\alpha +\frac{\zt}{2}(1+\gamma) \biggr) e_\tau
    -\frac{\zt}{2}\beta Q
    \biggr).
  \end{align*}
  By $1$-equivariant symmetry, the directional derivatives of $u$ along $e_1$ and $e_2$ satisfy
  \[
    \R \frac{\pr u}{\pr e_1}\bigg|_{y=0}
    = \frac{\pr u}{\pr e_2}\bigg|_{y=0},
  \]
  and thus
  \begin{align*}
    & \biggl( -\z\pry+\frac{\z}{y} \biggr) \beta e_r + \biggl( \z\pry - \frac{\z}{y} - \frac{\zt}{2}\lphi \biggr) \alpha e_\tau                \\
    & + \biggl( \frac{(\z-1)\zt}{2}(1 + \gamma) + \lphi\pry\gamma \biggr) e_\tau + \biggl( -\lphi\pry\beta + \frac{\zt}{2}\beta \biggr) Q
    \to 0
    \quad\mbox{as}\quad y\to 0.
  \end{align*}
  Using the asymptotics~\eqref{eq:lphi}, we obtain
  \[
    \oaz\alpha \to 0, \quad
    \oaz\beta \to 0, \quad
    \pry \gamma \to 0
    \quad\mbox{as}\quad y\to 0,
  \]
  where $\oaz$ and its formal adjoint $\asz$ are given by
  \[
    \oaz = -\pry +\frac{1}{y} = -y\pry\biggl(\frac{\cdot}{y}\biggr), \quad
    \asz = \pry +\frac{2}{y} = \frac{1}{y^2}\pry(y^2 \cdot),
  \]
  and their composition by
  \[
    \ham_0 = \asz\oaz = -\Delta + \frac{1}{y^2}.
  \]
  Similarly, computing $\Delta u$ using Lemma~\ref{le:frenet-basis},
  \begin{align*}
    \Delta u
    = \bigl(-\ham_0\alpha +\zt\pry\gamma\bigr) e_r -\ham_0\beta e_\tau  +\biggl(-\frac{\zt^2}{2}+ \biggl(\Delta-\frac{\zt^2}{2}\biggr)\gamma -\zt\biggl(\pry+\frac{\z}{y}\biggr)\alpha \biggr) Q.
  \end{align*}
  Then the smoothness of the solution gives
  \[
    |\ham_0\alpha| +|\ham_0\beta| \les 1.
  \]
  Let us focus on $\alpha$. Note from the definition of $\asz$ that $\ham_0\alpha = y^{-2}\pry(y^2\oaz\alpha)$. A direct integration using the boundary condition $\oaz\alpha(0) = 0$ yields
  \[
    \oaz\alpha(y) = \frac{1}{y^2}\int_0^y\ham_0\alpha(\sigma)\sigma^2 d\sigma = O(y).
  \]
  In addition, we observe that $\ham_0\alpha = -\pry^2\alpha +y^{-1}\oaz\alpha$, and thus
  \[
    |\pry^2\alpha| \les 1.
  \]
  Repeating this argument for the corresponding components of the higher spatial derivatives of $u$, and using $N\gg L$, gives the $C^\ell$-regularity of $\alpha$ for $1\le \ell\le 2p+2$. Let
  \[
    \alpha_{[0]} = \alpha, \quad
    \alpha_{[\ell+1]} =
    \begin{cases}
      \asz\alpha_{[\ell]} & \mbox{for $\ell$ odd},  \\
      \oaz\alpha_{[\ell]} & \mbox{for $\ell$ even},
    \end{cases}
    \quad 0\le \ell\le 2p+1.
  \]
  A direct induction gives
  \begin{equation}
    \label{pf:reg=app-alpha-bnd}
    |\alpha_{[\ell]}| \les \begin{cases}
      y^2 & \mbox{for $\ell$ odd},  \\
      y   & \mbox{for $\ell$ even},
    \end{cases}
    \quad 1 \le \ell \le 2p+2.
  \end{equation}
  This implies the Taylor expansion near the origin
  \[
    \alpha(y)
    = \sum_{\ell=1}^{2p+1}c_{\ell}y^{\ell} +O(y^{2p+2})
    \with (c_\ell)_{\ell=1}^{2p+1}\subset\mbR.
  \]
  Since $\asz\oaz=\ham_0$ and
  \[
    \ham_0 y^{\ell} = -(\ell^2-1)y^{\ell-2}
    \for \ell\in\mbZ,
  \]
  applying $\oaz,\asz$ successively to the expansion and using~\eqref{pf:reg=app-alpha-bnd} yields, by induction,
  \[
    c_{2\ell} = 0 \for 1\le\ell\le p,
  \]
  and hence~\ref{le:reg=app-taylor} for $\alpha$.
  The same conclusion holds for $\beta$ by a similar induction. For $\gamma$, we note that it is a smooth function of $\alpha^2+\beta^2$ near the origin, vanishing to second order in $(\alpha,\beta)$, since by Lemma~\itref{le:geo-prop}{it:geo-prop=1},
  \[
    |u|^2 = 1 \quad\Longrightarrow\quad
    \alpha^2 + \beta^2 + (1+\gamma)^2 = 1.
  \]
  Together with $\gamma(0)=0$, this yields
  \[
    \gamma = \sqrt{1-\alpha^2-\beta^2}-1 = O(y^2).
  \]
  Since the expansions of $\alpha$ and $\beta$ contain only odd powers, the Taylor expansion of $\gamma$ contains only even powers,
  which establishes~\ref{le:reg=app-taylor}.

  Proof of~\ref{le:reg=app-wei-bnd-finite}. Here $\varepsilon_{(\ell)}$ denotes the suitable derivative of order $\ell$ defined in~\eqref{eq:adb-seq}. Since $\oa-\oaz=\as-\asz=(\z-1)/y$ is smooth and odd near the origin, with $(\z-1)/y=O(y)$, the Taylor expansions in~\ref{le:reg=app-taylor} give $\varepsilon_{(2j)}=O(y)$ and $\varepsilon_{(2j+1)}=O(y^2)$ for the orders appearing below. Hence the required integrability for $y\le1$ follows from
  \[
    \int_{y\le 1}
    \frac{|\varepsilon_{(2\ell-1)}|^2}{y^6\jbk[\log y]^2\jbk[y]^{4(k-\ell)}} + \int_{y\le 1}\frac{|\varepsilon_{(2\ell)}|^2}{y^4\jbk[\log y]^2\jbk[y]^{4(k-\ell)}}
    \les \int_{y\le 1}\frac{\dd y}{y\jbk[\log y]^2}
    < +\infty.
  \]
  For $y\ge1$, we consider $\varepsilon = \alpha$. By Leibniz,
  \[
    |\pry^k\alpha|  \les |\pry^k
    \vv \cdot e_r| + \sum_{\ell=0}^{k-1} |\pry^\ell \vv|\,|\pry^{k-\ell} e_r|.
  \]
  From Lemma~\ref{le:frenet-basis},
  \[
    \pry e_r = -\frac{\zt}{2} Q,\quad \pry Q = \frac{\zt}{2} e_r
    \quad\Longrightarrow\quad
    |\pry^k e_r| \les \frac{1}{\jbk[y]^{k+1}}
    \for k\ge 1.
  \]
  Consequently, for $\ell\le k+1$,
  \begin{align*}
    & \int_{y\ge 1}
    \frac{|\alpha_{(2\ell-1)}|^2}{y^6\jbk[\log y]^2\jbk[y]^{4(k-\ell)}} + \int_{y\ge 1}\frac{|\alpha_{(2\ell)}|^2}{y^4\jbk[\log y]^2\jbk[y]^{4(k-\ell)}} \\
    & \les \int\frac{|\alpha|^2}{y^2} + \sum_{\ell=1}^{2k+2}\int|\pry^{\ell}\alpha|^2
    \les \int\frac{|\vv|^2}{y^2} + \sum_{\ell=1}^{2k+2}\int|\pry^{\ell} \vv|^2
    < +\infty.
  \end{align*}
  The same conclusion holds for $\beta$ by a similar argument. This shows~\ref{le:reg=app-wei-bnd-finite} for $y\ge 1$, and hence completes the proof.
\end{proof}

\section{Operator actions on admissible spaces}
\label{ap:adm-lemma}
The goal of this appendix is to prove Lemmas~\ref{le:ham-act-adb} and~\ref{le:ham-act-badb}. The reader can refer to~\citep{raphael_quantized_2014} for more details.

\begin{proof}[Proof of Lemma~\ref{le:ham-act-adb}]
  Proof of~\ref{it:ham-act-adb=ham}. By induction, it suffices to show
  \[
    \ham f \in
    \begin{cases}
      \mcA(\max(p_1-2,1),p_2-2) & \mbox{for $p_1$ odd},  \\
      \mcA(p_1-2,p_2-2)         & \mbox{for $p_1$ even}.
    \end{cases}
  \]
  For $y\le 1$, from \eqref{eq:ham} we observe that for $i\ge 0$,
  \begin{equation}
    \label{pf:adm=bounce-back}
    \ham y^{i} = (1-i^2)y^{i-2} +O(y^{i})
    =\begin{cases}
      O(y^{i-2}), & \mbox{$i\neq 1$}, \\
      O(y^{i}),   & \mbox{$i=1$},
    \end{cases}
  \end{equation}
  which demonstrates the bounce-back effect: $\ham y=O(y)$, and hence yields the $y\le 1$ part of (i). For $y\ge 1$, let $q_2=p_2-2$. Since
  \[
    (\ham f)\sd[i] = f\sd[i+2],
    \quad\mbox{for}\quad i\in\mbN
  \]
  and $q_2-i=p_2-(i+2)$, the two cases in Definition~\itref{def:adb}{it:def-adb=2} coincide for $(\ham f)\sd[i]$ and $f\sd[i+2]$. This gives the desired asymptotics at infinity.

  Proof of~\ref{it:ham-act-adb=ham-inv}. Let $u=\ham^{-1}f$. Again by induction, it suffices to show $u$ is \adb\ of degree $(p_1+2,p_2+2)$. Observe that $u\sd[i]= f\sd[i-2]$ for $i\ge 2$. Since $(p_2+2)-i=p_2-(i-2)$, the two cases in Definition~\itref{def:adb}{it:def-adb=2} again coincide, and hence the admissibility bounds for $u_{(i)}$, $i\ge2$, follow directly from those for $f_{(i-2)}$. Thus it remains to check (ii) for $u_{(1)}=\oa u$ and $u_{(0)}=u$. Recalling the properties of $\ham$ provided in Subsection~\ref{sub:lin-ham},
  \[
    \oa\lphi=0, \quad
    \oa\Gamma = -\Gamma' +\frac{(\lphi)'}{\lphi}\Gamma = -\frac{1}{y\lphi}.
  \]
  These together with \eqref{eq:inverse-ham} yield
  \begin{equation}
    \label{pf:adm=au-u}
    \oa u(y) = \frac{1}{y\lphi(y)}\int_0^{y}f(x)\lphi(x) x \dd x,
    \quad
    u(y) = -\lphi(y) \int_0^y \frac{\oa u(x)}{\lphi(x)}\dd x.
  \end{equation}
  Together with the asymptotics~\eqref{eq:lphi}, this yields the Taylor expansion for $y\le 1$:
  \[
    \oa u(y) = y^{p_1+1}\left(\sum_{i=0}^{p}c'_iy^{2i} +O(y^{2p+2})\right),
    \quad u(y) = y^{p_1+2}\left(\sum_{i=0}^{p}c''_iy^{2i} +O(y^{2p+2})\right).
  \]
  For $y\ge1$, we first estimate $\oa u$. If $p_2\ge 0$, then Definition~\itref{def:adb}{it:def-adb=2} and~\eqref{pf:adm=au-u} yield
  \[
    |\oa u(y)|
    \les 1+\int_1^y x^{p_2}\jbk[\log x]\dd x
    \les y^{p_2+1}\jbk[\log y].
  \]
  In the borderline case $p_2=-1$, the same definition gives $|f(x)|\les x^{-1}$, and hence
  \[
    |\oa u(y)|
    \les 1+\int_1^y\frac{\dd x}{x}
    \les \jbk[\log y]
    =y^{p_2+1}\jbk[\log y].
  \]
  Thus, for all $p_2\ge-1$, the second identity in~\eqref{pf:adm=au-u} and the asymptotic behavior $\lphi(y)\sim 2/y$ imply
  \[
    |u(y)|
    \les \frac{1}{y}\left(1+\int_1^y x^{p_2+2}\jbk[\log x]\dd x\right)
    \les y^{p_2+2}\jbk[\log y].
  \]
  This concludes the proof.
\end{proof}

\begin{proof}[Proof of Lemma~\ref{le:ham-act-badb}]
  Proof of~\ref{it:ham-act-badb=ham}. By induction, it suffices to show
  \[
    \ham f \in
    \begin{cases}
      \mcAb(\max(p_1-2,1),p_2-2) & \mbox{for $p_1$ odd},  \\
      \mcAb(p_1-2,p_2-2)        & \mbox{for $p_1$ even}.
    \end{cases}
  \]
  First for $y\le 1$, by definition we have $\ham f = \sum_{i=1}^{I}\wh{\delta}_i \ham\wh{f}_i$, where $\wh{f}_{i}$ satisfies the series expansion as in Definition~\itref{def:badb}{it:def-badb=1}. Then, as in the proof of Lemma~\ref{le:ham-act-adb}, $\ham\wh{f}_{i}$ also admits the bounce-back effect \eqref{pf:adm=bounce-back}, which implies the asymptotics for $y\le 1$. For $y\ge1$, the result is straightforward since
  \[
    (\ham f)\sd[i] = f\sd[i+2]
    \quad\mbox{for}\quad i\in\mbN.
  \]

  Proof of~\ref{it:ham-act-badb=ham-inv}. Let $u=\ham^{-1}f$. Again by induction, it suffices to show $u\in\mcAb(p_1+2,p_2+2)$. Note that $u\sd[i] =f\sd[i-2]$ for $i\ge 2$, so it remains to prove the result for $u\sd[1] = \oa u$ and $u\sd[0]=u$. We use \eqref{pf:adm=au-u} again to compute the asymptotics of $\oa u$ and $u$ for $y\le 1$ and $y\ge1$. The case $y\le 1$ is similar to part (ii) in the proof of Lemma~\ref{le:ham-act-adb}, as the $\wh\delta_{i}$ are not actually involved in the computations. For $y\ge1$, we obtain
  \[
    |\oa u|
    \les 1 + \int_1^y\frac{x^{p_2}}{\lgb}(F_{p_2} +K_{p_2})\dd x + \int_1^y x^{p_2-2}\jbk[\log x]^{c_{p_2}}\dd x.
  \]
  First, for the $F_{p_2}$ part, since $p_2\ge 1$,
  \begin{align*}
    \int_1^y\frac{x^{p_2}}{\lgb} F_{p_2} \dd x
    & = \frac{1}{\lgb} \int_1^y x^{p_2}\jbk[\log(\sqrt{b_1}x)] \one_{x\le 2S_0}\dd x      \\
    & \les \frac{y^{p_2+1}}{\lgb} F_{p_2+1} + \frac{S_0^{p_2+1}}{\lgb} \one_{y\ge 2S_0}.
  \end{align*}
  Then, for the $K_{p_2}$ part, we note that $K_{p_2}$ vanishes if $p_2<2$, and thus
  \[
    \int_1^y\frac{x^{p_2}}{\lgb} K_{p_2} \dd x
    = \int_1^y \frac{x^{p_2}}{\lgb} \frac{\one_{x\ge 2S_0}}{x^2 b_1} \dd x\, \one_{p_2\ge 2}
    \les \frac{y^{p_2+1}}{\lgb} K_{p_2+1}.
  \]
  Next, for the remaining integral,
  \begin{align*}
    \int_1^y x^{p_2-2}\jbk[\log x]^{c_{p_2}}\dd x
    & \les y^{p_2-1}\jbk[\log y]^{c_{p_2}+1}.
  \end{align*}
  We insert these estimates back into $|\oa u|$. Note that the fact that $p_2\ge 1$ implies the constant $1$ and the $S_0$ power term are absorbed:
  \begin{align*}
    \one_{y\le 2S_0} \les y^{p_2-1}\jbk[\log y]^{c_{p_2}+1},
    \quad
    \one_{y\ge 2S_0} \les \frac{S_0^{p_2+1}}{\lgb} \one_{y\ge 2S_0}
    \les \frac{y^{p_2+1}}{\lgb} K_{p_2+1},
  \end{align*}
  and thus
  \[
    |\oa u| \les \frac{y^{p_2+1}}{\lgb}(F_{p_2+1} + K_{p_2+1}) + y^{p_2-1} \jbk[\log y]^{c_{p_2}+1}.
  \]
  Now to treat $u$, using~\eqref{pf:adm=au-u} together with $\lphi\sim 1/y$ for $y\ge1$, we obtain
  \[
    |u| \les \frac{1}{y} \biggl(1 + \int_1^y\frac{x^{p_2+2}}{\lgb}(F_{p_2+1} + K_{p_2+1})\dd x + \int_1^y x^{p_2}\jbk[\log x]^{c_{p_2}+1}\dd x \biggr),
  \]
  and the desired estimate follows from the same three-integral treatment as for $\oa u$, with the exponents shifted accordingly. Now let us check the estimates for the $b_1$ derivatives. From \eqref{pf:adm=au-u},
  \begin{align*}
    \bigg|\frac{\pr^\ell \oa u}{\pr b_1^{\ell}}\bigg|
    & \les \int_0^1 \bigg|\frac{\pr^\ell f}{\pr b_1^{\ell}}\bigg| x^2 \dd x
    + \int_1^y \bigg|\frac{\pr^\ell f}{\pr b_1^{\ell}}\bigg| \dd x                                                                                                                   \\
    & \les \frac{1}{b_1^{\ell}\lgb} + \int_1^y \frac{x^{p_2}}{b_1^{\ell}\lgb}(\bar{F}_{p_2} + K_{p_2})\dd x + \int_1^y \frac{x^{p_2-2}\jbk[\log x]^{c_{p_2}}}{b_1^{\ell}\lgb} \dd x.
  \end{align*}
  The $\bar{F}_{p_2}$ part is bounded by
  \begin{align*}
    \int_1^y\frac{x^{p_2}}{b_1^{\ell}\lgb} \bar{F}_{p_2} \dd x
    & = \frac{1}{b_1^{\ell}\lgb^2} \int_1^y x^{p_2}\jbk[\log x]\one_{x\le 2S_0} \dd x                              \\
    & \les \frac{y^{p_2+1}}{b_1^{\ell}\lgb} \bar{F}_{p_2+1} +\frac{S_0^{p_2+1}}{b_1^{\ell}\lgb} \one_{y\ge 2S_0}.
  \end{align*}
  The remaining integrals are bounded as in the estimates for $\oa u$, with an extra factor $b_1^{-\ell}$.
  Note that $p_2\ge 1$ allows the absorption of the constant $1$ and the $S_0$ term:
  \[
    1 \les y^{p_2-1}\jbk[\log y]^{c_{p_2}+1},
    \quad
    S_0^{p_2+1} \one_{y\ge 2S_0}
    \les y^{p_2+1} K_{p_2+1},
  \]
  and thus
  \[
    \bigg|\frac{\pr^\ell \oa u}{\pr b_1^{\ell}}\bigg|
    \les \frac{y^{p_2+1}}{b_1^{\ell}\lgb} (\bar{F}_{p_2+1} + K_{p_2+1}) + \frac{y^{p_2-1}\jbk[\log y]^{c_{p_2}+1}}{b_1^{\ell}\lgb}.
  \]
  The bound for $|\frac{\pr^{\ell} u}{\pr b_1^{\ell}}|$ follows from \eqref{pf:adm=au-u} and an analogous computation. This concludes (ii).

  Proof of~\ref{it:ham-act-badb=vla}. We rewrite the operator $\vla = y\pry$ so that
  \[
    \vla f = -f -y f_{(1)} +(1+\z)f.
  \]
  For $y\le 1$, the expansion of $\vla f$ in Definition~\itref{def:badb}{it:def-badb=1} follows from a direct computation using the near-origin asymptotics of $\z$. For $y\ge 1$, it suffices to show
  \[
    y f\sd[1],\, (1+\z)f \in \mcAb(p_1,p_2).
  \]
  Consider the $y f\sd[1]$ part. We use Leibniz for $\ham^i$ in \citep[Lemma D.1]{raphael_quantized_2014}, which ensures for all $i\ge 1$ that there exist constants $c_{i+1}, c_{i+2}$ and a sequence of functions ${(P_{j,i})}_{j=1}^{i}$ such that
  \[
    (y f_{(1)})\sd[i]
    = c_{i+1} y f\sd[i+1] +c_{i+2}f\sd[i] +\sum_{j=1}^{i}P_{j,i}(y)f_{(j)},
    \quad
    |P_{j,i}(y)| \les \frac{1}{\jbk[y]^{2+i-j}}.
  \]
  From the estimates for $f\sd[i+1], f\sd[i]$, and the first inequality of~\eqref{eq:badb-rough-bnd} for the summation part, we have
  \[
    |(y f_{(1)})\sd[i]|
    \les \frac{y^{p_2-i}}{\lgb}\Bigl(F_{p_2-(i+1)} +K_{p_2-(i+1)} +F_{p_2-i} +K_{p_2-i}\Bigr) +y^{p_2-i-2}\jbk[\log y]^{c_{p_2}}.
  \]
  By the monotonicity of $F_j, K_j$ in $j$, we have
  \[
    F_{p_2-(i+1)}\les F_{p_2-i},
    \quad K_{p_2-(i+1)}\les K_{p_2-i},
  \]
  which yields the bound for $|(yf\sd[1])\sd[i]|$. The $b_1$-derivative estimate follows analogously from the second inequality of~\eqref{eq:badb-rough-bnd} and the monotonicity of $\bar{F}_j$ (yielding the same bound with $b_1^{-\ell}\lgb^{-1}$ in place of $\lgb^{-1}$ and $\bar{F}_j$ in place of $F_j$). To treat $(1+\z)f$, we note that the function $(1+\z)$ decays as $O(y^{-2})$ as $y\to+\infty$, and hence $\pry^\ell(1+\z) \les y^{-(2+\ell)}$. By Leibniz and~\eqref{eq:badb-rough-bnd}, we have $(1+\z)f \in \mcAb(p_1,p_2)$. This concludes (iii).

  Proof of~\ref{it:ham-act-badb=b1}. Let $u = b_1\tfrac{\pr f}{\pr b_1}$. For $y\le 1$, the estimate for $u$ is straightforward from
  \[
    \bigg|b_1\frac{\pr \wh{\delta}_{i}}{\pr b_1}\bigg| \les \frac{1}{\lgb} \les 1,
    \andd
    \bigg|\frac{\pr^{\ell}}{\pr b_1^{\ell}}\biggl(b_1\frac{\pr \wh{\delta}_{i}}{\pr b_1}\biggr)\bigg|
    \les \frac{1}{b_1^{\ell}\lgb}, \quad \forall\ell\ge 1.
  \]
  For $y\ge 1$, from the smoothness of $f$, we have $u\sd[i]= b_1\frac{\pr f\sd[i]}{\pr b_1}$, which implies
  \[
    |u\sd[i]|
    \les \frac{y^{p_2-i}}{\lgb}(\bar{F}_{p_2-i} +K_{p_2-i}) + \frac{y^{p_2-i-2}\jbk[\log y]^{c_{p_2}}}{\lgb}.
  \]
  Note that $\bar{F}_{j} \les F_{j}$ for all $j\in\mbZ$, since
  \[
    \frac{\jbk[\log y]}{\lgb} \les \frac{\jbk[\log (\sqrt{b_1} y)]}{\lgb} + \frac{1}{2} \les \jbk[\log (\sqrt{b_1}y)]\one_{j\ge 0} +\one_{j<0},
    \quad 1\le y\le 2S_0,
  \]
  which gives the bound for $u\sd[i]$. Moreover, the bound for the $b_1$-derivatives of $u$ follows directly from the fact $f\in\mcAb(p_1,p_2)$ and Leibniz:
  \[
    \bigg|\frac{\pr^\ell u\sd[i]}{\pr b_1^{\ell}}\bigg|
    \les \sum_{\ell_1+\ell_2=\ell}  \bigg|\frac{\pr^{\ell_1} b_1}{\pr b_1^{\ell_1}}\bigg|\, \bigg|\frac{\pr^{\ell_2+1} f\sd[i]}{\pr b_1^{\ell_2+1}}\bigg|
    \les b_1\bigg|\frac{\pr^{\ell+1} f\sd[i]}{\pr b_1^{\ell+1}}\bigg| + \bigg|\frac{\pr^{\ell} f\sd[i]}{\pr b_1^{\ell}}\bigg|.
  \]
  This concludes the proof.
\end{proof}

\section{Coercivity bounds}
\label{ap:coer-bnd}
In this appendix, we recall some coercivity estimates for the linear operator $\ham$, its conjugate $\ti{\ham}$, and the iterates of $\ham$. These are standard results established in~\citep{raphael_stable_2012, raphael_quantized_2014}. In what follows, we let $u:\mbR^2\to\mbR$ be a radially symmetric scalar function and define $\xm_M$ by~\eqref{eq:local-reson}. The measure in the integrals below is $y\dd y$.

\begin{lemma}[Coercivity of $\ham$ \citep{raphael_quantized_2014}]
  \label{le:coer-ham}
  Let $M$ be a sufficiently large constant. Then there exists $C(M)>0$ such that for any radially symmetric $u$ with
  \[
    \int\biggl(|\pry u|^2 +\frac{|u|^2}{y^2}\biggr) <+\infty
  \]
  and $(u, \xm_M) = 0$, we have
  \[
    (\ham u,u) \ge
    C(M)\int\biggl(|\pry u|^2 +\frac{|u|^2}{y^2}\biggr).
  \]
\end{lemma}

\begin{lemma}[Coercivity of $\ti{\ham}$ \citep{raphael_stable_2012}]
  \label{le:coer-ti-ham}
  Let $\ti{\ham}$ be the conjugate of $\ham$.
  There exists $C>0$ such that for any radial $u$ with
  \[
    \int\biggl(|\pry u|^2 +\frac{u^2}{y^2}\biggr) <+\infty,
  \]
  we have
  \[
    (\ti\ham\, u,u) = \|\as u\|_{L^2}^{2}
    \ge
    C\int\biggl(|\pry u|^2 +\frac{|u|^2}{y^2\jbk[\log y]^2}\biggr).
  \]
\end{lemma}

\begin{lemma}[Weighted coercivity bound \citep{raphael_quantized_2014}]
  \label{le:wgt-coer-bnd}
  Let $L\ge 1$ and $0\le k\le L$. Let $M=M(L)$ be a sufficiently large constant. Then there exists $C(M)>0$ such that for all radial $u$ with
  \[
    \int\frac{|\oa u|^2}{y^6\jbk[y]^{4k}\jbk[\log y]^2} +\int\frac{|u|^2}{y^4\jbk[y]^{4k+4}\jbk[\log y]^2}< +\infty
  \]
  and $(u, \xm_M) = 0$, we have
  \[
    \int\frac{|\ham u|^2}{y^4\jbk[y]^{4k}\jbk[\log y]^2} \ge C(M) \left\{
    \int\frac{|\oa u|^2}{y^6\jbk[y]^{4k}\jbk[\log y]^2} +\int\frac{|u|^2}{y^4\jbk[y]^{4k+4}\jbk[\log y]^2} \right\}.
  \]
\end{lemma}

\begin{lemma}[Coercivity of iterated $\ham$ \citep{raphael_quantized_2014}]
  \label{le:coer-e-2k2}
  Let $L\ge 1$ and $0\le k\le L$. Let $M=M(L)$ be a sufficiently large constant. Then there exists $C(M)>0$ such that for all radial $u$ with
  \begin{align*}
    \int|u_{(2k+2)}|^2 +\int\frac{|u_{(2k+1)}|^2}{y^2\jbk[y]^2}
    & +\sum_{j=1}^{k}\int \frac{|u_{(2j-1)}|^2}{y^6\jbk[y]^{4(k-j)}\jbk[\log y]^2} \\
    & +\sum_{j=0}^{k}\int\frac{|u_{(2j)}|^2}{y^4\jbk[y]^{4(k-j)}\jbk[\log y]^2}
    < +\infty
  \end{align*}
  and $(u, \ham^j\xm_M) = 0$ for $0\le j\le k$, we have
  \begin{align*}
    \int|u_{(2k+2)}|^2
    \ge C(M) \Biggl\{ \int \frac{|u_{(2k+1)}|^2}{y^2\jbk[\log y]^2}
    & +\sum_{j=1}^k\int \frac{|u_{(2j-1)}|^2}{y^6\jbk[y]^{4(k-j)}\jbk[\log y]^2}          \\
    & + \sum_{j=0}^k\int \frac{|u_{(2j)}|^2}{y^4\jbk[y]^{4(k-j)}\jbk[\log y]^2} \Biggr\}.
  \end{align*}
\end{lemma}

\begin{remark}
  \label{re:app-coerc=e-wp}
  By the definition~\eqref{eq:e2k} of the energy, we have
  \[
    \E_{2k+2}
    = \int|\ham^{k+1}\alpha|^2 +\int|\ham^{k+1}\beta|^2
    = \int|\alpha_{(2k+2)}|^2 +\int|\beta_{(2k+2)}|^2,
  \]
  where $\alpha, \beta$ are the first and second components of the radiation $\w$. This, together with the orthogonality~\eqref{eq:ortho}, implies that Lemma~\ref{le:coer-e-2k2} actually provides a coercive estimate for $\E_{2k+2}$ with $0\le k\le L$. In particular, it provides the coercive estimate for $\E_{2L+2}$, which is necessary for establishing the interpolation bounds in Appendix~\ref{ap:interp-bnd}.

\end{remark}

\section{Interpolation bounds}
\label{ap:interp-bnd}
In this appendix, we establish the interpolation bounds for the radiation $\w$, which are used extensively in Section~\ref{sec:energy-est}. These estimates follow from the bootstrap bounds~\eqref{pro:boot-e1-bnd}--\eqref{pro:boot-e2l2-bnd} and the coercivity estimate for $\E_{2L+2}$ in Lemma~\ref{le:coer-e-2k2}.

\begin{lemma}[Interpolation bounds for $\w$]
  \label{le:interp-wp}
  The following estimates hold. The explicit constant $C$ below and the implicit constant in $\les$ may depend on $M,K$.
  \begin{enumerate}[label={\rm (\roman*)}]
    \item weighted Sobolev bound.~\label{it:interp-wp=sobolev}
      For $0\le k\le L$,
      \[
        \sum_{i=0}^{2k+1}\int\frac{|\pry^{i}\wp|^2}{y^2\jbk[y]^{4k+2-2i}\jbk[\log y]^2} \les \E_{2k+2};\footnotemark
      \]
      \footnotetext{The implicit constant here is independent of $K$.}
    \item behavior near the origin.~\label{it:interp-wp=taylor-orig}
      For $y \le 1$, we have $\w=\ws +\rs$ with
      \[
        \ws = \sum_{k=1}^{L+1} \begin{pmatrix}
          c_{\alpha,k} \hs[.1] y^{2k-1} \\
          c_{\beta,k} \hs[.1] y^{2k-1}  \\
          c_{\gamma,k} \hs[.1] y^{2k}
        \end{pmatrix}
        \andd
        \rs = \begin{pmatrix}
          \ra \\ \rb\\ \rc
        \end{pmatrix},
      \]
      where for $1\le k\le L+1$, $0\le \ell\le 2L+1$,
      \begin{align*}
        |c_{\alpha,k}| +|c_{\beta,k}| +b_1^{-1}|c_{\gamma,k}|
        & \les \frac{b_1^{L+1}}{\lgb},                           \\
        |\pry^\ell r_{\alpha}| +|\pry^\ell r_{\beta}| +(b_1y)^{-1}|\pry^\ell r_{\gamma}|
        & \les y^{2L+1-\ell}\jbk[\log y] \frac{b_1^{L+1}}{\lgb}.
      \end{align*}

    \item lossy bound.~\label{it:interp-wp=lossy}
      \[
        \sum_{i=0}^{2k+1} \int\frac{\jbk[\log y]^C}{\jbk[y]^{4k+4-2i}}|\pry^i\w|^2
        \les \lgb^C \begin{cases}
          b_1^{(2k+1)\frac{2L}{2L-1}}, & 0\le k\le L-1,   \\
          b_1^{2L+2}                   & \mbox{for } k=L;
        \end{cases}
      \]
    \item general lossy bounds.~\label{it:interp-wp=gen-lossy}
      Let $(i,j)\in\mbN\times\mbN_{+}$,
      \[
        \int\frac{\jbk[\log y]^C}{\jbk[y]^{2j}}|\pry^{i}\w|^2
        \les \lgb^C \begin{cases}
          b_1^{(i+j-1)\frac{2L}{2L-1}}, & 2\le i+j\le 2L, \\
          b_1^{2L+1},                   & i+j =2L+1,      \\
          b_1^{2L+2},                   & i+j =2L+2;
        \end{cases}
      \]
      moreover,
      \[
        \int\frac{|\pry^i\w|^2}{\jbk[\log y]^2}
        \les \lgb^C \begin{cases}
          b_1^{(i-1)\frac{2L}{2L-1}}, & 2\le i\le 2L,      \\
          b_1^{2L+1},                & i=2L+1,             \\
          b_1^{2L+2},                  & i=2L+2;
        \end{cases}
      \]
    \item far-out pointwise bound.~\label{it:interp-wp=pw}
      Let $(i,j)\in\mbN\times\mbN$,
      \[
        \biggl\|\frac{\pry^i\w}{y^j}\biggr\|^2_{L^{\infty}(y\ge1)}
        \les \lgb^C \begin{cases}
          \dbs                        & \mbox{for $i+j =0$},    \\
          b_1^{(i+j)\frac{2L}{2L-1}}, & 1\le i+j\le 2L-1,       \\
          b_1^{2L+1}                  & \mbox{for $i+j =2L$},   \\
          b_1^{2L+2}                  & \mbox{for $i+j =2L+1$}.
        \end{cases}
      \]
  \end{enumerate}
\end{lemma}

\begin{proof}
  This result extends \citep[Lemma C.1]{raphael_quantized_2014} to the vector-valued setting considered here.
  Indeed, we recall from Remark~\ref{re:app-coerc=e-wp} that
  \[
    \E_{2L+2}[\w]
    = \int|\ham^{L+1}\alpha|^2 +\int|\ham^{L+1}\beta|^2,
  \]
  which coincides with the definition of the top energy $\E_{2L+2}$ in~\cite{raphael_quantized_2014}, if we consider only the first two components $\alpha, \beta$. Therefore, the estimates for \(\alpha\), \(\beta\), and hence for $\wp$, follow directly from \citep[Lemma C.1]{raphael_quantized_2014}. In what follows, we prove~\ref{it:interp-wp=taylor-orig}--\ref{it:interp-wp=pw} for \(\gamma\).

  \step[1, Expression of $\gamma$]
  By~\eqref{eq:constra} and $\hw=\wzt+\w$,
  \[
    \gamma = -(1+\ti{\gamma}_0) +\sqrt{1-|\wztp+\wp|^2}.
  \]
  Recall from~\eqref{pf:form-comp=f} that $f(x) = -1+\sqrt{1-x}$ and $f_{L+1}$ denotes its Taylor polynomial at the origin of degree $(L+1)$. Note from~\eqref{pf:form-comp=gamma} that $\gamma_0 = f_{L+1}(|\wzp|^2)$. Thus $\gamma$ can be rewritten as
  \[
    \gamma = \rrr_1 + \rrr_2 + \rrr_3,
    \with
    \begin{cases}
      \rrr_1 = f_{L+1}(|\wztp|^2) - \chi_{S_1}f_{L+1}(|\wzp|^2), \\
      \rrr_2 = f_{L+1}(|\wztp+\wp|^2) - f_{L+1}(|\wztp|^2),      \\
      \rrr_3 = f(|\wztp+\wp|^2) - f_{L+1}(|\wztp+\wp|^2).
    \end{cases}
  \]
  In addition,
  \begin{equation}
    \begin{split}
      \label{pf:interp-wp=mcj}
      \rrr_2
      & = (2\wztp\cdot\wp +|\wp|^2) \int_0^1 f'_{L+1}(|\wztp|^2 +\sigma(2\wztp\cdot\wp +|\wp|^2))\dd\sigma, \\
      \rrr_3
      & = \frac{|\wztp+\wp|^{2(L+2)}}{(L+1)!} \int_0^1 (1-\sigma)^{L+1} f^{(L+2)}(\sigma|\wztp+\wp|^2)\dd\sigma.
    \end{split}
  \end{equation}

  \step[2, Proof of~\ref{it:interp-wp=taylor-orig}]
  Note that $\rrr_1$ is supported in $y\sim 2S_1$, so for \(y\le 1\),
  \[
    \gamma = \rrr_2 +\rrr_3.
  \]
  The expression of $f_{L+1}$ and the decomposition $\wp = \ws^{\pp} + \rs^{\pp}$ give
  \begin{align*}
    \rrr_2
    & = \sum_{\ell=1}^{L+1} c_{\ell} \bigl(2\ws^{\pp}\cdot \wztp +|\ws^{\pp}|^2\bigr) \sum_{k=0}^{\ell-1} |\wztp+\ws^{\pp}|^{2k} |\wztp|^{2(\ell-1-k)} + \rrr_{21} \\
    & = \sum_{\ell=1}^{L+1} O(|\ws^\pp|) O((|\wztp| +|\ws^\pp|)^{2\ell-1}) + \rrr_{21}.
  \end{align*}
  Here, $\rrr_{21}$ denotes the nonlinear part involving $\rs^{\pp}$. Since $\wztp$ has an odd Taylor expansion and $\ws^{\pp}$ is an odd polynomial in $y$, with factors $b_1$ and $b_1^{L+1}\lgb^{-1}$ respectively, every monomial in the preceding expansion has size at most $b_1^{L+2}\lgb^{-1}$. Collecting these monomials according to their powers of $y$, we obtain
  \[
    \gamma = \sum_{\ell=1}^{L+1} c_{\gamma,\ell} \hs[.1] y^{2\ell} + r_{\gamma}
    \with
    |c_{\gamma,\ell}| \les \frac{b_1^{L+2}}{\lgb},
  \]
  where $r_{\gamma}$ contains three parts: $\rrr_3$, $\rrr_{21}$, and the terms of spatial order at least $2L+4$ in the Taylor expansion of $\rrr_2-\rrr_{21}$.
  It suffices to check that each part of $r_{\gamma}$ satisfies the derivative estimate in~\ref{it:interp-wp=taylor-orig}. The Taylor terms of spatial order at least $2L+4$ satisfy this directly. The same is true for $\rrr_{21}$, due to the regularity of $r_{\alpha}, r_{\beta}$ and the additional factor $b_1$ contributed by $\wztp$, $\wp$. It remains to consider $\rrr_3$. By the smoothness of $f$ and the bound
  \[
    \|\wztp +\wp\|_{L^{\infty}(y\le1)}\les b_1\ll1,
  \]
  the integral in $\rrr_3$ admits an expansion in even powers of $|\wztp+\wp|$. Thus,
  \[
    \rrr_3 = \sum_{k = L+2}^{N} C_{k}|\wztp+\wp|^{2k} + O(|\wztp+\wp|^{2(N+1)})
  \]
  for large enough $N$.
  Therefore, $\rrr_3$ carries a factor of size $b_1^{2(L+2)}$, negligible compared with $b_1^{L+2}\lgb^{-1}$. Moreover, since $\wztp+\ws^{\pp}=O(y)$ near the origin, $\rrr_3$ vanishes to order at least $2(L+2)$ in $y$. Therefore, $\rrr_3$ satisfies~\ref{it:interp-wp=taylor-orig}.

  \step[3, Proof of~\ref{it:interp-wp=lossy}] Denote
  \[
    \mcN_{k,i}[\eta] = \int \frac{\jbk[\log y]^C}{\jbk[y]^{4k+4-2i}}|\pry^i \eta|^2
    \for 0\le i\le2k+1.
  \]
  First, since $\rrr_1$ is supported in $y\sim S_1$, from~\eqref{pf:app-sol=wz-l-infty},
  \begin{equation}
    \label{pf:interp-wp=mcj1}
    \mcN_{k,i}[\rrr_1]
    \les b_1^2\lgb^C\int_{y\sim S_1}\frac{y\dd y}{y^{4k+4}}
    \les b_1^{2k+3}\lgb^C.
  \end{equation}
  Then in view of ~\ref{it:interp-wp=taylor-orig}, both $\rrr_2$ and $\rrr_3$ carry a factor of size $b_1^{L+2}$ for $y\le 1$, and thus
  \[
    \sum_{i=0}^{2k+1}\int_{y\le1}\jbk[\log y]^C
    \bigl(|\pry^i\rrr_2|^2+|\pry^i\rrr_3|^2\bigr)
    \les b_1^{2L+4}.
  \]
  It remains to estimate their $\mcN_{k,i}$ bounds on $y\ge1$. By the Fa\`{a} di Bruno formula and~\eqref{pf:interp-wp=mcj}, for $0\le i\le2L+2$, we have
  \begin{equation}\label{pf:interp-wp=rrr-deriv}
    |\pry^i\rrr_j|
    \les\sum_{n,m,\xi}\prod_{\ell=1}^{m}|\pry^{\xi_\ell}\wztp|
    \prod_{\ell=m+1}^{n}|\pry^{\xi_\ell}\wp|
    \for j=2,3,
  \end{equation}
  where the sum is taken over
  \[
    \begin{cases}
      2\le n\le2L+2,\quad 0\le m<n &\fors j=2,\\
      2L+4\le n\le2L+4+2i,\quad 0\le m\le n &\fors j=3,\\
    \end{cases}
  \]
  and
  \[
    \xi\in\mbN^n,\quad \xi_1+\dots+\xi_n=i.
  \]
  The terms containing only $\wztp$ cancel in the definition of $\rrr_2$, so $m=n$ occurs only for $\rrr_3$. Next, for the terms in $\rrr_3$ with $m=n$, from~\eqref{pf:app-sol=wz-l-infty}, we have
  \[
    \int_{y\ge1}\frac{\jbk[\log y]^C}{\jbk[y]^{4k+4-2i}}
    \prod_{\ell=1}^n|\pry^{\xi_\ell}\wztp|^2
    \les b_1^n\lgb^C\int_1^\infty\frac{\jbk[\log y]^C}{y^{4k+4}}y\dd y
    \les b_1^{2L+4}\lgb^C.
  \]
  Finally, we estimate $\rrr_2$ and the terms in $\rrr_3$ containing $\wp$, which correspond to $m<n$. Choosing $\xi_n\ge\xi_\ell$ for $m+1\le\ell\le n-1$, we have
  \[
    2\xi_\ell\le\xi_\ell+\xi_n\le i\le2k+1
    \quad\Longrightarrow\quad
    \xi_\ell\le k\le L\le2L-1.
  \]
  Thus each product satisfies
  \begin{equation}\begin{split}
      \label{pf:interp-wp=rrr}
      &\int_{y\ge1}\frac{\jbk[\log y]^C}{\jbk[y]^{4k+4-2i}}
      \prod_{\ell=1}^{m}|\pry^{\xi_\ell}\wztp|^2\prod_{\ell=m+1}^{n}|\pry^{\xi_\ell}\wp|^2 \\
      &\les b_1^m\lgb^C
      \prod_{\ell=m+1}^{n-1}\|\pry^{\xi_\ell}\wp\|_{L^\infty(y\ge1)}^2
      \int_{y\ge1}\frac{\jbk[\log y]^C|\pry^{\xi_n}\wp|^2}{\jbk[y]^{2(2k+2-i+\xi_1+\dots+\xi_m)}} \\
      &\les b_1^{d_{n,m,\xi}}\lgb^C,
    \end{split}\end{equation}
  where by~\ref{it:interp-wp=gen-lossy}--\ref{it:interp-wp=pw} for $\wp$, the exponent is given by
  \begin{itemize}
    \item for $k\le L-1$:
      \[
        d_{n,m,\xi}=(2k+1)\frac{2L}{2L-1}+m
        \ge(2k+1)\frac{2L}{2L-1};
      \]
    \item for $k=L$, $\xi_{m+1}+\dots+\xi_{n-1}=0$:
      \[
        d_{n,m,\xi}=2L+2+m\ge2L+2;
      \]
    \item for $k=L$, $\xi_{m+1}+\dots+\xi_{n-1}=1$:
      \[
        d_{n,m,\xi}=\frac{2L}{2L-1}+2L+1+m\ge2L+2;
      \]
    \item for $k=L$, $\xi_{m+1}+\dots+\xi_{n-1}\ge2$:
      \[
        d_{n,m,\xi}=(2L+1)\frac{2L}{2L-1}+m\ge2L+2.
      \]
  \end{itemize}
  Combining the above estimates with the bounds for $y\le1$, we obtain
  \[
    \sum_{i=0}^{2k+1}\bigl(\mcN_{k,i}[\rrr_2]+\mcN_{k,i}[\rrr_3]\bigr)
    \les \lgb^C
    \begin{cases}
      b_1^{(2k+1)\frac{2L}{2L-1}}, & 0\le k\le L-1,    \\
      b_1^{2L+2}                   & \fors k=L,
    \end{cases}
  \]
  Together with~\eqref{pf:interp-wp=mcj1}, this proves~\ref{it:interp-wp=lossy}.

  \step[4, Proof of~\ref{it:interp-wp=gen-lossy}, Part I]
  We first prove the second inequality in~\ref{it:interp-wp=gen-lossy}.
  For $y\le1$, from~\ref{it:interp-wp=taylor-orig}, we already have
  \[
    \sum_{i=0}^{2L+1}\int_{y\le1}\frac{|\pry^i\gamma|^2}{\jbk[\log y]^2}
    \les b_1^{2L+4}.
  \]
  For the top derivative, using the constraint~\eqref{eq:constra} we compute
  \[
    \biggl|\pry^{2L+2}\gamma
    +\frac{\hw^\pp\cdot\pry^{2L+2}\wp}{1+\wh\gamma}\biggr|
    \les \jbk[\log y] b_1^{L+2}\lgb^{-1},
  \]
  where the terms containing lower derivatives are controlled by~\ref{it:interp-wp=taylor-orig} for $\wp$. Writing $\pry^{2L+2}\wp$ in terms of $\ham^{L+1}\wp$ and lower-order derivatives, and using~\ref{it:interp-wp=sobolev}, we have
  \[
    \int_{y\le1}|\pry^{2L+2}\wp|^2
    \les\E_{2L+2}+\int_{y\le2}|\wp|^2
    \les \E_{2L+2}.
  \]
  Thus from~\ref{it:interp-wp=taylor-orig} and~\eqref{pro:boot-e2l2-bnd}, we obtain
  \[
    \int_{y\le1}|\pry^{2L+2}\gamma|^2
    \les b_1^2\E_{2L+2}+\frac{b_1^{2L+4}}{\lgb^2}
    \les \frac{b_1^{2L+4}}{\lgb^2}.
  \]
  For $y\ge1$, we use $\gamma=\rrr_1+\rrr_2+\rrr_3$ from Step~1. By~\eqref{pf:interp-wp=mcj}, every term in $\rrr_2$ contains at least one factor of $\wp$, and $\rrr_3$ contains at least $2L+4$ factors of either $\wztp$ or $\wp$. We therefore estimate $\rrr_1$, the terms in $\rrr_3$ containing only $\wztp$, and $\rrr_2$ together with the remaining terms in $\rrr_3$ respectively.

  First, since $\rrr_1$ is supported in $y\sim S_1$, from~\eqref{pf:app-sol=wz-l-infty}, as in Step~3, we have
  \[
    \int_{y\ge1}\frac{|\pry^i\rrr_1|^2}{\jbk[\log y]^2}
    \les b_1^2\lgb^C\int_{y\sim S_1}\frac{y\dd y}{y^{2i}}
    \les b_1^{i+1}\lgb^C.
  \]
  Next, we treat the terms in $\rrr_3$ containing only $\wztp$, which correspond to $m=n$ in~\eqref{pf:interp-wp=rrr-deriv}. Again by~\eqref{pf:app-sol=wz-l-infty}, we have
  \[
    \int_{y\ge1}\frac{1}{\jbk[\log y]^2}\prod_{\ell=1}^n|\pry^{\xi_\ell}\wztp|^2
    \les b_1^n\lgb^C\int_1^\infty \frac{y\dd y}{y^{2i}}
    \les b_1^{2L+4}\lgb^C.
  \]
  Finally, we treat $\rrr_2$ and the terms in $\rrr_3$ containing $\wp$, which correspond to $m<n$ in~\eqref{pf:interp-wp=rrr-deriv}. We choose $\xi_n\ge\xi_\ell$ for $m+1\le\ell\le n-1$.

  For $L\ge2$, this implies
  \[
    2\xi_{\ell} \le\xi_{\ell} +\xi_n\le i
    \quad\Longrightarrow\quad \xi_{\ell}\le\frac{i}{2}\le L+1\le2L-1.
  \]
  Then each product satisfies
  \begin{align*}
    &\int_{y\ge1}\frac{1}{\jbk[\log y]^2}\prod_{\ell=1}^m|\pry^{\xi_\ell}\wztp|^2
    \prod_{\ell=m+1}^n|\pry^{\xi_\ell}\wp|^2 \\
    &\les b_1^m\lgb^C
    \prod_{\ell=m+1}^{n-1}\|\pry^{\xi_\ell}\wp\|_{L^\infty(y\ge1)}^2
    \int_{y\ge1}\frac{|\pry^{\xi_n}\wp|^2}
    {y^{2(\xi_1+\dots+\xi_m)}\jbk[\log y]^2} \\
    &\les b_1^{d_{n,m,\xi}}\lgb^C,
  \end{align*}
  where by~\ref{it:interp-wp=gen-lossy}--\ref{it:interp-wp=pw} for $\wp$, the exponent is given by
  \begin{itemize}
    \item for $\xi_n+(\xi_1+\dots+\xi_m)=1$: this implies $\xi_n=1$ and $\xi_1=\dots=\xi_m=0$; using $\|\pry\wp\|_{L^2}^2\les\E_1\ll1$, we obtain
      \[
        d_{n,m,\xi}=\frac{2L}{2L-1}\sum_{\ell=m+1}^{n-1}\xi_{\ell} +m
        =(i-1)\frac{2L}{2L-1}+m;
      \]
    \item for $2\le\xi_n+(\xi_1+\dots+\xi_m)\le2L$:
      \[
        d_{n,m,\xi}=(i-1)\frac{2L}{2L-1}+m
        \ge(i-1)\frac{2L}{2L-1};
      \]
    \item for $\xi_n+(\xi_1+\dots+\xi_m)=2L+1$:
      \[
        d_{n,m,\xi}=\frac{2L}{2L-1}\sum_{\ell=m+1}^{n-1}\xi_{\ell} +2L+1+m
        \ge i;
      \]
    \item for $\xi_n+(\xi_1+\dots+\xi_m)=2L+2$: this implies $i=2L+2$ and $\xi_{m+1}=\dots=\xi_{n-1}=0$, and thus
      \[
        d_{n,m,\xi}=2L+2+m\ge i.
      \]
  \end{itemize}

  For $L=1$, the same estimates apply except when $i=4$ and two factors of $\wp$ each carry two derivatives. In that case, we have
  \[
    \|\pry^2\wp\|_{L^\infty(y\ge1)}^2
    \int_{y\ge1}\frac{|\pry^2\wp|^2}{\jbk[\log y]^2}
    \les b_1^3\lgb^C\,b_1^2\lgb^C
    \les b_1^5\lgb^C.
  \]
  Combining the above estimates and using
  $(i-1)\frac{2L}{2L-1}\ge i$ for $i=2L+1,2L+2$, we obtain
  \begin{equation}
    \label{pf:interp-wp=gamma-log}
    \int\frac{|\pry^i\gamma|^2}{\jbk[\log y]^2}
    \les\lgb^C\begin{cases}
      b_1^{(i-1)\frac{2L}{2L-1}},&2\le i\le2L,\\
      b_1^{2L+1},&i=2L+1,\\
      b_1^{2L+2},&i=2L+2,
    \end{cases}
  \end{equation}
  which proves the second inequality in~\ref{it:interp-wp=gen-lossy}.

  \step[5, Proof of~\ref{it:interp-wp=gen-lossy}, Part II]
  We now prove the first inequality in~\ref{it:interp-wp=gen-lossy}. If $i+j=2k+2$ for $0\le k\le L$, then $i\le2k+1$, and from~\ref{it:interp-wp=lossy} we obtain the desired bound:
  \[
    \int\frac{\jbk[\log y]^C|\pry^i\gamma|^2}{\jbk[y]^{2j}}
    =\int\frac{\jbk[\log y]^C|\pry^i\gamma|^2}{\jbk[y]^{4k+4-2i}}
    \les\lgb^C\begin{cases}
      b_1^{(2k+1)\frac{2L}{2L-1}},&k\le L-1,\\
      b_1^{2L+2} &\fors k=L.
    \end{cases}
  \]
  If $i+j=2k+1$ with $1\le k\le L$ and $j\ge2$, then by Cauchy--Schwarz,
  \begin{align*}
    \int\frac{\jbk[\log y]^C|\pry^i\gamma|^2}{\jbk[y]^{2j}}
    &\le\biggl(\int\frac{\jbk[\log y]^C|\pry^i\gamma|^2}{\jbk[y]^{2j+2}}\biggr)^{1/2}
    \biggl(\int\frac{\jbk[\log y]^C|\pry^i\gamma|^2}{\jbk[y]^{2j-2}}\biggr)^{1/2}.
  \end{align*}
  Now for $j=1$, we have $i=2k$ and thus
  \[
    \int\frac{\jbk[\log y]^C|\pry^{2k}\gamma|^2}{\jbk[y]^2}
    \le\biggl(\int\frac{\jbk[\log y]^{2C+2}|\pry^{2k}\gamma|^2}{\jbk[y]^4}\biggr)^{1/2}
    \biggl(\int\frac{|\pry^{2k}\gamma|^2}{\jbk[\log y]^2}\biggr)^{1/2}.
  \]
  From~\ref{it:interp-wp=lossy} and~\eqref{pf:interp-wp=gamma-log}, in both cases we obtain
  \begin{align*}
    \int\frac{\jbk[\log y]^C|\pry^i\gamma|^2}{\jbk[y]^{2j}}
    &\les\lgb^C\begin{cases}
      \Bigl(b_1^{(2k+1)\frac{2L}{2L-1}}b_1^{(2k-1)\frac{2L}{2L-1}}\Bigr)^{1/2},&k\le L-1,\\
      \left(b_1^{2L+2}b_1^{2L}\right)^{1/2} &\fors k=L,
    \end{cases}\\
    &=\lgb^C\begin{cases}
      b_1^{2k\frac{2L}{2L-1}},&k\le L-1,\\
      b_1^{2L+1} &\fors k=L,
    \end{cases}
  \end{align*}
  which proves the first inequality in~\ref{it:interp-wp=gen-lossy}.

  \step[6, Proof of~\ref{it:interp-wp=pw}]
  If $i+j\ge 1$, by Cauchy--Schwarz we have
  \begin{align*}
    \biggl\|\frac{\pry^i\gamma}{y^j}\biggr\|_{L^\infty(y\ge1)}^2
    &\les\int_{y\ge1}\frac{\jbk[\log y]^2|\pry^i\gamma|^2}{y^{2j+2}}
    +\int_{y\ge1}\frac{|\pry^{i+1}\gamma|^2}{y^{2j}\jbk[\log y]^2}\\
    &\les\lgb^C\begin{cases}
      b_1^{(i+j)\frac{2L}{2L-1}},&1\le i+j\le2L-1,\\
      b_1^{2L+1},&\fors i+j=2L,\\
      b_1^{2L+2},&\fors i+j=2L+1,
    \end{cases}
  \end{align*}
  where we use the first inequality in~\ref{it:interp-wp=gen-lossy} for the first integral and for the second when $j\ge1$, and~\eqref{pf:interp-wp=gamma-log} when $j=0$. If $i=j=0$, from~\eqref{eq:constra} and the smallness of $\wztp$ and $\wp$, we have directly
  \[
    \|\gamma\|_{L^\infty}^2
    \les\|\ti\gamma_0\|_{L^\infty}^2+\|\wztp+\wp\|_{L^\infty}^4
    \les\dbs,
  \]
  which proves~\ref{it:interp-wp=pw}.
  This completes the proof.
\end{proof}

\begin{lemma}[Refined interpolation for $\gamma$]
  \label{le:interp-gamma}
  The following estimates hold. The explicit constant $C$ below and the implicit constant in $\les$ may depend on $M,K$.
  \begin{enumerate}[label={\rm (\roman*)}]
    \item $\dot{H}^1$ bound:
      \label{it:interp-gamma=h1}
      \[
        \|\pry \gamma\|_{L^2}^2 \les \dbs\E_1 +b_1^2\lgb^C;
      \]
    \item unweighted bound:
      \label{it:interp-gamma=unweighted}
      \[
        \int_{y\ge1}|\pry^i\gamma|^2
        \les \lgb^C\begin{cases}
          b_1^{(i-1)\frac{2L}{2L-1}}, & 2\le i\le2L, \\
          b_1^{2L+1}, & i=2L+1, \\
          b_1^{2L+2}, & i=2L+2.
        \end{cases}
      \]
    \item weighted Sobolev bound:
      \label{it:interp-gamma=sobolev}
      \[
        \sum_{i=0}^{2L+2}\int\frac{|\pry^i\gamma|^2}{\jbk[y]^{4L+4-2i}} \les \dbs\E_{2L+2} +b_1^{\frac{1}{4L}}\frac{b_1^{2L+2}}{\lgb^2}.
      \]
  \end{enumerate}
\end{lemma}

\begin{proof}
  We prove~\ref{it:interp-gamma=h1} by taking one derivative of the constraint~\eqref{eq:constra}. For~\ref{it:interp-gamma=unweighted} and~\ref{it:interp-gamma=sobolev}, we refine the estimates in the proof of Lemma~\ref{le:interp-wp}.

  \step[1, Proof of~\ref{it:interp-gamma=h1}]
  By the constraint~\eqref{eq:constra} and $\hw=\wzt+\w$,
  \[
    \pry\gamma
    =-\pry\ti{\gamma}_0
    -\frac{(\wztp+\wp)\cdot\pry(\wztp+\wp)}{\sqrt{1-|\wztp+\wp|^2}}.
  \]
  Note that the denominator is bounded away from zero by Lemma~\ref{le:interp-wp} and the bootstrap bounds. Moreover, from \eqref{pf:app-sol=wz-l-infty}, the construction of $\wzt$ and Lemma~\ref{le:l2-norm-homo},
  \[
    \|\wztp\|_{L^\infty}^2 + \|\pry\wztp\|_{L^2}^2\les b_1\lgb^C,
    \quad
    \|\pry\ti{\gamma}_0\|_{L^2}^2\les b_1^2\lgb^C.
  \]
  Together with $\|\wp\|_{L^\infty}^2\les \E_1$ by the radial Sobolev inequality, this gives
  \begin{align*}
    \|\pry\gamma\|_{L^2}^2
    & \les \|\pry\ti{\gamma}_0\|_{L^2}^2 + \bigl(\|\wztp\|_{L^\infty}^2 + \|\wp\|_{L^\infty}^2\bigr) \bigl(\|\pry\wztp\|_{L^2}^2 + \|\pry\wp\|_{L^2}^2\bigr) \\
    & \les \dbs\E_1+b_1^2\lgb^C.
  \end{align*}

  \step[2, Proof of~\ref{it:interp-gamma=unweighted}]
  We first establish the corresponding bounds for $\wp$. Writing $\pry^{2k}\wp$ in terms of $\ham^k\wp$ and lower-order derivatives, we have
  \[
    \int_{y\ge1}|\pry^{2k}\wp|^2
    \les\E_{2k}+\sum_{j=0}^{2k-1}\int_{y\ge1}\frac{|\pry^j\wp|^2}{y^{2(2k-j)}}
    \for 1\le k\le L+1.
  \]
  By Lemma~\itref{le:interp-wp}{it:interp-wp=gen-lossy} for $\wp$ and the bootstrap bound~\ref{pro:boot-e2k-bnd}--\ref{pro:boot-e2l2-bnd}, this gives the even orders in~\eqref{pf:interp-gamma=planar-unweighted} below. For $1\le k\le L$, integration by parts and Cauchy--Schwarz give
  \[
    \int_{y\ge1}|\pry^{2k+1}\wp|^2
    \les\biggl(\int_{y\ge1}|\pry^{2k}\wp|^2\biggr)^{1/2}
    \biggl(\int_{y\ge1}|\pry^{2k+2}\wp|^2\biggr)^{1/2}
    +\frac{b_1^{2L+2}}{\lgb^2},
  \]
  where the boundary terms at $y=1$ are controlled by Lemma~\itref{le:interp-wp}{it:interp-wp=taylor-orig} for $\wp$. Thus
  \begin{equation}\label{pf:interp-gamma=planar-unweighted}
    \int_{y\ge1}|\pry^i\wp|^2
    \les\lgb^C\begin{cases}
      b_1^{(i-1)\frac{2L}{2L-1}},&2\le i\le2L,\\
      b_1^{2L+1},&i=2L+1,\\
      b_1^{2L+2},&i=2L+2.
    \end{cases}
  \end{equation}
  For $\gamma$, we repeat the estimates for $y\ge1$ in Step~4 of the proof of Lemma~\ref{le:interp-wp}, omitting the factor $\jbk[\log y]^{-2}$. The only change is the estimate of
  \[
    \int_{y\ge1}\frac{|\pry^{\xi_n}\wp|^2}{y^{2(\xi_1+\dots+\xi_m)}}.
  \]
  For $\xi_1+\dots+\xi_m>0$, we use Lemma~\itref{le:interp-wp}{it:interp-wp=gen-lossy} for $\wp$. For $\xi_1+\dots+\xi_m=0$, we use~\eqref{pf:interp-gamma=planar-unweighted} if $\xi_n\ge2$, and $\|\pry\wp\|_{L^2}^2\les\E_1$ if $\xi_n=1$. The exponents $d_{n,m,\xi}$ and the bound for $L=1$ are unchanged, yielding~\ref{it:interp-gamma=unweighted}.

  \step[3, Proof of~\ref{it:interp-gamma=sobolev}]
  For $y\le1$, by Lemma~\itref{le:interp-wp}{it:interp-wp=taylor-orig} and the estimate for $\pry^{2L+2}\gamma$ in Step~4 of its proof, we have
  \begin{equation}\label{pf:interp-gamma=local-endpoint}
    \sum_{i=0}^{2L+2}\int_{y\le1}|\pry^i\gamma|^2
    \les\frac{b_1^{2L+4}}{\lgb^2}.
  \end{equation}
  For $y\ge1$, we again use $\gamma=\rrr_1+\rrr_2+\rrr_3$. First, from~\eqref{pf:app-sol=wz-l-infty} and the support of $\rrr_1$, we have
  \[
    \sum_{i=0}^{2L+2}\int_{y\ge1}\frac{|\pry^i\rrr_1|^2}{\jbk[y]^{4L+4-2i}}
    \les b_1^2\lgb^C\int_{y\sim S_1}\frac{y\dd y}{y^{4L+4}}
    \les b_1^{2L+3}\lgb^C.
  \]
  Next, for the terms in $\rrr_3$ containing only $\wztp$, which correspond to $m=n$ in~\eqref{pf:interp-wp=rrr-deriv}, we have
  \[
    \int_{y\ge1}\frac{1}{\jbk[y]^{4L+4-2i}}\prod_{\ell=1}^n|\pry^{\xi_\ell}\wztp|^2
    \les b_1^n\lgb^C\int_1^\infty\frac{y\dd y}{y^{4L+4}}
    \les b_1^{2L+4}\lgb^C.
  \]
  Finally, for $\rrr_2$ and the terms in $\rrr_3$ containing $\wp$, we use~\eqref{pf:interp-wp=rrr-deriv} with $m<n$. We choose $\xi_n\ge\xi_\ell$ for $m+1\le\ell\le n-1$.

  For $L\ge2$, we have
  \[
    2\xi_\ell\le\xi_\ell+\xi_n\le i
    \quad\Longrightarrow\quad
    \xi_\ell\le\frac{i}{2}\le L+1\le2L-1.
  \]
  Except when $m=0$ and $\xi_{m+1}+\dots+\xi_{n-1}=0$, each product integral satisfies
  \begin{align*}
    &\int_{y\ge1}\frac{1}{\jbk[y]^{4L+4-2i}}\prod_{\ell=1}^m|\pry^{\xi_\ell}\wztp|^2
    \prod_{\ell=m+1}^n|\pry^{\xi_\ell}\wp|^2\\
    &\les b_1^m\lgb^C
    \prod_{\ell=m+1}^{n-1}\|\pry^{\xi_\ell}\wp\|_{L^\infty(y\ge1)}^2
    \int_{y\ge1}\frac{|\pry^{\xi_n}\wp|^2}
    {y^{2(2L+2-i+\xi_1+\dots+\xi_m)}} \\
    &\les b_1^{d_{n,m,\xi}}\lgb^C.
  \end{align*}
  The last integral has total order
  \[
    1\le\xi_n+2L+2-i+(\xi_1+\dots+\xi_m)
    =2L+2-(\xi_{m+1}+\dots+\xi_{n-1})\le2L+2.
  \]
  For total order at least two, we use Lemma~\itref{le:interp-wp}{it:interp-wp=gen-lossy} for $\wp$ when the power of $y$ in the denominator is positive, and~\eqref{pf:interp-gamma=planar-unweighted} when it is zero. Together with Lemma~\itref{le:interp-wp}{it:interp-wp=pw} for $\wp$, this gives the following exponents:
  \begin{itemize}
    \item for $m\ge1$, $\xi_{m+1}+\dots+\xi_{n-1}=0$:
      \[
        d_{n,m,\xi}=2L+2+m\ge2L+3;
      \]
    \item for $\xi_{m+1}+\dots+\xi_{n-1}=1$:
      \[
        d_{n,m,\xi}=\frac{2L}{2L-1}+2L+1+m
        \ge2L+2+\frac1{2L-1};
      \]
    \item for $2\le\xi_{m+1}+\dots+\xi_{n-1}\le2L$:
      \[
        d_{n,m,\xi}=(2L+1)\frac{2L}{2L-1}+m
        \ge2L+2+\frac2{2L-1};
      \]
    \item for $\xi_{m+1}+\dots+\xi_{n-1}=2L+1$, we have $\xi_n=1$, $i=2L+2$ and $\xi_1=\dots=\xi_m=0$; using $\|\pry\wp\|_{L^2}^2\les\E_1\ll1$, we obtain
      \[
        d_{n,m,\xi}=(2L+1)\frac{2L}{2L-1}+m
        \ge2L+2+\frac2{2L-1}.
      \]
  \end{itemize}
  It remains to consider $m=0$ and $\xi_{m+1}+\dots+\xi_{n-1}=0$, so that $\xi_n=i$. We keep one factor of $\wp$ inside the integral and bound the other factors of order zero by $\|\wp\|_{L^\infty}\les1$. For $i=2L+2$, writing $\pry^{2L+2}\wp$ in terms of $\ham^{L+1}\wp$ and lower-order derivatives, we have
  \begin{align*}
    \int_{y\ge1}|\wp|^2|\pry^{2L+2}\wp|^2
    &\les\|\wp\|_{L^\infty}^2\E_{2L+2}
    +\left\|\frac{\wp}{y}\right\|_{L^\infty(y\ge1)}^2
    \sum_{j=0}^{2L+1}\int_{y\ge1}\frac{|\pry^j\wp|^2}{y^{2(2L+1-j)}}\\
    &\les\dbs\E_{2L+2}+b_1^{2L+1+\frac{2L}{2L-1}}\lgb^C.
  \end{align*}
  For $i\le2L+1$, we have
  \[
    \int_{y\ge1}\frac{|\wp|^2|\pry^i\wp|^2}{y^{2(2L+2-i)}}
    \le\left\|\frac{\wp}{y}\right\|_{L^\infty(y\ge1)}^2
    \int_{y\ge1}\frac{|\pry^i\wp|^2}{y^{2(2L+1-i)}}
    \les b_1^{2L+1+\frac{2L}{2L-1}}\lgb^C.
  \]

  For $L=1$, the same estimates apply except when $i=4$ and two factors of $\wp$ each carry two derivatives. As in Step~2, this contribution is bounded by $b_1^5\lgb^C$.
  Combining the above estimates with~\eqref{pf:interp-gamma=local-endpoint}, and using $\frac1{2L-1}>\frac1{4L}$, we obtain
  \[
    \sum_{i=0}^{2L+2}\int\frac{|\pry^i\gamma|^2}{\jbk[y]^{4L+4-2i}}
    \les\dbs\E_{2L+2}+b_1^{\frac1{4L}}\frac{b_1^{2L+2}}{\lgb^2},
  \]
  which proves~\ref{it:interp-gamma=sobolev}. This completes the proof.
\end{proof}

\section{Pointwise bounds for iterated operators}
\label{ap:proof-omega-pw}

\begin{proof}[Proof of Lemma~\ref{le:omega-pw-bnd}]
  We prove the lemma by induction on $j$.

  \emph{Case $j=1$}: By definition $\rho\in\{0,1\}$. If $\rho=0$, then $n=0$, $\xi=(\xi_1,\xi_2)\in\mbN^2$. From the definition~\eqref{eq:mham} of $\mham$,
  \[
    |\pry^{i}\Omega^{1}_{\rho}\w| = |\pry^{i}\R\mham\w|
    \le |\pry^{i}\ham\beta|  +|\pry^{i}(\ham\alpha -\zt\pry\gamma)|,
    \quad i\ge 0.
  \]
  Together with the asymptotics
  \begin{equation}\label{pf:omega-pw=asymp}
    |\pry^{\ell}\z| \les \frac{1}{y^{\ell}}, \quad
    |\pry^{\ell}\zt|+ |\pry^{\ell}(V-1)| \les \frac{1}{y^{2+\ell}},
    \quad \ell\ge 0, \; y\ge 1,
  \end{equation}
  this implies that there exists $C(i,1)\ge 1$ such that
  \[
    |\pry^{i}\Omega^{1}_{\rho}\w|
    \le C(i,1) \sum_{\ell=0}^{i+2}\frac{|\pry^{\ell}\w|}{y^{i+2-\ell}}.
  \]
  Setting $\xi_1=\ell$ and $\xi_2=i+2-\ell$ yields the desired estimate.

  If $\rho=1$, then $n=1$, $\xi=(\xi_1,\xi_2,\xi_3)\in\mbN^3$. Similarly, from \eqref{pf:omega-pw=asymp},
  \begin{align*}
    & |\pry^{i}\Omega^{1}_{\rho}\w|
    = |\pry^{i}(\vv\times\mham\w)|                                                                            \\
    & \le C(i,1) \sum_{i_1+i_2=i} |\pry^{i_1}\vv| \biggl\{
    |\pry^{i_2}\Delta\w|
    +\biggl|\pry^{i_2}\biggl(\frac{V\w}{y^2}\biggr)\biggr|
    +|\pry^{i_2}(\zt\pry\w)|
    +\biggl|\pry^{i_2}\biggl(\frac{\z\zt\w}{y}\biggr)\biggr| \biggr\}                                         \\
    & \le C(i,1) \sum_{i_1+i_2=i} |\pry^{i_1}\vv| \sum_{i_3=0}^{i_2+2} \frac{|\pry^{i_3}\w|}{y^{i_2+2-i_3}}.
  \end{align*}
  Setting $\xi_1=i_1$, $\xi_2=i_3$ and $\xi_3=i_2+2-i_3$ yields the desired estimate. This proves the case $j=1$.

  \emph{Case $j\to j+1$}: For any $\rho\in\{0,1\}^{j+1}$, we let $\sigma\in\{0,1\}^j$ be such that $\sigma_\ell = \rho_\ell$ for all $1\le \ell\le j$. Then one of the following two must hold:
  \[
    \rho = (\sigma_{1},\dots,\sigma_{j}, 0)
    \quad\mbox{or}\quad
    \rho = (\sigma_{1},\dots,\sigma_{j}, 1).
  \]
  If $\rho = (\sigma_{1},\dots,\sigma_{j}, 0)$, then we have $|\sigma|_1= |\rho|_1 =n$ and
  \[
    |\pry^{i}\Omega^{j+1}_{\rho}\w|
    = |\pry^{i}\Omega^{j}_{\sigma}(\R\mham\w)|
    \quad\mbox{for all}\quad i\ge 0.
  \]
  By induction, we have
  \[
    |\pry^{i}\Omega^{j}_{\sigma}(\R\mham\w)|
    \le C(i,j) \sum_{\xi}
    \Biggl(\prod_{1\le k\le n}|\pry^{\xi_{k}}\vv|\Biggr)
    \frac{|\pry^{\xi_{n+1}}(\R\mham\w)|}{y^{\xi_{n+2}}},
  \]
  where the sum is taken over
  \[
    \xi\in\mbN^{n+2},\quad |\xi|_1=i+2j,\quad |\xi|_+\ge2.
  \]
  From the result for $j=1$,
  \[
    |\pry^{\xi_{n+1}}(\R\mham\w)|
    \le C(\xi_{n+1},1) \sum_{\ell=0}^{\xi_{n+1}+2}\frac{|\pry^{\ell}\w|}{y^{\xi_{n+1}+2-\ell}}.
  \]
  We put these estimates together and use the following summation index:
  \[
    \zeta = (\xi_1, \dots, \xi_{n}, \ell, \xi_{n+1}+\xi_{n+2}+2-\ell)\in\mbN^{n+2},
  \]
  which satisfies
  \[
    |\zeta|_1 = |\xi|_1 +2 = i+2(j+1)
    \andd
    |\zeta|_+ = \xi_{n+1}+\xi_{n+2}+2 \ge 2.
  \]
  This yields
  \begin{align*}
    |\pry^{i}\Omega^{j+1}_{\rho}\w|
    \le C(i,j+1)
    \sum_{\substack{|\zeta|_1=i+2(j+1) \\ |\zeta|_{+}\ge 2}}
    \biggl(\prod_{1\le k\le n}|\pry^{\zeta_{k}}\vv|\biggr)
    \frac{|\pry^{\zeta_{n+1}}\w|}{y^{\zeta_{n+2}}},
  \end{align*}
  where $C(i,j+1)\ge1$ depends only on $i,j$. This yields the desired estimate for $\rho=(\sigma_1,\dots,\sigma_j,0)$.

  Finally, if $\rho=(\sigma_1,\dots,\sigma_j,1)$, then $|\sigma|_1=|\rho|_1-1=n-1$. Again by induction,
  \begin{align*}
    |\pry^{i}\Omega^{j+1}_{\rho}\w|
    & = |\pry^{i}\Omega^{j}_{\sigma}(\vv\times\mham\w)| \\
    & \le C(i,j) \sum_{\xi}
    \biggl(\prod_{1\le k\le n-1}|\pry^{\xi_{k}}\vv|\biggr) \frac{|\pry^{\xi_{n}}(\vv\times\mham\w)|}{y^{\xi_{n+1}}},
  \end{align*}
  where the sum is taken over
  \[
    \xi\in\mbN^{n+1},\quad |\xi|_1=i+2j,\quad |\xi|_+\ge2.
  \]
  From the result for $j=1$,
  \[
    |\pry^{\xi_{n}}(\vv\times\mham\w)|
    \le C(\xi_{n},1)
    \sum_{\ell_1+\ell_2=\xi_{n}} |\pry^{\ell_1}\vv|
    \sum_{\ell_3=0}^{\ell_2+2}
    \frac{|\pry^{\ell_3}\w|}{y^{\ell_2+2-\ell_3}}.
  \]
  We put these together and set the following summation index:
  \[
    \zeta=(\xi_1, \dots, \xi_{n-1}, \ell_1, \ell_3, \xi_{n+1}+\ell_2+2-\ell_3)\in\mbN^{n+2},
  \]
  which satisfies
  \[
    |\zeta|_1 = |\xi|_1 +2 = i+2(j+1)
    \andd
    |\zeta|_+ = \xi_{n+1} +\ell_2 +2 \ge 2.
  \]
  This yields
  \[
    |\pry^{i}\Omega^{j+1}_{\rho}\w|
    \le C(i,j+1)
    \sum_{\substack{|\zeta|_1=i+2(j+1)\\ |\zeta|_+\ge 2}}
    \biggl(\prod_{1\le k\le n} |\pry^{\zeta_{k}}\vv|\biggr)
    \frac{|\pry^{\zeta_{n+1}}\w|}{y^{\zeta_{n+2}}},
  \]
  where $C(i,j+1)$ is enlarged if necessary to cover both cases. Finally, redefining the constants by
  \[
    C(i,j) = \max_{0\le k\le i}\max_{1\le\ell\le j} C(k,\ell)
  \]
  ensures that $C(i,j)$ is monotone increasing in $i,j$. This completes the proof.
\end{proof}

\section{Lower-order mixed energy estimates}
\label{ap:proof-mix-low-eng}
This appendix is devoted to the proof of Lemma~\ref{le:mix-low-eng}, which is used for closing the bootstrap in the proof of Proposition~\ref{pro:boot}.
\begin{lemma*}[Mixed energy estimate for $\E_{2k+2}$]
  For any $0\le k\le L-1$, there exists a constant $C > 1$ independent of $K, M$, such that
  \[
    \frac{\dd}{\dd t}\biggl\{\frac{1}{\la^{4k+2}}\biggl[\E_{2k+2} + O\Bigl(b_1^{\frac12+(2k+1)\frac{2L}{2L-1}}\Bigr)\biggr]\biggr\}
    \le \frac{C(M,K)}{\la^{4k+4}} b_1^{1+(2k+1)\frac{2L}{2L-1}}\lgb^{K/2+C}.
  \]
\end{lemma*}

\begin{proof}
  From the equation~\eqref{eq:w2k-eq} for $\W_{2k}^0$, we compute
  \begin{align*}
    \frac{\dd}{\dd t} \int|\W_{2k+2}^0|^2
    & = -2\int\W_{2k+2}^0 \cdot \R\mhaml\W_{2k+2}^0
    +2\int\W_{2k+2}^0 \cdot \bigl[\prt, (\R\mhaml)^{k+1}\bigr]\Wp \\
    & \quad +2\int\W_{2k+2}^0\cdot (\R\mhaml)^{k+1}\R^2\W_2^1
    -\frac{2}{\la^2} \int\W_{2k+2}^0\cdot (\R\mhaml)^{k+1}\mF_{\la}^{\pp}.
  \end{align*}
  Note that the first term on the right-hand side vanishes
  \[
    \int\W_{2k+2}^0 \cdot \R\mhaml\W_{2k+2}^0
    =\frac{1}{\la^{4k+4}}\int\ma\w_{2k+2}^0\cdot \R\ma\w_{2k+2}^0 =0.
  \]
  It suffices to control the remaining terms.

  \step[1, Morawetz correction]
  Consider the commutator term. By induction, we have
  \[
    \bigl[\prt, (\R\mhaml)^{k+1}\bigr]
    = \sum_{\ell=0}^k (\R\mhaml)^{\ell}\bigl[\prt, \R\mhaml\bigr](\R\mhaml)^{k-\ell},
  \]
  which together with~\eqref{pf:eng=ma-mham-comm} yields
  \begin{align*}
    \int\W_{2k+2}^0 \cdot \bigl[\prt, & (\R\mhaml)^{k+1}\bigr]\Wp
    = \frac{b_1}{\la^{4k+4}} \sum_{\ell=0}^{k}\int\w_{2k+2}^0\cdot \R^{k+1}\mhamp^\ell\bigl(g_2\mhamp^{k-\ell}\wp\bigr)                                                                            \\
    & \quad -\frac{1}{\la^{4k+4}}\biggl(b_1+\frac{\la_s}{\la}\biggr) \sum_{\ell=0}^{k}\int\w_{2k+2}^0\cdot \R^{k+1}\mhamp^\ell\bigl(g_2\mhamp^{k-\ell}\wp\bigr).
  \end{align*}
  To treat the first term on the right-hand side above, we apply the low-order Morawetz $\mora_k$ defined by~\eqref{def:mora-func}. We note that the proof of Lemma~\itref{le:mora-est}{le:mora-est=eq} is purely algebraic: it uses only the equation~\eqref{eq:w2k-eq} and the structure~\eqref{pf:vr-est=a-t-p}, and no energy bound. Hence a direct replacement of $L$ with $k$ for Lemma~\itref{le:mora-est}{le:mora-est=eq} gives the Morawetz estimate corresponding to $\E_{2k+2}$:
  \begin{align*}
    \frac{\dd\mora_k}{\dd t}
    \le -\frac{b_1}{\la^{4k+4}} \sum_{\ell=0}^{k}\int\w_{2k+2}^0\cdot \R^{k+1}\mhamp^\ell\bigl(g_2\mhamp^{k-\ell}\wp\bigr) + \sum_{j=1}^{3} \QQ_{\mora_k,j} + \mF_{\mora_k},
  \end{align*}
  where $\QQ_{\mora_k,j}, \mF_{\mora_k}$ are defined by \eqref{le:mora-est=qq1}--\eqref{le:mora-est=f} with $L$ replaced by $k$. Putting this together with the energy equality of $\E_{2k+2}$ and using the scaling law~\eqref{eq:scaling-w-w}, we obtain
  \begin{equation}
    \begin{split}
      \label{pf:low-eng=expand}
      & \quad \frac{\dd}{\dd t} \biggl\{ \int|\W_{2k+2}^0|^2 + 2\mora_k \biggr\}                                                                                  \\
      & \le -\frac{2}{\la^{4k+4}}\biggl(b_1+\frac{\la_s}{\la}\biggr) \sum_{\ell=0}^{k}\int\w_{2k+2}^0\cdot \R^{k+1}\mhamp^\ell\bigl(g_2\mhamp^{k-\ell}\wp\bigr) \\
      & \quad + 2\sum_{j=1}^{3} \QQ_{\mora_k,j} +\frac{2}{\la^{4k+4}} \int\w_{2k+2}^0\cdot (\R\mham)^{k+1}\R^2\w_2^1                                            \\
      & \quad + 2\mF_{\mora_k} -\frac{2}{\la^{4k+4}} \int\w_{2k+2}^0\cdot (\R\mham)^{k+1}\mF^{\pp}.
    \end{split}
  \end{equation}
  We now estimate each part of this inequality in turn.

  \step[2, The energy functional]
  By the scaling law~\eqref{eq:scaling-w-w} and \eqref{eq:w2k0-e2k}, we have
  \[
    \int|\W_{2k+2}^0|^2  = \frac{\E_{2k+2}}{\la^{4k+2}}.
  \]
  We now show the smallness of $\mora_{k}$. Recall the definition in~\eqref{def:mora-func}. From the explicit expression of $g_1$ and Lemma~\itref{le:interp-wp}{it:interp-wp=sobolev}, \ref{it:interp-wp=lossy}, the first integral in $\mora_k$ is bounded by
  \begin{align*}
    \left|b_1\int\W_{2k}^{0}\cdot\R G_1\ma_{\la}\W_{2k}^0\right|
    & \les \frac{b_1}{\la^{4k+2}} \biggl\|\frac{\jbk[\log y]}{\jbk[y]^2}\w_{2k}^0\biggr\|_{L^2}
    \biggl\|\frac{\ma\w_{2k}^0}{\jbk[y]\jbk[\log y]}\biggr\|_{L^2}                               \\
    & \les \frac{1}{\la^{4k+2}}
    b_1^{1+(2k+1)\frac{2L}{2L-1}}\lgb^{C(K)},
  \end{align*}
  where we remark that Lemma~\itref{le:interp-wp}{it:interp-wp=lossy} used here provides different bounds in the cases $k\le L-2$ and $k=L-1$, but both of these serve our purpose.
  For the other integral, we note that it is non-vanishing only when $k\ge 1$. Thus for $k\ge 1$, we claim that
  \begin{equation}
    \label{pf:low-eng=rep-mham}
    \Bigl\|\mhamp^{\ell_1}\bigl(g_2\mhamp^{\ell_2}\wp\bigr)\Bigr\|_{L^2}^2
    \les b_1^{(2k+1)\frac{2L}{2L-1}}\lgb^{C(K)},
    \quad \ell_1+\ell_2 = k-1,k.
  \end{equation}
  This follows by repeating the proof of~\eqref{pf:mix-eng=rep-ham-claim}. Then from the self-adjointness of $\mhamp$,
  \begin{align*}
    \left| b_1\int\W_{2(2k-\ell)}^0\cdot\R G_2\W_{2\ell}^0 \right|
    & \les \frac{b_1}{\la^{4k+2}} \|\w_{2k+2}^0\|_{L^2} \Bigl\|\mhamp^{k-1-\ell} \bigl( g_2\mhamp^\ell\wp \bigr) \Bigr\|_{L^2} \\
    & \les \frac{1}{\la^{4k+2}}
    b_1^{1+(2k+1)\frac{2L}{2L-1}}\lgb^{C(K)}.
  \end{align*}
  Thus we obtain
  \begin{equation}
    \label{pf:low-eng=mora-bdd}
    \begin{split}
      |\mora_k| & \les \frac{1}{\la^{4k+2}}b_1^{1+(2k+1)\frac{2L}{2L-1}}\lgb^{C(K)} \\
      & \les \frac{1}{\la^{4k+2}}b_1^{\frac12+(2k+1)\frac{2L}{2L-1}}.
    \end{split}
  \end{equation}

  \step[3, The remaining quadratic terms] We estimate the first and second lines of \eqref{pf:low-eng=expand}. Using~\eqref{le:mod-eq-1},~\eqref{pro:boot-e2k-bnd} and \eqref{pf:low-eng=rep-mham} (or \ref{it:interp-wp=lossy} if $k=0$),
  \begin{equation}
    \begin{split}
      \label{pf:low-eng=qua-modu}
      & \quad \left|\frac{1}{\la^{4k+4}}\biggl(b_1+\frac{\la_s}{\la}\biggr) \sum_{\ell=0}^{k}\int\w_{2k+2}^0\cdot \R^{k+1}\mhamp^\ell\bigl(g_2\mhamp^{k-\ell}\wp\bigr)\right| \\
      & \les \frac{D(s)}{\la^{4k+4}} \|\w_{2k+2}^0\|_{L^2} \sum_{\ell=0}^{k} \Bigl\|\mhamp^{\ell} \bigl( g_2\mhamp^{k-\ell}\wp \bigr) \Bigr\|_{L^2}                           \\
      & \les \frac{1}{\la^{4k+4}} b_1^{2+(2k+1)\frac{2L}{2L-1}} \lgb^{C(K)}.
    \end{split}
  \end{equation}
  For $\QQ_{\mora_{k},j}$ ($j=1,2,3$), the estimate follows by the same argument as in Steps 1--3 of the proof of Proposition~\ref{pro:mix-eng}. We only demonstrate the details for $\QQ_{\mora_{k}, 1}$. By the modulation equation~\eqref{le:mod-eq-1} and Lemma~\itref{le:mora-est}{le:mora-est=bdd}, the first and third terms in $\QQ_{\mora_{k}, 1}$ are bounded by
  \begin{align*}
    & \Biggl|-C^{k}_{k}\prt b_1\int\W_{2k}^0\cdot \R G_1\ma_{\la}\W_{2k}^0
    +\sum_{\ell=0}^{k-1}C^{k}_{\ell}\prt b_1\int\W_{2(2k-\ell)}^0\cdot \R G_2\W_{2\ell}^0\Biggr| \\
    & \les \Bigl|\frac{\prt b_1}{b_1}\Bigr|\, |\mora_{k}|
    \les \frac{1}{\la^{4k+4}} b_1^{2+(2k+1)\frac{2L}{2L-1}} \lgb^{C(K)}.
  \end{align*}
  Note that by the definition in~\eqref{def:mora-func} the summation part above vanishes when $k=0$.
  For the second term, from Lemma~\ref{le:est-before-mix}, Lemma~\itref{le:interp-wp}{it:interp-wp=sobolev} and \eqref{pro:boot-e2k-bnd},
  \begin{align*}
    \left|b_1\int\W_{2k}^0\cdot\R\prt G_1\ma_{\la}\W_{2k}^0\right|
    & \les \frac{b_1^2}{\la^{4k+4}}
    \biggl\|\frac{\w_{2k}^0}{\jbk[y]^3}\biggr\|_{L^2}
    \biggl\|\frac{\ma\w_{2k}^0}{\jbk[y]^2}\biggr\|_{L^2}                 \\
    & \les \frac{1}{\la^{4k+4}} b_1^{2+(2k+1)\frac{2L}{2L-1}} \lgb^{C(K)}.
  \end{align*}
  For the fourth term (which vanishes when $k=0$), using the self-adjointness of $\mhamp$, and \eqref{pf:low-eng=rep-mham},
  \[
    \left|b_1\int\W_{2(2k-\ell)}^0\cdot\R\prt G_2\W_{2\ell}^0\right|
    \les \frac{b_1^2}{\la^{4k+4}} \|\w_{2k+2}^0\|_{L^2}
    \bigl\|\mhamp^{k-\ell-1}\bigl(g_4\mhamp^\ell\wp\bigr)\bigr\|_{L^2},
  \]
  where $g_4=4g_2+\vla g_2$ is defined in the proof of Lemma~\ref{le:est-before-mix}. For $y\le1$, from Corollary~\itref{co:dore}{it:dore=w}, Lemma~\itref{le:preserv}{it:preserv=op}, and the even series expansion of $g_4$, we have $\mhamp^{k-\ell-1}(g_4\mhamp^\ell\wp)\in\mcD_{L-k+2}$ with factor $b_1^{L+1}\lgb^{-1}$. Hence
  \[
    \Bigl\|\mhamp^{k-\ell-1}\bigl(g_4\mhamp^\ell\wp\bigr)\Bigr\|_{L^2(y\le 1)}^2
    \les C(M,K)\frac{b_1^{2L+2}}{\lgb^2}.
  \]
  For $y\ge1$, applying Lemma~\ref{le:omega-pw-bnd} twice and Lemma~\itref{le:interp-wp}{it:interp-wp=sobolev},
  \begin{align*}
    \Bigl\|\mhamp^{k-\ell-1}\bigl(g_4\mhamp^\ell\wp\bigr)\Bigr\|_{L^2(y\ge 1)}^2
    & \les \sum_{\ell=0}^{2k-2}\int_{y\ge 1}\frac{|\pry^{\ell}\wp|^2}{y^{2(2k+4-\ell)}} \\
    & \les C(M) b_1^{2+(2k+1)\frac{2L}{2L-1}} \lgb^{C(K)}.
  \end{align*}
  Inserting these back, we obtain
  \begin{equation}
    \label{pf:low-eng=qua-q-mora}
    |\QQ_{\mora_{k}, 1}|
    \les \frac{1}{\la^{4k+4}} b_1^{2+(2k+1)\frac{2L}{2L-1}} \lgb^{C(K)}.
  \end{equation}
  The same commutator expansion as in Step 2 of Proposition~\ref{pro:mix-eng} gives
  \[
    |\QQ_{\mora_k,2}|\les\frac{C(M,K)}{\la^{4k+4}}b_1^{2+(2k+1)\frac{2L}{2L-1}}\lgb^{C(K)}.
  \]
  For $\QQ_{\mora_k,3}$, we first integrate $\ma$ by parts and distribute the powers of $\mhamp$ between the two factors, as in Step 3 of that proof. Lemma~\itref{le:interp-wp}{it:interp-wp=taylor-orig}, \ref{it:interp-wp=gen-lossy}, \ref{it:interp-wp=pw}, and the definition of $\w_2^1$ give
  \[
    \left\|\frac{\mhamp^k\w_2^1}{\jbk[y]}\right\|_{L^2}^2
    +\left\|\mhamp^{k-j-1}(g_2\mhamp^j\w_2^1)\right\|_{L^2}^2
    \les C(M,K)b_1^{1+(2k+1)\frac{2L}{2L-1}}\lgb^{C(K)},
  \]
  where the second norm is included for $0\le j\le k-1$ only when $k\ge1$. Indeed, for $y\ge1$ the derivative orders and decay add up to at least $2k+3$; the general lossy bound gives $b_1^{(2k+2)\frac{2L}{2L-1}}$ for $k\le L-2$, and $b_1^{2L+1}$ for $k=L-1$. For $y\le1$, we use only derivatives up to order $2k+2\le2L$, so the Taylor bound applies. Pairing with the remaining radiation factor and keeping the prefactor $b_1$ therefore yields
  \[
    |\QQ_{\mora_k,3}|\les\frac{C(M,K)}{\la^{4k+4}}b_1^{\frac32+(2k+1)\frac{2L}{2L-1}}\lgb^{C(K)}.
  \]
  Now, we treat the integral involving $\w_2^1$. For $k=L-1$, we use~\eqref{pf:mix-eng=w21-claim} and the bootstrap bounds directly:
  \[
    \left|\int\w_{2L}^0\cdot(\R\mham)^L\R^2\w_2^1\right|
    \le\sqrt{\E_{2L}}\,\|\mhamp^L\w_2^1\|_{L^2}
    \les C(M,K)b_1^{2L+1}\lgb^{K/2-1}.
  \]
  It remains to consider $0\le k\le L-2$.
  Note that from the definition in~\eqref{def:w2k1}, we have
  \[
    \w_2^1 = \hw\times\mham\w + \R\mham(\gamma e_3).
  \]
  For $\hw\times\mham\w$, by Corollary~\itref{co:dore}{it:dore=w}, \ref{it:dore=hw} and Lemma~\itref{le:preserv}{it:preserv=op}--\ref{it:preserv=cross}, we have $\mhamp^{k+1}(\hw\times\mham\w)\in\mcD_{L-k-1}$ with factor $b_1^{L+2}\lgb^{-1}$. Since $k\le L-2$, we bound its squared $L^2$ norm on $y\le1$ by $C(M,K)b_1^{2L+4}\lgb^{-2}$. From this bound and~\eqref{pro:boot-e2k-bnd},
  \[
    \left| \int_{y\le 1}\w_{2k+2}^0\cdot(\R\mham)^{k+1}\R^2(\hw\times\mham\w) \right|
    \les C(M,K)b_1^{2+(2k+1)\frac{2L}{2L-1}} \lgb^{C(K)}.
  \]
  For $y\ge 1$, we split the integral according to $\hw = \wzt +\w$. The $\wzt$ part is controlled directly by a crude estimate using Lemma~\ref{le:omega-pw-bnd}, \eqref{pf:w2l2-e=wzt-l-infty}, and Lemma~\itref{le:interp-wp}{it:interp-wp=gen-lossy}. For the $\w$ part, we observe that
  \begin{align*}
    \mhamp^{k+1}\R^2(\w\times\mham\w)
    & = -\ham^{k+1} \Bigl\{ -\beta\Delta\gamma -\gamma\ham\beta +\beta\zt\Bigl( \pry+\tfrac{\z}{y} \Bigr)\alpha \Bigr\} e_1                              \\
    & \quad - \ham^{k+1} \Bigl\{ \gamma\ham\alpha +\alpha\Delta\gamma -\gamma\zt\pry\gamma -\alpha\zt\Bigl( \pry+\tfrac{\z}{y} \Bigr)\alpha \Bigr\} e_2,
  \end{align*}
  where the terms with undifferentiated $\gamma$ include $\gamma\ham^{k+2}\beta e_1-\gamma\ham^{k+2}\alpha e_2$. Since $k\le L-2$, their pairing with $\w_{2k+2}^0$ is bounded directly by
  \[
    \|\gamma\|_{L^\infty}\sqrt{\E_{2k+2}\E_{2k+4}}
    \les C(M,K)b_1^{(2k+2)\frac{2L}{2L-1}}\lgb^{C(K)}.
  \]
  For the terms containing derivatives of $\gamma$, we use its expression in the proof of Lemma~\ref{le:interp-gamma}. We control terms containing a derivative of order $2k+4$ of one factor of $\wp$ by $\E_{2k+4}$ as above, with the lower-order terms estimated by Lemma~\itref{le:interp-wp}{it:interp-wp=gen-lossy}. In all other products, we use derivatives up to order $2k+3\le2L-1$ of each factor; Leibniz and~\ref{it:interp-wp=pw}, with the derivative of lower order in $L^\infty$, give the same bound. Since $\frac{2L}{2L-1}>1$, the extra positive power of $b_1$ absorbs the $K$-dependent logarithmic loss.

  Recall that $\R\mham(\gamma e_3)=-\zt\pry\gamma\,e_2$. By Lemmas~\itref{le:interp-wp}{it:interp-wp=taylor-orig} and~\itref{le:preserv}{it:preserv=op}, we have $(\R\mham)^{k+2}(\gamma e_3)\in\mcD_{L-k-1}$ with factor $b_1^{L+2}\lgb^{-1}$. By this bound for $y\le1$ and Leibniz for $y\ge1$,
  \begin{align*}
    \|\ham^{k+1}(\zt\pry\gamma)\|_{L^2}^2
    &\les C(M,K)\frac{b_1^{2L+4}}{\lgb^2}
    +\sum_{i=0}^{2k+3}\int_{y\ge1}\frac{|\pry^i\gamma|^2}{y^{2(2k+5-i)}} \\
    &\les C(M,K)b_1^{3+(2k+1)\frac{2L}{2L-1}}\lgb^{C(K)}.
  \end{align*}
  Here $2k+5\le2L+1$, so the last inequality follows from Lemma~\itref{le:interp-wp}{it:interp-wp=gen-lossy}. Pairing with $\w_{2k+2}^0$ and using~\eqref{pro:boot-e2k-bnd} gives $C(M,K)b_1^{\frac32+(2k+1)\frac{2L}{2L-1}}\lgb^{C(K)}$. The extra factor $b_1^{1/2}$ absorbs the logarithmic loss. Combining these bounds with the case $k=L-1$ yields
  \begin{equation}
    \label{pf:low-eng=qua-w21}
    \left|\int \w_{2k+2}^0\cdot(\R\mham)^{k+1}\R^2\w_2^1\right|
    \les C(M,K)b_1^{1+(2k+1)\frac{2L}{2L-1}}\lgb^{K/2+C}.
  \end{equation}

  \step[4, The $\mF$ terms]
  Since $\R^{\ast} = -\R$, we have
  \begin{gather*}
    \int\mhamp^{k+1}\wp\cdot\R\mhamp\wp = 0,\\
    \int\R\w_{2k}^0\cdot\R g_1\ma\w_{2k}^0
    +\int\w_{2k}^0\cdot\R g_1\ma\R\w_{2k}^0=0,\\
    \int\R\w_{2(2k-j)}^0\cdot\R g_2\w_{2j}^0
    +\int\w_{2(2k-j)}^0\cdot\R g_2\R\w_{2j}^0=0.
  \end{gather*}
  In view of $\z = -1+\zt/2$, the leading part of $\vtha_s\z\R\w$ is just $-\vtha_s\R\w$, which is cancelled in $\mF_{\mora_{k}}$ and the pairing term of $\w_{2k+2}^0$ and $(\R\mham)^{k+1}\mF^{\pp}$. Thus it suffices to estimate \[
    \mFb = \mF + \vtha_s\R\wp
    = \psizt +\modt +\frac12 \vtha_s\zt\R\w +\w\times(\mham\wzt +\hp\lphi).
  \]
  We claim that
  \begin{equation}
    \begin{split}
      \label{pf:low-eng=mf-claim}
      & \bigl\|(\R\mham)^{k+1}\mFb^\pp\bigr\|_{L^2}^2
      +\bigl\|\mhamp^{k-j-1}\bigl(g_2\mhamp^j\mFb^\pp\bigr)\bigr\|_{L^2}^2       \\
      & +\biggl\|\frac{\jbk[\log y]}{\jbk[y]^2}\mhamp^k\mFb^\pp\biggr\|_{L^2}^2
      +\biggl\|\frac{\jbk[\log y]}{\jbk[y]}\ma\mhamp^k\mFb^\pp\biggr\|_{L^2}^2   \\
      & \quad \les C(M,K) b_1^{2+(2k+1)\frac{2L}{2L-1}}\lgb^{C}.
    \end{split}
  \end{equation}
  Here the second term is included for $0\le j\le k-1$ only when $k\ge1$. Recall the definition~\eqref{le:mora-est=f} for $\mora = \mora_{L}$, where we replace $L$ by $k$ to get the expression of $\mF_{\mora_k}$. The first term in $\mF_{\mora_k}$ is bounded using Lemma~\ref{le:coer-ti-ham}, \eqref{pf:low-eng=mf-claim}, and the bootstrap bound~\eqref{pro:boot-e2k-bnd},
  \begin{align*}
    \left|\frac{b_1}{\la^2}\int(\R\mhaml)^k\mFb_{\la}^\pp\cdot \R G_1\ma_\la\W_{2k}^0\right|
    & \les \frac{b_1}{\la^{4k+4}}
    \biggl\|\frac{\ma\w_{2k}^0}{\jbk[y]\jbk[\log y]}\biggr\|_{L^2} \biggl\|\frac{\jbk[\log y]}{\jbk[y]^2}\mhamp^k\mFb^\pp\biggr\|_{L^2} \\
    & \les \frac{C(M,K)}{\la^{4k+4}} b_1^{1+(2k+1)\frac{2L}{2L-1}}\lgb^{K/2+C}.
  \end{align*}
  For the second term,
  \begin{align*}
    \left|\frac{b_1}{\la^2}\int\W_{2k}^0\cdot \R G_1\ma_\la(\R\mhaml)^k\mFb_{\la}^\pp\right|
    & \les \frac{b_1}{\la^{4k+4}} \biggl\|\frac{\w_{2k}^0}{\jbk[y]^2\jbk[\log y]}\biggr\|_{L^2}
    \biggl\|\frac{\jbk[\log y]}{\jbk[y]}\ma\mhamp^k\mFb^\pp\biggr\|_{L^2}                           \\
    & \les \frac{C(M,K)}{\la^{4k+4}} b_1^{1+(2k+1)\frac{2L}{2L-1}}\lgb^{K/2+C}.
  \end{align*}
  Other terms in $\mF_{\mora_k}$ are treated similarly using~\eqref{pf:low-eng=rep-mham} and \eqref{pf:low-eng=mf-claim}. Here the prefactor $b_1$, together with the $b_1$ in the square root of~\eqref{pf:low-eng=mf-claim}, gives $b_1^{2+(2k+1)\frac{2L}{2L-1}}\lgb^{C(K)}$, which absorbs the logarithmic loss in~\eqref{pf:low-eng=rep-mham}. Thus we have
  \begin{equation}
    |\mF_{\mora_{k}}|
    \les \frac{C(M,K)}{\la^{4k+4}} b_1^{1+(2k+1)\frac{2L}{2L-1}}\lgb^{K/2+C}.
  \end{equation}
  The last term in~\eqref{pf:low-eng=expand} is treated using \eqref{pf:low-eng=mf-claim}, the identity $\|\w_{2k+2}^0\|_{L^2}=\sqrt{\E_{2k+2}}$, and the bootstrap bound~\eqref{pro:boot-e2k-bnd}. This gives the logarithmic factor $\lgb^{K/2+C}$. Therefore,
  \[
    |\mF_{\mora_{k}}| + \left|\frac{1}{\la^{4k+4}}\int\w_{2k+2}^0\cdot (\R\mham)^{k+1}\mFb^{\pp}\right|
    \les \frac{C(M,K)}{\la^{4k+4}} b_1^{1+(2k+1)\frac{2L}{2L-1}}\lgb^{K/2+C}.
  \]
  Inserting \eqref{pf:low-eng=mora-bdd}, \eqref{pf:low-eng=qua-modu}, the bounds for $\QQ_{\mora_k,j}$, \eqref{pf:low-eng=qua-w21}, and the above inequality into~\eqref{pf:low-eng=expand} yields the desired energy estimate for $\E_{2k+2}$. The extra powers of $b_1$ absorb the logarithmic losses in the quadratic terms. It remains to show~\eqref{pf:low-eng=mf-claim}. The contribution of $\psizt$ follows from Lemma~\itref{le:local-app-sol}{it:local-app-sol=sob} and \ref{it:local-app-sol=weig}. We estimate the other three terms below.

  \step[5, Contribution of $\modt$]
  Recall from~\eqref{le:local-app-sol=modt} that
  \begin{align*}
    \modt & =
    \sum_{k=1}^{L}(\prs a_k +A_k)\biggl(\phiat[k] +\chi_{S_1}\biggl(\frac{\pr\x}{\pr a_k} +\frac{\pr\gamma_0}{\pr a_k}e_3\biggr)\biggr)                                                                 \\
    & \quad +\sum_{k=1}^{L}(\prs b_k +B_k)\biggl(\phibt[k] +\chi_{S_1}\biggl(\frac{\pr\x}{\pr b_k} +\frac{\pr\gamma_0}{\pr b_k}e_3\biggr) + \delta_{1,k}\frac{\pr\chi_{S_1}}{\pr b_1}\wz \biggr) \\
    & \quad +(\vtha_s+a_{1})\bigl(\lphi\bigl(\etwo -\eone\times\wzt \bigr) +\z\R\wzt \bigr)                                                                                                       \\
    & \quad -\Bigl(\frac{\la_s}{\la} +b_1\Bigr)\bigl(\lphi\bigl(\eone +\etwo\times\wzt\bigr) +\vla\wzt \bigr).
  \end{align*}
  From the property of $\mcA$,
  \[
    \int|(\R\mham)^{k+1}\phiat[j]|^2
    \les 1 +\int_{1}^{2S_1} \biggl|\frac{y^{2j-1}\jbk[\log y]}{y^{2k+2}}\biggr|^2
    \les 1+b_1^{2(k+1-j)}\lgb^C.
  \]
  From the rough bound~\eqref{eq:badb-rough-bnd},
  \begin{align*}
    \int\biggl|(\R\mham)^{k+1}\biggl[ \chi_{S_1} \biggl(\frac{\pr\x}{\pr a_j}+ \frac{\pr\gamma_0}{\pr a_j} e_3 \biggr)\biggr]\biggr|^2
    & \les \sum_{\ell=j+1}^{L+2} b_1^{2(\ell-j)} \Biggl(1+\int_{1}^{2S_1}\biggl|\frac{y^{2\ell-1}}{y^{2k+2}}\biggr|^2\Biggr) \\
    & \les b_1^2 + b_1^{2(k+1-j)}\lgb^C.
  \end{align*}
  The counterpart involving the $b_j$ derivative yields the same bound. Recall that the additional cutoff term admits
  \[
    b_1\frac{\pr\chi_{S_1}}{\pr b_1}
    =\Bigl(\frac12+\frac1{\lgb}\Bigr)\vla\chi_{S_1}
    = O(1)\one_{y\sim S_1}.
  \]
  The profile estimates then give
  \[
    \biggl\|(\R\mham)^{k+1}\biggl(\frac{\pr\chi_{S_1}}{\pr b_1}\wz\biggr)\biggr\|_{L^2}^2
    \les 1+b_1^{2k}\lgb^C,
  \]
  which is the $j=1$ bound above. Moreover, since $\ham\lphi = 0$,
  \[
    \int|(\R\mham)^{k+1}(\lphi\, e_i)|^2 = 0
    \for i=1,2.
  \]
  From the explicit construction of $\wzt$ given by Lemma~\ref{le:app-sol} and \ref{le:local-app-sol},
  \begin{align*}
    & \sum_{i=1}^2\bigl\|(\R\mham)^{k+1}\bigl(\lphi(e_i\times\wzt)\bigr)\bigr\|_{L^2}^2 \\
    & +\|(\R\mham)^{k+1}(\z\R\wzt)\|_{L^2}^2
    +\|(\R\mham)^{k+1}(\vla\wzt)\|_{L^2}^2
    \les \lgb^C,
  \end{align*}
  which is again covered by the $j=1$ term.
  These together with~\eqref{pf:mod-eq=d} and the bootstrap bound~\eqref{pro:boot-e2l2-bnd} give
  \begin{align*}
    \bigl\|(\R\mham)^{k+1}\modt\bigr\|_{L^2}^2
    & \les \sum_{j=1}^L |D(s)|^2 \bigl(1+b_1^{2(k+1-j)}\lgb^C\bigr) \\
    & \les b_1^{2+(2k+1)\frac{2L}{2L-1}}\lgb^C,
  \end{align*}
  where we use the inequality:
  \[
    2k+4 \ge 2+ (2k+1)\frac{2L}{2L-1}
    \for k\le L-1.
  \]

  \step[6, Contribution of $\vtha_s\zt\R\w$]
  We again split the integral domain into the regions $y\le1$ and $y\ge1$ and estimate the two parts using Lemmas~\ref{le:preserv} and~\ref{le:omega-pw-bnd}, respectively.
  From \eqref{le:mod-eq-1},
  \[
    |(\R\mham)^{k+1}(\vtha_s \zt\R\w)|
    \les (|a_1| + D(s)) \sum_{\sigma}|\Omega^{k+1}_\sigma (\zt\R\w)|.
  \]
  For $y\le1$, by Corollary~\itref{co:dore}{it:dore=w} and Lemma~\itref{le:preserv}{it:preserv=mult}--\ref{it:preserv=op}, we have $(\R\mham)^{k+1}(\zt\R\w)\in\mcD_{L-k}$ with factor $b_1^{L+1}\lgb^{-1}$. Since $k\le L-1$, we have
  \begin{align*}
    \|(\R\mham)^{k+1}(\vtha_s \zt\R\w)\|_{L^2(y\le 1)}^2
    & \les C(M,K)(a_1^2 + D(s)^2)b_1^{2L+2}\lgb^{-2} \\
    & \les C(M,K)\frac{b_1^{2L+4}}{\lgb^2}.
  \end{align*}
  For $y\ge1$, we use Lemma~\itref{le:interp-wp}{it:interp-wp=gen-lossy}. Since $\ell+(2k+4-\ell)=2k+4$, and $a_1^2+D(s)^2\les b_1^2$, we obtain
  \begin{align*}
    & \|(\R\mham)^{k+1}(\vtha_s \zt\R\w)\|_{L^2(y\ge 1)}^2 \\
    & \les (a_1^2 + D(s)^2) \sum_{\ell=0}^{2k+2} \int_{y\ge 1}\frac{|\pry^{\ell}\wp|^2}{y^{2(2k+4-\ell)}} \\
    & \les C(M,K)\lgb^{C(K)}\begin{cases}
      b_1^{2+(2k+3)\frac{2L}{2L-1}},&0\le k\le L-2,\\
      b_1^{2L+4},&k=L-1,
    \end{cases}\\
    & \les C(M,K)b_1^{4+(2k+1)\frac{2L}{2L-1}}\lgb^{C(K)}.
  \end{align*}
  This yields the first estimate of~\eqref{pf:low-eng=mf-claim} for $\vtha_s\zt\R\w$.
  The remaining three squared norms in~\eqref{pf:low-eng=mf-claim} have the same total order $2k+4$, counting the decay of $\zt$ and the weights. Lemma~\itref{le:interp-wp}{it:interp-wp=gen-lossy} gives the same bound, including their logarithmic weights. The extra factor $b_1^2$ absorbs $\lgb^{C(K)}$.

  \step[7, Contribution of $\w\times(\mham\wzt+\hp\lphi)$]
  Recalling the decomposition~\eqref{pf:mix-eng=wzto-wztt}, we have $\wzt = \wzto + \wztt$ and the cancellation
  \begin{align*}
    \mham\wzto+\hp\lphi
    & = O(D(s))\lphi e_1 +O(D(s))\lphi e_2 + a_1\zt\Bigl(\pry+\frac{\z}{y}\Bigr)(\chi_{S_1}T_1)e_3 \\
    & \quad +(\chi_{S_1}-1)\mhamp\wz^1 -2\pry\chi_{S_1}\pry\wz^1 -\Delta\chi_{S_1}\wz^1.
  \end{align*}
  Note on the right-hand side that the second line vanishes in the region $y\le S_1$.
  For $y\le 1$, from the definition in~\eqref{def:omega},
  \begin{align*}
    & |(\R\mham)^{k+1}(\w\times(\mham\wzt+\hp\lphi))|    \\
    & \les D(s) |\Omega^{k+1}_{(0,\dots,0)}(\w\times\lphi e_1)| +D(s) |\Omega^{k+1}_{(0,\dots,0)}(\w\times\lphi e_2)| \\
    & \quad + |a_1|\,\Bigl|\Omega^{k+1}_{(0,\dots,0)}\Bigl[\w\times\zt\Bigl(\pry+\frac{\z}{y}\Bigr)(\chi_{S_1}T_1)e_3\Bigr]\Bigr|
    +|\Omega^{k+1}_{(0,\dots,0)}(\w\times\mham\wztt)|.
  \end{align*}
  The first line belongs to $\mcD_{L+1}$ with factor $D(s)+|a_1|$, while $\wztt\in\mcD_{L+2}$ with factor $b_1^2$. By Corollary~\itref{co:dore}{it:dore=w} and the smoothness of these profile factors, the cross products with $\w$ belong to $\mcD_{L+1}$ with factor $(D(s)+|a_1|+b_1^2)b_1^{L+1}\lgb^{-1}$, as in the proof of Lemma~\itref{le:preserv}{it:preserv=cross}. By Lemma~\itref{le:preserv}{it:preserv=op}, after $k+1$ applications of $\R\mham$, we have functions in $\mcD_{L-k}$ with the same factor, and hence
  \begin{align*}
    & \bigl\|(\R\mham)^{k+1}(\w\times(\mham\wzt+\hp\lphi))\bigr\|_{L^2(y\le 1)}^2 \\
    & \les C(M,K)(D(s)^2+a_1^2+b_1^4)\frac{b_1^{2L+2}}{\lgb^2}
    \les b_1^{2+(2k+1)\frac{2L}{2L-1}}.
  \end{align*}
  For $y\ge1$, the term with $a_1$ has the same decay as in Step 6, up to powers of $\jbk[\log y]$, and obeys the same bound. The term with $D(s)$ is smaller by~\eqref{pf:mod-eq=d}. For the remaining terms,~\eqref{pf:mix-eng=wzt-hp-lphi} and Leibniz bound each of the four squared norms in~\eqref{pf:low-eng=mf-claim} by
  \begin{align*}
    & Cb_1^3\lgb^C\sum_{i=0}^{2k+2}\int_{1\le y\le2S_1}\frac{\jbk[\log y]^C|\pry^i\w|^2}{y^{2(2k+2-i)}}
    +Cb_1^2\sum_{i=0}^{2k+2}\int_{y\ge S_1}\frac{\jbk[\log y]^C|\pry^i\w|^2}{y^{2(2k+3-i)}} \\
    & \qquad \les C(M,K)b_1^{3+(2k+1)\frac{2L}{2L-1}}\lgb^{C(K)}.
  \end{align*}
  Indeed, Lemma~\itref{le:interp-wp}{it:interp-wp=gen-lossy} applies with $i+j=2k+2$ and $2k+3$, respectively. For $i=2k+2$ in the first sum, its logarithmically weighted bound gives
  \[
    \int_{1\le y\le2S_1}\jbk[\log y]^C|\pry^{2k+2}\w|^2
    \les\lgb^{C+2}\int\frac{|\pry^{2k+2}\w|^2}{\jbk[\log y]^2}
    \les b_1^{(2k+1)\frac{2L}{2L-1}}\lgb^{C(K)},
  \]
  since $\jbk[\log y]\les\lgb$ on this interval. The extra factor $b_1$ absorbs $\lgb^{C(K)}$.
  This establishes~\eqref{pf:low-eng=mf-claim} for $\w\times(\mham\wzt+\hp\lphi)$, and hence completes the proof of Lemma~\ref{le:mix-low-eng}.
\end{proof}

\section*{Statements and declarations}
\noindent\textbf{Data Availability declaration.}
Data sharing is not applicable to this article as no datasets were generated or analyzed during the current study.

\medskip
\noindent\textbf{Funding declaration.}
J. J. was supported by ERC grant 101117126 (INSOLIT). L. Z. was supported by the National Natural Science Foundation of China under Grants 12341102 and 12271497.

\bibliographystyle{amsplain}
\bibliography{bib2609}

\end{document}